\documentclass[twoside,11pt]{article}

\usepackage{blindtext}  

\usepackage[preprint]{jmlr2e}

\usepackage{placeins}
\usepackage{booktabs}
\usepackage{multirow}

\usepackage{lastpage}

\ShortHeadings{A Stochastic GDA Method With Backtracking}{A Stochastic GDA Method With Backtracking}
\firstpageno{1}

\usepackage[margin=1in]{geometry}
\usepackage[utf8]{inputenc}
\usepackage{algorithm}
\usepackage{algpseudocode}
\usepackage{gensymb}
\usepackage{amsmath}
\usepackage{bbold}
\usepackage{amssymb}
\usepackage{caption}
\usepackage{subcaption}
\usepackage{bbding}
\usepackage{dsfont}
\usepackage{todonotes}
\usepackage{thmtools}
\usepackage{cleveref}
\usepackage{hyperref}
\hypersetup{
    colorlinks=true,
    citecolor=blue,
    linkcolor=blue,
    filecolor=magenta,      
    urlcolor=cyan,
    pdfpagemode=FullScreen,
    }
\usepackage{enumitem}
\setlist[enumerate]{itemsep=1pt}

\newtheorem{assumption}{Assumption}

\usepackage{array}
\usepackage{graphicx}

\makeatletter
\newcommand{\thickhline}{%
    \noalign {\ifnum 0=`}\fi \hrule height 1pt
    \futurelet \reserved@a \@xhline
}
\newcolumntype{"}{@{\hskip\tabcolsep\vrule width 1pt\hskip\tabcolsep}}
\makeatother

\usepackage[font=scriptsize]{caption}
\usepackage{mathtools}

\def\grad{\nabla}

\def\ba{\mathbf{a}}

\def\bw{\mathbf{w}}
\def\bx{\mathbf{x}}  
\def\by{\mathbf{y}}
\def\bz{\mathbf{z}}

\def\bE{\mathbf{E}}

\def\bI{\mathbf{I}}

\def\bP{\mathbf{P}}

\def\p{{\boldsymbol{\pi}}}

\def\cA{\mathcal{A}}

\def\cD{\mathcal{D}}
\def\cE{\mathcal{E}}
\def\cF{\mathcal{F}}
\def\cG{\mathcal{G}}

\def\cL{\mathcal{L}}
\def\cM{\mathcal{M}}
\def\cN{\mathcal{N}}
\def\cO{\mathcal{O}}
\def\cP{\mathcal{P}}

\def\cR{\mathcal{R}}

\def\cU{\mathcal{U}}

\def\cX{\mathcal{X}}
\def\cY{\mathcal{Y}}

\def\smskip{\smallskip}

\def\texitem#1{\par\smskip\noindent\hangindent 25pt
               \hbox to 25pt {\hss #1 ~}\ignorespaces}

\def\norm#1{\|#1\|}

\newcommand{\BEAS}{\begin{eqnarray*}}
\newcommand{\EEAS}{\end{eqnarray*}}
\newcommand{\BEA}{\begin{eqnarray}}
\newcommand{\EEA}{\end{eqnarray}}
\newcommand{\BEQ}{\begin{eqnarray}}
\newcommand{\EEQ}{\end{eqnarray}}
\newcommand{\BIT}{\begin{itemize}}
\newcommand{\EIT}{\end{itemize}}
\newcommand{\BNUM}{\begin{enumerate}}
\newcommand{\ENUM}{\end{enumerate}}

\newcommand{\BA}{\begin{array}}
\newcommand{\EA}{\end{array}}

\newcommand{\reals}{\mathbb{R}}
\newcommand{\integers}{\mathbb{Z}}

\newcommand{\Rank}{\mathop{\bf rank}}
\newcommand{\Tr}{\mathop{\bf Tr}}

\def\fprod#1{\left\langle#1\right\rangle}
\def\prox#1{\mathbf{prox}_{#1}}
\DeclareMathOperator*{\argmax}{\mathbf{argmax}}
\DeclareMathOperator*{\argmin}{\mathbf{argmin}}

\newcommand{\dom}{\mathop{\bf dom}}

\usepackage{xcolor}

\newif\ifpagenumbering
\pagenumberingtrue

\pagenumberingfalse
\newsavebox{\theorembox}
\newsavebox{\lemmabox}
\newsavebox{\defnbox}
\newsavebox{\assbox}
\savebox{\theorembox}{\noindent\bf Theorem}
\savebox{\lemmabox}{\noindent\bf Lemma}
\savebox{\defnbox}{\noindent\bf Definition}
\newtheorem{defn}{\usebox{\defnbox}}

\usepackage{soul}
\usepackage[normalem]{ulem}
\DeclareUnicodeCharacter{2212}{~}

\newcommand{\sgdab}{\texttt{SGDA-B}{}}
\newcommand{\rbsgda}{\texttt{RB-SGDA}{}}
\newcommand{\rbagda}{\texttt{RB-SAGDA}{}}
\newcommand{\agda}{\texttt{AGDA}{}}
\newcommand{\gda}{\texttt{GDA}{}}
\newcommand{\smagda}{\texttt{sm-AGDA}{}}

\newcommand{\tiada}{\texttt{TiAda}{}}
\newcommand{\neada}{\texttt{NeAda}{}}
\newcommand{\vrlm}{\texttt{VRLM}{}}

\usepackage{hyperref}
\usepackage{cleveref}
\crefname{assumption}{Assumption}{Assumptions}
\crefname{theorem}{Theorem}{Theorems}
\crefname{lemma}{Lemma}{Lemmas}

\usepackage{relsize}

\def\xz#1{\textcolor{black}{#1}}
\def\sa#1{\textcolor{black}{#1}}
\def\rv#1{\textcolor{black}{#1}}
\def\saa#1{\textcolor{black}{#1}}
\def\xqs#1{\textcolor{black}{#1}}

\def\mg#1{\textcolor{black}{#1}}
\DeclareMathOperator{\gt}{\widetilde\nabla}
\DeclareMathOperator{\bom}{\boldsymbol{\omega}}
\DeclareMathOperator{\bzt}{\boldsymbol{\zeta}}
\DeclareMathOperator{\bxi}{\boldsymbol{\xi}}

\def\Ek#1{\mathbf{E}\Big[#1~\Big|~\cF^{k}\Big]}
\def\E#1{\mathbf{E}\left[#1\right]}
\def\p#1{\mathbf{P}\left(#1\right)}

\def\tz{\tilde\bz_{(\ell)}}
\def\tx{\tilde\bx_{(\ell)}}
\def\ty{\tilde y_{(\ell)}}
\def\txi{\tilde\bxi_{(\ell)}}

\definecolor{darkgreen}{rgb}{0.2,0.8,0.1}
\def\mg#1{\textcolor{black}{#1}}

\def\mgb#1{\textcolor{black}{#1}}
\def\mgrev#1{\textcolor{black}{#1}}

\begin{document}

\title{\vspace*{-7mm} A Stochastic GDA Method With Backtracking For Solving Nonconvex 
Concave Minimax Problems}

\author{\name Necdet Serhat Aybat \email nsa10@psu.edu \\
       \addr Department of Industrial \& Manufacturing Engineering\\
       Penn State University\\
       University Park, PA 16802, USA
       \AND
       \name Qiushui Xu \email qjx5019@psu.edu \\
       \addr Department of Industrial \& Manufacturing Engineering\\
       Penn State University\\
       University Park, PA 16802, USA
       \AND
       \name Xuan Zhang \email xxz358@psu.edu \\
       \addr Department of Industrial \& Manufacturing Engineering\\
       Penn State University\\
       University Park, PA 16802, USA
       \AND
       \name Mert G\"urb\"uzbalaban \email mg1366@rutgers.edu \\
       \addr Department of Management Science and Information Systems\\
       Rutgers University\\
       Piscataway, NJ 08854, USA}

\editor{}
\maketitle
\vspace*{-2mm}
\begin{abstract} 
We propose a stochastic GDA (\mg{g}radient \mg{d}escent \mg{a}scent) method with backtracking (\sgdab) to solve nonconvex-
concave~\sa{(NCC)}
minimax problems of the form: $\min_{\bx} \max_y \sum_{i=1}^N \sa{g_i(x_i)}+f(\bx,y)-h(y)$, where $h$ and $g_i$ for $i=1,\cdots,N$ are closed, convex functions, and \sa{for some $L,\mu\geq 0$}, $f$ is $L$-smooth and $f(\bx,\cdot)$ is $\mu$-strongly concave for all $\bx$ in the problem domain.  
We consider 
the stochastic setting where 
one only has an access to an unbiased stochastic oracle of $\nabla f$ with a finite variance \rv{bound $\sigma^2$}. \rv{While most of the existing methods assume knowledge of 
$L$, 
\rv{$\mu$ and/or $\sigma^2$}, 
\sgdab{} is agnostic to 
all of these problem parameters.} Moreover, \sgdab~\mg{can support random block-coordinate updates}.
\rv{In the deterministic setting, i.e., $\sigma^2=0$ and one can compute $\nabla f$ exactly, \sgdab{} can compute an $\epsilon$-stationary point within $\cO(L\kappa^2/\epsilon^2)$ and $\cO(L^3/\epsilon^4)$ gradient calls when $\mu>0$ and $\mu=0$, respectively, where $\kappa\triangleq L/\mu$.} \sa{In the stochastic setting, i.e., $\sigma^2>0$, for any $p\in(0,1)$ and $\epsilon>0$,} 
it can compute an $\epsilon$-stationary point \saa{with high probability, which requires} $\cO(L \kappa^3  \epsilon^{-4} \log^2(1/p))$ and \sa{$\tilde\cO(L^4\epsilon^{-7}\log^2(1/p))$ stochastic oracle calls, \saa{with probability at least $1-p$,} when $\mu>0$ and $\mu=0$, respectively}. To our knowledge, \sgdab{} is the first GDA-type method with backtracking to solve \sa{NCC} minimax problems \mg{and achieves the best complexity among the methods that are agnostic to \rv{$L$, $\mu$ and $\sigma^2$}.} We \mg{also} provide numerical results for \sgdab{} on a distributionally robust learning problem \mg{illustrating the potential performance gains that can be 
\saa{achieved} by \sgdab{}.}
\end{abstract}

\vspace*{-2mm}
\section{Introduction}
\vspace*{-2mm}
\sa{Let $\mathcal{X}_i=\mathbb{R}^{n_i}$ for $i\in \mathcal{N} \triangleq\{1,2, \ldots, N\}$ and \rv{$\mathcal{Y}=\mathbb{R}^{n_y}$} be finite dimensional Euclidean spaces, and define} $\mathbf{x}\triangleq\left[x_i\right]_{i \in \mathcal{N}} \in \Pi_{i \in \mathcal{N}} \mathcal{X}_i \triangleq \mathcal{X}=\mathbb{R}^{n_x}$ where $\rv{n_x} \triangleq \sum_{i \in \mathcal{N}} n_i$. \sa{Optimization problems with \mg{such} block-coordinate structure appear in many important applications including, training support vector machines~\citep{chang2008coordinate}, \mg{layer-wise training of deep learning models}~\citep{layer_wise2021}, compressed sensing~\citep{li2009coordinate}, regularized regression~\citep{wu2008coordinate} \mg{and} truss topology design~\citep{richtarik2012efficient}.}  In this paper, we study the following class of \emph{non-convex minimax} problems: \vspace*{-1mm}
\begin{equation}
\begin{aligned}\label{eq:main-problem}
\min_{\mathbf{x} \in \mathcal{X}} \max _{y \in \mathcal{Y}} \mathcal{L}(\mathbf{x}, y) \triangleq 
g(\bx)+f(\mathbf{x}, y)-h(y) \quad \mbox{with} \quad g(\bx)\triangleq \sum_{i\in\cN}g_i(x_i),
\end{aligned}
\vspace*{-1mm}
\end{equation}
where $f:\cX\times\cY\to \mathbb{R}$ is differentiable with a Lipschitz gradient, $f$ is \sa{possibly} non-convex in $\bx$ and (strongly) concave in $y$ with modulus $\mu\geq 0$, $h:\cY\to\reals\cup\{+\infty\}$ and $g_i:\cX_i\to\reals\cup\{+\infty\}$ for $i\in\cN$ are closed convex functions. \sa{We consider two scenarios: \mg{$(i)$} \textit{deterministic setting}, where we assume that one can compute $\grad f$ exactly, \mg{and} \mg{$(ii)$} \textit{stochastic setting}, where one can only access to stochastic estimates of $\grad f$, see \cref{assumption:noise} for details.} If $\mu>0$, then we call the \sa{minimax problem in~\eqref{eq:main-problem}} as \textit{weakly convex-strongly concave} (WCSC), whereas \sa{for} $\mu=0$, we call it as \textit{weakly convex-merely concave}~(WCMC). Both problems arise frequently in \mg{many} applications including \textit{constrained optimization} of weakly-convex objectives based on Lagrangian duality \citep{li2021augmented}, Generative Adversarial Networks
~\citep{goodfellow2014generative}
\mg{and} distributional robust learning with weakly convex 
loss functions such as those arising in deep learning \citep{gurbuzbalaban2020stochastic,rafique1810non}. 
These applications admitting the formulation in~\eqref{eq:main-problem} often require (\textit{i}) handling non-convexity in high dimensions and (\textit{ii}) exploiting the block structure in the primal variable. \sa{Indeed, for these machine learning~(ML) problems, the setting with larger model size compared to the number of data points 
has 
\mg{attracted significant active research, where it has been commonly observed that larger models can often generalize better, i.e., they perform better on unseen data 
~\citep{nakkiran2021deep}.}} 
\sa{In addition to \mg{the above-mentioned} 
examples, another particular example for \eqref{eq:main-problem} with a block-structure would be the \textit{distributed computation} setting for WCSC/WCMC problems~\citep{zhang2023jointly}, i.e., $\min_{x\in\reals^d}\max_{y\in\cY}\sum_{i\in\cN} f_i(x,y)$ over a network of processing nodes represented by a connected undirected graph $\cG=(\cN,\cE)$ with $\cN$ denoting the set of computing nodes and $\cE \mg{\subseteq} \cN\times\cN$ denoting the edges of $\cG$ through which incident nodes can 
\saa{communicate} with each other; in case that synchronous parallel processing is available, one can reformulate the problem as $\min_{\bx\in\cX}\{\max_{y\in\cY}\sum_{i\in\cN} f_i(x_i,y):\ x_i=x_j~\forall~(i,j)\in\cE\}$, i.e., $\cX_i=\reals^d$ for all $i\in\cN$; hence, $n_x=|\cN| d$. 
}

\textbf{\mg{\textit{Our goals:}}
} The use of first-order primal-dual~(FOPD) methods have turned out to be an effective approach to tackle with \eqref{eq:main-problem}. Majority of 
existing 
    FOPD methods require the knowledge of \textit{global Lipschitz constant} for $\grad f$, \rv{\textit{concavity modulus} related to $f$ and also \textit{a variance bound} for the first-order stochastic oracle}; these are relatively standard assumptions when deriving rate statements for first-order~(FO) schemes. That being said, in many practical settings either these constants may not be readily available or using the global constants generally leads to conservative updates, e.g., \rv{smaller step sizes and/or large sample sizes,} 
    \sa{slowing down the} convergence in practice. One avenue for removing such a requirement lies in adopting 
    step-size search or backtracking schemes to exploit \textit{local} Lipschitz constants \rv{and concavity modulus} when determining the step sizes. For deterministic 
    \sa{convex-concave minimax} problems, 
    \sa{backtracking} methods that adjust the step size 
    \sa{adaptively 
    through estimating local Lipschitz constants have proven to be very effective both in theory\footnote{\sa{These methods relying on backtracking achieve the same complexity guarantees (up to $\cO(1)$ constant) with those requiring the availability of global Lipschitz constants.}} and practice, e.g., see~\citep{malitsky2018first,malitsky2018proximal,hamedani2018_RBC,jalilzadeh2019doubly,hamedani2021primal,jiang2022generalized}}. However, for \emph{stochastic non-convex} problems such as \eqref{eq:main-problem}, we are not aware of any \textit{backtracking}/\textit{step-size search} strategy with iteration complexity guarantees. \sa{In this paper, our aim is to fill this gap \rv{even if the variance bound $\sigma^2$ is unknown,} and develop a backtracking technique 
    to be incorporated into a \textit{single-loop stochastic FOPD method}\footnote{It is more appropriate to adopt single-loop algorithms
    as subroutines while solving large-scale problems efficiently --usually in methods with nested loops, inner iterations are terminated when a sufficient optimality condition holds and these conditions are usually very conservative, leading to excessive number of inner iterations. Indeed, single loop algorithms are preferable compared to their multi-loop 
    alternatives in many settings, e.g., see~\citep{zhang2020single} for a discussion.} for the nonconvex minimax problem in \eqref{eq:main-problem}.} 
    
    Our secondary goal is to allow 
    for \mg{block-}coordinate updates within the backtracking scheme we design. Indeed, there are some practical scenarios for which adopting randomized block coordinate updates~\sa{(RBC)} would be beneficial, \mg{which we discuss next. 
    } 
    

\sa{\textit{Scenario I:} \saa{Consider the setting with} a very large primal dimension, i.e., $n_x\gg 1$. 
\saa{Suppose that} while it is impractical to 
compute $\grad_\bx f$ at every iteration,
the problem has a \emph{coordinate-friendly} structure~\citep{peng2016coordinate}, i.e., {for any $i\in\cN$,} the
{amount of work}
to compute the partial-gradient $\grad_{x_i} f$ 
{is}
${n_i/n\approx} 1/N$ fraction of the {work required for} 
$\grad_{\bx} f$ 
computation.} In this setup, randomized block updates will lead to low per-iteration cost and possibly small memory overhead, \mg{e.g., \citep{nesterov2012efficiency,hamedani2018_RBC,fercoq-bianchi,xiao2019dscovr}.} 

\sa{\textit{Scenario II:} In the nonconvex regime, even if the problem is not coordinate friendly, adopting random block coordinate updates might still be beneficial as it improves the generalization power of particular machine learning models, e.g., layer-by-layer training of deep neural networks is commonly used in practice for this purpose, e.g., see~\citep{nakamura2021block}.} In this setup, randomized block updates can help with generalization in deep learning, even though \sa{the per-iteration complexity} may not \mg{necessarily} improve.

    \rv{Therefore, our objective in this paper is to design an efficient stochastic first-order primal-dual method with \emph{randomized} block-coordinate updates that can generate an $\epsilon$-stationary point (see Definition~\ref{def:eps-stationarity}) of the structured non-convex minimax problem {given} in~\eqref{eq:main-problem} with high probability and without requiring the {knowledge of the} Lipschitz constant $L$, concavity modulus $\mu$ and/or the variance bound $\sigma^2$ for the stochastic oracle, \mg{which are typically unknown a priori}. 
    } 


    \textbf{\textit{Challenges:}} In contrast to the convex setting, \sa{i.e., 
    (strongly) convex-(strongly) concave minimax problems,} the \textit{nonconvexity} 
    introduces 
    significant extra challenges in designing a reasonable \textit{backtracking} 
    \sa{\textit{condition} 
    --\saa{which is employed to decide whether} a candidate step size should be accepted or rejected through checking the condition.} 
    Excluding the \textit{multi-loop} methods such as~\citep{nouiehed2019solving,lin2020near,ostrovskii2021efficient,kong2021accelerated}, in {the} non-convex setting the convergence guarantees for 
    all \sa{\textit{single-loop}} FO methods that we are aware of, except for very few, e.g., \citep{lu2020hybrid,zhang2020single,yang2022faster,xu2023unified}, are provided 
    in terms of the partial sum of gradient norms corresponding to the primal function or 
    \sa{its} Moreau envelope 
    over the past iterations, e.g., \citep{lin2020gradient,boct2020alternating,sebbouh2022randomized,huang2021efficient,chen2022accelerated}, rather than having guarantees \sa{involving the gradients of the coupling function $f$ evaluated at} the last iterate or at an ergodic average. 
    \sa{Clearly, computing the gradient of the \textit{primal function} or of the \textit{Moreau envelope} at each iteration to check a backtracking condition would be impractical \rv{even} in the deterministic setting, {as the former is typically not easy to compute exactly and it requires solving an optimization problem while the latter one requires evaluating the proximal 
    map of the primal function \citep{chen2022accelerated}.} 
    More importantly, this computation would be impossible for the stochastic setting as we can only access to the stochastic estimates of $\grad f$ at a given point. Thus, designing a backtracking condition that would lead to the state-of-the-art
    complexity guarantees for WCSC and WCMC problems
    and that is \textit{easy} to check in practice
    requires developing new analysis techniques (different from those used for the convex setting as they do not extend to the non-convex setting) to provide guarantees for quantities that are readily available such as the \textit{stochastic gradient map} involving the coupling function --see Definition~\ref{def:s-gradmapping} and note that the computation cost of the (stochastic) gradient map is significantly lower compared to the impractical computational burden associated with computing the norm of the gradient of the primal function\footnote{\mg{
    This quantity 
    can be computed based on the Danskin's theorem; but, this would require solving an optimization problem \sa{at every time the backtracking condition is checked, which would be} computationally expensive.}}.}
    
    Another source of difficulty, 
    in the analysis of single-loop methods for the non-convex setting is 
    \sa{the \textit{necessity} of imposing} \saa{an} appropriate \textit{time-scale separation} between the primal and dual updates in attaining convergence for the nonconvex setting~{\citep{li2022convergence}}. \sa{For instance, in both \citep{lin2020gradient,boct2020alternating}, the convergence is shown for \mg{dual stepsize} $\eta_y=\Theta(\frac{1}{L})$ and \mg{primal stepsize} $\eta_x=\Theta(\frac{1}{\kappa^2 L})$ for deterministic WCSC minimax problems where $L$ is the Lipschitz constant of the gradient of $f$ and $\kappa=L/\mu$, i.e., the ratio $\eta_y/\eta_x=\cO(\kappa^2)$ is needed for the convergence analysis; however, it is not clear how one can incorporate backtracking on $\eta_y$ and preserve the ratio between the primal and dual step sizes with some theoretical guarantees. 
    For the deterministic setting, there are very \textit{few} backtracking methods for the weakly convex-weakly concave~(WCWC) setting under Weak MVI or negative-comonotonicity assumptions, e.g.,~\citep{lee2021fast,pethick2023escaping}; however, these conditions are not guaranteed to hold \saa{in general for the WCSC or WCMC minimax problems we consider in this paper in the form of~\eqref{eq:main-problem}} --see the discussion in~\cref{sec-appendix-deterministic} on existing methods for WCWC problems.} 
    
    \rv{Finally, another challenge, maybe the main one, is that, even if we design a backtracking condition based on stochastic gradient estimates, the absence of a \textit{known} variance bound for the first-order stochastic oracle makes it difficult to control the stochastic error in the backtracking condition while simultaneously searching for a reasonable mini-batch sample size. This, in turn, necessitates dynamically estimating the variance via a suitable estimator, along with 
    conducting reliable tests to assess the accuracy of that variance estimate itself.}
 \textbf{\textit{\mg{Our contributions:}}} 
 For stochastic WCSC or WCMC minimax problems in the form of \eqref{eq:main-problem}, the algorithm \sgdab{} we propose in this paper achieves the best 
 sample complexity for computing an $\epsilon$-stationary point (\mg{in the sense of Definition~\ref{def:eps-stationarity}}) 
 among the existing methods that are agnostic to the global Lipschitz constant $L$, \rv{concavity modulus $\mu$ and to the variance bound $\sigma^2$ for the stochastic oracle.}
 More precisely, our main contribution is to show that without knowing/using $L,\mu$ and $\sigma^2$ values, \saa{for 
 given $\epsilon>0$, \sgdab{} can compute an $\epsilon$-stationary point of \eqref{eq:main-problem} with the following guarantees:} for deterministic WCSC and WCMC problems, \sgdab{} requires 
 $\cO(L\kappa^2/\epsilon^2)$ and $\cO(L^3/\epsilon^4)$ gradient calls, respectively; furthermore, we also show that for any $p\in(0,1)$, \sgdab{} can generate an $\epsilon$-stationary point w.p. at least $1-p$ requiring $\cO(L\kappa^3\epsilon^{-4}\log^2(1/p))$ and $\sa{\tilde\cO}(L^4\epsilon^{-7}\log^2(1/p))$ stochastic gradient calls for stochastic WCSC and WCMC problems, respectively. \mg{To our knowledge,} these are the best stochastic gradient complexities among all the first-order methods that are \rv{\textit{agnostic} to $L,\mu$ and $\sigma^2$,} and designed for WCSC and WCMC problems of the form in~\eqref{eq:main-problem} -- see the comparison provided in Table~\ref{table_stoc} 
 \rv{for the stochastic 
 setting which is the main focus of our paper. For brevity, our results for the deterministic gradient oracle setting are summarized in Table~\ref{table_deter} in Appendix \ref{sec-appendix-deterministic}, along with a discussion of existing results from the literature.} 

\sa{\textbf{\textit{Outline:}}
In Section~\ref{sec:past_work}, we discuss the existing work closely-related to our paper. Later in \cref{sec_pre}, we state our assumptions and provide the algorithmic framework \sgdab{} with backtracking. In \cref{sec:WCSC} and \cref{sec:WCMC}, we establish the convergence properties of \sgdab{} on WCSC and WCMC minimax problems, respectively. Finally, in \cref{sec:numerics}, on problems with synthetic and real data, we test \sgdab{} against other 
methods that are agnostic to \rv{the parameters
$L,\mu$ and $\sigma^2$,} and compare the results against the benchmark obtained using other state-of-the-art methods which require \mg{the knowledge of} $L,\mu$ and/or $\sigma^2$ to properly fix \saa{their} step sizes.}

\textbf{\emph{Notations.}} The set $\integers_+$ denotes 
\mg{non-negative} integers, \mg{whereas $\integers_{++}\triangleq \integers_+ \saa{\setminus} \{ 0\} $ denotes the set of positive integers.} The same 
convention is also used for $\reals_+$ and $\reals_{++}$. $\cU[1,N]$ denotes the uniform distribution on the set \rv{$[N]\triangleq\{1,2,\dots,N\}$}. \rv{The norm $\norm{\cdot}$ denotes the Euclidean norm for vectors and the spectral norm for matrices.} Given a closed convex function $h:\cY\to\reals\cup\{+\infty\}$, the proximal map is defined as $\prox{h}(y')\triangleq \argmin_{y\in\cY}h(y)+\frac{1}{2}\norm{y-y'}^2$ for all $y'\in\cY$. \rv{For $a,b\in\reals$, $a \vee b \triangleq \max\{a,b\}$ and $a \wedge b = \min\{a,b\}$.} 
Throughout the text, For some given $p\in(0,1)$, we use ``w.p. $p$" as an abbreviation for ``with probability $p$."
\renewcommand{\arraystretch}{1.4}
    \begin{table}[!htbp]
    \captionsetup{justification=raggedright, singlelinecheck=false}
        \centering
        \resizebox{\textwidth}{!}{
        \begin{tabular}{|l"l|c|c|c|c|l|l|}
            \hline
            \textbf{Work} & \textbf{Ref.} & \textbf{Agnostic} 
            & $g$ & $h$ & \textbf{BCU} & \textbf{Complexity} $(\mu>0)$ & \textbf{Complexity} $(\mu=0)$ \\
            \thickhline
            ${}^*$\gda{} & ~\cite{lin2020gradient} & \XSolidBrush & \XSolidBrush & $\mathds{1}_{Y}$ & \XSolidBrush & M1:\ $\cO\left(\rv{(\kappa^3 L B_0+ \kappa^2L^2\cD_y^2)} \epsilon^{-4}\right)$ & M3: $\cO(\rv{L^3\ell^2\cD_y^2\hat B_0}\epsilon^{-8})$\\
            \hline
            ${}^*$\agda{} & ~\cite{boct2020alternating} & \XSolidBrush & \CheckmarkBold & \CheckmarkBold & \XSolidBrush & M1: $\cO(\rv{(\kappa^3 L B_0+ \kappa^2L^2\cD_y^2)} \epsilon^{-4})$ & M3: $\cO(\rv{L^3\ell^2\cD_y^2 B_0}\epsilon^{-8})$\\
            \hline
             ${}^\dag$\texttt{SMDA} & ~\cite{huang2021efficient} & \XSolidBrush & \CheckmarkBold & \CheckmarkBold & \XSolidBrush & M1: $\cO(\rv{\kappa^{5}\mu^{-2} B_0}
             \epsilon^{-4})$ & \qquad\qquad\XSolidBrush \\
            \hline
             ${}^\dag$\texttt{Acc-MDA} & ~\cite{huang2022accelerated} & \XSolidBrush & \xqs{$\mathds{1}_{\sa{X}}$} & \xqs{$\mathds{1}_{Y}$} & \XSolidBrush & M1: $\cO(\kappa^{3.5}L^{1.5}\rv{B_0^{1.5}}\epsilon^{-3}\log(\epsilon^{-1}))$ & \qquad\qquad\XSolidBrush \\
            \hline
            \vrlm{} & ~\cite{mancino2023variance} & \XSolidBrush & \CheckmarkBold & \CheckmarkBold & \XSolidBrush & M1: $\cO(\kappa^3 L \rv{B_0}\epsilon^{-3})$ & \qquad\qquad\XSolidBrush \\
            \hline
            \smagda{} & ~\cite{yang2022faster} & \XSolidBrush & \XSolidBrush & \XSolidBrush & \XSolidBrush & M2: $\cO(\kappa^2 L \rv{B_0}\epsilon^{-4})$ & \qquad\qquad\XSolidBrush \\
            \thickhline
            ${}^\ddag$\neada{} & ~\cite{junchinest}  & {$L,\mu,\sigma^2$} & \XSolidBrush & $\mathds{1}_{Y}$ & \XSolidBrush & M2: $
            {{\cO}(\rv{(\kappa^2G^2+\kappa L+B_0)^4}\epsilon^{-4}\log(\epsilon^{-1}))}$ & \qquad\qquad\XSolidBrush\\
            \hline
            ${}^\S$\tiada{} & ~\cite{li2022tiada}  & {$L,\mu,\sigma^2$} & \XSolidBrush & $\mathds{1}_{Y}$ & \XSolidBrush & M2: 
            $\saa{\cO(\kappa^{12}\rv{G^6}\bar{L}^4\mu^{-4}\epsilon^{-4})}$ & \qquad\qquad\XSolidBrush \\
            \hline
            ${}^\P${\sgdab} &  \textbf{Ours} & \rv{$L,\sigma^2$} & \CheckmarkBold & \CheckmarkBold & \CheckmarkBold & M2: $\cO({\kappa^3 L \rv{B_0}\log^2(p^{-1})\epsilon^{-4}})$ & M2: $\sa{\tilde{\cO}}(L^4\cD_y^3\rv{B_0}\saa{\log^2(p^{-1})}\epsilon^{-7})$\\
            \hline
             ${}^\P${\sgdab} &  \textbf{Ours} & \rv{$L,\mu,\sigma^2$} & \CheckmarkBold & \CheckmarkBold & \CheckmarkBold & M2: \rv{$\cO({\hat{\kappa}^3 \hat{L} \rv{B_0} \log^2(p^{-1})\epsilon^{-4}})$ }& M2: $\sa{\tilde{\cO}}(\rv{\hat{L}^4}\cD_y^3\rv{B_0}\saa{\log^2(p^{-1})}\epsilon^{-7})$\\
            \hline
        \end{tabular}%
        }
        \raggedright
        \caption{\sa{Comparison of methods agnostic to \rv{$L,\mu$ and $\sigma^2$}, i.e., \sgdab{} (this paper), NeAda~\citep{junchinest} and TiAda~\citep{li2022tiada}, with \textit{single-loop} methods that require $L,\mu$ and $\sigma^2$ for solving \textit{stochastic} WCSC ($\mu>0$) and WCMC ($\mu=0$) minimax problems. {The column ``Agnostic" indicates whether the method requires knowing $L,\mu$ and/or $\sigma^2$, and the column ``BCU" indicates whether \saa{the method can use} block-coordinate updates.} In the columns we state complexity, we indicate the metric used for defining $\epsilon$-stationarity; 
        \saa{the definitions of 
        metrics M1, M2, and M3 are provided at the beginning of \cref{sec:past_work}. All the methods listed above, except for \sgdab, provide guarantees for these metrics in the expected sense, while \sgdab{} give guarantees for M2 with high probability. In the table $\cD_y$ denotes the diameter of the dual domain, and \rv{$B_0=F(\bx^0)-F^*$ denotes the initial suboptimality with respect to the primal function $F(\cdot)$ where $F^*=\min_{\bx\in\cX}F(\bx)$.} The derivation of complexity results for the methods listed here is provided in \cref{sec:complexity_comparison} of the supplementary material.}
        }\\ \sa{\textbf{Table Notes:}} ${}^*$For the case $\mu=0$, \cite{lin2020gradient} and \cite{boct2020alternating} 
        assume that \rv{$f(\cdot,y)$} is $\ell$-Lipschitz for each $y\in\cY$. \rv{Here $\hat B_0=F_\lambda(\bx^0)-F^*$ for $\lambda=\frac{1}{2L}$, where $F_\lambda(\bx^0)$ denotes the Moreau envelope of $F(\cdot)$.} \sa{${}^\dag$The complexity results reported here are different than those in \citep{huang2021efficient,huang2022accelerated}. The issues in their proofs leading to the wrong complexity 
        bounds are explained in \cite[Appendix I]{zhang2022sapd+}.} \saa{${}^\ddag$In~\citep{junchinest}, it is assumed that there exists \rv{a constant $G>0$ such that $\norm{\tilde\grad f(x,y;\xi)}\leq G$} for any $x\in \cX$ and $y\in Y$}. \sa{${}^\S$}
        In~\citep{li2022tiada} it is assumed that (i) \saa{$G\in(0,\infty)$ exists as in~\citep{junchinest}}; (ii) \saa{$\grad^2_{xy}f$ and $\grad^2_{yy}f$ are $\bar{L}$-Lipschitz;} 
        (iii) 
        $F(\cdot)$ is upper-bounded; and (iv) $y_*(x)\in{\rm int}(Y)$ for $x\in\cX$.\\
        {For any $\delta>0$, \mg{the 
        bound for TiAda is at least} 
        \saa{$\cO(G^{6-\frac{4\delta}{2+\delta}}\kappa^{12-\frac{6\delta}{2+\delta}}\bar L^{4-\frac{2\delta}{2+\delta}} \mu^{-4+\frac{2\delta}{2+\delta}} \epsilon^{-4+\frac{2\delta}{2+\delta}} + G^{2+\delta}\kappa^{(2+\frac{\delta}{2})} L^{(2+\frac{\delta}{2})} \epsilon^{-(4+\delta)})$.}
        In the table, we state the complexity for $\delta>0$ sufficiently small. ${}^\P$ \rv{In the complexity bounds for \sgdab, $p\in (0,1)$ is an algorithm parameter bounding the number of backtracking iterations, i.e., \sgdab{} requires at most $\cO(\log(\cR))$ backtracking iterations to generate an $\epsilon$-stationary point in (M2) metric with probability at least $1-p$. When $\mu>0$ is known, $\cR=\kappa$, and for the case $\mu>0$ is unknown, $\cR\triangleq\max\{\mu^{0}/\mu,L/L^{0},1\}$ and $\hat{\kappa}\triangleq\hat{L}/\hat{\mu}$, where $\hat{L}\triangleq\cR L^{0}$ and $\hat{\mu}\triangleq \mu^{0}/\cR$--here, $\mu^{0}\geq\mu$ and $L^{0}\leq L$ denote the initial estimates of the concavity modulus $\mu$ and the local Lipschitz constant of $\grad f$ at $(\bx^0,y^0)$, respectively. For the case $\mu$ is known, we set $\mu^0=\mu$ and $L^0=\mu/\gamma$ for some $\gamma\in(0,1)$.} 
        \vspace*{-0.75cm}
        }} 
        \label{table_stoc}
    \end{table}

\section{Related work}\label{sec:past_work}
We first briefly discuss some commonly adopted metrics for defining $\epsilon$-stationarity. The relation between these different metrics are discussed in \cref{sec:metrics}.
\vspace*{-2mm}
\begin{enumerate}
    \item[(i)] \sa{When $f$ is WCSC and $g(\cdot)=0$ in \eqref{eq:main-problem}, a commonly used metric for stationarity is \mg{the} gradient of the primal function, i.e., $\bx_\epsilon$ is $\epsilon$-stationary if $\norm{\grad \Phi(\bx_\epsilon)}\leq \epsilon$, where $\Phi(\bx)\triangleq \max_{y\in\cY}f(\bx,y)-h(y)$; moreover, using the gradient map, one can extend $\epsilon$-stationarity to the case with $g(\cdot)\neq 0$. Indeed, given $\epsilon>0$, $\bx_\epsilon$ is called $\epsilon$-stationary in metric \textbf{(M1)} if $\norm{\cG(\bx_\epsilon)}\leq \epsilon$ for some $\alpha>0$, where $\cG(\bx)\triangleq [\bx-\prox{\alpha g}(\bx-\alpha\grad \Phi(\bx))]/\alpha$.} 
    \item[(ii)] \sa{
    Furthermore, $\epsilon$-stationarity can also be defined using the coupling function $f$; 
    indeed, 
    for any given 
    $\eta_x,\eta_y>0$, let $G_\bx(\bx_\epsilon,y_\epsilon)\triangleq\big(\bx_\epsilon-\prox{\eta_x g}\big(\bx_\epsilon-\eta_x\grad_\bx f(\bx_\epsilon,y_\epsilon)\big)\big)/\eta_x$ and $G_y(\bx_\epsilon,y_\epsilon)\triangleq\big(\prox{\eta_y h}\big(y_\epsilon+\eta_y\grad_y f(\bx_\epsilon,y_\epsilon)\big)-y_\epsilon\big)/\eta_y$, then we call $(\bx_\epsilon,y_\epsilon)$ $\epsilon$-stationary\footnote{\rv{If $(\bx_\epsilon,y_\epsilon)$ is $\epsilon$-stationary in \textbf{(M2)}, then $ {\rm dist}\big(0,\grad_{\bx}f(\bx_\epsilon,y_\epsilon)+\partial g(\bx_\epsilon)\big)\leq (L+1)\epsilon$, and ${\rm dist}\big(0,\grad_{y}f(\bx_\epsilon,y_\epsilon)-\partial h(y_\epsilon)\big)\leq (L+1)\epsilon$, i.e., the distance of $\mathbf{0}$ to the subdifferential of $\cL$ is $\cO(\epsilon)$, which is independent of the step sizes $\eta_x,\eta_y$.}} in metric \textbf{(M2)} if $\norm{G(\bx_\epsilon,y_\epsilon)}\leq \epsilon$, 
    where $G(\bx_\epsilon,y_\epsilon)=[G_\bx(\bx_\epsilon,y_\epsilon)^\top, G_y(\bx_\epsilon,y_\epsilon)^\top]^\top$.}
    \item[(iii)] \sa{An alternative definition of $\epsilon$-stationarity when $g(\cdot)\neq 0$ is to use the Moreau envelope of the primal function. Indeed, $\bx_\epsilon$ is $\epsilon$-stationary in metric \textbf{(M3)} if $\norm{\grad F_\lambda(\bx_\epsilon)}\leq \epsilon$ for some $\lambda\in(0, \frac{1}{L})$, where $F=g+\Phi$ is the primal function and $F_{\lambda}:\bx\mapsto \min_{\bw\in\cX} F(\bw) + \frac{1}{2\lambda} \norm{\bw-\bx}^2$ is the Moreau envelope of $F$.}
\end{enumerate}
    \sa{Our objective in this paper is to incorporate a step size search within the (proximal) GDA method, analyzed in~\citep{lin2020gradient,boct2020alternating}, to make it agnostic to $L$ and $\mu$ \rv{even when the variance bound $\sigma^2$ is unknown}. For the reasons discussed above in the introduction, we would like to design a backtracking condition based on the stochastic gradient of the coupling function $f$ and provide some \rv{high-probability} guarantees in terms of metric \textbf{(M2)} for the new method. \saa{In general, compared to \textbf{(M1)} and \textbf{(M3)}, providing guarantees in \textbf{(M2)} requires checking a \textit{simpler} backtracking condition depending on the stochastic gradient map which can be easily calculated in practice; moreover, $\epsilon$-stationarity in \textbf{(M2)} also implies $\cO(\epsilon)$--stationarity with respect to \textbf{(M1)} and \textbf{(M3)} metrics --see~\cref{sec:metrics} of the supplementary material.}
    Since we are building on the (proximal) GDA method, our main goal is to design the new backtracking framework that can obtain similar complexity bounds with those established in~\citep{lin2020gradient,boct2020alternating} \rv{(where it is assumed that $L$, $\mu$ and $\sigma^2$ are all known)} --see Tables~\ref{table_stoc} and~\ref{table_deter}. 
    We choose the proximal GDA method 
    for backtracking due to its simplicity and its decent 
    complexity bounds with a reasonable $\cO(1)$ constant for the case $L$, $\mu$ and $\sigma^2$ are known; that said, one can also consider designing backtracking for other classes of single-loop methods, e.g.,~\citep{xu2023unified,yang2022faster}.} To put our work into context, we first briefly mention the two recent papers~\citep{xu2023unified,yang2022faster} that also provide guarantees involving the coupling function $f$ --both methods require that the global Lipschitz constant \rv{and the concavity modulus $\mu$} be known (to properly initialize the step sizes).
    
    Assuming $g$ and $h$ are indicator functions of some convex compact sets, $X$ and $Y$, respectively; \cite{xu2023unified} consider \eqref{eq:main-problem} in the \textit{deterministic} WCSC and WSMC settings with number of blocks $N=1$, and it does not consider the stochastic setting. The method in~\citep{xu2023unified} can compute $(\bx_\epsilon,y_\epsilon)$ such that $\norm{G(\bx_\epsilon,y_\epsilon)}\leq \epsilon$ in $\cO(\rv{\kappa^6} L^2\epsilon^{-2})$ and $\cO(L^4\cD_y^4\epsilon^{-4})$ gradient calls when $\mu>0$ and $\mu=0$, respectively \rv{--the derivation of these complexity bounds is provided in \cref{sec:complexity_comparison} of the online supplementary document.} Since it provides guarantees for both $\grad_x f$ and $\grad_y f$ --hence, easy to check in practice for the deterministic setting with $\sigma^2=0$, it may potentially be used to design a backtracking technique similar to the one proposed in our paper, i.e., searching for an admissible dual stepsize while imposing an appropriate time-scale separation between the primal and dual step sizes; however, it is not clear how their can be extended to the \textit{stochastic} setting especially when $\sigma^2$ is not known.
    \sa{On the other hand, the 
    Smoothed AGDA \mg{(\smagda)} method proposed in~\citep{yang2022faster} can handle the stochastic \saa{WCSC} setting when $\sigma^2$ is known, and provides the best known guarantees for the \textbf{(M2)} metric; 
    however, these results are only valid for the smooth case, i.e., $h(\cdot)=g(\cdot)=0$, and it is not trivial to extend the proof by~\cite{yang2022faster} \rv{either to the more general setup considered in \eqref{eq:main-problem}, which involves closed convex functions $g$ and $h$, or to the stochastic setting with an unknown variance bound $\sigma^2$.}}
    
    \sa{It is worth emphasizing that the existence of closed convex functions $g,h$ complicates the backtracking analysis even in \mg{the} deterministic setting,
    and dealing with these non-smooth terms in the stochastic setting is inherently more difficult; indeed, due to use of proximal operation for $g$ and $h$, \saa{the effect of noise (as a result of noisy gradients) showing up in the error bound for stationarity, which we call the \textit{variance term},} cannot be directly controlled through choosing a \textit{small} step size --this issue can also be observed in existing papers: compare the variance term of the error bound in \citep{lin2020gradient}, which assumes $g(\cdot)=0$, 
    with that in \citep{boct2020alternating}, which 
    considers $g$ and $h$ as in our paper. Indeed, according to~\cite[Eq. (C.12)]{lin2020gradient} (which is a result of the second inequality in~\cite[Lemma~C5]{lin2020gradient}) the variance \saa{term} can be controlled by adopting small step size $\eta_x>0$ --as the variance term scales with $\eta_x\sigma^2$,
    while \cite[Proof of Thm 4.2]{boct2020alternating} shows that $\sum_{k=1}^K\E{\norm{w_k}^2}=\cO(K\sigma^2/M)$ is \textit{independent} of $\eta_x$, where $w_k$ denotes a subgradient of the primal function $g(\cdot)+\max_{y\in\cY}f(\cdot,y)-h(y)$, and $M$ is the batch size. Finally, for the stochastic setting, even if we consider the smooth case, i.e., $g(\cdot)=h(\cdot)=0$,  it is not trivial to extend the proof by \cite{yang2022faster} 
    \saa{to use} backtracking as the guarantees given in the paper are on exact gradients of the coupling function $\grad f$ rather than its stochastic estimate $\tilde \grad f$; therefore, one cannot directly use those guarantees on $\grad f$ to accept or reject step sizes within a backtracking scheme as they cannot be verified in practice.
    }

    \sa{In this work, we 
    show that a simple backtracking technique 
    can be incorporated into the proximal GDA algorithm 
    with complexity guarantees matching (up to $\cO(1)$ constant) the results given by \cite{lin2020gradient,boct2020alternating}; 
    unlike \citep{lin2020gradient,boct2020alternating,xu2023unified,yang2022faster} discussed above,} our results do not require knowing \rv{the Lipschitz constant $L$, the concavity modulus $\mu$ and the variance bound $\sigma^2$}. In summary, to the best of our knowledge, there are no backtracking results for \textit{stochastic} WCSC and WCMC problems in the form of~\eqref{eq:main-problem}, and  our \rv{results 
    provide 
    the first guarantees for WCMC and WCSC minimax problems
    with possibly non-smooth closed convex 
    $h$ and $g$ as in \eqref{eq:main-problem}, and \sa{\sgdab{} algorithm proposed in this paper achieves} the best known complexity in terms of $\epsilon$ 
    without requiring any prior knowledge of the Lipschitz constant and concavity modulus associated with f, or of the variance bound of the stochastic oracle.} 
    \sa{Next, we briefly mention some existing work that are closely related to ours.} 

\textbf{{Existing work for the stochastic {setting}.
}} \sa{To the best of our knowledge, efficient
\textit{stochastic backtracking}~(SB) step-size search techniques that can exploit the structure of minimax problems are still missing in the literature even for the convex-concave minimax problems. Although there are some recent developments on SB for stochastic optimization, e.g., \citep{jin2021high,paquette2018stochastic,jin2023sample}, 
{these methods} do not extend to more 
general minimax problems. {There are also} methods developed
for stochastic monotone variational inequalities~(VI), e.g., \citep{iusem2019variance,antonakopoulos2021sifting}; 
that being said, these methods are not applicable for WCSC and WCMC minimax problems as the corresponding VI is \textit{not} monotone anymore, and we are not aware of any stochastic VI methods that can deal with such problems using adaptive stepsizes.} 

\sa{For solving \textit{stochastic} WCSC minimax problems, the only FOPD algorithms that are \textit{agnostic} to 
$L$ and $\mu$, and \sa{that} use adaptive step sizes are \neada{} (double-loop) by~\cite{junchinest} and \tiada{} (single loop) by~{\cite{li2022tiada}}, which 
are both based on AdaGrad stepsizes~{\citep{duchi2011adaptive}} and these methods do not use backtracking {step-size} search. The analyses of these algorithms are given for smooth \sa{WCSC} problems with 
$g(\cdot)=0$ and $h(\cdot)=\mathds{1}_{Y}(\cdot)$, i.e., the indicator function of a closed convex set $Y$\footnote{\label{fn:tiada-interior}In~\citep{li2022tiada}, \rv{it is assumed that for any $x\in\cX$, $y_*(x)\in{\rm int}(Y)$; that is why they adopt the stationarity measure $\norm{\grad f}^2$. This is important for \tiada{} as it uses AdaGrad-type step sizes, see \citep[Remark 1]{li2022tiada}.}}.} 
\cite{li2022tiada} showed that \tiada{} can compute $(x_\epsilon,y_\epsilon)$ such that $\E{\norm{\grad f(x_\epsilon,y_\epsilon)}^2}\leq \epsilon^2$ with at most $C^{\rm TiAda}_{\epsilon,\kappa}=\cO(\epsilon^{-(4+\delta)})$ oracle complexity for any $\delta>0$ --however, how $\cO(1)$ constant depends on $\kappa$ is not discussed in~\cite{li2022tiada}; indeed, in \cref{sec:complexity_comparison}, we argue that for any given $\delta>0$, \saa{$C^{\rm TiAda}_{\epsilon,\kappa}=\Omega({\kappa^{2+\frac{\delta}{2}}}\epsilon^{-(4+\delta)}+\kappa^{12-\frac{6\delta}{2+\delta}}\epsilon^{-4+\frac{2\delta}{2+\delta}})$.}
On the other hand,
\cite{junchinest} showed that the oracle complexity of \neada{} can be bounded by $C^{\rm NeAda}_{\epsilon,\kappa}=\cO(\epsilon^{-4}\log(1/\epsilon))$ to compute $(x_\epsilon,y_\epsilon)$ such that $\E{\norm{\grad_x f(x_\epsilon,y_\epsilon)}}\leq \epsilon$ and $\E{\norm{y_\epsilon-y^*(x_\epsilon)}}\leq \epsilon$; in \cref{sec:complexity_comparison}, we argue that $C^{\rm NeAda}_{\epsilon,\kappa}=\Omega(\kappa^8\epsilon^{-4}\log(1/\epsilon))$. While the single-loop method \tiada{} is agnostic to the strong concavity modulus $\mu$ and to the stochastic oracle variance bound $\sigma^2$, \rv{its convergence analysis is given under some assumptions stronger than ours:} 
\textit{(i)} $\grad^2_{xy} f$ and $\grad^2_{yy} f$ are Lipschitz, i.e., the coupling function $f$ is twice differentiable in $y$ with a Lipschitz second derivative; 
\textit{(ii)} \textit{stochastic} gradients are uniformly bounded, i.e., $\exists G>0$ such that $\|\tilde\grad f(x,y)\|\leq G$ for all $(x,y)\in\cX\times Y$, which clearly does not hold even for additive Gaussian noise structure; \rv{\textit{(iii)} $y_*(x)\in{\rm int}(Y)$ for all $x\in\cX$,} and \textit{(iv)} the primal function $\Phi(\cdot)\triangleq \max_{y\in Y} f(\cdot,y)$ is bounded above, i.e., $\exists B_\Phi>0$ such that $\Phi(x)\leq B_\Phi$ for all $x\in\cX$, which is a strong assumption considering that $\Phi$ is weakly convex and $\cX$ is a vector space. \rv{In our analysis of \sgdab{} for the case $\sigma^2$ is \textit{not} known, after we provide the analysis under the bounded \textit{stochastic} gradient assumption as in \citep{li2022tiada} for \tiada{}, we next discuss how this assumption can be relaxed to sub-Gaussian noise setting.}

\sa{In our 
numerical tests (see Figures~\ref{fig:Q},~\ref{fig:tildeQ} and~\ref{figs:tilde LG Regression}), the practical performance of \neada{} and 
\tiada{} did not match their nice theoretical complexity bounds 
    --this may be due to large $\cO(1)$ constants with \textit{high} powers of $\kappa$ arising in their bounds; indeed, {both of them}
    \sa{adopt AdaGrad-like step size sequences, and while this choice makes the algorithm agnostic to $L$ and $\mu$, it also potentially causes issues in practice:}
    since initial iterates are potentially far from stationarity (with gradient norms \textit{not} close to $0$), in the initial transient phase of the algorithm the step sizes 
    decrease very quickly, as they are inversely proportional \sa{to the partial sum of
    gradient norms over all past iterations}. These quickly decaying step sizes 
    cause slow convergence in practice 
    \sa{(see Figures~\ref{fig:Q} and~\ref{fig:tildeQ})
    }.}
\section{Assumptions and Algorithmic Framework}\label{sec_pre}
\begin{assumption}\label{ass1}
    Let $g_i:\cX_i\to\reals\cup\{+\infty\}$ for $i\in\cN$, $f:\cX\times\cY\to \mathbb{R}$, $h:\cY\to\reals\cup\{+\infty\}$ such that
    \begin{enumerate}
        \item [(i)] $g_i$, for $i\in \mathcal{N}$, and $h$ are closed convex functions;
        \item [(ii)] $f$ is differentiable on an open set containing $\dom g \times \dom h$ such that $\grad f$ is $L$-Lipschitz on $\dom g \times \dom h$; furthermore, $f(\bx,\cdot)$ is (strongly) concave with modulus $\mu$ for each $\bx\in\dom g$ for some $\mu\geq 0$;
        \item [(iii)] $-\infty<F^*\triangleq \inf_{\bx\in\cX} F(\bx)$, where $F\triangleq g+\Phi$ for $\Phi:\cX\to\reals$ such that $\Phi(\cdot)\triangleq\max_{y\in\cY}f(\cdot,y)-h(y)$. \saa{Suppose we are given $\bar F\in\reals$ such that $\bar F\leq F^*$.}
    \end{enumerate}
\end{assumption}
\begin{remark}
\label{rem:barF}
    \rv{In many applications $\bar F$ is readily available from the structure of the problem or through Lagrangian duality (via weak duality). A canonical application of \eqref{eq:main-problem} is the DRO problem where one is interested in minimizing the loss under the worst-case data distribution, see~\eqref{eq:dro-problmm} in the numerical section. In this scenario, $f(\bx,y)$ takes the form $\sum_{i=1}^d y_i\ell_i(\bx)$ where $\ell_i(\bx)\geq 0$ is the loss value for data point $i=1,\ldots, d$ and $y=[y_i]_{i=1}^d$ such that $y$ lies in a unit simplex; thus, it is apparent that one can set $\bar F=0$. 
    {Similarly, for many other minimax problems arising in machine learning such as fair learning \citep{singh2023minimax}, adversarial learning \citep{madry2017towards} or statistical learning subject to model constraints involve \rv{non-negative} loss functions, 
    one can set $\bar{F}=0$.}}
\end{remark}
\begin{defn}\label{def-ystar}
    Define ${y}_{*}:\cX\to\cY$ such that ${y}_{*}(\bx)\triangleq\argmax_{y\in\cY}f(\bx,y)-h(y)$ for all $\bx\in\dom g$.
\end{defn}

\sa{In many practical settings \mg{partial gradients} $\grad_{\bx} f$ and $\grad_y f$ may not be 
available or computing them may be very costly; but, \mg{rather we have access to their noisy unbiased stochastic estimates at a given point \sa{$(\bx,y)$} through a stochastic oracle. This would be 
the natural setting in machine learning applications where the gradients are estimated from randomly sampled subsets (batches) of data {\citep{bottou2018optimization}}. We make the following assumption \sa{on the stochastic oracles} which is \sa{also commonly 
adopted} in the literature~{\citep{zhang2022sapd+,zhang2021robust}}}. 
}
\sa{\begin{assumption}
\label{assumption:noise}
Let $\{\gt_{x_i} f\}_{i\in\cN}$ and $\gt_y f$ denote stochastic oracles for $\{\grad_{x_i} f\}_{i\in\cN}$ and $\grad_y f$, respectively, satisfying the following assumptions for all $(\bx,y)\in\dom g\times\dom h$:
{\small
\begin{enumerate}
    \item [(i)] $\mathbf{E}_{\omega}[\gt_{x_i} f(\bx,y;\omega)|\bx,y]=\grad_{x_i} f(\bx,y),\quad
    \forall~i\in\cN$;
    \item[(ii)] $\mathbf{E}_{\zeta}[\gt_{y} f(\bx,y;\zeta)|\bx,y]=\grad_{y} f(\bx,y)$; 
    \item[(iii)] $\exists~\sigma_x\geq 0:\ \mathbf{E}_{\omega}\Big[\norm{\gt_{x_i} f(\bx,y;\omega) - \grad_{x_i} f(\bx,y)}^2~\Big|~\bx,y\Big]\leq\sigma_x^2/N,\quad
    \forall~i\in\cN$;
    \item[(iv)] $\exists~\sigma_y\geq 0:\ \mathbf{E}_{\zeta}\Big[\norm{\gt_{y} f(\bx,y;\zeta) - \grad_{y} f(\bx,y)}^2~\Big|~\bx,y\Big]\leq\sigma_y^2$.
\end{enumerate}}%
\end{assumption}}%
\sa{Here, we consider the setting where we are allowed to call the stochastic oracles in batches of size $M\geq 1$. 
The result in \citep{lin2019gradient}, assuming $g(\cdot)=0$ and 
$h(\cdot)=\mathds{1}_{Y}(\cdot)$ in~\eqref{eq:main-problem} for some closed convex set $Y$, can be extended to cover purely stochastic case of $M=1$; however, this would degrade the complexity from $\cO(\epsilon^{-4})$ for $M=\cO(\epsilon^{-2})$ to $\cO(\epsilon^{-5})$ for $M=\cO(1)$. 
The 
\saa{\smagda{}} method proposed by~\cite{yang2022faster} has $\cO(\epsilon^{-4})$ complexity guarantee for $M=\cO(1)$; however, \saa{this requires setting $\eta_x=\eta_y=\cO(\frac{\epsilon^2}{\sigma^2L\kappa})$} \rv{for some $\sigma^2\geq\max\{\sigma_x^2,\sigma_y^2\}$} and this result is only valid for the smooth case, i.e., $h(\cdot)=g(\cdot)=0$; furthermore, \saa{it is not trivial to extend the proof by~\cite{yang2022faster}, showing the $\cO(\epsilon^{-4})$ complexity result with $M=1$,} to the more general setup we considered in \eqref{eq:main-problem} involving possibly \textit{nonsmooth} closed convex functions $g$ and $h$. 
There are only very few \textit{single-loop} methods that can handle \textit{stochastic} WCSC problem in \eqref{eq:main-problem} with arbitrary closed convex functions $g$ and $h$; indeed, \citep{boct2020alternating,mancino2023variance,huang2021efficient} are the only ones we are aware of. The methods by~\cite{boct2020alternating} and \cite{huang2021efficient} both can guarantee computing an $\epsilon$-stationary point with $\cO(\epsilon^{-4})$ complexity and using batch size $M=\cO(\epsilon^{-2})$ \saa{--\cite{boct2020alternating} require setting $\eta_y=1/L$, $\eta_x=\frac{1}{L\kappa^2}$, $M=\cO(\kappa^2\sigma^2\epsilon^{-2})$, and 
\texttt{SMDA} in \citep{huang2021efficient} requires $\eta_y=\frac{1}{L}$, $\eta_x=\frac{1}{\kappa^3}$ and $M=\cO(\frac{\kappa^2\sigma^2}{\mu^2\epsilon^2})$;} on the other hand, using variance reduction \citep{huang2021efficient,mancino2023variance} can achieve $\cO(\epsilon^{-3})$ complexity ---\cite{mancino2023variance} require batch size $M=\cO(1)$ \saa{and step sizes $\eta_y=\cO(\frac{\epsilon}{\sigma L \kappa})$ and $\eta_x=\cO(\frac{\epsilon}{\sigma L \kappa^3})$, and \texttt{VR-SMDA} in~\citep{huang2021efficient} requires setting $\eta_y=\frac{1}{L\kappa}$, $\eta_x=\frac{1}{\kappa^4}$, and sampling large batch size of $\cO(\frac{\kappa^2\sigma^2}{\mu^2\epsilon^2})$ once in every $q=\cO(\frac{\kappa}{\mu\epsilon})$ iterations which all use small batch size of $\cO(\frac{\kappa}{\mu\epsilon})$}. That being said, 
\cite{boct2020alternating,mancino2023variance,huang2021efficient} all define $\epsilon$-stationarity using the gradient of the primal function, \saa{i.e., using metric \textbf{(M1)}}; therefore, these methods are not suitable for backtracking and they require knowledge of \rv{$L$, $\mu$ and a variance bound $\sigma^2$, where $\sigma\triangleq\max\{\sigma_x,\sigma_y\}$.} Indeed, we are not aware of any work that achieves $\cO(\epsilon^{-4})$ complexity with $M=\cO(1)$ when there are arbitrary closed convex functions $g$ and $h$ as in~\eqref{eq:main-problem} and that 
has \saa{\textit{direct}} guarantees 
for the \textit{coupling function}, \rv{i.e., in metric \textbf{(M2)}}, which 
\saa{is desirable} for being able to implement backtracking on the dual step size. In the rest of this paper, we restrict the discussion to a method that does not require a bound on \rv{$L$, $\mu$ and $\sigma^2$, i.e., agnostic to $L$, $\mu$ and $\sigma^2$,} and still achieving $\cO(\epsilon^{-4})$ complexity with batch size of $M=\cO(\epsilon^{-2})$.}

\mg{We next introduce the stochastic gradient estimates $s_{\bx}$ and $s_{\by}$ that are obtained based on $M_x, M_y\in \integers_+$ calls to the stochastic oracle and the corresponding stochastic gradient maps $\tilde G_{\bx}$ and $\tilde G_{y}$, which arise frequently in the study of proximal-gradient methods~{\citep{chen2022accelerated,xu2023unified,lu2020hybrid}.}} 
\begin{defn}\label{def:s-gradmapping}
\sa{Given $\mg{M_x, M_y}\in\integers_{++}$ and 
$\eta_x,\eta_y>0$, define $s_{\bx},\tilde G_{\bx}:\cX\times\cY\to\cX$ 
and $s_y,\tilde G_y:\cX\times\cY\to\cY$ such that $s_{\bx}(\bx,y;\bom)=[s_{x_i}(\bx,y;\bom)]_{i=1}^N\in\cX$ and $\tilde G_{\bx}(\bx,y;\bom)=[\tilde G_{x_i}(\bx,y;\bom)]_{i=1}^N\in\cX$, 
{\small
\begin{align*}
    &s_{x_i}(\bx,y;\bom)\triangleq \frac{1}{M_x} \sum_{j=1}^{M_x}\gt_{x_i} f({\bx},y;\omega_j),\quad \sa{\tilde G_{\sa{x_i}}(\bx,y;\bom)}\sa{\triangleq \Big[{x}_i - \prox{\eta_x g_i} \Big({x}_i-\eta_x s_{x_i}(\bx,y;\bom)\Big)\Big]/\eta_x}, \quad\forall~i\in\cN,\\
    &s_y(\bx,y;\bzt)\triangleq \frac{1}{M_y} \sum_{j=1}^{M_y}\gt_{y} f({\bx},y;\zeta_j),\quad \tilde G_y(\bx, y;\bzt)\triangleq \Big[\prox{\eta_y h} \Big({y}+\eta_y s_y(\bx,y;\bzt)\Big)-{y}\Big]/\eta_y,
\end{align*}}%
are defined for all $( \bx, y)\in\dom f\times \dom g$, and $\bom\triangleq[\omega_j]_{j=1}^{M_x}$, $\bzt\triangleq [\zeta_j]_{j=1}^{M_y}$ such that both $\{\omega_j\}_{j=1}^{M_x}$ and $\{\zeta_j\}_{j=1}^{M_y}$ are i.i.d. \mg{corresponding to the randomness 
of the stochastic oracle} and satisfy \cref{assumption:noise}. 
For notational simplicity, let $\bxi\triangleq [\bom^\top \bzt^\top]^\top$, and we define $s(\bx,y;\bxi)\triangleq [s_{\bx}(\bx,y;\bom)^\top s_{y}(\bx,y;\bzt)^\top ]^\top$ and $\tilde G(\bx,y;\bxi)\triangleq [\tilde G_{\bx}(\bx,y;\bom)^\top \tilde G_{y}(\bx,y;\bzt)^\top ]^\top$. \rv{Moreover, for $M_x,M_y\geq 2$, we define the sample variance function $V_{\bxi}:\dom g\times \dom h\to\reals\times\reals$ such that $V_{\bxi}(\bx,y)=(v^x,v^y)$ and}}\looseness=-1
\rv{\small
\begin{align}
\label{eq:sample_variance}
    v^x=\frac{1}{M_x-1}\sum_{j=1}^{M_x}\norm{\gt_{\bx} f(\bx,y;\omega_j) - s_{\bx}(\bx,y;\bom)}^2,
    \quad
     v^y=\frac{1}{M_y-1}\sum_{j=1}^{M_y}\norm{\gt_{y} f(\bx,y;\zeta_j) - s_{y}(\bx,y;\bzt)}^2.
\end{align}}%

\sa{Furthermore, let $G_{x_i}$ for $i\in\cN$ and $G_y$ denote the deterministic gradient maps, i.e., for any $(\bx,y)\in\dom g\times\dom h$, $G_{{x_i}}(\bx,y)\triangleq \Big[{x}_i - \prox{\eta_x g_i} \Big({x}_i-\eta_x \grad_{x_i} f(\bx,y)\Big)\Big]/\eta_x$ for $i\in\cN$, and $G_{y}(\bx,y)\triangleq \Big[\prox{\eta_y h} \Big(y+\eta_y \grad_{y} f(\bx,y)\Big)-y\Big]/\eta_y$. 
Finally, let $G_{\bx}(\bx,y)\triangleq [G_{x_i}(\bx,y)]_{i=1}^N\in\cX$, and we define $G(\bx,y)\triangleq [G_{\bx}(\bx,y)^\top G_{y}(\bx,y)^\top ]^\top$.}
\end{defn}

\begin{remark} \mg{By the i.i.d. nature of $\{\omega_j\}_{j=1}^{M_x}$ and $\{\zeta_j\}_{j=1}^{M_y}$, it follows from \cref{{assumption:noise}} that 
\begin{equation}
\E{\norm{s_{\bx}(\bx,y;\bom)-\grad_{\bx} f(\bx,y)}^2} \leq \sigma_x^2/M_x, \quad 
\E{\norm{s_{\by}(\bx,y;\bzt)-\grad_{\by} f(\bx,y)}^2} \leq \sigma_y^2/M_y.
\label{ineq-grad-var-bound}
\end{equation}}%
\end{remark}
\vspace*{-9mm}
\begin{defn}\label{def:eps-stationarity}
\sa{Given $\epsilon>0$, $(\bx_\epsilon,y_\epsilon)\in\dom g\times\dom h$} is an $\epsilon$-stationary point if $\norm{G(\bx_\epsilon,y_\epsilon)}\leq \epsilon$.
\end{defn}

\rv{To get tighter high-probability guarantees, we also consider the setting where the stochastic noise is bounded almost surely in addition to \cref{assumption:noise}.} 
\begin{assumption}
\label{as:bounded-oracle}
    There exist constants 
    $B_x,B_y<\infty$ such that $B_x\geq \sup\{\norm{\gt_x f(\bx,y;\bom)-\grad_x f(\bx,y)}:\ \bx\in\dom g,~y\in\dom h\}$ and $B_y\geq \sup\{\norm{\gt_y f(\bx,y;\bzt)-\grad_y f(\bx,y)}:\ \bx\in\dom g,~y\in\dom h\}$ a.s. 
\end{assumption} 
\rv{Assumption~\ref{as:bounded-oracle} will often hold for smooth objectives over compact domains. Indeed, suppose the domain $\dom g\times\dom h$ is compact and $f$ admits a finite-sum representation of the form $f(\bx,y) = \sum_{j=1}^P f_j(\bx,y)$ with each component function $f_j(\bx,y)$ being continuously differentiable for $j\in[P]$ --for this setting, $\nabla f_j(\bx,y)$ is bounded for $j\in[P]$ over the domain. Therefore, in this case, when the stochastic gradient oracle is constructed by mini-batching, i.e., when $f_j$ is sampled uniformly at random with replacement, the randomness generated by the oracle is bounded a.s. and Assumption~\ref{as:bounded-oracle} will naturally hold.} 

\begin{remark}
    In \cref{sec:subGaussian}, we {further relax Assumption~\ref{as:bounded-oracle} and} provide high-probability convergence guarantees for \sgdab{} under the light-tail noise setting where we assume that the stochastic gradient oracle is norm-subGaussian.
\end{remark}
Before we describe the proposed \sgdab{} method, we state a concentration result for $s(\bx,y;\bxi)$ under Assumption~\ref{as:bounded-oracle}.

\begin{theorem}
	\label{thm:concentration-bounded}
    Let $X\in\reals^d$ be a random vector with mean $\mathbf{0}$, variance $\sigma^2$ and covariance $\Sigma$, i.e., $\mathbf{E}[X]=\mathbf{0}$ and $\mathbf{E}[\norm{X}^2]=\sigma^2$. Suppose that there exists $B>0$ such that $\norm{X}\leq B$ a.s. Let $\{X_j\}_{j=1}^M\subset\reals^d$ be i.i.d. random vectors following the same distribution with $X$. Then, the sample mean $\bar X=\frac{1}{M}\sum_{j=1}^M X_j$ satisfies
	\begin{align}
		\label{eq:mean-concentration-bounded}
		\p{\norm{\bar X}>t}\leq (d+1) \exp\Big(-M\frac{\sigma^2}{B^2}\cdot H\Big(\frac{Bt}{\sigma^2}\Big)\Big),\quad\forall~t>0,
	\end{align}
    {where $H:\reals_+\to\reals$ is defined as $H(u)=(1+u)\log(1+u)-u$.}
    
	Moreover, let $v=\frac{1}{M-1}\sum_{j=1}^M\norm{X_j-\bar{X}}^2$ denote the sample variance {for $M\geq 2$}. Then,
	{\small
		\begin{subequations}
        \label{eq:svar-tail-bounds}
			\begin{align}
				\p{\Big(1-\frac{1}{M}\Big)v-\sigma^{2}\leq -t}&\leq \exp \left(-M \frac{\sigma^2}{B^2}H\Big(\frac{t}{2\sigma^2}\Big)\right)
				+(d+1)\exp\Big(-M\frac{\sigma^2}{B^2}\cdot H\Big(\frac{B\sqrt{t/2}}{\sigma^2}\Big)\Big), 
				\label{eq:svar-lower-tail-bounded}\\
				\p{\Big(1-\frac{1}{M}\Big)v-\sigma^{2}\geq t}&\leq \exp \left(-M \frac{\sigma^2}{B^2}\cdot H\Big(\frac{t}{\sigma^2}\Big)\right),\quad\forall~t>0, \label{eq:svar-upper-tail-bounded}
			\end{align}
	\end{subequations}}%
    Finally, for any $\alpha\in (0,1)$, $\bar\sigma\geq\sigma$ and $c>0$, 
    it holds that
	\begin{equation}
		\label{eq:concentration-subGaussian-bounded}
		\p{\Big(1-\frac{1}{M}\Big)v\leq (1+c)\bar\sigma^{2}}\geq 1-\alpha,\qquad \forall~M\in\integers_+:\ M\geq \max\Big\{\frac{1}{H(c)}\cdot\frac{B^2}{\bar\sigma^2}\cdot\log\Big(\frac{1}{\alpha}\Big),~2\Big\}.
	\end{equation}
\end{theorem}
\begin{proof}
{The proof is based on leveraging the concentration properties of independent bounded random variables,} and is provided in~\cref{sec:bounded-proofs}. 
\end{proof}

\begin{algorithm}[bht!]
\caption{SGDA with Backtracking (\sgdab)}
\label{alg:GDA-B}
\small
\begin{algorithmic}[1]
\State{\bf Input:} $\epsilon,\delta,\rv{\sigma^0_x, \sigma^0_y,\mu^0, L^0 > 0}$ \rv{: $L^0 > \mu^0$}, 
$\rv{p,\bar p,}\gamma\in(0,1)$, \rv{$c\in(0,1)$,} $(\bx^0,y^0)\in 
\dom g\times \dom h$, 
\saa{$\rv{F_0},\bar{F}\in\mathbb{R}$}
\State{\bf Initialize:} $\rho\gets\frac{\sqrt{1+{12}/N}-1}{24}$,\quad $T\gets\lceil\log_2(\rv{3/p})\rceil$, \quad 
\rv{$(\underline{\mu}, \bar{\sigma}_x, \bar{\sigma}_y)$ are set as in \eqref{eq:mu-sigma-initial}}
\State \rv{$\widetilde L\gets L^0,\quad \widetilde\mu\gets\mu^0,\quad \widetilde\sigma_x\gets\sigma_x^0,\quad \widetilde\sigma_y\gets\sigma_y^0$}
\For{$\ell=\rv{0},\dots
$}\label{alg:outer_iter}
\State \rv{$\eta_y\gets 1/ \widetilde L$,\quad $\eta_x\gets N\rho\widetilde\mu^2\eta_y^3$} \label{algeq:step-sizes}
\State $K\gets \saa{\lceil\frac{64N}{\epsilon^2\eta_x}\left(
\rv{F_0-\bar{F}+6\rho\delta\widetilde\mu}\right)\rceil}$,\quad 
\rv{$M_x \gets M_x(\epsilon,\widetilde\sigma_x^2,p,\bar p)$,} 
\quad \rv{$M_y \gets M_y(\epsilon,\widetilde\sigma_y^2,\eta_y,\widetilde\mu,p,\bar p)$}

\Comment{$M_x(\cdot)$ and $M_y(\cdot)$ functions will be specified later}
\label{algeq:M}
\For{$t=1,2,\dots,T$} \Comment{\saa{$T$ independent, parallel \rbsgda \; runs}}
    \State $\sa{\Big(\tilde\bx_{(t,\ell)},\tilde y_{(t,\ell)},\tilde G(\tilde\bx_{(t,\ell)},\tilde y_{(t,\ell)};\tilde\bxi_{(t,\ell)}),\rv{v^x_{(t,\ell)}, v^y_{(t,\ell)}}\Big)}\gets {\rm \sa{\rbsgda}}(\eta_x,\eta_y,M_x,M_y,N,K,\bx^0,y^0)$
    \label{algeq:RBSGDA_runs}
\EndFor
\State \sa{$\tilde S_{(t,\ell)}\gets\norm{\tilde G(\tilde\bx_{(t,\ell)},\tilde y_{(t,\ell)};\tilde\bxi_{(t,\ell)})}^2$} \label{algeq:stochastic_S}
\State \sa{$t^*\gets\argmin\{\tilde{S}_{(t,\ell)}:\ t=1,\ldots,T\}$}
\State \rv{$\tilde S_{(\ell)}\gets\tilde S_{(t^*,\ell)}$,\quad $v_{(\ell)}^x\gets v^x_{(t^*,\ell)}$, \quad $v_{(\ell)}^y\gets v^y_{(t^*,\ell)}$,\quad $(\tx,\ty,\txi)\gets (\tilde\bx_{(t^*,\ell)},\tilde y_{(t^*,\ell)},\tilde\bxi_{(t^*,\ell)})$}
\If{$\bar\sigma_x<\infty$ \textbf{or} \rv{$(1-\frac{1}{M_x})v_{(\ell)}^x\leq (1+c)\widetilde\sigma_x^2$}}
\label{eq:sigma-x-test}
\State $\mbox{flag-x}\gets \textbf{True}$
\EndIf
\If{$\bar\sigma_y<\infty$ \textbf{or} \rv{$(1-\frac{1}{M_y})v_{(\ell)}^y\leq (1+c)\widetilde\sigma_y^2$}} \label{eq:sigma-y-test}
\State $\mbox{flag-y}\gets \textbf{True}$
\EndIf
\If{$
        \sa{\rv{\tilde S_{(\ell)}}
        \leq \frac{\epsilon^2}{4}}$ \rv{\textbf{and} \rm{flag-x} \textbf{and} \rm{flag-y}}} \label{algeq:stop}
    \State  \sa{$(\bx_\epsilon,y_\epsilon,\bxi_\epsilon)\gets(\tx,\ty,\txi)$}
    \State \Return \sa{$(\bx_\epsilon,y_\epsilon)$} \label{algeq:output}
\EndIf
\State{$\widetilde{L}\gets \widetilde{L}/\sa{\gamma}$,\quad \rv{$\widetilde \mu \gets \max\{\widetilde \mu \gamma, \underline{\mu}\}$,\quad $\widetilde\sigma_x\gets\min\{\widetilde\sigma_x/\gamma,\bar{\sigma}_x\}$,\quad $\widetilde\sigma_y\gets\min\{\widetilde\sigma_y/\gamma,\bar{\sigma}_y\}$}} \label{eq:L-update} 
\EndFor
\end{algorithmic}
\end{algorithm}
\begin{algorithm}[bht!]
\caption{$(\tilde\bx,\tilde y,\sa{\tilde G},v^x,v^y)=$\sa{\texttt{RB-SGDA}}($\eta_x,\eta_y,M_x,M_y,N,K,\bx^0,y^0$)}
\label{alg:GDA}
\small
\begin{algorithmic}[1]
\State \sa{$\tilde k\gets\cU[0,K-1]$}
\For{$k=0,\dots,\tilde k$}
\State \sa{$i_k\gets\cU[1,N]$ \Comment{$i_k$ is distributed uniformly on $\{1,\ldots,N\}$}}
\State \sa{$\bx^{k+1}\gets\bx^{k}$}
\State \sa{$x_{i_{k}}^{k+1}\gets \prox{\eta_x g_{i_k}}\Big(x_{i_{k}}^k-\eta_x s_{x_{i_k}}(\bx^k,y^k;\bom^k)\Big)$} \label{eq:GDA_x}
\State\sa{$y^{k+1}\gets \prox{\eta_y h}\Big(y^{k}+\eta_y s_y(\bx^k,y^k;\bzt^k)\Big)$} \label{eq:GDA_y}
\EndFor
\State \sa{$(\tilde\bx,\tilde y)\gets(\bx^{\tilde k},y^{\tilde k}),\quad \tilde G\gets \tilde G(\bx^{\tilde k},y^{\tilde k};\bxi^{\tilde k})$},\quad \rv{$(v^x,v^y)\gets V_{\bxi^{\tilde k}}(\bx^{\tilde k},y^{\tilde k})$} \label{algeq:RB-SGDA_N}
\Comment{$V$ is defined in \eqref{eq:sample_variance}}
\State \Return \sa{$(\tilde\bx,\tilde y, \mg{\tilde G}, v^x, v^y)$} 
\end{algorithmic}
\end{algorithm}
\rv{Next, we state Stochastic GDA with Backtracking~(\sgdab) in~\cref{alg:GDA-B}. We first discuss how certain algorithm parameters are initialized: If $\mu>0$ is known, then in \sgdab{} we initialize $\mu^0=\mu$ and similarly in case $\sigma_x$ (or $\sigma_y$) is known, then we set $\sigma_x^0=\sigma_x$ (or $\sigma_y^0=\sigma_y$); furthermore, let}\looseness=-10 
\rv{
\begin{align}
\label{eq:mu-sigma-initial}
    \underline{\mu}\gets \begin{cases}
    \mu,\ \mbox{if $\mu$ is known,}\\
    0,\ \mbox{otherwise;}
\end{cases}
\bar{\sigma}_x\gets \begin{cases}
    \sigma_x,\ \mbox{if $\sigma_x$ is known,}\\
    +\infty,\ \mbox{otherwise;}
\end{cases}
\bar{\sigma}_y\gets \begin{cases}
    \sigma_y,\ \mbox{if $\sigma_y$ is known,}\\
    +\infty,\ \mbox{otherwise.}
\end{cases}
\end{align}}%
\rv{\sgdab{} takes tolerance parameter $\epsilon>0$, the variance bound estimates $\sigma^0_x, \sigma^0_y>0$, concavity modulus and Lipschitz constant estimates $0<\mu^{0}<L^{0}$, algorithm parameters $p,\bar p, {c},\gamma\in (0,1)$ related to the random stopping condition, the backtracking coefficient $\gamma\in (0,1)$, $F_0,\bar F\in\reals$ and $\delta>0$ as input.} We next 
discuss these parameters: $\epsilon$ is the tolerance, i.e., \sgdab{} is guaranteed to stop w.p.1, and when it terminates, it generates an $\cO(\epsilon)$-stationary point $(\bx_\epsilon,y_\epsilon)$ \rv{with probability at least $1-\bar p$} 
\saa{(see~Theorem \ref{ncsc_thm})}; 
\rv{the inputs $\delta$ and $F_0$ should satisfy $F_0\geq F(\bx_0)$ and $\delta\geq \norm{y_0-y_*(\bx_0)}^2$,} and $\bar F$ is a lower bound on $F^*=\inf_{\bx\in\cX} F(\bx)$ --in many ML problems, the primal objective is to minimize a loss function for which $F^*\geq 0$ (see Remark~\ref{rem:barF}); $\gamma\in (0,1)$ is the backtracking parameter, every time the backtracking condition in line~\ref{algeq:stop} fails, Lipschitz constant estimate $\widetilde L$, concavity modulus estimate $\widetilde\mu$, and the variance bound estimates $\widetilde \sigma_x,\widetilde \sigma_y$ are updated according to $\widetilde{L}\gets \widetilde{L}/\sa{\gamma}$,  $\widetilde\mu \gets \max\{\widetilde\mu\gamma,\underline{\mu}\}$, $\widetilde\sigma_x\gets\min\{\widetilde\sigma_x/\gamma,\bar{\sigma}_x\}$, and $\widetilde\sigma_y\gets\min\{\widetilde\sigma_y/\gamma,\bar{\sigma}_y\}$, respectively, as in line~\ref{eq:L-update} of \cref{alg:GDA-B}; finally, the 
parameter $p\in (0,1)$ bounds the number of backtracking iterations, i.e., \sgdab{} requires at most $\cO(\log(\rv{\cR}))$ backtracking iterations with probability at least $1-p$, \rv{where $\cR\triangleq\max\{\mu^{0}/\mu,L/L^{0},\sigma_x/\sigma_x^{0},\sigma_y/\sigma_y^{0},1\}$}, requiring $\cO(\rv{\hat{\kappa}^3 \hat{L}\log^2(1/\min\{p,\bar p\})}\epsilon^{-4})$ stochastic oracle calls for the stochastic WCSC setting, where $\hat\kappa\triangleq\hat L/\hat \mu$ with $\hat{L}\triangleq\cR L^{0}$ and $\hat{\mu}\triangleq\max\{\underline{\mu}, \mu^{0}/\cR\}$ \rv{(see Theorem~\ref{ncsc_thm} and the discussion in Section~\ref{sec:unknown-bounds})}. 
\rv{In each backtracking iteration, given the estimates $\widetilde{L}$, $\widetilde{\mu}$, $\widetilde{\sigma}_x$ and $\widetilde{\sigma}_y$, 
\sgdab{} makes \saa{$T=\lceil\log_2(3/p)\rceil$ parallel calls} to the randomized block stochastic gradient descent ascent method (\rbsgda) which is displayed in~\cref{alg:GDA} and will be elaborated further in \cref{sec-rb-sgda}.} For the deterministic case with $\sigma_x=\sigma_y=0$, both $p$ and $\bar p$ are set to $0$; hence, we initialize 
$T=1$
together with $M_x=M_y=1$.
\begin{remark}
\label{rem:x0-delta-choice}
\rv{Given $\bx^0\in\dom g$, \sgdab{} requires an input $\delta>0$ such that $\delta\geq \norm{y^0-y_*(\bx^0)}^2$. In case the dual domain is bounded, i.e., $\cD_y \triangleq \sup_{y_1,y_2\in \dom h} \norm{y_1 - y_2} < \infty$, one can simply set $\delta=\cD_y^2$. For instance, for the DRO problem discussed in \cref{rem:barF}, $\dom h$ is the unit simplex; hence, $\cD_y=1$. Alternatively, if $\mu$ is known and $\sigma_y=0$, then one can choose $\delta=\norm{s^0-\grad_y f(\bx^0,y^0)}^2/\mu^2$ for any $s^0\in\partial h(y^0)$, and in practice when $\sigma_y>0$ one can estimate $\grad_y f(\bx^0,y^0)$ with $\frac{1}{M}\sum_{j=1}^{M}\gt_{y} f({\bx}^0,y^0;\zeta_j)$ for some large $M$.} \rv{Another practical strategy that works well in practice is to use adaptive methods to reinitialize $y^0$ to some inexact maximizer $y^0_*$ of the initialization problem $\max_y \cL(\bx^0,y)$.
{For $\mu>0$}, given any $\tilde\epsilon>0$, computing $y^0_*\in\dom h$ such that $\max_{y\in\cY}\cL(\bx^0,y)\leq{\cL(\bx^0,y^0_*)}+\tilde\epsilon$ requires $\cO(\sqrt{\kappa}\log(1/\tilde\epsilon))$ gradient calls for the case $\sigma_y=0$ without requiring to know $L$ and/or $\mu$ --see \cite{calatroni2019backtracking,rebegoldi2022scaled} for details.
On the other hand, for the case $\sigma_y>0$, to compute $y^0_*$ such that $\max_{y\in\cY}\cL(\bx^0,y)\leq\E{\cL(\bx^0,y^0_*)}+\tilde\epsilon$ or $\E{\norm{y^0_*-y_*(\bx^0)}}\leq \tilde\epsilon$ where $\mu,L$ and $\sigma_y$ are possibly unknown, $\tilde\cO(\sigma_y^2/(\mu\tilde\epsilon))$ stochastic oracle calls is sufficient for 
 SPG~\citep{rosasco2020convergence}. Furthermore, one can also use online convex optimization algorithms, e.g., Generalized AdaGrad~\citep{junchinest}, FreeRexMomentum~\citep{NIPS2017_6aed000a} and Epoch-GD method~\citep{hazan2014beyond}, to reinitialize $y^0$; indeed, by online-to-batch conversion the regret rate for these methods can be translated into rates for strongly convex stochastic optimization. Assuming $\cD_y<\infty$ and there exists $G>0$ such that $\norm{\gt_y f(\bx_0,y;\zeta)}\leq G$ for all $y\in\dom h$ w.p. 1, \citep[Theorem 3.4.]{junchinest} shows $\cO(T^{1-\alpha})$ regret for any $\alpha\in(0,1)$, \citep[Corollary~11]{NIPS2017_6aed000a} shows $\tilde\cO(G^2/\mu)$ regret, and \citep[Theorem 9]{hazan2014beyond} assuming $\mu$ is known shows $\cO(G^2/\mu)$ regret; thus, the convergence rate for Generalized AdaGrad, FreeRexMomentum and Epoch-GD are $\cO(1/T^{\alpha})$, $\tilde\cO(1/T)$ and $\cO(1/T)$, respectively. To summarize, in practice one can run these adaptive methods on $\max_y \cL(\bx^0,y)$ to produce an inexact maximizer $y^0_*$; indeed, for small $\tilde\epsilon>0$, after conservatively running these methods $\cO(1/\tilde\epsilon^2)$ iterations, one can set $F_0=\cL(\bx^0,y^0_*)+\tilde\epsilon$, $\delta=2\tilde\epsilon/\tilde\mu$ and $y^0=y^0_*$ -- we found this approach very stable and robust for choosing input parameters $F_0$ and $\delta$ in Section~\ref{sec:numerics} for the numerical experiments.}
\end{remark}%
\begin{remark}
\label{rem:mu-L-initialization}
    \rv{Our convergence results hold for all initialization choices of $0<\mu^0\leq L^0$. That being said, one can use the following approach in practice to initialize these parameter values. Given some $M\in\integers_+$, consider $M$ randomly generated points $\{(\bx_i, y_i)\}_{i=1}^{M}\subset\dom g\times\dom h$. For any 
    $\bz_i = (\bx_i, y_i)$ and $\bz_j = (\bx_j, y_j)$, it holds that $L(\bz_i, \bz_j) \triangleq |\langle \grad f(\bz_i) - \grad f(\bz_j), \bz_i - \bz_j\rangle| /\norm{\bz_i - \bz_j}^2 \leq L$. Thus, we can initialize $L^0=\max\{L(\bz_i, \bz_j): 1\leq i\leq M,\ i+1\leq j\leq M\}$. Furthermore, since $\mu \leq \langle \grad_y f(\bx, y) - \grad_y f(\bx, y'), y-y' \rangle / \norm{y - y'}^2$ for all $\bx\in\dom g$ and $y,y'\in\dom h$, we can use a similar approach to initialize $\mu^0$ as well.}
\end{remark}

\sa{\sgdab, using the batch size of $M_x=\cO(\mg{1+}\rv{\widetilde\sigma_x^2}/\epsilon^2)$ and $M_y=\cO(\mg{1+}\rv{\widetilde\sigma_y^2}/\epsilon^2)$, 
terminates w.p.1 
(see Lemma~\ref{lem:stopping_probability}), and at the termination, it generates $(\bx_\epsilon,y_\epsilon)$ such that $\norm{\tilde G(\bx_\epsilon,y_\epsilon;\bxi_\epsilon)}\leq \epsilon$ w.p.1. 
\mg{Next, we show that the \saa{deterministic gradient map $G(\bx_\epsilon,y_\epsilon)$ is close to its stochastic estimate $\tilde G(\bx_\epsilon,y_\epsilon;\bxi_\epsilon)$} in a probabilistic sense 
with this batchsize selection, \saa{implying that the output $(\bx_\epsilon,y_\epsilon)$ has a small deterministic gradient map with high probability for an appropriate choice of $M_x,M_y$}.} 
}
\begin{lemma}
\label{lem:Chebyshev}
    Suppose \cref{ass1,assumption:noise} hold. For any given sample sizes $M_x,M_y\in\integers_+$, there exists $b:\reals_{++}\times\integers_+\times\integers_+\to\reals_+$ such that $b(r,M_x,M_y)$ {is} decreasing in all the arguments, and 
\begin{equation}
\label{eq:diff-bound}
    \p{\norm{\tilde G(\bx,y,\bxi)-G(\bx,y)}>r}\leq b(r,M_x,M_y),\quad \forall~r>0,\quad \forall~(\bx,y)\in\dom g\times \dom h.
\end{equation} 
    Indeed, under \cref{assumption:noise}, \eqref{eq:diff-bound} holds for $b(r,M_x,M_y)=\frac{1}{r^2}\Big(\frac{\sigma_x^2}{M_x}+\frac{\sigma_y^2}{M_y}\Big)$;
moreover, if \cref{as:bounded-oracle} also holds, then 
$b(r,M_x,M_y)=(n_x+1)\exp\Big(-M_x\frac{\sigma_x^2}{B_x^2}\cdot H\Big(\frac{B_x r}{\sqrt{2}\sigma_x^2}\Big)\Big)+(n_y+1)\exp\Big(-M_y\frac{\sigma_y^2}{B_y^2}\cdot H\Big(\frac{B_y r}{\sqrt{2}\sigma_y^2}\Big)\Big)$ satisfies \eqref{eq:diff-bound} as well.
\end{lemma}
\begin{proof}
    Note that for any given $\bx\in\dom g$, $y\in\dom h$ and random realization $\bxi$, we have {\small$\tilde G_\bx(\bx,y;\bxi)-G_\bx(\bx,y)=\left[\prox{\eta_x g}(\bx-\eta_x s_{\bx}(\bx,y;\bom))-\prox{\eta_x g}(\bx-\eta_x \grad_{\bx} f(\bx,y))\right]/\eta_x$}; thus, 
    the non-expansivity of the proximal map implies that $\norm{\tilde G_\bx(\bx,y;\bom)-G_\bx(\bx,y)}\leq \norm{s_{\bx}(\bx,y;\bom)-\grad_{\bx} f(\bx,y)}$. The same argument also gives us $\norm{\tilde G_y(\bx,y;\bzt)-G_y(\bx,y)}\leq \norm{s_{y}(\bx,y;\bzt)-\grad_{y} f(\bx,y)}$; hence, combining the two, we conclude that the following inequality \rv{holds with probability $1$}:
     \begin{align}
        \label{eq:nonexpansive_prox}
        \norm{\tilde G(\bx,y;\bxi)-G(\bx,y)}\leq \norm{s(\bx,y;\bxi)-\grad f(\bx,y)}.
    \end{align}
    Note that we have $\mathbf{E}[s(\bx,y;\bxi)]=\grad f(\bx,y)$ \mg{as well as} 
 $\mathbf{E}[\norm{s(\bx,y;\bxi)-\grad f(\bx,y)}^2]\leq \sigma_x^2/M_x+\sigma_y^2/M_y$ \mg{from \cref{ineq-grad-var-bound}}. Then, using
 the Chebyshev's inequality leads to the desired result. 
Next, suppose that both \cref{assumption:noise,as:bounded-oracle} hold. Since for any $r>0$,
\begin{equation}
\label{eq:G-xy-split}
{\small
\begin{aligned}
    \MoveEqLeft\p{\norm{\tilde G(\bx,y;\bxi)-G(\bx,y)}>r}\\
    &\leq \p{\norm{\tilde G_\bx(\bx,y;\bom)-G_\bx(\bx,y)}>\frac{r}{\sqrt{2}}}+\p{\norm{\tilde G_\by(\bx,y;\bzt)-G_\by(\bx,y)}>\frac{r}{\sqrt{2}}},
\end{aligned}}%
\end{equation}
using \eqref{eq:nonexpansive_prox} and \eqref{eq:mean-concentration-bounded} in \cref{thm:concentration-bounded}, we get the desired result.
\end{proof}
\begin{remark}
\label{rem:small-eps}
\rv{Suppose the variance bounds $\sigma_x$ and $\sigma_y$ are unknown. We assume that \cref{assumption:noise,as:bounded-oracle} hold, and in \sgdab{}, the sample sizes for the $\ell$-th iteration satisfy $M_x\geq \frac{128}{\epsilon^2}\widetilde\sigma_x^2$ and $M_y\geq \frac{128}{\epsilon^2}\widetilde\sigma_y^2$. This choice implies that
     $M_x \geqslant \frac{128}{\epsilon^{2}}\left(\frac{\sigma_x^0}{\gamma^\ell}\right)^{2} \geqslant \frac{128}{\epsilon^{2}}\left(\sigma_x^0\right)^{2}$ for all iterations $\ell\geq 0$ since the estimate $\tilde\sigma_x$ is updated using $\widetilde\sigma_x\gets\min\{\widetilde\sigma_x/\gamma,\bar{\sigma}_x\}$ starting from any given $\sigma_x^{0}>0$ as long as the stopping condition in \emph{\texttt{line}}~\ref{algeq:stop} in \sgdab{} does not hold. Similarly, $M_y\geq \frac{128}{\epsilon^2}\big(\sigma_y^0\big)^2$ for all $\ell\geq 0$.} 
    
    \rv{For any  $(\bx,y)$, let \rv{$(v^x,v^y)\gets V_{\bxi}(\bx,y)$ as defined in \eqref{eq:sample_variance}.} Thus, it follows from \cref{thm:concentration-bounded} that for any $\alpha\in(0,1)$, $\p{(1+c)\sigma_x^2\geq (1-\frac{1}{M_x})v^x}\geq 1-\alpha$ holds whenever
    {\small
    \begin{align*} 
    \frac{128}{\epsilon^{2}}\left(\sigma_x^{0}\right)^{2} \geqslant \frac{1}{H(c)} \frac{B_x^2}{\sigma_x^2} \log \frac{1}{\alpha}\quad
 \mgrev{\iff} \quad  
 8\sqrt{2}\sigma_x^{0} \frac{\sigma_x}{B_x} \frac{\sqrt{H(c)}}{\sqrt{\log(1/{\alpha})}} \geqslant \epsilon>0.
\end{align*}}%
A similar argument is also valid for $\p{(1+c)\sigma_y^2\geq (1-\frac{1}{M_y})v^y}\geq 1-\alpha$.} 

\rv{For $c\in (0,1]$, it holds that $H(c)\geq \frac{3}{8}c^2$; therefore, for all sufficiently small $\epsilon>0$, i.e., 
{\small
$$0<\epsilon\leq 4\sqrt{3}\cdot c\min\Big\{\sigma_x^0~\frac{\sigma_x}{B_x},\sigma_y^0~\frac{\sigma_y}{B_y}\Big\}/\sqrt{\log(1/\alpha)}$$}%
and $c\in (0,1]$, one has $\p{(1+c)\sigma_x^2\geq (1-\frac{1}{M_x})v_{(t,\ell)}^x}\geq 1-\alpha$ and $\p{(1+c)\sigma_y^2\geq (1-\frac{1}{M_y})v^y_{(t,\ell)}}\geq 1-\alpha$ for all $\ell\geq 0$ and $t=1,\ldots,T$ where $\Big(v_{(t,\ell)}^x,v_{(t,\ell)}^y\Big)=V_{\tilde\bxi_{(t,\ell)}}\Big(\tilde\bx_{(t,\ell)},\tilde y_{(t,\ell)}\Big)$ as defined in \emph{\texttt{\cref{algeq:RBSGDA_runs}}} of \sgdab. Thus, to handle a larger range of tolerance values $\epsilon>0$, one can pick larger initial values for the variance bound estimates $\sigma_x^0,\sigma_y^0>0$, which are free design parameters.}
\end{remark}
\vspace*{-5mm}
\section{Weakly \mg{C}onvex-\mg{S}trongly \mg{C}oncave~(WCSC) \mg{S}etting 
}
\label{sec:WCSC}
\saa{In the next section, we consider \rbsgda{} using Jacobi-type updates. In~\cref{sec:GS}, we also consider \rbsgda{} with Gauss-Seidel updates.}
\subsection{\sa{\rbsgda} with Jacobi Updates}\label{sec-rb-sgda}
\sa{Let $\{i_k\}_{k\geq 0}$ be a random sequence with elements i.i.d. following $\cU[1,N]$. For $k\geq 0$, consider the following \mg{
\saa{proximal 
GDA algorithm with Jacobi-type} block updates}:} 
\mg{
\begin{eqnarray*}\bx_{i_k}^{k+1} &=&  \prox{\eta_x g_{i_k}}\Big(\bx_{i_k}^k-\eta_{x} s_{x_{i_k}}(\bx^k,y^k;\bom^k)\Big), \quad\mbox{and}\quad
\bx_{\saa{i}}^{k+1} =\bx_{\saa{i}}^{k} \quad \mbox{if} \quad \saa{i}\neq i_k, \\
y^{k+1} &=& \prox{\eta_y h}\Big( y^{k}+\eta_y s_y(\bx^k,y^k;\bzt^k)\Big),
\end{eqnarray*}
}%
\mg{where the source of randomness ($\bom^k$ and $\bzt^k$) 
is due to the stochastic oracle satisfying \cref{assumption:noise}.  
}%
\mg{
\saa{Here, only the randomly chosen primal block $i_k$, \rv{i.e., $\bx_{i_k}$,} is updated; 
therefore,} we call this method \textit{random-block stochastic} 
\saa{GDA} (\rbsgda)}. 

\mg{We next 
\saa{define} the natural filtration generated by \saa{the iterate sequence}.}
\begin{defn}
\label{def:filtration}
\sa{Let $\cF^k\triangleq\sigma(\{\bz^k\})$
denote the $\sigma$-algebra generated by 
$\bz^k=(\bx^k,y^k)$ for $k\geq 0$.}
\end{defn}
The following lemma, \mg{conditioned on 
\saa{$\cF^k$,} \saa{provides us with the useful identities in \eqref{lemeq:cond-exp-Gik} and \eqref{lemeq:cond-exp-gradPhi_ik},} and it controls the variance of the stochastic gradients \saa{based on sample sizes $M_x$ \mg{and} $M_y$.} 
}
\begin{lemma}
\label{lem:cond-exp}
    \rv{Given arbitrary mini-batch sizes $M_x,M_y\in\integers_{+\mgrev{+}}$}, \sa{it holds for all $k\geq 0$ that
    {\small
    \begin{subequations}
    \label{lemeq:cond-exp}
        \begin{align}
        &\Ek{\|\tilde G_{x_{i_k}}(\bx^k,y^k;\bom^k)\|^2} = \frac{1}{N}\Ek{\norm{\tilde G_{\bx}(\bx^k,y^k;\bom^k)}^2}, \label{lemeq:cond-exp-Gik}\\
        &\Ek{\norm{\grad_{x_{i_k}}f(\bx^k,y^k)-\grad_{x_{i_k}} \Phi(\bx^k)}^2}=\frac{1}{N}\norm{\grad_{\bx}f(\bx^k,y^k)-\grad \Phi(\bx^k)}^2,\label{lemeq:cond-exp-gradPhi_ik}\\
        &\Ek{\norm{s_{x_{i_k}}(\bx^k,y^k;\bom^k)-\grad_{x_{i_k}} f(\bx^k,y^k)}^2}\leq \frac{1}{N}\frac{\sigma_x^2}{M_x},
        \label{lemeq:cond-exp-variance}\\
        &\Ek{\norm{s_y(\bx^k,y^k;\bzt^k)-\grad_{y} f(\bx^k,y^k)}^2}\leq \frac{\sigma_y^2}{M_y}.
        \label{lemeq:cond-exp-variance-y}
    \end{align}
    \end{subequations}}}%
    \vspace*{-0.6cm}
\end{lemma}
\begin{proof}
    \sa{The results trivially follow from the tower law of expectation.}
\end{proof}
\mg{
\mg{We next show that} after a
proximal gradient \saa{block} update \saa{for any given block $i^*\in\cN$,}  
the stochastic gradient map corresponding to \saa{the chosen} block, i.e., \saa{$\widetilde G_{x_{i^*}}$, can be controlled by three terms: (\textit{i}) change in the primal function, (\textit{ii}) the variance of the stochastic gradients, and (\textit{iii}) a term 
\saa{related to the difference between $\grad_{x_{i^*}} f$ and $\grad_{x_{i^*}}\Phi$}.
}
}
\begin{lemma}
\label{lem:x-norm-bound}
    Given $(\bar\bx,\bar{y})\in \dom g \times \dom h$ and \sa{$i^*\in\cN$}, let $\bar\bx^+\in\cX$ be such that \sa{$\bar x^+_i=\bar{x}_i$ for $i\in\cN\setminus\{i^*\}$, and $\bar x^+_{i^*}=\prox{{\eta_x} g_{i^*}}\Big(\bar{x}_{i^*}-\eta_x s_{x_{i^*}}(\bar\bx,\bar y; \bom)\Big)$, where \sa{$ \bom=[ \omega_j]_{j=1}^{M_x}$ and} $\{ \omega_j\}_{j=1}^{M_x}$ are i.i.d. random variables such that $\{\gt_{x_{i^*}}f(\bar\bx,\bar{y};{\omega}_j)\}_{j=1}^{M_x}$ satisfy \cref{assumption:noise}.} Then,
\begin{equation}
{\small
    \begin{aligned}
    \label{eq:x-norm-bound_tmpp_gdab}
          &\Big(\frac{\eta_x}{2}-\eta_x^2 L\kappa\Big)\sa{\mathbf{E}_{\bom}}[\norm{\tilde G_{x_{i^*}}(\bar{\bx},\bar{y};\sa{\bom})}^2]\\
          &\leq   F(\bar\bx) - \sa{\mathbf{E}_{\bom}}\Big[F(\bar\bx^+)-\frac{\eta_x}{2}\norm{s_{x_{i^*}}(\bar\bx,\bar y;\bom)-\grad_{x_{i^*}} f(\bar\bx,\bar{y})}^2\Big]+\frac{\eta_x}{2}\norm{\grad_{x_{i^*}}f(\bar\bx,\bar{y})-\grad_{x_{i^*}} \Phi(\bar\bx)}^2.
    \end{aligned}}%
\end{equation}
\end{lemma}
\begin{proof}
   Since $\Phi$ is $\sa{L}(1+\kappa)$-smooth from \cite[Lemma A.5]{nouiehed2019solving}, \sa{the definition $\bar\bx^+$ implies that}
\begin{equation}
{\small
   \begin{aligned}
          \label{eq:Phi-smooth_block}
    \Phi(\bar\bx^+)&\leq \Phi(\bar\bx)+\fprod{\grad \Phi(\bar\bx), \bar\bx^+-\bar\bx}+ L\kappa   \|\bar\bx^+-\bar\bx\|^2 \\
    &=\Phi(\bar\bx)+\fprod{\grad_{x_{i^*}} \Phi(\bar\bx), \bar x^+_{i^*}-\bar{x}_{i^*}}+ L\kappa   \|\bar x_{i^*}^+-\bar{x}_{i^*}\|^2,
   \end{aligned}}%
\end{equation}
\sa{where we used \rv{$(\bar\bx^+-\bar\bx)_i=\mathbf{0}$ for $i\in\cN\setminus\{i^*\}$} and \mg{$\kappa\geq 1$}.}
The definition of $\bar\bx^+$ \sa{also} implies that $\frac{1}{\eta_x}(\bar{x}_{i^*}- \bar x^+_{i^*})-s_{x_{i^*}}(\bar\bx,\bar y;\bom)\in\sa{\partial g_{i^*}(\bar x^+_{i^*})}$ and \sa{$g(\bar \bx^+)-g(\bar\bx) = g_{i^*}(\bar x^+_{i^*}) - g_{i^*}(\bar x_{i^*})$;} therefore, we get
\sa{\small
\begin{equation}
\hspace{-0.07in}
    \begin{aligned}
    F(\bar\bx^+)
    &\leq F(\bar\bx)+\fprod{\grad_{x_{i^*}} \Phi(\bar\bx), \bar x^+_{i^*}-\bar x_{i^*}}
    +L\kappa   \| \bar x_{i^*}^+-\bar x_{i^*}\|^2
    +\fprod{\frac{1}{\eta_x}(\bar x_{i^*}- \bar x^+_{i^*})- s_{x_{i^*}}(\bar\bx,\bar y;\bom),\bar x^+_{i^*}-\bar x_{i^*}}\\
    &= F(\bar\bx)\saa{+}\fprod{\tilde G_{x_{i^*}}(\bar\bx,\bar{y};\bom), \bar x^+_{i^*}-\bar x_{i^*}}+L\kappa \| \bar x^+_{i^*}-\bar x_{i^*}\|^2+\fprod{\grad_{x_{i^*}} \Phi(\bar\bx)-s_{x_{i^*}}(\bar\bx,\bar y;\bom), \bar x^+_{i^*}-\bar x_{i^*}}\\
    &= F(\bar\bx)-\Big(\eta_x-\eta_x^2L\kappa\Big)\norm{\tilde G_{x_{i^*}}(\bar\bx,\bar{y};\bom)}^2+\eta_x\fprod{s_{x_{i^*}}(\bar\bx,\bar y;\bom)-\grad_{x_{i^*}} \Phi(\bar\bx), \tilde G_{x_{i^*}}(\bar\bx,\bar{y};\bom)}\\
    &\leq  F(\bar\bx)-\Big(\eta_x-\eta_x^2 L\kappa\Big)\norm{\tilde G_{x_{i^*}}(\bar\bx,\bar{y};\bom)}^2+\frac{\eta_x}{2}\norm{\tilde G_{x_{i^*}}(\bar\bx,\bar{y};\bom)}^2+\frac{\eta_x}{2}\norm{s_{x_{i^*}}(\bar\bx,\bar y;\bom)-\grad \Phi_{x_{i^*}}(\bar\bx)}^2\\
    &=  F(\bar\bx)-\Big(\frac{\eta_x}{2}-\eta_x^2L\kappa\Big)\norm{\tilde G_{x_{i^*}}(\bar\bx,\bar{y};\bom)}^2 +\frac{\eta_x}{2}\norm{s_{x_{i^*}}(\bar\bx,\bar y;\bom)-\grad_{x_{i^*}}f(\bar\bx,\bar{y})}^2\\
    &\quad +\frac{\eta_x}{2}\norm{\grad_{x_{i^*}}f(\bar\bx,\bar{y})-\grad \Phi_{x_{i^*}}(\bar\bx)}^2+\eta_x\fprod{s_{x_{i^*}}(\bar\bx,\bar y;\bom)-\grad_{x_{i^*}}f(\bar\bx,\bar{y}),~\grad_{x_{i^*}}f(\bar\bx,\bar{y})-\grad \Phi_{x_{i^*}}(\bar\bx)},
    \end{aligned}
\end{equation}}%
where the first inequality follows from \eqref{eq:Phi-smooth_block} and 
convexity of $g$, the equalities follow from the definition of \sa{$\tilde G_\bx(\bar\bx,\bar{y};\bom)$}, and the last inequality follows from Young's inequality. Rearranging terms and then taking the expectation of both sides of the inequality give the desired result.
\end{proof}
In the following lemma, we \saa{consider the effect of the proximal gradient ascent step in the dual, and} provide a fundamental inequality for proving the convergence result. 
\begin{theorem}\label{thm_H}
For any given $\bar\bx\in\dom g$, let $H:\cY\to\reals\cup\{+\infty\}$ such that $H(y)\triangleq h(y)-f(\bar\bx,y)$ for $y\in\cY$. Then, for any given $\bar{y}\in\dom h$,
\begin{equation}
    \label{eq:str-cvx-H}
    \begin{split}
    H(y)\geq &H( \bar y^+)-\fprod{\tilde G_y(\bar\bx,\bar{y};\sa{\bzt}), y-\bar{y}}+\fprod{s_y(\bar\bx,\bar y;\bzt)-\grad_y f(\bar \bx,\bar y), y -\bar y^+}\\
    &+\frac{\eta_y}{2}(2-L\eta_y)\norm{\tilde G_y(\bar\bx,\bar{y};\sa{\bzt})}^2+\frac{\mu}{2}\norm{ y-\bar{y}}^2
    \end{split}
\end{equation}
holds for all $y\in\cY$, where $\bar y^+=\prox{\eta_y h}\Big(\bar{y}+\eta_y s_y(\bar\bx,\bar y;\bzt)\Big)$ \sa{for $\bzt=[\zeta_j]_{j=1}^{M_y}$ such that} $\{\zeta_j\}_{j=1}^{M_y}$ are i.i.d. random variables and $\{\gt_y f(\bar\bx,\bar y;{\zeta}_j)\}_{j=1}^{M_y}$ satisfy \cref{assumption:noise}.
\end{theorem}
\begin{proof}
     Define $\widetilde H:\cY\to\reals\cup\{+\infty\}$ such that $\widetilde H(y)\triangleq h(y)-f(\bar\bx,\bar{y})-\fprod{s_y(\bar\bx,\bar y;\bzt), y-\bar{y}}+\frac{1}{2\eta_y}\norm{ y-\bar{y}}^2$ for all $ y\in\cY$. Clearly, $ \bar y^+=\argmin_{ y\in\cY}\widetilde H( y)$. The first-order optimality condition implies that
    \begin{align*}
        s_y(\bar\bx,\bar y;\bzt)-\tilde G_y(\bar\bx,\bar{y};\sa{\bzt})=s_y(\bar\bx,\bar y;\bzt)+\frac{1}{\eta_y}(\bar{y}- \bar y^+)\in\partial h( \bar y^+);
    \end{align*}
    therefore, using the convexity of $h$, \saa{for any $y\in\cY$,} we get
    \begin{align}
    \label{eq:subgradient_h}
    \begin{split}
        h(y)&\geq h( \bar y^+)+\fprod{s_y(\bar\bx,\bar y;\bzt)-\tilde G_y(\bar\bx,\bar{y};\sa{\bzt}), y- \bar y^+}\\
        & = h(\bar y^+)+\fprod{s_y (\bar\bx, \bar y;\bzt), y-\bar y} + \fprod{s_y (\bar \bx,\bar y;\bzt),\bar y-\bar y^+} - \fprod{ \tilde G_y(\bar \bx,\bar y;\bzt), y-\bar y^+}.
    \end{split}
    \end{align}
    Moreover, since $f(\bar\bx,\cdot)$ is $\mu$-strongly concave, \saa{for any $y\in\cY$,} we also get
    \begin{equation}
    \label{eq:str-concavity}
        \begin{split}
            -f(\bar\bx, y)-\frac{\mu}{2}\norm{y-\bar{y}}^2
            &\geq -f(\bar\bx,\bar{y})-\fprod{\grad_yf(\bar\bx,\bar{y}), y-\bar{y}}.\\
        \end{split}
    \end{equation}
   Next, summing \eqref{eq:subgradient_h} and \eqref{eq:str-concavity} gives
   \begin{equation}
   \label{eq:H-bound-1}
   {\small
    \begin{split}
       \MoveEqLeft{H(y)-\frac{\mu}{2}\norm{y-\bar{y}}^2}\\
        &\geq \widetilde H(\bar y^+)-\frac{1}{2\eta_y}\norm{ \bar y^+-\bar{y}}^2-\fprod{\tilde G_y(\bar\bx,\bar{y};\bzt), y- \bar y^+}+\fprod{s_y (\bar \bx,\bar y;\bzt)-\grad_y f(\bar\bx,\bar y), y-\bar y}\\
        & = \widetilde H( \bar y^+)+\frac{\eta_y}{2}\norm{\tilde G_y(\bar\bx,\bar{y};\bzt)}^2-\fprod{\tilde G_y(\bar\bx,\bar{y};\bzt), y-\bar{y}}+\fprod{s_y (\bar \bx,\bar y;\bzt)-\grad_y f(\bar\bx,\bar y), y-\bar y}.
    \end{split}}
   \end{equation}
    Since $\grad f$ is $L$-Lipschitz,
    \begin{equation}
        \label{eq:H-bound-2}
        \begin{split}
        \widetilde H(\bar y^+)
        &=h(\bar y^+)-f(\bar\bx,\bar{y})-\fprod{\grad_y f(\bar\bx,\bar{y}), \bar y^+-\bar{y}}+\frac{L}{2}\norm{\bar y^+-\bar{y}}^2\\
        &\quad+\fprod{s_y(\bar\bx,\bar y;\bzt)-\grad_y f(\bar \bx,\bar y), \bar y -\bar y^+}+\frac{1}{2}\Big(\frac{1}{\eta_y}-L\Big)\norm{\bar y^+-\bar{y}}^2\\
        &\geq H(\bar y^+)+\frac{1}{2}\Big(\frac{1}{\eta_y}-L\Big)\norm{\bar y^+-\bar{y}}^2+\fprod{s_y(\bar\bx,\bar y;\bzt)-\grad_y f(\bar \bx,\bar y), \bar y -\bar y^+}\\
        &= H( \bar y^+)+\frac{\eta_y}{2}(1-L\eta_y)\norm{\tilde G_y(\bar\bx,\bar{y};\bzt)}^2+\fprod{s_y(\bar\bx,\bar y;\bzt)-\grad_y f(\bar \bx,\bar y), \bar y -\bar y^+}.\\
        \end{split}
    \end{equation}
    Thus, \eqref{eq:H-bound-1} and \eqref{eq:H-bound-2} together imply the desired inequality in~\eqref{eq:str-cvx-H}.
\end{proof}
\begin{corollary}
\label{cor:str-concavity}
For any given $(\bar\bx,\bar{y})\in\dom g\times\dom h$ \sa{and $\eta_y>0$,}
\begin{equation}
{\small
\begin{aligned}
    \MoveEqLeft 2\eta_y\fprod{\tilde G_y(\bar\bx,\bar{y};\bzt),\bar{y}-{y}_{*}(\bar\bx)}+\eta_y^2\norm{\tilde G_y(\bar\bx,\bar{y};\bzt)}^2\\
    &\leq-\eta_y\mu\norm{\bar{y}-{y}_{*}(\bar\bx)}^2-\eta_y^2(1-L\eta_y)\norm{\tilde G_y(\bar\bx,\bar{y};\bzt)}^2+2\eta_y \fprod{s_y(\bar\bx,\bar y;\bzt)-\grad_y f(\bar\bx,\bar y),\bar y^+-y_*(\bar\bx)}.
\end{aligned}}
\end{equation}
\end{corollary}
\begin{proof}
\sa{The definition of $y_*(\cdot)$ 
gives $H(y_*(\bar\bx))\leq H(\bar y^+)$; therefore, substituting $y=y_*(\bar\bx)$ in \eqref{eq:str-cvx-H}, rearranging the terms and multiplying both sides by $2\eta_y$ immediately implies the desired inequality.}
\end{proof}
Next, we bound the average squared norm of stochastic gradient maps over past iterations in expectation by a sum of three terms: 
\textit{(i)} expected primal 
\rv{descent}, i.e.,  $
\E{F(\bx^0)-F(\bx^K)}$, \textit{(ii)} 
\rv{cumulative} dual {suboptimality},
i.e., $
\sum_{k=0}^{K-1}\E{\|y_*(\bx^k) - y^k\|^2}$
, and \textit{(iii)} a constant noise term, i.e., $\delta^2_x/M_x+\delta^2_y/M_y$, \rv{for any given arbitrary mini-batch sizes $M_x,M_y\in\integers_{+\mgrev{+}}$.}
\begin{lemma}\label{lemma:primal-bound-gda}
For any $K\geq 1$ \mg{and $\eta_x>0$}, the \rbsgda{}
iterate sequence $\{\bx^{k},y^{k}\}_{k\geq 0}$ satisfies
\sa{
\begin{equation}\label{eq:partial_sum_grad_delta_tmpp_jac}
    \begin{aligned}
        &\Big(\frac{1}{2} - \eta_xL\kappa\Big)\sum_{k=0}^{K-1}\mathbf{E}\Big[\norm{\tilde G_{\bx}(\bx^{k},y^{k};\bom^k)}^2\Big]+\frac{1}{4} \sum_{k=0}^{K-1}\mathbf{E}\Big[\norm{\tilde G_y(\bx^{k},y^{k};\bzt^k)}^2\Big]\\
        &\leq \mathbf{E}\Big[\frac{N}{\eta_x}\Big(F(\bx^0)-F(\bx^K)\Big)
        +\frac{3}{2} \Big(L^2+\frac{1}{\eta_y^2}\Big)
        \sum_{k=0}^{K-1} \delta^{k}\Big]+K\Big(\frac{\sigma_x^2}{M_x}+ \frac{\sigma_y^2}{M_y}\Big),
    \end{aligned}
\end{equation}}%
where $\sa{\delta^k\triangleq}\|{y}^k_{*}-{y}^k\|^2$ and ${y}^k_{*}\triangleq{y}_{*}(\bx^k)$ \sa{for $k\geq 0$}.
\end{lemma}
\begin{proof}
Fix $k\geq 1$, since $(\bx^{k}, y^{k})\in\dom g\times \dom h$, we can invoke \sa{\rv{Lemma}~\ref{lem:x-norm-bound} with $(\bar\bx,\bar{y})=(\bx^{k}, y^{k})$ and $i^*=i_k$}, implying $\saa{\bar\bx^+}=\sa{\bx^{k+1}}$; thus, it follows from \eqref{eq:x-norm-bound_tmpp_gdab} and \rv{Lemma}~\ref{lem:cond-exp} that
\begin{align*}
\begin{split}
    &\frac{1}{N}\Big(\frac{\eta_x}{2}-\eta_x^2 L\kappa\Big)\Ek{\norm{\tilde G_{\bx}(\bx^k,y^k;\bom^k)}^2}\\
    & \leq   F(\bx^k) - \Ek{F(\bx^{k+1})}+\frac{\eta_x}{\sa{2}M_x N}\sigma_x^2+\frac{\eta_x}{\sa{2}N}\norm{\grad\Phi(\bx^k)-\grad_{\bx} f(\bx^k,y^k)}^2\\
    &\leq F(\bx^{k}) - \Ek{F(\bx^{k+1})}+\frac{\eta_x}{2M_x N}\sigma_x^2+\frac{\eta_x}{2N}L^2\delta^k,\\
\end{split}
\end{align*}
where in the second inequality we use $L$-Lipschitz continuity of $\grad_{\bx} f(\bx^{k},\cdot)$, the definition of \saa{$\delta^{k}$}, and the identity $\grad\Phi(\bx^{k})=\grad_{\bx} f(\bx^{k},{y_*^k})$. Thus, \saa{dividing both sides by $\frac{\eta_x}{N}$ and taking the total expectation,} we get
\sa{\begin{align}
\label{eq:xt-norm-bound_tmpp}
\begin{split}
    \mg{\Big(\frac{1}{2}-\eta_x L\kappa\Big)}\E{\norm{\tilde G_{\bx}(\bx^k,y^k;\bom^k)}^2}\leq \E{\frac{\mg{N}}{\eta_x}\Big(F(\bx^{k}) - F(\bx^{k+1})\Big)+\frac{1}{\mg{2}}L^2\delta^k}+\frac{1}{\mg{2}}\frac{\sigma_x^2}{M_x}.
\end{split}
\end{align}}%
Furthermore, 
${y}_*^k\triangleq {y}_{*}(\bx^{k})$ implies that \saa{$G_y(\bx^{k},{y}_*^k)=\mathbf{0}$, i.e., ${y}_*^k=\prox{\sa{\eta_y} h}\Big({y}_*^k+\eta_y
    \grad_y f(\bx^{k},{y}_*^k)\Big)$.}
Therefore, 
\begin{equation}
{\small
\begin{aligned}
\MoveEqLeft{\frac{1}{4}\norm{\tilde G_y(\bx^{k},y^{k};\bzt^k)}^2}\\ 
    &=\frac{1}{4}\norm{\tilde G_y(\bx^{k},y^{k};\bzt^k)\rv{+}\Big({y}_*^k -\prox{\sa{\eta_y} h}\Big({y}_*^k+\eta_y
    \grad_y f(\bx^{k},{y}_*^k)\Big)\Big)/\eta_y}^2\\ 
    &=\frac{1}{4\eta_y^2}\norm{\prox{\eta_y h} \Big(y^{k}+\eta_y s_y(\bx^k,y^k;\bzt^k)\Big)-\prox{\eta_y h} \Big({y^{k}_*}+\eta_y \grad_y f(\bx^{k},{y}_*^k)\Big)+{y}_*^k-y^{k}}^2\\ 
    &\leq\frac{1}{4\eta_y^2}\Big(2\norm{y^{k}-{y^{k}_*}+\eta_y s_y(\bx^k,y^k;\bzt^k)-\eta_y \grad_y f(\bx^{k},{y}_*^k)}^2+2\norm{{y}_*^k-y^{k}}^2\Big)\\ 
    &\leq\frac{1}{4\eta_y^2}\Big(4\norm{y^{k}-{y^{k}_*}+\eta_y \grad_y f(\bx^{k},{y}^k)-\eta_y \grad_y f(\bx^{k},{y}_*^k)}^2\\ 
    &\qquad\qquad{+4\eta_y^2\norm{s_y(\bx^k,y^k;\bzt^k)-\grad_y f(\bx^{k},{y}^k)}^2+2\norm{{y}_*^k-y^{k}}^2\Big)}\\ 
    &\leq \norm{ s_y ( \bx^k,y^k; \bzt^k)-\grad_y f(\bx^{k},y^{k})}^2
    +\Big(L^2 + \frac{3}{2\eta_y^2}\Big) \delta^{k},
\end{aligned}}%
\label{eq:yt-norm-bound_tmpp}
\end{equation}
where we used non-expansivity of the proximal map, \sa{concavity of $f(\bx_{k},\cdot)$ (which implies that $\norm{y^{k}-{y^{k}_*}+\eta_y \grad_y f(\bx^{k},{y}^k)-\eta_y \grad_y f(\bx^{k},{y}_*^k)}^2\leq \norm{y^k-y_*^k}^2+\eta_y^2\norm{\grad_y f(\bx^{k},{y}^k)-\grad_y f(\bx^{k},{y}_*^k)}^2$),} $L$-Lipschitz continuity of $\grad_y f(\bx_{k},\cdot)$, the definition of $\delta_{k}$, and the triangular inequality. Then by taking the \sa{conditional} expectation of both sides, we can conclude that
\begin{equation}\label{eq:expect-gy-bound}
    \begin{aligned}
        &\frac{1}{4}\Ek{\norm{\tilde G_y (\bx^{k},y^{k};\bzt^k)}^2 }\\
        &\leq\Ek{
         {\norm{ s_y ( \bx^k,y^k; \bzt^k)-\grad_y f(\bx^{k},y^{k})}^2}
        } +(L^2 + \frac{3}{2\eta_y^2}) \sa{\delta^{k}}
        \leq \frac{\sigma_y^2}{M_y} + \Big(L^2 + \frac{3}{2\eta_y^2}\Big) \sa{\delta^{k}}.
    \end{aligned}
\end{equation}
\sa{Next, combining \eqref{eq:xt-norm-bound_tmpp} and \sa{the total expectation of} \eqref{eq:expect-gy-bound}, then summing} \saa{the resulting inequality from $k=0$ to $K-1$, will lead} to the desired inequality.
\end{proof}
\mg{
Next, we investigate the decay properties of the dual error \saa{$\delta^k\triangleq \|y_*^k - y^k\|^2$} in expectation.}
\begin{lemma}\label{lemma:deltat_tmpp}
    Suppose \rv{$\eta_y \in(0, 1/\widetilde\mu)$ for some $0<\widetilde\mu\leq\mu$}, and let $a\triangleq \frac{1}{2}\frac{\rv{\widetilde\mu}\eta_y}{1-\rv{\widetilde\mu}\eta_y}$. Then, 
    \begin{equation*}
    {\small
    \begin{aligned}
        \mathbf{E}\Big[\sa{\delta^k}\Big]\leq& \mathbf{E}\Big[\sa{\delta^0} \sa{C_1^k}+ \sum_{i=1}^k C_1^{k-i}\Big(C_2\|\tilde G_{\bx}(\bx^{i-1},y^{i-1};\bom^{i-1})\|^2/N -C_3 \|\tilde G_y(\bx^{i-1},y^{i-1};\bzt^{i-1})\|^2 \Big)\Big] + \sa{C_4}\sa{\frac{\sigma_y^2}{M_y}}\sum_{i=0}^{k-1}C_1^{i},
    \end{aligned}}%
    \end{equation*}
     for $k\geq 1$, where $\sa{C_1}\triangleq 
     \saa{(1+a)(1-\rv{\widetilde\mu}\eta_y)} \mg{< 1}$, $\sa{C_2}\triangleq \frac{2 \kappa^2}{\rv{\widetilde\mu}}\frac{\eta_x^2}{\eta_y}$, $\sa{C_3}\triangleq 
     \saa{(1+a)\eta_y^2(1-L\eta_y)}$, \saa{$C_4\triangleq 
     2(1+a)\eta_y^2$.}
\end{lemma}
\begin{proof}
Indeed, \sa{for all} $a>0$, 
\begin{equation}\label{eq:bound-deltat-1_tmpp}
    \begin{aligned}
    \delta^k =&  \|{y}_*(\bx^k) - {y}_*(\bx^{k-1}) + {y}_*(\bx^{k-1}) - y^k\|^2\\
        \leq &
        (1+a)\| {y}_*(\bx^{k-1}) - y^k \|^2
        +
         \left(1+\frac{1}{a}\right)\|{y}_*(\bx^k) - {y}_*(\bx^{k-1}) \|^2
         \\
        \leq &
        (1+a)\| {y}_*(\bx^{k-1}) - y^k \|^2
        +
         \left(1+\frac{1}{a}\right)\kappa^2\|\bx^k - \bx^{k-1} \|^2,
    \end{aligned}
\end{equation}
where the second inequality is 
\sa{due to ${y}_{*}(\cdot)$ being $\kappa$-Lipschitz from \cite[Lemma A.3]{nouiehed2019solving}.}
Next, using the fact \saa{that} $\|\bx^k-\bx^{k-1}\|^2=\eta_x^2\| \tilde G_{x_{i_{k-1}}}(\bx^{k-1},y^{k-1};\bom^{k-1})\|^2$ for $k\geq 1$, we obtain
\begin{equation} \label{eq:bound-deltat-2_tmpp}
    \begin{aligned}
          \delta^k
        \leq &(1+a)\| {y}_*(\bx^{k-1}) - y^k \|^2
         +
         \left(1+\frac{1}{a}\right)\kappa^2\eta_x^2\| \tilde G_{x_{i_{k-1}}}(\bx^{k-1},y^{k-1};\bom^{k-1})\|^2,\quad\forall~k\geq 1.
    \end{aligned}
\end{equation}
\sa{Moreover,} since $y^k=
\sa{y^{k-1}+\eta_y \tilde G_y(\bx^{k-1},y^{k-1};\bzt^{k-1})}$ for $k\geq 1$, \sa{we also have}
\begin{equation*}
    \begin{aligned}
        \|{y}_{*}(\bx^{k-1})-y^k\|^2&= \|y^{k-1}+\eta_y \tilde G_y(\bx^{k-1},y^{k-1};\bzt^{k-1})-{y}_{*}(\bx^{k-1})\|^2\\
        &=
        \delta^{k-1}+2\eta_y \fprod{ \tilde G_y(\bx^{k-1},y^{k-1};\bzt^{k-1}),~y^{k-1}-{y}_{*}(\bx^{k-1})}+\eta_y^2\|\tilde G_y(\bx^{k-1},y^{k-1};\bzt^{k-1})\|^2.
    \end{aligned}
\end{equation*}
Now, invoking Corollary~\ref{cor:str-concavity} with $(\bar\bx,\bar{y})=(\bx^{k-1},y^{k-1})$ implies that
\begin{equation*}
    \begin{aligned}
        \|{y}_{*}(\bx^{k-1})-y^k\|^2&\leq (1-\mu \eta_y)\delta^{k-1} - \eta_y^2(1- \eta_y L)\|\tilde G_y(\bx^{k-1},y^{k-1};\bzt^{k-1})\|^2\\
        &\quad +2\eta_y \fprod{s_y(\bx^{k-1}, y^{k-1};\bzt^{k-1})-\grad_y f(\bx^{k-1}, y^{k-1}),y^{k}-y_*(\bx^{k-1})}.
    \end{aligned}
\end{equation*}
\saa{Therefore,} \rv{since $\widetilde\mu\leq \mu$,} we have
\begin{equation}\label{eq14_tmpp}
    \begin{aligned}
        \delta^k&\leq (1+a)\left(1-\rv{\widetilde\mu} \eta_y\right)\delta^{k-1} -(1+a)\eta_y^2(1- \eta_y L)\|\tilde G_y(\bx^{k-1},y^{k-1};\bzt^{k-1})\|^2
        \\
        &\quad + 2(1+a)\eta_y \fprod{s_y(\bx^{k-1},y^{k-1};\bzt^{k-1})-\grad_y f(\bx^{k-1}, y^{k-1}),y^k-y_*(\bx^{k-1})}\\
        &\quad + \left(1+\frac{1}{a}\right)\kappa^2\eta_x^2\| \tilde G_{x_{i_{k-1}}}(\bx^{k-1},y^{k-1};\bom^{k-1})\|^2.
    \end{aligned}
\end{equation}
Let \rv{$a = \frac{1}{2}\frac{\widetilde\mu\eta_y}{1-\widetilde\mu \eta_y}$ \sa{--we have $a> 0$ since \rv{$\eta_y< 1/\widetilde\mu$}}.
This choice} implies that
\begin{equation}\label{eq_delta_one_tmpp}
{\small
    \begin{aligned}
        \mathbf{E}[\delta^k~|~\cF^{k-1}]
        &\leq C_1\delta^{k-1} + \mathbf{E}\Big[C_2\| \tilde G_{x_{i_{k-1}}}(\bx^{k-1},y^{k-1};\bom^{k-1})\|^2 - C_3\|\tilde G_y(\bx^{k-1},y^{k-1};\bzt^{k-1})\|^2~\Big|~\cF^{k-1}\Big]\\
        &\quad 
        +\saa{\frac{1}{\eta_y}C_4}\mathbf{E}\Big[\fprod{s_y(\bx^{k-1}, y^{k-1};\bzt^{k-1})-\grad_y f(\bx^{k-1}, y^{k-1}),~y^k-y_*(\bx^{k-1})}~\Big|~\cF^{k-1}\Big].
    \end{aligned}}%
\end{equation}
{Let $\hat{y}^k\triangleq\prox{\sa{\eta_y} h}\Big(y^{k-1}+\eta_y \grad_y f(\bx^{k-1},y^{k-1})\Big)$; 
the inner product 
in \cref{eq_delta_one_tmpp} can be bounded as
\begin{equation}\label{eq:cond-bound-y-bias}
{\small
    \begin{aligned}
        \MoveEqLeft\mathbf{E}\Big[
            \fprod{
                s_y(\bx^{k-1}, y^{k-1};\bzt^{k-1})-\grad_y f(\bx^{k-1}, y^{k-1}),~y^k-y_*(\bx^{k-1})}~\Big|~\sa{\cF^{k-1}}
        \Big]
        \\
        =& \mathbf{E}\Big[
            \fprod{
                s_y(\bx^{k-1}, y^{k-1};\bzt^{k-1})-\grad_y f(\bx^{k-1}, y^{k-1}),~y^k-\hat{y}^k)}~\Big|~\sa{\cF^{k-1}}
        \Big]
        \\
        \leq
        & \eta_y \mathbf{E}\Big[
            \|s_y(\bx^{k-1}, y^{k-1};\bzt^{k-1})-\grad_y f(\bx^{k-1}, y^{k-1})\|^2~\Big|~\sa{\cF^{k-1}}
        \Big]
        \leq \eta_y\frac{\sigma_y^2}{M_y},
    \end{aligned}}%
\end{equation}}%
\sa{where in the first equality we used the fact that $\mathbf{E}\Big[s_y(\bx^{k-1}, y^{k-1};\bzt^{k-1})~\Big|~\cF^{k-1}\Big]=\grad_y f(\bx^{k-1},y^{k-1})$ and $\saa{\hat y^k}-y_*(\bx^{k-1})$ is $\cF^{k-1}$-measurable; in the first inequality, we used Cauchy-Schwarz inequality and the nonexpansivity of the proximal map; and the final inequality follows from \cref{assumption:noise}.
Therefore, \cref{eq_delta_one_tmpp} and \cref{eq:cond-bound-y-bias} together with \cref{lemeq:cond-exp-Gik} imply that}
\begin{equation}\label{eq_delta_one_tmpp_exp}
{\small
    \begin{aligned}
        \mathbf{E}\Big[ \delta^k \Big]\leq  \mathbf{E} \Big[C_1\delta^{k-1}+ C_2 \norm{\tilde G_{\sa{\bx}} (\bx^{k-1},y^{k-1};\bom^{k-1})}^2/N-C_3\norm{\tilde G_y (\bx^{k-1},y^{k-1};\bzt^{k-1})}^2\Big] + \saa{C_4\frac{\sigma_y^2}{M_y}.}
    \end{aligned}}%
\end{equation}
\sa{Thus, \cref{eq_delta_one_tmpp_exp} implies the desired result 
for all $k\geq 1$.}
\end{proof}


\begin{lemma}\label{linesearch_scratch_stoc}
 \sa{Let $\eta_x>0$ and $\eta_y\in (0,1/\rv{\widetilde\mu})$ \rv{for some $0<\widetilde\mu\leq\mu$}.} 
 Then, \sa{for any \saa{$K>1$,} 
 }%
 \begin{equation}\label{eq:line-search-sketch_tmpp}
 {\small
    \begin{aligned} 
        &\mathbf{E}\Big[\sa{\|\tilde G_y(\bx^{K-1},y^{K-1};\bzt^{K-1})\|^2+\Big(2 - 4\eta_x\kappa L\Big)\|\tilde G_{\bx}(\bx^{K-1},y^{K-1};\bom^{K-1})\|^2}\\
        &\quad+\sum_{k=0}^{K-2} \Big(1+ 6\sa{C_3}\Big(L^2+\frac{1}{\eta_y^2}\Big) \sum_{i=k}^{K-2}\sa{C_1^{i-k}}\Big)\|\tilde G_y(\bx^{k},y^{k};\bzt^k)\|^2\\
        &\quad+\saa{4}\sum_{k=0}^{K-\sa{2}}\Big(\frac{1}{2} - \eta_x\kappa L-\frac{3}{2N}C_2\Big(L^2+\frac{1}{\eta_y^2}\Big) \sum_{i=k}^{K-2}\sa{C_1^{i-k}}\Big)\|\tilde G_{\bx}(\bx^{k},y^{k};\bom^k)\|^2\Big]\\
        &\leq 4\mathbf{E}\Big[\frac{N}{\eta_x}\Big(F(\bx^0)-F(\bx^K)\Big)+\delta_0{\sum_{k=0}^{K-1} \frac{3}{2}\Big(L^2+\frac{1}{\eta_y^2}\Big) \sa{C_1^k}}\Big]\\
        &\quad+4K\Big(\frac{\sigma_x^2}{M_x}+\Big(1 + \frac{3\eta_y}{\rv{\widetilde\mu}}\Big(L^2+\frac{1}{\eta_y^2}\Big)\frac{2-\eta_y\rv{\widetilde\mu}}{1-\eta_y\rv{\widetilde\mu}}\Big)\frac{\sigma_y^2}{M_y}\Big).
    \end{aligned}}%
\end{equation}
\end{lemma}
\begin{proof}
\sa{By using \cref{eq:partial_sum_grad_delta_tmpp_jac}} and \rv{Lemma}~\ref{lemma:deltat_tmpp}, we can obtain 
\begin{equation}\label{eq_sum_grad_original_tmpp}
{\small
    \begin{aligned}
        &\Big(\frac{1}{2} - \eta_x\kappa L\Big)\sum_{k=0}^{K-1}\mathbf{E}\Big[\|\tilde G_{\bx}(\bx^{k},y^{k};\bom^k)\|^2\Big]
        { + \frac{1}{4}\sum_{k=0}^{K-1}\mathbf{E}\Big[\|\tilde G_y(\bx^{k},{y}^{k};\bzt^k)\|^2\Big]}\\
        \leq &\mathbf{E}\Big[\frac{\sa{N}}{\eta_x}\Big(F(\bx^0)-F(\bx^K)\Big)+\frac{3}{2}\Big(L^2+\frac{1}{\eta_y^2}\Big) \sum_{k=0}^{K-1} \delta^{k}\Big]+K\Big(\frac{\sigma_x^2}{M_x}+\frac{\sigma_y^2}{M_y}\Big)\\
        \leq & \mathbf{E}\Big[\frac{\sa{N}}{\eta_x}\Big(F(\bx^0)-F(\bx^K)\Big)+\delta^0\frac{3}{2}\Big(L^2+\frac{1}{\eta_y^2}\Big)\sum_{k=\sa{0}}^{K-1} C_1^k\\
        &\quad {+C_2\frac{3}{2N}\Big(L^2+\frac{1}{\eta_y^2}\Big)\sum_{k=\sa{1}}^{K-1} \sum_{i=1}^{k} \left( \|\tilde G_{\bx}(\bx^{i-1},y^{i-1};\bom^{i-1})\|^2 C_1^{k-i}\right)} \\
        &\quad {- \sa{C_3}\frac{3}{2}\Big(L^2+\frac{1}{\eta_y^2}\Big)\sum_{k=\sa{1}}^{K-1}\sum_{i=1}^{k} \left( \|\tilde G_y(\bx^{i-1},y^{i-1};\bzt^{i-1})\|^2 C_1^{k-i}\right)}\Big]\\
        &\quad +K\frac{\sigma_x^2}{M_x}+
        K\left(
        1 + \frac{3}{2}\Big(L^2+\frac{1}{\eta_y^2}\Big)\sa{C_4}\sum_{k=0}^{K-1}C_1^k
        \right)\frac{\sigma_y^2}{M_y},
    \end{aligned}}%
\end{equation}%
{where $C_1,C_2,C_3$ \sa{and $C_4$} are defined in \rv{Lemma}~\ref{lemma:deltat_tmpp}; and we use the fact $\sum_{k=\sa{1}}^{K-1}\sum_{i=0}^{k-1}C_1^i\leq K\sum_{k=0}^{K-1}C_1^k$.}
Furthermore, by rearranging terms, we have
\begin{equation}\label{eq_delta_gradx_tmpp}
    \begin{aligned}
        \sum_{k=\sa{1}}^{K-1} \sum_{i=1}^{k} \|\tilde G_{\bx}(\bx^{i-1},y^{i-1};\bom^{i-1})\|^2 C_1^{k-i}= \sum_{k=0}^{K-2}\|\tilde G_{\bx}(\bx^{k},y^{k};\bom^k)\|^2 \sum_{i=k}^{K-2} C_1^{i-k},
    \end{aligned}
\end{equation}
\begin{equation}\label{eq_delta_grady_tmpp}
    \begin{aligned}
        &\sum_{k=\sa{1}}^{K-1}\sum_{i=1}^{k} \|\tilde G_y(\bx^{i-1},y^{i-1};\bzt^{i-1})\|^2  C_1^{k-i}= \sum_{k=0}^{K-2}\|\tilde G_y(\bx^{k},y^{k};\bzt^{k})\|^2 \sum_{i=k}^{K-2} C_1^{i-k}.
    \end{aligned}
\end{equation}
\sa{Next, using the bound $\sum_{k=0}^{K-1}C_1^k\leq \frac{1}{1-C_1}$ and substituting the values 
\saa{of $C_1$ and $C_4$,} we get}
\begin{equation}\label{sigmay_bd}
    \begin{aligned}
        1 + \frac{3}{2}\Big(L^2+\frac{1}{\eta_y^2}\Big)\sa{C_4}\sum_{k=0}^{K-1}C_1^k
        \leq 1 + \frac{3\eta_y}{\rv{\widetilde\mu}}\Big(L^2+\frac{1}{\eta_y^2}\Big)\frac{2-\eta_y\rv{\widetilde\mu}}{1-\eta_y\rv{\widetilde\mu}}.
    \end{aligned}
\end{equation}
By plugging \eqref{eq_delta_gradx_tmpp}, \eqref{eq_delta_grady_tmpp} and \eqref{sigmay_bd} into \eqref{eq_sum_grad_original_tmpp}, \saa{and then multiplying both sides by $4$,} we obtain the desired inequality.
\end{proof}
\saa{Having established a bound on the stochastic gradient map for \rbsgda~iterate sequence in \rv{Lemma}~\ref{linesearch_scratch_stoc}, we are now ready to study \sgdab. Both sides of error bound in \eqref{eq:line-search-sketch_tmpp} depend on $\eta_x,\eta_y$ and \rv{possibly} unknown problem constants \rv{$L,\kappa,\sigma_x,\sigma_y$} in a very complicated way; therefore, in the next section we focus on designing a practical backtracking condition that allows us to select $\eta_x$ and $\eta_y$ without knowing $L$ \rv{and/or $\mu$} (hence, $\kappa$) and that will lead to desirable theoretical guarantees.}
\subsection{\rv{Stopping guarantees for \rbsgda}}
\rv{Our aim in this section is to design a practical backtracking condition.}
\saa{In the next lemma below, we show that \rv{when problem parameter estimates satisfy $\widetilde\mu\leq \mu$ and $\widetilde L\geq L$ (hence, $\eta_y=1/\widetilde L$ is sufficiently small, i.e., $\eta_y\leq \frac{1}{L}$) and when the batchsizes $M_x,M_y$ are sufficiently large,} we can control the stochastic gradient map in expectation. This is one of our main results helping us to derive a backtracking condition 
that is guaranteed to eventually hold due to our use of backtracking on $\eta_y$. Another important 
\rv{implication} of this lemma is that it prescribes us how to select $\eta_x$ in practice depending on $\eta_y$ 
\rv{and the concavity modulus estimate $\widetilde\mu$} while still ensuring the necessary \textit{time-scale separation}.}

\begin{lemma}\label{lem:linesearch}
    Given some $\gamma\in(0,1)$, \rv{suppose $\eta_y=1/\widetilde L$ and $\eta_x = N\rho\cdot\eta_y^3\widetilde\mu^2$ for some $\widetilde L, \widetilde\mu>0$ such that $\widetilde\mu \leq \mu\leq L\leq \widetilde L$,} and $\rho = \frac{\sqrt{1+\frac{12}{N}}-1}{24}\in (0,1)$. Let $\{\bx^k,y^k\}_{k\geq 0}$ be generated by \rbsgda. Then,
    \begin{equation}\label{eq:line-search-condition}
    {\small
    \begin{aligned}
        \saa{\frac{1}{K}}\sum_{k=0}^{K-1}\mathbf{E}\Big[\|\tilde G(\bx^{k},{y}^{k};\bxi^k)\|^2\Big]
        \leq \frac{4N}{\eta_x K}\mathbf{E}\Big[F(\bx^0)-F(\bx^K)+\rv{6 \rho\delta^0\widetilde\mu}
          \Big]+\frac{4\sigma_x^2}{M_x}+ \Big(1+ \frac{6}{\eta_y \rv{\widetilde\mu}}\cdot\frac{ 2 - \eta_y \rv{\widetilde\mu}}{ 1-\eta_y \rv{\widetilde\mu}}\Big)\frac{4\sigma_y^2}{M_y}
    \end{aligned}}%
\end{equation}
     holds for all $K\geq 1$, where $\tilde G(\bar\bx,\bar y;\bar \bxi)\triangleq [\tilde G_{\bx}(\bar\bx,\bar y;\bar \bom)^\top \tilde G_y(\bar \bx,\bar y; \bar \bzt)^\top]^\top$ for all $(\bar\bx, \bar y)\in\dom g\times \dom h$.
\end{lemma}
\begin{proof}
\sa{
Note $\eta_y\in(0,\frac{1}{L}]$ implies that \saa{$C_1,C_3\geq 0$}, \rv{which are defined in Lemma~\ref{lemma:deltat_tmpp}}. Thus, the sum of all terms related to $\tilde G_y(\cdot,\cdot;\cdot)$ on the left-hand side of \cref{eq:line-search-sketch_tmpp} can be lower bounded by \saa{$
\sum_{k=0}^{K-1}\norm{\tilde G_y(\bx^k,y^k;\bzt^k)}^2$}.
Next, we provide a lower bound for the sum of all terms related to $\tilde G_{\bx}(\cdot,\cdot;\cdot)$ on the left-hand side of \cref{eq:line-search-sketch_tmpp}. 
\saa{Since $C_1\geq 0$, 
we have $\sum_{i=k}^{K-2} C_1^{i-k}\leq \sum_{k=0}^\infty C_1^k = \frac{2}{\rv{\widetilde\mu} \eta_y}$.} Therefore, using $\eta_x = N\rho\cdot\eta_y^3 \rv{\widetilde\mu^2}$, we obtain that
\begin{equation}\label{eq:long-LS-leftx2}
    \begin{aligned}
        &\frac{1}{2} - \eta_x \kappa L-\frac{3}{2N}C_2\Big(L^2+\frac{1}{\eta_y^2}\Big) \sum_{i=k}^{K-2}\sa{C_1^{i-k}}\\
        &\geq \frac{1}{2}- \eta_x \kappa L - \frac{3}{N}\frac{\kappa^2}{\rv{\widetilde\mu}} \frac{\eta_x^2}{\eta_y}\Big(L^2 + \frac{1}{\eta_y^2}\Big) \frac{2}{\eta_y\rv{\widetilde\mu}}= \frac{1}{2}- N\rho \cdot  L^2\eta_y^3 \rv{\frac{\widetilde\mu^2}{\mu}} - 6 N \rho^2 L^2  \eta_y^4 \rv{\frac{\widetilde\mu^2}{\mu^2}} \Big(L^2+\frac{1}{\eta_y^2}\Big)\\
        &\geq \frac{1}{2}- N\rho \cdot \eta_y \rv{\widetilde\mu} - 6 N \rho^2 \eta_y^2 \Big(\frac{1}{\eta_y^2}+\frac{1}{\eta_y^2}\Big)\geq \frac{1}{2}- N \rho - 12 N \rho^2= \frac{1}{4},
    \end{aligned}
\end{equation}
where the first inequality follows from $\sum_{i=k}^{K-2}C_1^{i-k}\leq \frac{2}{\eta_y\rv{\widetilde\mu}}$; the second inequality follows from $L\leq \frac{1}{\eta_y}$ \rv{and $0<\widetilde\mu\leq\mu$}; in the third inequality we used the fact that \rv{$\eta_y\widetilde\mu\leq 1$}; and in the last equality follows from our choice of $\rho$, i.e.,
$ \rho = \frac{\sqrt{1+\frac{12}{N}}-1}{24}\in (0,1)$. \saa{The definition of $\rho$ implies that $N\rho=\frac{1}{4}-12N\rho^2\leq\frac{1}{4}$; hence, using $\eta_y\leq\frac{1}{L}$ and \rv{$\widetilde\mu \leq \mu$}, we get $\eta_x=N\rho\cdot\eta_y^3\rv{\widetilde\mu^2}<\frac{1}{4L\kappa^2}\leq\frac{1}{4L\kappa}$ since $\kappa\geq 1$.}
Thus, the sum of all terms related to $\tilde G_{\bx}(\cdot,\cdot;\cdot)$ on the left-hand side of \cref{eq:line-search-sketch_tmpp} can be lower bounded by \saa{$
\sum_{k=0}^{K-1}\norm{\tilde G_{\bx}(\bx^k,y^k;\bom^k)}^2$.}}

\sa{Next, we provide \saa{a bound for the noise term on the right-hand side of \eqref{eq:line-search-sketch_tmpp} related to $\sigma_y^2$,} i.e.,
\begin{equation}
    \begin{aligned}
        &1 + \frac{3\eta_y}{\rv{\widetilde\mu}} \Big(L^2 + \frac{1}{\eta_y ^2}\Big)\frac{2-\eta_y \rv{\widetilde\mu}}{1-\eta_y \rv{\widetilde\mu}}\rv{\leq  1+ \frac{6}{\eta_y\widetilde\mu} \cdot\frac{2-\eta_y \widetilde\mu}{1-\eta_y \widetilde\mu}},
    \end{aligned}
\end{equation}
\saa{which follows from $L\leq \frac{1}{\eta_y}$.} 
Using the same arguments we have used so far for bounding the other terms, we can bound the term on the right-hand side of \cref{eq:line-search-sketch_tmpp} related to $\delta^0$ \rv{multiplied by $\eta_x$}, i.e.,
\begin{equation}\label{eq:long-LS-right}
    \begin{aligned}
        &\delta^0 \frac{3\eta_x}{2}\Big(L^2 + \frac{1}{\eta_y^2}\Big) \sum_{k=0}^{K-1} C_1^k\leq \delta^0 \frac{3 N \rho \eta_y^3 \rv{\widetilde\mu^2}}{2} \frac{2}{\eta_y^2}\frac{2}{\rv{\widetilde\mu} \eta_y} \rv{ \leq 6 \delta^0 N \rho \widetilde\mu}.
    \end{aligned}
\end{equation}
The desired inequality follows from combining all these bounds with \cref{eq:line-search-sketch_tmpp}, and then dividing both sides by 
\saa{$K$}.}
\end{proof}
\saa{
Next, we argue that given any failure probability $p\in(0,1)$ and tolerance $\epsilon>0$, \rv{whenever $\tilde L\geq L$, $\tilde \mu\leq \mu$, $\widetilde\sigma_x^2\geq\sigma_x^2$ and $\widetilde\sigma_y^2\geq\sigma_y^2$ hold,} by executing $\log_2(1/p)$ \rbsgda~runs in parallel one can generate $(\tilde\bx,\tilde y)$ such that $\norm{\widetilde G(\tilde\bx,\tilde y;\tilde\bxi)}\leq\epsilon$ w.p. at least $1-p$.}
\begin{lemma} 
\label{lem:small-eta}
    \sa{Suppose \cref{ass1,assumption:noise} hold. Given any $(\bx^0,y^0)\in\dom g\times\dom h$, $\epsilon>0$ and failure probability $p\in (0,1)$, let $\{\bz^k_t\}_{k=0}^K$ with $\bz^k_t\triangleq (\bx^k_t,y^k_t)$ for $t=1,\ldots,\rv{T}$ be the iterate sequences corresponding to
    \rv{$T\geq T_p\triangleq\lceil\log_2(1/p)\rceil$} independent \rbsgda~runs all starting from $(\bx^0,y^0)\in\dom g\times\dom h$ for some parameters $\eta_y,\eta_x, K$ and $M_x,M_y$ such that \rv{$\eta_y=1/\tilde L$, and $\eta_x=N\rho\widetilde\mu^2\eta_y^3$ for some $\tilde L,\tilde\mu$ satisfying $0<\widetilde\mu\leq\mu$ and $L\leq \tilde L$,} \rv{$M_x\geq  \mg{\lceil} \frac{128}{\epsilon^2}\rv{\widetilde\sigma_x^2}\saa{+1} \mg{\rceil}$ and $M_y \geq \mg{\lceil} (1+\frac{6}{\rv{\widetilde\mu} \eta_y} \frac{2-\rv{\widetilde\mu} \eta_y}{1-\rv{\widetilde\mu} \eta_y})\frac{128}{\epsilon^2}\rv{\widetilde\sigma_y^2}\saa{+1}\mg{\rceil}$ for some $\widetilde\sigma_x^2,\widetilde\sigma_y^2>0$ such that $\widetilde\sigma_x^2\geq \sigma_x^2$ and $\widetilde\sigma_y^2\geq \sigma_y^2$} and $K=\lceil\frac{64N}{\epsilon^2\eta_x}\left(\rv{F_0-\bar{F}+6\rho\delta\widetilde\mu}\right)\rceil$ \rv{for some given $F_0$ and $\delta$ such that $F_0\geq F(\bx_0)$ and $\delta\geq \delta^0$}. \mg{Let $\tilde k_t$ for $t=1,2,\dots, \rv{T}$ be i.i.d. with a uniform distribution on $\{0,\ldots,K-1\}$.}
    Then
    $\saa{\min_{t=1,\ldots,\rv{T}}}\tilde S_t\leq \frac{\epsilon^2}{4}$ holds with probability at least \rv{$1-p$}, \saa{where
    $\tilde S_t\triangleq\|\tilde G(\bx_t^{\mg{\tilde k}},y_t^{\mg{\tilde k}};\bxi_t^{\mg{\tilde k}})\|^2$ for $t=1,\ldots,\rv{T}$  --for 
    simplicity, we dropped the subscript in $\tilde k_t$}}.
\end{lemma}
\begin{proof} \mg{For every $k \in \{0,1,\dots, K-1\}$ fixed, }
    \rv{$\{\bx_t^k,y_t^k\}_{t=1}^{T}$} 
    are i.i.d. 
    \mg{Consequently,} we also have $\{\tilde S_t\}_{t=1}^{T}$ are i.i.d. \rv{since $\{\tilde k_t\}_{t=1}^{T}$ are i.i.d.}; hence,
\begin{align}
\mathbf{P}\Big(\min_{t=1,\ldots,T}\tilde S_t\leq\frac{\epsilon^2}{\sa{4}}\Big)= 1-\mathbf{P}\Big(\min_{t=1,\ldots,T}\tilde S_t>\frac{\epsilon^2}{\sa{4}}\Big)=1-\prod_{t=1}^{T}\mathbf{P}\Big(\tilde S_t>\frac{\epsilon^2}{\sa{4}}\Big).
\end{align}
For any $t=1,\ldots,T$, we have
{\small
\begin{align*}
   \mathbf{P}\Big(\tilde S_t>\frac{\epsilon^2}{\sa{4}}\Big)\leq \frac{4}{\epsilon^2}\mathbf{E}[\tilde S_t]\leq \frac{16}{\epsilon^2}\Big(\frac{N}{K\eta_x}\mathbf{E}\Big[F(\bx^0)-F(\bx_t^K)+\rv{6 \rho \delta^0\widetilde\mu}
          \Big]+\frac{\sigma_x^2}{M_x}+ \Big(1+ \frac{6}{\eta_y \rv{\widetilde\mu}}\frac{2 - \eta_y \rv{\widetilde\mu}}{1-\eta_y \rv{\widetilde\mu}}\Big)\frac{\sigma_y^2}{M_y}\Big),
\end{align*}}%
where in the first inequality we use Markov's inequality, and the second inequality follows from \rv{Lemma}~\ref{lem:linesearch} and \saa{$\mathbf{E}[\tilde S_t]\leq \frac{1}{K}\sum_{k=0}^{K-1}\mathbf{E}\Big[\|\tilde G(\bx_t^{k},{y}_t^{k};\bxi_t^k)\|^2\Big]$ --which is implied by $\tilde k_t\sim\cU[0,K-1]$.} 
Furthermore, since $F\Big(\bx_t^K\Big)\geq F^*\geq \bar{F}$, \rv{$F_0\geq F(\bx^0)$, $\delta\geq \delta^0$, $\widetilde\sigma_x^2\geq \sigma_x^2$ and $\widetilde\sigma_y^2\geq \sigma_y^2$}, the particular choice of $K$, $T\geq T_p$ 
and the condition on $M_x,M_y$ given in the hypothesis imply that
\begin{equation*}
\small
\mathbf{P}\Big(\tilde S_t>\frac{\epsilon^2}{\sa{4}}\Big)\leq \frac{1}{2}\implies \mathbf{P}\Big(\min_{t=1,\ldots,T}\tilde S_t\leq\frac{\epsilon^2}{\sa{4}}\Big)\geq 1-\Big(\frac{1}{2}\Big)^{T}\geq \rv{1-p}.\vspace*{-9mm}
\end{equation*}
\end{proof}
\vspace*{-6mm}
\begin{remark}
    \rv{For 
    $N=1$, rather than running \rbsgda{}(\cref{alg:GDA}) for $\tilde k\sim U[0,K-1]$ iterations and 
    returning $(\bx^{\tilde k},y^{\tilde k})$; instead, one can run it for $K$ iterations and 
    the output in Line~\ref{algeq:RB-SGDA_N} of \rbsgda{} is set to $(\tilde\bx,\tilde y)\gets (\tilde \bx^{k},\tilde y^{k})$ for
    $\tilde k\gets\argmin\{\|\tilde G(\bx^{k},y^{k};\bxi^k)\|:\ k=0,\ldots,K-1\}$ and $\tilde G\gets\big(\frac{1}{K}\sum_{k=0}^{K-1}\|\tilde G(\bx^{k},y^{k};\bxi^k)\|^2\big)^{1/2}$. Lemma~\ref{lem:small-eta} continues to hold for this variant of \rbsgda{} as well.} 
\end{remark}
All the discussion in this section until this point was about various properties of \rbsgda. In the next section, we establish the theoretical guarantees of \sgdab{} based on these results.
\subsection{\rv{Convergence guarantees for \sgdab{}}}
We should emphasize that \rv{in case a true value for either $\mu$, $L$, $\sigma_x$ or $\sigma_y$ is known, then the corresponding algorithm parameter, i.e., $\mu^0$, $L^0$, $\sigma^0_x$ or $\sigma^0_y$, is set to the known parameter value; otherwise, in the worst case scenario, when none of these true values are known, \sgdab{} is still guaranteed to converge for \textit{arbitrary} initialization of $\mu^0,L^0,\sigma^0_x,\sigma^0_y,\epsilon,\gamma,p,\bar p,c$ such that $0<\mu^0\leq L^0$, $\sigma^0_x,\sigma^0_y>0$ and $\gamma,p, \bar p,c\in (0,1)$ and $\epsilon>0$ sufficiently small, i.e., $\epsilon=\cO\Big(\min\{\sigma_x^{(0)}\frac{\sigma_x}{B_x},~\sigma_y^{(0)}\frac{\sigma_y}{B_y}\}\frac{c}{\sqrt{\log(1/p)}}\Big)$.}

\begin{defn}
\label{def:sigma-xy}
    Let $\sigma^2\triangleq \sigma_x^2+\sigma_y^2$. Under \cref{assumption:noise}, let $\gt f$ denote the stochastic oracle for $\grad f$. \rv{For any $(\bx,y)\in\dom g\times\dom h$, let $\sigma_x^2(\bx,y)$ and $\sigma_y^2(\bx,y)$ denote the variances of the oracles $\gt_{\bx} f(\bx,y;\bom)$ and $\gt_{y}f(\bx,y;\bzt)$ evaluated at $(\bx,y)$, respectively; and we define $\sigma^2(\bx,y)\triangleq \sigma_x^2(\bx,y)+\sigma_y^2(\bx,y)$; hence, $\sigma^2(\bx,y)\leq \sigma^2$.} 
\end{defn}
For $\bxi\triangleq [\bom^\top \bzt^\top]^\top$, since $\mathbf{E}_{\bxi}[\gt f(\bx,y;\bxi)]=\grad f(\bx,y)$ for any $(\bx,y)\in\dom g\times\dom h$, \cref{assumption:noise} and Definition~\ref{def:sigma-xy} together imply for all $(\bx,y)\in\dom g\times\dom h$ that 
    \begin{equation}
    \label{eq:simple-v-bound}
    \begin{aligned}
        \mathbf{E}_{\bom}[\norm{\gt_{\bx} f(\bx,y;\bom)-\grad_{\bx} f(\bx,y)}^2]&=\sigma_x^2(\bx,y)\leq \sigma_x^2,\\
        \mathbf{E}_{\bzt}[\norm{\gt_y f(\bx,y;\bzt)-\grad_y f(\bx,y)}^2]&=\sigma_y^2(\bx,y)\leq \sigma_y^2.
    \end{aligned}
    \end{equation}
In 
this section, 
we 
consider the stochastic estimates for $\grad f(\bx,y)$ and $G(\bx,y)$, namely $s(\bx,y;\bxi)\triangleq [s_{\bx}(\bx,y;\bom)^\top s_{y}(\bx,y;\bzt)^\top ]^\top$ and $\tilde G(\bx,y;\bxi)\triangleq [\tilde G_{\bx}(\bx,y;\bom)^\top \tilde G_{y}(\bx,y;\bzt)^\top ]^\top$ as defined in \rv{Definition}~\ref{def:s-gradmapping}, 
where $\bxi\triangleq [\bom^\top \bzt^\top]^\top$.

\rv{In the rest, we provide convergence guarantees 
assuming the bounded noise setting in \cref{as:bounded-oracle} -- in \cref{sec:subGaussian}, we also investigate the light tail noise setting under \cref{ass3}. More precisely, we consider the following scenarios.\footnote{In \cref{sec:subGaussian}, we investigate the same questions by replacing \cref{as:bounded-oracle} with a subGaussian assumption.}
\begin{enumerate}
    \item [(i)] \textbf{Known bounds:} We consider the setting with \textit{known} $\sigma_x^2,\sigma_y^2$ under \cref{assumption:noise,as:bounded-oracle}.
    \item [(ii)] \textbf{Unknown bounds:} We consider the case with \textit{unknown} $\sigma_x^2,\sigma_y^2$ under \cref{assumption:noise,as:bounded-oracle}.
\end{enumerate}}%
We also compare our bounds with the crude bounds implied by only \cref{assumption:noise} without resorting to \cref{as:bounded-oracle} or the subGaussian type light tail assumption considered in \cref{sec:subGaussian}.
\subsubsection{\rv{Convergence guarantees for \sgdab{} with known bounds}}
\label{sec:known-bounds}
Suppose $\sigma_x^2,\sigma_y^2$ satisfying \eqref{eq:simple-v-bound} are known--see Definition~\ref{def:sigma-xy}. We first argue that \sgdab{} stops w.p. 1 and that the number of backtracking iterations is $\cO(\log(\rv{\cR}))$ w.p. at least $1-p$, \rv{where 
\begin{align}
\label{eq:R}
   \cR\triangleq\max\Big\{\mu^{0}/\mu,~L/L^{0},~1\Big\}. 
\end{align}}%
\mg{After presenting this lemma,} 
\rv{in \cref{thm:p-bound} and Corollary~\ref{cor:sample-complexity-bounded},} we show that the output $(\bx_\epsilon,y_\epsilon)$ is $\cO(\epsilon)$-stationary according to Definition~\ref{def:eps-stationarity}, i.e., $\norm{G(\bx_\epsilon,y_\epsilon)}=\cO(\epsilon)$, with high probability. Finally, \cref{ncsc_thm} is our main result on \sgdab, establishing its oracle complexity.
\begin{lemma}
\label{lem:stopping_probability}
    \sa{Suppose \cref{ass1,assumption:noise} hold. \rv{Given an arbitrary failure probability $p\in (0,1)$, for any given tolerance $\epsilon>0$,} suppose we execute \sgdab{} with $K=\lceil\frac{64N}{\epsilon^2\eta_x}\left(\rv{F_0-\bar{F}+6\rho\delta\widetilde\mu}\right)\rceil$ \rv{for some given $F_0$ and $\delta$ such that $F_0\geq F(\bx_0)$ and $\delta\geq \delta^0$}, \rv{$M_x\geq  \mg{\lceil} \frac{128}{\epsilon^2}\rv{\sigma_x^2}\saa{+1} \mg{\rceil}$ and $M_y \geq \mg{\lceil} (1+\frac{6}{\rv{\widetilde\mu} \eta_y} \frac{2-\rv{\widetilde\mu} \eta_y}{1-\rv{\widetilde\mu} \eta_y})\frac{128}{\epsilon^2}\rv{\sigma_y^2}\saa{+1}\mg{\rceil}$, and \rv{$T=\lceil\log_2(1/p)\rceil$}. Then,} \sgdab{} stops w.p. 1, and the probability that \rv{the stopping condition $\tilde S_{(t^*,\ell)}
        \leq \frac{\epsilon^2}{4}$} holds within at most $\bar{\ell}\triangleq\lceil\log_{\frac{1}{\gamma}}(\rv{\cR})\rceil$ backtracking iterations, i.e., for some $\ell\leq \lceil\log_{\frac{1}{\gamma}}(\rv{\cR})\rceil$, is at least $1-p$, \rv{where $\cR$ is defined in \eqref{eq:R}}, $\tilde S_{(t^*,\ell)}=\min_{t=1,\ldots,T}\tilde S_{(t,\ell)}$ and $\tilde S_{(t,\ell)}=\norm{\tilde G(\tilde\bx_{(t,\ell)},\tilde y_{(t,\ell)};\tilde\bxi_{(t,\ell)})}^2$ for $t\in[T]$. Furthermore, $\p{\tilde S_{(t^*,\ell)}
        \leq \frac{\epsilon^2}{4}}\geq 1-p$ for all $\ell\geq \bar\ell$.}
\end{lemma}
\begin{proof}
\rv{Consider the $\ell$-th backtracking iteration for some $\ell\geq 0$. Recall that for the $\ell$-th backtracking iteration of \sgdab, the parameter estimates are set to $\widetilde L=L^0/\gamma^\ell$ and $\widetilde\mu=\max\{\mu^{0} \gamma^\ell,\underline{\mu}\}$; hence, during the $\ell$-th backtracking iteration, the step sizes $\eta_y=\gamma^\ell/L^0$ and $\eta_x = N\rho\Big(\max\{\mu^{0} \gamma^\ell,\underline{\mu}\}\Big)^2\eta_y^3$ are adopted.} 
\saa{Next we define the probability of \sgdab~stopping at \rv{the} $\ell$-th backtracking iteration --with initial point $(\bx^0,y^0)$ fixed, \rv{this probability is a function of 
$\ell\geq 0$ when algorithm parameters are set as described above.}
For any $\ell\in\integers_+$, 
 consider 
 \rv{$t^*\triangleq\argmin_{t=1,\ldots,T}\tilde S_{(t,\ell)}$,} and define
 \rv{
 \begin{align}
 \label{eq:ql}
     q(\ell)\triangleq\mathbf{P}\Big(A_{(\ell)}\Big),\ \mbox{where}\ A_{(\ell)}\triangleq\Big\{\tilde S_{(t^*,\ell)}\leq\frac{\epsilon^2}{4}\Big\}=\Big\{\norm{\tilde G(\tilde\bx_{(t^*,\ell)},\tilde y_{(t^*,\ell)};\tilde\bxi_{(t^*,\ell)})}\leq\frac{\epsilon}{2}\Big\},\quad\forall~\ell\geq 0.
 \end{align}}}
 Let \saa{$I\in\integers_{+}$} denote the random stopping time of \sgdab,
    i.e., \rv{$\mathbf{P}(I=0)=q(0)$} and $\mathbf{P}(I=\ell)=q(\ell)\Pi_{i=1}^{\ell-1}(1-q(i))$ for \rv{$\ell\geq 1$. Note that 
    for all $\ell\geq \bar{\ell}=\lceil\log_{\frac{1}{\gamma}}(\rv{\cR})\rceil$, we have $\widetilde L\geq \mu$ and $\widetilde\mu\leq \mu$; therefore, Lemma~\ref{lem:small-eta} implies that $\p{A_{(\ell)}}\geq 1-p$ for all $\ell\geq \bar{\ell}$. Since $\{A_{(\ell)}\}_{\ell\geq 0}$ are independent events} and $\rv{\p{A_{(\ell)}}}=q(\ell)\geq 1-p$ for $\ell\geq \bar{\ell}$, we can conclude that $\mathbf{P}(I<\infty)=1$, i.e., \sgdab~stops with probability 1. Moreover, we also have
    {\small
    \begin{align*}
        \mathbf{P}(I\leq \bar{\ell})&=\rv{q(0)}+\sum_{\ell=\rv{1}}^{\bar{\ell}}q(\ell)\Pi_{i=\rv{0}}^{\ell-1}(1-q(i))\\
        &\geq \rv{\inf\{\alpha_0+\sum_{\ell=1}^{\bar{\ell}-1}\alpha_\ell\Pi_{i=0}^{\ell-1}(1-\alpha_i)+(1-p)\Pi_{i=0}^{\bar{\ell}-1} (1-\alpha_i):\ \alpha_i\in[0,1],~i=0,\ldots,\bar{\ell}-1\}}\geq 1-p,
    \end{align*}}%
    \saa{which immediately follows from a simple induction argument.}
\end{proof}
\vspace*{-3mm}
\begin{theorem}
\label{thm:p-bound}
    Let $\bar\ell\triangleq\lceil\log_{\frac{1}{\gamma}}(\rv{\cR})\rceil$. Under the premise of Lemma~\ref{lem:stopping_probability}, for any $(\bx^0,y^0)\in\dom g\times\dom h$, $\epsilon>0$, $p\in (0,1)$ and all other parameters chosen as in the input line of \cref{alg:GDA-B} with \rv{$T\triangleq\lceil\log_2(1/p)\rceil$}, \sgdab{} stops w.p.1 returning $(\bx_\epsilon,y_\epsilon)$ such that
{\small
\begin{equation*}
\begin{aligned}
    \p{\norm{G\left(\bx_\epsilon,y_\epsilon\right)}\leq\epsilon}\geq 1-\sum_{l=0}^{\bar\ell-1} b_\ell\Big(\frac{\epsilon}{2}\Big)-\sum_{l\geq \bar\ell} b_\ell\Big(\frac{\epsilon}{2}\Big) p^{\ell-\bar\ell},
\end{aligned}
\end{equation*}}%
where $b_\ell(\epsilon)\triangleq b(\epsilon,M_x,M_y)$ for $M_x,M_y$ values corresponding to the backtracking iteration $\ell\geq 0$ and $b(\epsilon,M_x,M_y)$ is defined in Lemma~\ref{lem:Chebyshev}, \rv{depending on whether \cref{assumption:noise} holds alone or together with \cref{as:bounded-oracle}.}
\end{theorem}
\begin{proof}
    \rv{For \sgdab, within the backtracking iteration $\ell\geq \rv{0}$, suppose $M_x,M_y\in\integers_{+\mgrev{+}}$ satisfying the condition in Lemma~\ref{lem:stopping_probability} are chosen depending only on $\epsilon,\sigma_x^2,\sigma_y^2,\eta_y,\widetilde\mu$ and $\ell$. Moreover, for the backtracking iteration $\ell\geq \rv{0}$, the parameter estimates are set to $\widetilde L=L^0/\gamma^\ell$ and $\widetilde\mu=\max\{\mu^{0} \gamma^\ell,\underline{\mu}\}$; hence, the step sizes $\eta_y,\eta_x$ chosen as in line~\ref{algeq:step-sizes} of \sgdab, and the sample sizes $M_x,M_y$ together with the iteration budget $K$ for \rbsgda{} chosen as in line~\ref{algeq:M} of \sgdab{}
are all functions of $\ell$ and also of other fixed problem parameters, i.e., $\bar{F},N,\sigma_x^2,\sigma_y^2$, and algorithm parameters, i.e., $\mu^0,L^0,\bx^0,y^0,F_0,\delta,\epsilon,\gamma,p,\bar p$.}  \saa{In the rest, for the ease of notation, we suppress the dependency of $\eta_y$, $\eta_x$, $K$, $M_x$, $M_y$ on $\ell$ and other parameters.} Let $\{\big(\bx^k_{(t,\ell)},y^k_{(t,\ell)}\big)\}_{k=0}^{K-1}$ for $t=1,\ldots,T$ denote the \rbsgda~iterate sequences for $T$ independent runs, all initialized from $(\bx^0,y^0)$ and using $\eta_y,\eta_x,K,M_x,M_y$ corresponding to given backtracking iteration counter $\ell\geq 0$.

For any given $\ell\geq 0$, suppose the sample sizes $M_x,M_y$ corresponding to the given backtracking iteration $\ell$ are deterministic integers satisfying the condition in Lemma~\ref{lem:stopping_probability}, and let $b_\ell(\cdot)\triangleq b(\cdot,M_x,M_y)$ as $M_x$ and $M_y$ depend on $\ell$. Next, for any given \rv{$\ell\geq 0$} and $t\in\{1,\ldots,T$\}, let $\tilde k\sim\cU[0,K-1]$. \saa{According to \sgdab, for \rv{any fixed $\ell\geq 0$}, we have} $\tilde S_{(t,\ell)}=\norm{\tilde
    G\big(\tilde\bx_{(t,\ell)},\tilde y_{(t,\ell)};\tilde\bxi_{(t,\ell)}\big)}^2$, where $(\tilde\bx_{(t,\ell)},\tilde y_{(t,\ell)};\tilde\bxi_{(t,\ell)})=(\bx^{\tilde k}_{(t,\ell)},y^{\tilde k}_{(t,\ell)};\bxi^{\tilde k}_{(t,\ell)})$ for all $t=1,\ldots, T$ \rv{--here, $\tilde k$ is sampled for each 
    $t$ \mg{in an i.i.d.~fashion}}. Moreover, for any given $\ell\geq 0$, define $t^*=\argmin\{\tilde S_{(t,\ell)}: t=1,\ldots,T\}$ and set $(\tilde\bx_{(\ell)},\tilde y_{(\ell)},\tilde \bxi_{(\ell)})=(\tilde\bx_{(t^*,\ell)},\tilde y_{(t^*,\ell)};\tilde\bxi_{(t^*,\ell)})$.
    Due to symmetry, $t^*$ is uniformly distributed on the support $\{1,\ldots,T\}$ since $\{\tilde S_{(t,\ell)}\}_{t=1}^T$ are i.i.d. random variables for each fixed $\ell\geq 0$.

Recall the stopping time $I\in\integers_+$ defined in the proof of \rv{Lemma}~\ref{lem:stopping_probability}, i.e., 
\begin{align*}
     I\triangleq\min\Big\{\ell\in\integers_+:\ \norm{\tilde G\big(\tilde\bx_{(\ell)},\tilde y_{(\ell)};\tilde \bxi_{(\ell)}\big)}\leq \frac{\epsilon}{2}\Big\};
 \end{align*}
 hence, $(\bx_\epsilon,y_\epsilon)=(\tilde\bx_{(I)},\tilde y_{(I)})$, i.e., $(\bx_\epsilon,y_\epsilon,\bxi_\epsilon)=(\tilde\bx_{(\ell)},\tilde y_{(\ell)},\tilde\bxi_{(\ell)})$ when $I=\ell$, and we have \rv{$\mathbf{P}(I=0)=q(0)$} and $\mathbf{P}(I=\ell)=q(\ell)\Pi_{i=0}^{\ell-1}(1-q(i))$ for \rv{$\ell\geq 1$, where $q(\ell)$ is defined in \eqref{eq:ql} for $\ell\geq 0$.} In the rest, let $p(\ell)\triangleq \mathbf{P}(I=\ell)$ for $\ell\geq 0$, i.e., $p(\ell)$ denotes the probability of \sgdab{} stopping at the $\ell$-th backtracking iteration for $\ell\geq 0$; hence, \rv{the output $(\bx_\epsilon,y_\epsilon)$ of the algorithm (see line~\ref{algeq:output} of \sgdab{}) satisfies} $(\bx_\epsilon,y_\epsilon)=(\tilde\bx_{(\ell)},\tilde y_{(\ell)})$ with probability $p(\ell)$. Therefore, we have
\begin{equation}
\label{eq:det-gmap}
{\small
\begin{aligned}
    \MoveEqLeft\p{\norm{G\left(\bx_\epsilon,y_\epsilon\right)}\leq \epsilon}\\
    & = \sum_{\ell\geq 0} \p{\norm{G\left(\tx,\ty\right)}\leq \epsilon,\ I=\ell}\\
    & = \sum_{\ell\geq 0} \p{\norm{G\left(\tx,\ty\right)}\leq \epsilon,\ \norm{\tilde G\big(\tx,\ty;\txi\big)}\leq \frac{\epsilon}{2},\ \norm{\tilde G\big(\tilde\bx_{(\ell')},\tilde y_{(\ell')};\tilde \bxi_{(\ell')}\big)}>\frac{\epsilon}{2},\ \ell'=0,\ldots,\ell-1}\\
    & = \sum_{\ell\geq 0} \p{\norm{G\left(\tx,\ty\right)}\leq \epsilon,\ \norm{\tilde G\big(\tx,\ty;\txi\big)}\leq \frac{\epsilon}{2}}\Pi_{\ell'=0}^{\ell-1}(1-q(\ell'))\\
    & = \sum_{\ell\geq 0} \p{\norm{G\left(\tx,\ty\right)}\leq \epsilon\ \bigg|\ \norm{\tilde G\big(\tx,\ty;\txi\big)}\leq \frac{\epsilon}{2}}q(\ell)\Pi_{\ell'=0}^{\ell-1}(1-q(\ell'))=\sum_{\ell\geq 0} w(\ell) p(\ell),
\end{aligned}}%
\end{equation}
where $w(\ell)\triangleq \p{\norm{G\left(\tx,\ty\right)}\leq \epsilon\ \bigg|\ \norm{\tilde G\big(\tx,\ty;\txi\big)}\leq \frac{\epsilon}{2}}$ for $\ell\geq 0$, and the third equality follows from $\tilde \bxi_{(\ell')}$ and $\tilde \bxi_{(\ell'')}$ being independent for all $\ell',\ell''\in\integers_+$ such that $\ell'\neq\ell''$. Thus, to show that the event $\{\norm{G\left(\bx_\epsilon,y_\epsilon\right)}\leq \epsilon\}$ holds with high probability, we will lower bound $w(\ell)$ for all $\ell\in\integers_+$. Towards this goal, we define another important quantity of interest: $w_0(\ell)\triangleq \p{\norm{\tilde G\big(\tx,\ty;\txi\big)}\leq \frac{\epsilon}{2}\ \bigg|\ \norm{G\left(\tx,\ty\right)}> \epsilon}$ for $\ell\geq 0$. Clearly, for all $\ell\in\integers_+$, the following relation between $w(\ell)$ and $w_0(\ell)$ holds:
\begin{equation}
\label{eq:w}
    w(\ell)= 1- w_0(\ell)\frac{\p{\norm{G\left(\tx,\ty\right)}> \epsilon}}{\p{\norm{\tilde G\big(\tx,\ty;\txi\big)}\leq \frac{\epsilon}{2}}}\geq 1- w_0(\ell)/q(\ell),
\end{equation}
where the inequality follows from the definition of $q(\ell)$ and 
$
\rv{\p{\norm{G\left(\tx,\ty\right)}> \epsilon}\in [0,1]}$. Let $\tz\triangleq (\tx,\ty)$ for $\ell\in\integers_+$. For any $(\bx,y)\in\dom g\times\dom h$, let $\bz=(\by,y)$ and $G(\bz)=G(\bx,y)$ as given in Definition~\ref{def:s-gradmapping} with step sizes $\eta_x$ and $\eta_y$ corresponding to $\ell$-th backtracking iteration. Then, for all $\ell\geq 0$, we trivially have
\begin{equation}
\label{eq:w0}
    w_0(\ell)=\E{\p{\big\|\tilde G\big(\tz;\txi\big)\big\|\leq \frac{\epsilon}{2}~\big|~\tz}\ \bigg|\ \norm{G\left(\tz\right)}> \epsilon}.
\end{equation}
For any $\bz$, $\big\|\tilde G\big(\bz;\txi\big)\big\|\leq \frac{\epsilon}{2}$ implies $\big\|\tilde G\big(\bz;\txi\big)-G(\bz)\big\|\geq \norm{G(\bz)}-\frac{\epsilon}{2}$. Therefore, \eqref{eq:w0} implies that
\begin{equation}
\label{eq:w0-bound}
    \begin{aligned}
        w_0(\ell)
        &\leq \E{\p{\big\|\tilde G\big(\tz;\txi\big)-G(\tz)\big\|\geq \norm{G(\tz)}-\frac{\epsilon}{2}~\big|~\tz}\ \bigg|\ \norm{G\left(\tz\right)}> \epsilon}\\
        & \leq \E{b_\ell\Big(\norm{G(\tz)}-\frac{\epsilon}{2}\Big)~\bigg|\ \norm{G\left(\tz\right)}> \epsilon}\leq b_\ell\Big(\frac{\epsilon}{2}\Big)
    \end{aligned}
\end{equation}
where in the second inequality we used the fact that \eqref{eq:diff-bound} from Lemma~\ref{lem:Chebyshev} holds for any $(\bx,y)\in\dom g\times \dom h$, and the third inequality follows from the fact that $b_\ell(r)$ is a deterministic function that is decreasing in $r$ for all $\ell\geq 0$. Thus, combining \eqref{eq:det-gmap}, \eqref{eq:w} and \eqref{eq:w0-bound}, we obtain
\begin{equation*}
\begin{aligned}
    \MoveEqLeft\p{\norm{G\left(\bx_\epsilon,y_\epsilon\right)}\leq \epsilon}\\
    &=\sum_{\ell\geq 0} w(\ell) p(\ell)\geq \sum_{l\geq 0}\Big(1-b_\ell\Big(\frac{\epsilon}{2}\Big)/q(\ell)\Big)p(\ell)=1-\sum_{l\geq 0}b_\ell\Big(\frac{\epsilon}{2}\Big)\Pi_{\ell'=0}^{\ell-1}(1-q(\ell')),
\end{aligned}
\end{equation*}
where we used $\sum_{\ell\geq 0}p(\ell)=1$ and $p(\ell)=q(\ell)\Pi_{\ell'=0}^{\ell-1}(1-q(\ell'))$ for the equality. Note that Lemma~\ref{lem:stopping_probability} shows that $q(\ell)\geq 1-p$ for all $\ell\geq \bar\ell\triangleq\lceil\log_{\frac{1}{\gamma}}(\rv{\cR})\rceil$; therefore, $\Pi_{\ell'=0}^{\ell-1}(1-q(\ell'))\leq p^{\ell-\bar\ell}$ for $\ell\geq\bar\ell$. Thus, we get the desired result. 
\end{proof}
\begin{corollary}
\label{cor:sample-complexity-bounded}
Under \cref{ass1}, for any given $\bar p\in(0,1)$, $\p{\norm{G\left(\bx_\epsilon,y_\epsilon\right)}\leq \epsilon}\geq 1-\bar p$ holds when $M_x$ and $M_y$ are chosen as follows: for any $C\geq \frac{\pi^2}{3\bar p}$ and any $\epsilon\in(0,~\bar\epsilon)$,
\begin{enumerate}
    \item[(i)] under \cref{assumption:noise,as:bounded-oracle}, let $M_x=\frac{32}{\epsilon^2}\sigma_x^2\Big(\log(n_x+1)+\log\Big(C(\ell+1)^2\Big)\Big)$ and $M_y=\frac{32}{\epsilon^2}\sigma_y^2\Big(\log(n_y+1)+ \Big(1+\frac{6}{\rv{\widetilde\mu} \eta_y} \frac{2-\rv{\widetilde\mu} \eta_y}{1-\rv{\widetilde\mu} \eta_y}\Big)\log\Big(C(\ell+1)^2\Big)\Big)$ for $\bar\epsilon=6\sqrt{2}\cdot\min\Big\{\frac{\sigma_x^2}{B_x},\frac{\sigma_y^2}{B_y}\Big\}$;
    \item[(ii)] under \cref{assumption:noise,as:bounded-oracle}, let $M_x=\frac{16}{\epsilon^2}\Big(\sigma_x^2+\frac{B_x\epsilon}{6\sqrt{2}}\Big)\Big(\log(n_x+1)+\log\Big(C(\ell+1)^2\Big)\Big)$ and $M_y=\frac{16}{\epsilon^2}\Big(\sigma_y^2+\frac{B_y\epsilon}{6\sqrt{2}}\Big)\Big(\log(n_y+1)+ \Big(1+\frac{6}{\rv{\widetilde\mu} \eta_y} \frac{2-\rv{\widetilde\mu} \eta_y}{1-\rv{\widetilde\mu} \eta_y}\Big)\log\Big(C(\ell+1)^2\Big)\Big)$ for $\bar\epsilon=\infty$;
    \item[(iii)] under \cref{assumption:noise}, let $M_x=\frac{4\sigma_x^2}{\epsilon^2}C(\ell+1)^2$ and $M_y=\frac{4\sigma_y^2}{\epsilon^2}C\Big(1+\frac{6}{\rv{\widetilde\mu} \eta_y} \frac{2-\rv{\widetilde\mu} \eta_y}{1-\rv{\widetilde\mu} \eta_y}\Big)(\ell+1)^2$ for $\bar\epsilon=\infty$.
\end{enumerate}
\end{corollary}
\begin{proof}
Recall that under \cref{assumption:noise,as:bounded-oracle}, according to Lemma~\ref{lem:Chebyshev}, we get
{\small
    \begin{align*}
        b_\ell\Big(\frac{\epsilon}{2}\Big)
        &=(n_x+1)\exp\Big(-M_x\frac{\sigma_x^2}{B_x^2}\cdot H\Big(\frac{B_x \epsilon}{2\sqrt{2}\sigma_x^2}\Big)\Big)+(n_y+1)\exp\Big(-M_y\frac{\sigma_y^2}{B_y^2}\cdot H\Big(\frac{B_y \epsilon}{2\sqrt{2}\sigma_y^2}\Big)\Big)\\
        &\leq (n_x+1)\exp\Big(-\frac{M_x\epsilon^2/16}{\sigma_x^2+\frac{B_x\epsilon}{6\sqrt{2}}}\Big)+(n_y+1)\exp\Big(-\frac{M_y\epsilon^2/16}{\sigma_y^2+\frac{B_y\epsilon}{6\sqrt{2}}}\Big).
    \end{align*}}%
    For the scenario \textit{(i)}, in case the variance bounds $(\sigma_x^2,\sigma_y^2)$ are known, for $\epsilon\in\Big(0,~6\sqrt{2}\cdot\min\Big\{\frac{\sigma_x^2}{B_x},\frac{\sigma_y^2}{B_y}\Big\}\Big)$, our choice of $M_x$ and $M_y$ immediately implies that
        $b_\ell\Big(\frac{\epsilon}{2}\Big) \leq \frac{2}{C}\cdot\frac{1}{(\ell+1)^2}$;
    similarly, for the scenario \textit{(ii)}, in case 
    $(\sigma_x^2,\sigma_y^2)$ together with stochastic gradient norm bounds $(B_x,B_y)$ are known, our choice of $M_x$ and $M_y$ for this scenario also implies that $b_\ell\Big(\frac{\epsilon}{2}\Big) \leq \frac{2}{C}\cdot\frac{1}{(\ell+1)^2}$. The same same bound also holds for the scenario \textit{(iii)}; indeed, Lemma~\ref{lem:Chebyshev} implies that $b_\ell(\epsilon/2)=\frac{4}{\epsilon^2}\Big(\frac{\sigma_x^2}{M_x}+\frac{\sigma_y^2}{M_y}\Big)$, which leads to the bound above. Therefore, for all three scenarios, from \cref{thm:p-bound}, for $C\geq \frac{\pi^2}{3\bar p}$, we get
    {\small
\begin{equation*}
\p{\norm{G\left(\bx_\epsilon,y_\epsilon\right)}\leq\epsilon}\geq 1-\frac{2}{C}\sum_{\ell\geq 0} \frac{1}{(\ell+1)^2}=1-\frac{\pi^2}{3C}\geq 1-\bar p.
\vspace*{-7mm}
\end{equation*}}%
\vspace*{-7mm}
\end{proof}
\begin{remark}
    When compared to scenario \textit{(iii)}, which only requires \cref{assumption:noise}, scenarios \textit{(ii)} and \textit{(iii)} has a better $\bar p$ dependence, i.e., $M_x,M_y=\cO(1/\bar p)$ for scenario \textit{(iii)} improves to $M_x,M_y=\cO(\log(1/\bar p))$ under the additional \cref{as:bounded-oracle}. Furthermore, for all three scenarios sample sizes $M_x$ and $M_y$ are increasing functions of $\ell$. That being said, setting $\bar p=p\in(0,1)$, Lemma~\ref{lem:stopping_probability} shows that w.p. at least $1-\bar p$, \sgdab{} stops within $\bar{\ell}\triangleq\lceil\log_{\frac{1}{\gamma}}(\rv{\cR})\rceil$ backtracking iterations. Therefore, it holds with probability at least $1-\bar p$ that for scenarios \textit{(i)} and \textit{(ii)}, $M_x,M_y=\cO(\log\log(\cR))$ while for scenario \textit{(iii)} we have $M_x,M_y=\cO(\log(\cR))$.
\end{remark}
\saa{Finally, in~\cref{ncsc_thm} we state our main result on the oracle complexity of \sgdab. In~\cref{sec:GS}, we consider a variant of \sgdab{} calling \rbsgda{} with Gauss-Seidel~(GS) updates, and we show that all the results shown for Jacobi-type updates continue to hold for the GS variant as well.}
\begin{theorem}\label{ncsc_thm}
Suppose \cref{ass1,assumption:noise,as:bounded-oracle} hold \rv{with known variance bounds\footnote{If one knows $B_x$ and $B_y$ rather than $\sigma_x^2$ and $\sigma_y^2$, since $B_x^2\geq \sigma_x^2$ and $B_y^2\geq \sigma_y^2$, replacing $\sigma_x^2$ and $\sigma_y^2$ with $B_x^2$ and $B_y^2$ in our choice of $M_x$ and $M_y$ as in (i) of Corollary~\ref{cor:sample-complexity-bounded}, we can allow for a wider range of tolerance values $\epsilon>0$, i.e., $\bar\epsilon\geq 6\sqrt{2}\min\{B_x,B_y\}$. On the other hand, if $B_x,B_y,\sigma_x^2,\sigma_y^2$ are all known, then choosing $M_x$ and $M_y$ as in (ii) of Corollary~\ref{cor:sample-complexity-bounded} implies that our result holds for all $\epsilon>0$, i.e., $\bar\epsilon=+\infty$.}} $\sigma_x^2$ and $\sigma_y^2$.
Given any 
$(\bx^0,y^0)\in\dom g\times\dom h$, $\epsilon, \mu^{0},L^{0}>0$ and $\gamma,p,\rv{\bar p}\in (0,1)$ such that $L^{0}>\mu^{0}$ and $\epsilon\in(0,\bar\epsilon)$ for $\bar\epsilon$ as defined in Corollary~\ref{cor:sample-complexity-bounded}, and let $\hat{L}\triangleq\cR L^{0}$ and $\hat{\mu}\triangleq\max\{\underline{\mu}/\gamma, \mu^{0}/\cR\}$ and $\cR\triangleq\max\{\mu^{0}/\mu,L/L^{0},1\}$.
Then, \sgdab, displayed in \cref{alg:GDA-B}, \rv{with $M_x$ and $M_y$ chosen as in (i) of Corollary~\ref{cor:sample-complexity-bounded},} stops w.p.1 returning
$(\bx_\epsilon,y_\epsilon)$ satisfying \rv{$\p{\norm{G\left(\bx_\epsilon,y_\epsilon\right)}\leq \epsilon}\geq 1-\bar p$.} 
{Moreover, with probability at least $1-p$, \sgdab~stops within $\bar{\ell}\triangleq\lceil\log_{\frac{1}{\gamma}}(\cR)\rceil$ 
backtracking iterations, which require $\cO\Big(\hat{L}\hat{\kappa}^2\frac{1}{\epsilon^2}\Big)$ \rbsgda{} iterations and the corresponding oracle complexity is 
$$\cO\Big({\frac{\hat{L}\hat{\kappa}^2}{\epsilon^4}\Big(
F_0-\bar{F}+
\delta\Big)}\Big[\Big(\log(n_x)+\log\Big(\frac{1}{\bar p}\Big)\Big)\rv{\sigma_x^2}+\Big(\rv{\log(n_y)+\hat{\kappa}\log\Big(\frac{1}{\bar p}\Big)}\Big)\rv{\sigma_y^2}\Big]\log(\cR)\log\Big(\frac{1}{\rv{p}}\Big)\Big).$$
For the special case of $\sigma_x^2,\sigma_y^2=0$ and $N=1$, i.e., deterministic case, 
setting $M_x,M_y=1$ and $T=1$, \sgdab{} can generate $(\bx_\epsilon,y_\epsilon)$ such that $\norm{ G(\bx_\epsilon,y_\epsilon)} \leq \epsilon$ with 
gradient complexity of $\cO\Big(\hat{L}\hat{\kappa}^2\frac{1}{\epsilon^2}\Big)$.}
\end{theorem}
\begin{proof}
\rv{For $\ell\geq 0$, given the initial values $\mu^0,L^0$ such that $0<\mu^0<L^0$, suppose $\mu^{(\ell)}$ and $L^{(\ell)}$ denote the $\widetilde \mu$ and $\widetilde L$ values generated by \sgdab{} for the $\ell$-th backtracking iteration such that $\mu^{(0)}=\mu^0$ and $L^{(0)}=L^0$; hence, $L^{(\ell)}\leq \hat{L}/\gamma$ and $\mu^{(\ell)}\geq\gamma\hat{\mu}$ for all $\ell\geq 0$.} Since $\eta_y$ is initialized at ${1/L^{0}}$, and $\eta_y$ is decreasing monotonically every time 
\sa{the backtracking} condition in line~\ref{algeq:stop} of \sgdab{} does not hold, we have $\eta_y<{1/\mu^{0}}$ throughout the algorithm. Now suppose the condition in line~\ref{algeq:stop} holds for some 
\sa{$\ell\geq 0$}, i.e.,
\begin{equation}
\sa{\tilde S_{(t^*,\ell)}\triangleq\min_{t=1,\ldots,T} \tilde S_{(t,\ell)}
        \leq \epsilon^2/4,}\label{eq:stop-cond}
\end{equation}
where $t^*\in\{1,\ldots,T\}$ is the index achieving the minimum and \sa{$\tilde S_{(t,\ell)}\triangleq \|\tilde G(\bx_{(t,\ell)}^{\tilde k},y_{(t,\ell)}^{\tilde k};\bxi_{(t,\ell)}^{\tilde k})\|^2
$ for $\tilde k$ chosen uniformly at random\footnote{The random index $\tilde k$ is sampled for each $t\in\{1,\ldots, T\}$; that said, for the sake of notational simplicity, we do not explicitly state dependence of $\tilde k$ on $t$.} from $\{0,\ldots, K-1\}$ for each $t=1,\ldots,T$. Thus, for $(\bx_\epsilon,y_\epsilon,\bxi_\epsilon)\triangleq (\bx_{(t^*,\ell)}^{\tilde k},y_{(t^*,\ell)}^{\tilde k};\bxi_{(t^*,\ell)}^{\tilde k})$, we have $\|\tilde G(\bx_\epsilon,y_\epsilon;\bxi_\epsilon)\|\leq \frac{\epsilon}{2}$.}
\sa{According to \cref{lem:stopping_probability}, \sgdab{} stops within {$\bar{\ell}\triangleq \lceil\log_\frac{1}{\gamma}(\cR)\rceil$} backtracking iterations with probability at least $1-p$, i.e., $\mathbf{P}(I\leq \bar{\ell})\geq 1-p$, \saa{where $I\in\integers_{++}$ denotes the random stopping time of \sgdab~--for details, see \cref{lem:stopping_probability}}. Under the 
event $I\leq \bar{\ell}$, 
\saa{consider} the worst\mg{-}case scenario in terms of the oracle complexity, i.e., the backtracking condition in line~\ref{algeq:stop} of \sgdab{} does not hold for $\ell<\bar{\ell}$ and it holds when $\ell=\bar{\ell}$. Note that for $\ell=\bar{\ell}$, 
we would have $\rv{\widetilde L}={L^{0}}/\gamma^{\bar{\ell}}\geq L$ {and $\rv{\widetilde\mu} = \max\{\mu^{0}\gamma^{\bar{\ell}},\underline{\mu}\}\leq\mu$}; hence, $\eta_y=\rv{1/\widetilde L}\leq 1/ L$. \rv{Moreover, for all $\ell\in\integers$ such that $0\leq \ell\leq\bar \ell$,} {we can also conclude that $\rv{\widetilde L=L^{(\ell)}}\leq L^{0}/\gamma^{\lceil\log_\frac{1}{\gamma}(\cR)\rceil}\leq \cR L^{0}/\gamma = \hat{L}/\gamma$ and $\rv{\widetilde \mu = \mu^{(\ell)}} \geq \max\{\mu^{0}\gamma^{\lceil\log_\frac{1}{\gamma}(\cR)\rceil}, \underline{\mu}\}\geq \max\{\gamma\mu^{0}/\cR, \underline{\mu}\} = \gamma\hat{\mu}$.} 
Thus, under the event $I\leq \bar{\ell}$, which holds with probability at least $1-p$, \sgdab{} can generate $(\bx_\epsilon,y_\epsilon)$ such that
$\norm{\tilde G(\bx_\epsilon,y_\epsilon;\bxi_\epsilon)} \leq \epsilon/2$, \saa{and the worst-case oracle complexity bound for the \sgdab{} is given by $C_\epsilon\triangleq T\sum_{\ell=\rv{0}}^{\bar{\ell}}K^{(\ell)}(M_x^{(\ell)}+M_y^{(\ell)})$, where $M_x^{(\ell)}$ and $M_y^{(\ell)}$ are the sample size values $M_x$ and $M_y$ for the $\ell$-th backtracking iteration that are chosen as in (i) of Corollary~\ref{cor:sample-complexity-bounded}, i.e., under the event $I\leq\bar\ell$, \sgdab{} requires at most $\bar\ell$ backtracking iterations and for each $\ell\in\{\rv{0},\ldots,\bar\ell\}$, $T\triangleq\lceil\log_2(1/p)\rceil$ independent \rbsgda{} runs are executed in parallel for \rv{at most} $K^{(\ell)}\triangleq \lceil\frac{64}{\epsilon^2}\big(\big(\eta_y^{(\ell)}\big)^3{(\mu^{(\ell)})^2}\rho\big)^{-1}\left(
\rv{F_0}-\bar{F}+
\rv{6\rho\delta\mu^{(\ell)}}\right)\rceil$ \rv{many \rbsgda{} iterations} with dual step size $\eta_{y}^{(\ell)}\triangleq\frac{1}{L^{(\ell)}}=\frac{1}{{L^{0}}}\gamma^{\ell}$ {and $\mu^{(\ell)} = \max\{\mu^{0}\gamma^\ell,\underline{\mu}\}$}, and in each one of these \rbsgda{} iterations $M_x^{(\ell)}+M_y^{(\ell)}$ stochastic oracles are called,
where \rv{$M_x^{(\ell)}$ and $M_y^{(\ell)}$ are chosen as in (i) of Corollary~\ref{cor:sample-complexity-bounded} with $C=\frac{\pi^2}{3\bar p}$; hence, under the event $I\leq \bar\ell$, we have} $M_x^{(\ell)} \leq \frac{32}{\epsilon^2} \sigma_x^2 \Big(\log(n_x+1)+\log(\rv{4}(\bar\ell+1)^2/\bar p)\Big)\rv{\triangleq\bar M_x}$ and $M_y^{(\ell)} \leq \frac{32}{\epsilon^2} \sigma_y^2 \Big(\log(n_y+1)+\Big(1+\frac{6}{{\mu^{(\ell)}} \eta_y^{(\ell)}} \frac{2-{\mu^{(\ell)}} \eta_y^{(\ell)}}{1-{\mu^{(\ell)}} \eta_y^{(\ell)}}\Big)\log(\rv{4}(\bar\ell+1)^2/\bar p)\Big)\rv{\triangleq\bar M_y^{(\ell)}}$ as $\pi^2/3\leq 4$. Therefore, 
the worst case complexity $C_\epsilon$ under the 
event $I\leq\bar{\ell}$ 
can be further bounded above as follows:}
{\small
\begin{align*}
    C_\epsilon \leq T\cdot\sum_{\ell=\rv{0}}^{\lceil \log_{\frac{1}{\gamma}}(\rv{\cR})\rceil} \frac{64}{\epsilon^2}\frac{1}{\saa{\big(\eta_{y}^{(\ell)}\big)^3}(\rv{\mu^{(\ell)}})^2\rho}\left(
    F_0-\bar{F}+
    6\rho\delta\mu^{(\ell)}\right)\cdot\rv{\Big(\bar M_x+\bar M_y^{(\ell)}\Big).}
\end{align*}}%
\saa{For each $\ell\in\{\rv{0},\ldots,\bar\ell\}$, since {$\mu^{(\ell)}\eta_{y}^{(\ell)}\leq \frac{\mu^{0}}{L^{0}}$}, we have \rv{$\frac{2-\mu^{(\ell)}\eta_{y}^{(\ell)}}{1-\mu^{(\ell)}\eta_{y}^{(\ell)}}\leq 1+\frac{L^{0}}{L^{0}-\mu^{0}}$.}} Therefore, using the lower bound {$\saa{\eta_{y}^{(\ell)}}\geq {\gamma}/{\hat{L}}$, $\mu^{(\ell)}\geq\gamma\hat{\mu}$ and \xz{the fact $T = \lceil \log_2(\frac{1}{p})\rceil$}} {and $\hat{\kappa}\triangleq \hat{L}/\hat{\mu}$}, we can bound the complexity as follows:
{{\small
\begin{align*}
    C_\epsilon
    =\cO\Bigg( 
    &\frac{1}{\gamma^5}\log_2\Big(\frac{1}{p}\Big)\log_{\frac{1}{\gamma}}(\cR)\frac{\hat{L}\hat{\kappa}^2}{\epsilon^4}\left(\frac{F_0-\bar{F}}{\rho}+
    6\rho\delta\mu^{0}\right)\cdot\\
    &\left[\sigma_x^2\left(\rv{\log(n_x)+\log\Big(\frac{1}{\bar p}\Big)}\right)+\sigma_y^2\left(\rv{\log(n_y)}+\hat{\kappa}\frac{1}{\gamma^2} \rv{\frac{L^{0}}{L^{0}-\mu^{0}}}\log\Big(\frac{1}{\bar p}\Big)\right)\right]\Bigg),
\end{align*}}}%
which completes the proof.}
\end{proof}
\vspace*{-4mm}
\begin{remark}
    For simplicity suppose $p=\bar p\in(0,1)$. \saa{For the case $\sigma_x,\sigma_y>0$, the oracle complexity of {$\cO(\hat{\kappa}^3 \hat{L}\log^2(1/p)\epsilon^{-4})$} for \sgdab{} is established in \cref{ncsc_thm}. 
    Moreover, for the case $\sigma_x>0$ and $\sigma_y=0$, the worst-case complexity improves from {$\cO\Big(\hat{L}\hat{\kappa}^3\log^2(\frac{1}{p}){\epsilon^{-4}}\Big)$ to $\cO\Big(\hat{L}\hat{\kappa}^2\log^2(\frac{1}{p}){\epsilon^{-4}}\Big)$, i.e., a {$\hat{\kappa}$}-factor improvement.}}
\end{remark}
{\begin{remark}
\label{rem:known-mu}
For simplicity suppose $p=\bar p\in(0,1)$. When $\mu$ is known, setting $\mu^{0} = \mu = \underline{\mu}$ and $L^{0} = \mu/\gamma$ implies that $\cR = \max\{1,\gamma\kappa\}$; hence, $\hat{L} = \cR L^{0}=\max\{\mu/\gamma, L\}$ and $\hat{\mu}\triangleq\max\{\underline{\mu}/\gamma, \mu^{0}/\cR\}=\mu/\gamma$. Therefore, $\hat\kappa=\hat L/\hat \mu=\kappa\gamma$, and the oracle complexity can be bounded by $\cO(\kappa^3 L\log^2(1/p)\epsilon^{-4})$.
\end{remark}}

\subsubsection{\rv{Convergence guarantees for \sgdab{} with unknown bounds}}
\label{sec:unknown-bounds}
In this section, we consider the setting with \rv{unknown $\sigma_x^2,\sigma_y^2$ under Assumptions~\ref{assumption:noise} and~\ref{as:bounded-oracle}.} Similar to \cref{sec:known-bounds}, we first argue that \sgdab{} stops w.p. 1 and that the number of backtracking iterations is $\cO(\log(\rv{\cR}))$ w.p. at least $1-p$, \rv{where 
\begin{align}
\label{eq:R-unknown}
   \cR\triangleq\max\Big\{\mu^{0}/\mu,~L/L^{0},~\sigma_x/\sigma_x^{0},~\sigma_y/\sigma_y^{0},~1\Big\}.
\end{align}}%
For the sake of notational simplicity, in this section we will consider a slightly different backtracking condition, without loss of generality. Indeed, according to~\cref{alg:GDA-B}, \sgdab{} stops at the $\ell$-th iteration if $(1-\frac{1}{M_x})v_{(t^*,\ell)}^x\leq (1+c)\widetilde\sigma_x^2$, $(1-\frac{1}{M_y})v_{(t^*,\ell)}^y\leq (1+c)\widetilde\sigma_y^2$ and $\tilde S_{(t^*,\ell)}
        \leq \frac{\epsilon^2}{4}$ hold simultaneously, where $\Big(v_{(t,\ell)}^x,v_{(t,\ell)}^y\Big)=V_{\tilde\bxi_{(t,\ell)}}\Big(\tilde\bx_{(t,\ell)},\tilde y_{(t,\ell)}\Big)$ for $t\in[T]$ as defined in {{\cref{algeq:RBSGDA_runs}}} of \sgdab, and $t^*=\argmin_{t\in[T]}\tilde S_{(t,\ell)}$. In the analysis below, we replace this condition with a slightly simpler one as given below:
        \begin{equation}
        \label{eq:stop-cond-simple}
        {\small
        \begin{aligned}
            \mathrm{StopCond}(\ell):\quad \tilde S_{(t^*,\ell)}
        \leq \frac{\epsilon^2}{4}\quad\mbox{\textbf{and}}\quad \Big(1-\frac{1}{M_x}\Big)v_{(t^*,\ell)}^x+\Big(1-\frac{1}{M_y}\Big)v_{(t^*,\ell)}^y\leq (1+c)(\widetilde\sigma_x^2+\widetilde\sigma_y^2),
        \end{aligned}}%
        \end{equation}
        where $t^*$ is defined as before, i.e., $t^*=\argmin_{t\in[T]}\tilde S_{(t,\ell)}$.
\begin{lemma}
\label{lem:stopping_probability-unknown}
    \sa{\rv{Given an arbitrary failure probability $p\in (0,1)$ and the sample variance test parameter $c\in(0,1)$, for any given tolerance $\epsilon>0$ sufficiently small, i.e., $\epsilon=\cO\Big(\min\{\sigma_x^{0}\frac{\sigma_x}{B_x},~\sigma_y^{0}\frac{\sigma_y}{B_y}\}\frac{c}{\sqrt{\log(1/p)}}\Big)$}, suppose we execute \sgdab{} with $K=\lceil\frac{64N}{\epsilon^2\eta_x}\left(\rv{F_0-\bar{F}+6\rho\delta\widetilde\mu}\right)\rceil$ \rv{for some given $F_0$ and $\delta$ such that $F_0\geq F(\bx^0)$ and $\delta\geq \delta^0$}, \rv{$M_x\geq  \mg{\lceil} \frac{128}{\epsilon^2}\rv{\widetilde\sigma_x^2}\saa{+1} \mg{\rceil}$ and $M_y \geq \mg{\lceil} (1+\frac{6}{\rv{\widetilde\mu} \eta_y} \frac{2-\rv{\widetilde\mu} \eta_y}{1-\rv{\widetilde\mu} \eta_y})\frac{128}{\epsilon^2}\rv{\widetilde\sigma_y^2}\saa{+1}\mg{\rceil}$, and \rv{$T=\lceil\log_2(2/p)\rceil$}. Then,} \sgdab{} stops w.p. 1, and the probability that the stopping condition $\mathrm{StopCond}(\ell)$ in \eqref{eq:stop-cond-simple} holds within at most $\bar{\ell}\triangleq\lceil\log_{\frac{1}{\gamma}}(\rv{\cR})\rceil$ backtracking iterations, i.e., for some $\ell\leq \lceil\log_{\frac{1}{\gamma}}(\rv{\cR})\rceil$, is at least $1-p$, \rv{where $\cR$ is defined in \eqref{eq:R-unknown}}. Furthermore, the probability that $\mathrm{StopCond}(\ell)$ holds is at least $1-p$ for all $\ell\geq \bar \ell$.}
\end{lemma}
\begin{proof}
\rv{Consider the $\ell$-th backtracking iteration for some $\ell\geq 0$. Recall that for the $\ell$-th backtracking iteration of \sgdab, the parameter estimates are set to $\widetilde L=L^{0}/\gamma^\ell$, $\widetilde\mu=\max\{\mu^{0} \gamma^\ell,\underline{\mu}\}$, $\widetilde\sigma_x=\max\{\sigma_x^{0}/ \gamma^\ell,\bar{\sigma}_x\}$ and $\widetilde\sigma_y=\max\{\sigma_y^{0}/ \gamma^\ell,\bar{\sigma}_y\}$; hence, during the $\ell$-th backtracking iteration, the step sizes $\eta_y=\gamma^\ell/L^{0}$ and $\eta_x = N\rho\Big(\max\{\mu^{0} \gamma^\ell,\underline{\mu}\}\Big)^2\eta_y^3$ are adopted.} 

Next we define the conditional probability of \sgdab~stopping at \rv{the} $\ell$-th backtracking iteration given that it did not stop earlier--with initial point $(\bx^0,y^0)$ fixed, \rv{this probability is a function of 
$\ell\geq 0$ when algorithm parameters are set as described above.}
For any $\ell\in\integers_+$, let $\big(\tilde\bx_{(t,\ell)},\tilde y_{(t,\ell)}\big)$ for $t=1,\ldots,T$ denote the random \rbsgda~output for $T$ independent runs, all initialized from $(\bx^0,y^0)$ and using $\eta_y,\eta_x,K,M_x,M_y$ corresponding to given backtracking 
counter $\ell$.
Moreover, let $\Big(v_{(t,\ell)}^x,v_{(t,\ell)}^y\Big)=V_{\tilde\bxi_{(t,\ell)}}\Big(\tilde\bx_{(t,\ell)},\tilde y_{(t,\ell)}\Big)$ for $t=1,\ldots,T$ denote the sample variances as defined in {\texttt{\cref{algeq:RBSGDA_runs}}} of \sgdab, using $V$ function given in~\eqref{eq:sample_variance}.
 Consider \rv{$v_{(t^*,\ell)}^x$, $v_{(t^*,\ell)}^y$ and} $\tilde S_{(t^*,\ell)}$, where $t^*\triangleq\argmin_{t=1,\ldots,T}\tilde S_{(t,\ell)}$, and let $q(\ell)\triangleq\mathbf{P}\Big(A_{(\ell)}^s\cap A_{(\ell)}^v\Big)$ for the events $A_{(\ell)}^s$ and $A_{(\ell)}^v$ defined as 
 \rv{
 \begin{align}
 \label{eq:ql-unknown}
     A_{(\ell)}^s=\Big\{\tilde S_{(t^*,\ell)}\leq\frac{\epsilon^2}{4}\Big\},\quad A_{(\ell)}^v=\Big\{(1-\frac{1}{M_x})v_{(t^*,\ell)}^x+(1-\frac{1}{M_y})v_{(t^*,\ell)}^y\leq (1+c)(\widetilde\sigma_x^2+\widetilde\sigma_y^2)\Big\}.
 \end{align}}%

\rv{For any failure probability $p\in(0,1)$, 
\cref{thm:concentration-bounded} and Remark~\ref{rem:small-eps} imply that for all sufficiently small tolerance $\epsilon>0$, i.e., $0<\epsilon=\cO\big(c/\sqrt{\log(1/p)}\big)$, we can conclude that \begin{equation}
\label{eq:p-sandwich}
{\small
    \begin{aligned}
    \p{\Big(1-\frac{1}{M_x}\Big)v_{(t^*,\ell)}^x+\Big(1-\frac{1}{M_y}\Big)v_{(t^*,\ell)}^y\leq (1+c)(\sigma_x^2+\sigma_y^2)}\geq 1-p/2.
    \end{aligned}}%
\end{equation}}%
 
 Let \saa{$I\in\integers_{+}$} denote the random stopping time of \sgdab,
    i.e., \rv{$\mathbf{P}(I=0)=q(0)$} and $\mathbf{P}(I=\ell)=q(\ell)\Pi_{i=\rv{0}}^{\ell-1}(1-q(i))$ for \rv{$\ell\geq 1$. Note that 
    for all $\ell\geq \bar{\ell}=\lceil\log_{\frac{1}{\gamma}}(\rv{\cR})\rceil$, we have $\widetilde L\geq \mu$, $\widetilde\mu\leq \mu$, $\widetilde\sigma_x^2\geq\sigma_x^2$ and $\widetilde\sigma_y^2\geq\sigma_y^2$; therefore, Lemma~\ref{lem:small-eta} implies that $\p{A_{(\ell)}^s}\geq 1-p/2$ for all $\ell\geq \bar{\ell}$. Moreover, since $\widetilde\sigma_x^2\geq\sigma_x^2$ and $\widetilde\sigma_y^2\geq\sigma_y^2$ for $\ell\geq \bar\ell$,
    \eqref{eq:p-sandwich} implies that $\p{A_{(\ell)}^v}\geq 1-p/2$ for all $\ell\geq \bar\ell$. Thus, $q(\ell)\geq \p{A_{(\ell)}^s}+\p{A_{(\ell)}^v}-1\geq 1-p$ for $\ell\geq \bar{\ell}$. Let $A_{(\ell)}\triangleq A_{(\ell)}^s\cap A_{(\ell)}^v$ for $\ell\geq 0$. Since $\{A_{(\ell)}\}_{\ell\geq 0}$ are independent events} and $\rv{\p{A_{(\ell)}}}=q(\ell)\geq 1-p$ for $\ell\geq \bar{\ell}$, we can conclude that $\mathbf{P}(I<\infty)=1$, i.e., \sgdab~stops with probability 1. Moreover, it follows from the same induction argument in the proof of Lemma~\ref{lem:stopping_probability} that $\mathbf{P}(I\leq \bar{\ell})\geq 1-p$.
\end{proof}
Next, we show that \sgdab{} output $(\bx_\epsilon,y_\epsilon)$ is $\cO(\epsilon)$-stationary with high probability.
\begin{theorem}
\label{thm:p-bound-unknown}
    Let $\bar\ell\triangleq\lceil\log_{\frac{1}{\gamma}}(\rv{\cR})\rceil$ for $\cR$ as in \eqref{eq:R-unknown}. Under the premise of Lemma~\ref{lem:stopping_probability-unknown}, for any $(\bx^0,y^0)\in\dom g\times\dom h$ and $p\in (0,1)$, 
   \sgdab{} with \rv{$T=\lceil\log_2(2/p)\rceil$} and employing $\mathrm{StopCond}(\ell)$ in \eqref{eq:stop-cond-simple} stops w.p.1 returning $(\bx_\epsilon,y_\epsilon)$ such that
\begin{equation*}
\begin{aligned}
\p{\norm{G\left(\bx_\epsilon,y_\epsilon\right)}\leq\epsilon}\geq 1-\sum_{\ell=0}^{\bar\ell-1} \tilde b_\ell(\epsilon)-\sum_{\ell\geq \bar\ell} \tilde b_\ell(\epsilon) p^{\ell-\bar\ell},
\end{aligned}
\end{equation*}
where
$\tilde b_\ell(\epsilon)\triangleq \tilde b^x_\ell(\epsilon)+\tilde b^y_\ell(\epsilon)$ and under Assumption~\ref{as:bounded-oracle}, $\tilde b^x_\ell(\epsilon),\tilde b^y_\ell(\epsilon)$ take the form:
{\small
\begin{align*}
    \tilde b^x_\ell(\epsilon)=(n_x+1)\Big[\exp\Big(-\frac{M_x\epsilon^2/16}{2(1+c)\tilde\sigma^2+\frac{B_x\epsilon}{6\sqrt{2}}}\Big)+\exp\left(\frac{-M_x(1+c)\tilde\sigma_x^2/4}{\sigma_x^2+B_x\sigma/6}\right)\Big]+\exp \left(-M_x \frac{(1+c)^2}{10}\frac{\tilde\sigma_x^4}{B_x^2\sigma_x^2}\right),\\
    \tilde b^y_\ell(\epsilon)=(n_y+1)\Big[\exp\Big(-\frac{M_y\epsilon^2/16}{2(1+c)\tilde\sigma^2+\frac{B_y\epsilon}{6\sqrt{2}}}\Big)+\exp\left(\frac{-M_y(1+c)\tilde\sigma_y^2/4}{\sigma_y^2+B_x\sigma/6}\right)\Big]+\exp \left(-M_y \frac{(1+c)^2}{10}\frac{\tilde\sigma_y^4}{B_y^2\sigma_y^2}\right),
\end{align*}}%
for $M_x,M_y,\tilde\sigma_x,\tilde\sigma_y$ denoting values for the backtracking iteration $\ell\geq 0$, and $\tilde\sigma^2=\tilde\sigma_x^2+\tilde\sigma_y^2$.
\end{theorem}
\begin{proof}
\rv{Since we consider the setting with both $\sigma_x^2$ and $\sigma_y^2$ are unknown, we set $\bar\sigma_x=\bar\sigma_y=+\infty$ in \cref{alg:GDA-B}. Thus, for \sgdab, within the backtracking iteration $\ell\geq \rv{0}$, the parameter estimates are set to $\widetilde L=L^0/\gamma^\ell$, $\widetilde\mu=\max\{\mu^{0} \gamma^\ell,\underline{\mu}\}$, $\widetilde\sigma_x=\sigma_x^{0}/ \gamma^\ell$ and $\widetilde\sigma_y=\sigma_y^{0}/ \gamma^\ell$; hence, the step sizes $\eta_y,\eta_x$ chosen as in line~\ref{algeq:step-sizes} of \sgdab, and the sample sizes $M_x,M_y$ together with the iteration budget $K$ for \rbsgda{} chosen as in Lemma~\ref{lem:stopping_probability-unknown}
are all functions of $\ell$ and also of other fixed problem parameters, i.e., $\bar{F},N,\sigma_x^2,\sigma_y^2$, and algorithm parameters, i.e., $\mu^0,L^0,\bx^0,y^0,F_0,\delta,\epsilon,\gamma,p,\bar p$.}

\saa{In the rest, for the ease of notation, we suppress the dependency of $\eta_y$, $\eta_x$, $K$, $M_x$, $M_y$ on $\ell$ and other parameters.} Let $\{\big(\bx^k_{(t,\ell)},y^k_{(t,\ell)}\big)\}_{k=0}^{K-1}$ for $t=1,\ldots,T$ denote the \rbsgda~iterate sequences for $T$ independent runs, all initialized from $(\bx^0,y^0)$ and using $\eta_y,\eta_x,K,M_x,M_y$ corresponding to given backtracking iteration counter $\ell\geq 0$. Fix an arbitrary $\ell\geq 0$, suppose the sample sizes $M_x,M_y$ corresponding to the given backtracking iteration $\ell$ are deterministic integers satisfying the condition in Lemma~\ref{lem:stopping_probability-unknown}. 
Next, for any given $\ell\geq 0$ and $t\in\{1,\ldots,T$\}, let $\tilde k_t\sim\cU[0,K-1]$. \saa{According to \sgdab, for \rv{any fixed $\ell\geq 0$}, we have} $\tilde S_{(t,\ell)}=\norm{\tilde
    G\big(\bx^{\tilde k_t}_{(t,\ell)},y^{\tilde k_t}_{(t,\ell)};\bxi^{\tilde k_t}_{(t,\ell)}\big)}^2$ for all $t=1,\ldots, T$ \rv{--here, $\tilde k_t$ is sampled for each 
    $t$ \mg{in an i.i.d.~fashion}}. Moreover, for any given $\ell\geq 0$, define $t^*=\argmin\{\tilde S_{(t,\ell)}: t=1,\ldots,T\}$ and set $(\tilde\bx_{(\ell)},\tilde y_{(\ell)};\tilde \bxi_{(\ell)})=(\bx^{\tilde k_t}_{(t,\ell)},y^{\tilde k_t}_{(t,\ell)};\bxi^{\tilde k_t}_{(t,\ell)})$ for $t=t^*$. 
    Due to symmetry, $t^*$ is uniformly distributed on the support $\{1,\ldots,T\}$ since $\{\tilde S_{(t,\ell)}\}_{t=1}^T$ are i.i.d. random variables for each 
    $\ell\geq 0$. Moreover, in \sgdab, we set $v_{(\ell)}^x=v_{(t^*,\ell)}^x$ and $v_{(\ell)}^y=v_{(t^*,\ell)}^y$ for all $\ell\in\integers_+$.

Recall the stopping time $I\in\integers_+$ defined in the proof of \rv{Lemma}~\ref{lem:stopping_probability-unknown}, i.e.,
{\small
\begin{align*}
     I\triangleq\min\Big\{\ell\in\integers_+:\ \norm{\tilde G\big(\tilde\bx_{(\ell)},\tilde y_{(\ell)};\tilde \bxi_{(\ell)}\big)}\leq \frac{\epsilon}{2},\ \Big(1-\frac{1}{M_x}\Big)v_{(\ell)}^x+\Big(1-\frac{1}{M_y}\Big)v_{(\ell)}^y\leq (1+c)(\widetilde\sigma_x^2+\widetilde\sigma_y^2)\Big\};
 \end{align*}}%
 hence, $(\bx_\epsilon,y_\epsilon)=(\tilde\bx_{(I)},\tilde y_{(I)})$, i.e., $(\bx_\epsilon,y_\epsilon;\bxi_\epsilon)=(\tilde\bx_{(\ell)},\tilde y_{(\ell)};\tilde\bxi_{(\ell)})$ when $I=\ell$, and we have \rv{$\mathbf{P}(I=0)=q(0)$} and $\mathbf{P}(I=\ell)=q(\ell)\Pi_{i=0}^{\ell-1}(1-q(i))$ for \rv{$\ell\geq 1$, where $q(\ell)\triangleq\mathbf{P}\Big(A_{(\ell)}^s\cap A_{(\ell)}^v\Big)$ for $A_{(\ell)}^s$ and $A_{(\ell)}^v$ as defined in \eqref{eq:ql-unknown} for $\ell\geq 0$, i.e., $q(0)=\p{I=0}$ and $q(\ell)=\p{I=\ell~\mid~I>\ell-1}$ for $\ell\geq 1$.} In the rest, let $p(\ell)\triangleq \mathbf{P}(I=\ell)$ for $\ell\geq 0$, i.e., $p(\ell)$ denotes the probability of \sgdab{} stopping at the $\ell$-th backtracking iteration for $\ell\geq 0$; hence, \rv{the output $(\bx_\epsilon,y_\epsilon)$ of the algorithm (see line~\ref{algeq:output} of \sgdab{}) satisfies} $(\bx_\epsilon,y_\epsilon)=(\tilde\bx_{(\ell)},\tilde y_{(\ell)})$ with probability $p(\ell)$. Therefore, the same arguments we used in the proof of \cref{thm:p-bound} lead to
\begin{equation}
\label{eq:det-gmap-unknown}
{\small
\begin{aligned}
    \p{\norm{G\left(\bx_\epsilon,y_\epsilon\right)}\leq \epsilon} = \sum_{\ell\geq 0} \p{\norm{G\left(\tx,\ty\right)}\leq \epsilon\ \bigg|\ A_{(\ell)}}q(\ell)\Pi_{\ell'=0}^{\ell-1}(1-q(\ell'))=\sum_{\ell\geq 0} w(\ell) p(\ell),
\end{aligned}}%
\end{equation}
where $A_{(\ell)}\triangleq A_{(\ell)}^s\cap A_{(\ell)}^v$ as defined in \eqref{eq:ql-unknown}, $w(\ell)\triangleq \p{\norm{G\left(\tx,\ty\right)}\leq \epsilon\ \bigg|\ A_{(\ell)}}$ for $\ell\geq 0$. Let
\rv{
\begin{equation}
    \label{eq:C1C2-sets}
 \begin{aligned}
     C_{(\ell)}^1 &\triangleq \Big\{\norm{G\left(\tx,\ty\right)}>\epsilon,\quad \sigma^2\big(\tx,\ty\big)\leq 2(1+c)\widetilde\sigma^2\Big\},\\
     C_{(\ell)}^2 &\triangleq \Big\{\norm{G\left(\tx,\ty\right)}>\epsilon,\quad \sigma^2\big(\tx,\ty\big)> 2(1+c)\widetilde\sigma^2\Big\},
 \end{aligned}
 \end{equation}}%
 for $\ell\geq 0$, where $\sigma^2\big(\tx,\ty\big)\triangleq \sigma_x^2\big(\tx,\ty\big)+\sigma_y^2\big(\tx,\ty\big)$ and $\widetilde\sigma^2\triangleq \widetilde\sigma_x^2+\widetilde\sigma_y^2$. Thus,
\begin{equation}
\label{eq:w0-C1C2-sets}
{\small
\begin{aligned}
w(\ell)=1-\p{A_{(\ell)}~\big|~C_{(\ell)}^1}\cdot\frac{\p{C_{(\ell)}^1}}{\p{A_{(\ell)}}}-\p{A_{(\ell)}~\big|~C_{(\ell)}^2}\cdot\frac{\p{C_{(\ell)}^2}}{\p{A_{(\ell)}}}. 
\end{aligned}}%
 \end{equation}
In the rest, we adopt the following notation: for any $(\bx,y)\in\dom g\times\dom h$, we define $\bz=(\by,y)$ and $G(\bz)=G(\bx,y)$ as given in Definition~\ref{def:s-gradmapping} with step sizes $\eta_x$ and $\eta_y$ corresponding to $\ell$-th backtracking iteration.
 
Similar to the proof of \cref{thm:p-bound}, we define two other important quantities of interest. Let $w_1(\ell)\triangleq \p{A_{(\ell)}~\big|~C_{(\ell)}^1}$ and $w_2(\ell)\triangleq \p{A_{(\ell)}~\big|~C_{(\ell)}^2}$, where $\tz=(\tx,\ty)$ for $\ell\in\integers_+$. Note that
\begin{subequations}
\label{eq:w-bounds}
\begin{align}
   w_1(\ell) &= \E{\p{A_{(\ell)}~\big|~\tz\,}~\big|~C_{(\ell)}^1}\leq \E{\p{A^s_{(\ell)}~\big|~\tz\,}~\big|~C_{(\ell)}^1},
   \label{eq:w1-bound-0}\\
   w_2(\ell) &= \E{\p{A_{(\ell)}~\big|~\tz\,}~\big|~C_{(\ell)}^2}\leq \E{\p{A^v_{(\ell)}~\big|~\tz\,}~\big|~C_{(\ell)}^2},
   \label{eq:w2-bound-0}
\end{align}
\end{subequations}
where the inequalities follow from $A_{(\ell)}\triangleq A_{(\ell)}^s\cap A_{(\ell)}^v$ for $\ell\geq 0$. Next, we upper bound both $w_1(\ell)$ and $w_2(\ell)$ in order to obtain a lower bound for $w(\ell)$ for $\ell\geq 0$. 

First, to bound $w_1(\ell)$, we observe that $A_{(\ell)}^s=\Big\{\tilde S_{(t^*,\ell)}\leq\frac{\epsilon^2}{4}\Big\}=\{\norm{\tilde G(\tz;\txi)}\leq \frac{\epsilon}{2}\}$.
Moreover, note that for any $(\bx,y)\in\dom g\times\dom h$, $\big\|\tilde G\big(\bx,y;\txi\big)\big\|\leq \frac{\epsilon}{2}$ implies $\big\|\tilde G\big(\bx,y;\txi\big)-G(\bx,y)\big\|\geq \norm{G(\bx,y)}-\frac{\epsilon}{2}$; thus,
{\small
\begin{align*}
    \p{A^s_{(\ell)}~\big|~\tz\,}\leq \p{\big\|\tilde G\big(\tz;\txi\big)-G(\tz)\big\|\geq \norm{G(\tz)}-\frac{\epsilon}{2}~\Big|~\tz\,}.
\end{align*}}%
Therefore, \eqref{eq:w1-bound-0} and Lemma~\ref{lem:Chebyshev} together imply that
{\small
\begin{align*}
    w_1(\ell) 
    &\leq \E{(n_x+1) \exp\left(-M_x\frac{\sigma_x^2(\tz)}{B_x^2}H\left(\frac{B_x}{\sqrt{2}}\cdot\frac{\norm{G(\tz)}-\frac{\epsilon}{2}}{ \sigma_x^2(\tz)}\right)\right)~\Big|~C_{(\ell)}^1}\\
    &\quad +\E{(n_y+1) \exp\left(-M_y\frac{\sigma_y^2(\tz)}{B_y^2}H\left(\frac{B_y}{\sqrt{2}}\cdot\frac{\norm{G(\tz)}-\frac{\epsilon}{2}}{ \sigma_x^2(\tz)}\right)\right)~\Big|~C_{(\ell)}^1}\\
    & \leq \sup_{t_1,t_2\geq 0}\quad (n_x+1)\exp\Big(-M_x\frac{t_1}{B_x^2}\cdot H\Big(\frac{B_x (t_0-\frac{\epsilon}{2})}{\sqrt{2}t_1}\Big)\Big)+(n_y+1)\exp\Big(-M_y\frac{t_2}{B_y^2}\cdot H\Big(\frac{B_y (t_0-\frac{\epsilon}{2})}{\sqrt{2}t_2}\Big)\Big)\\
    & \qquad \mbox{s.t.}\qquad t_0>\epsilon,~ t_1,t_2\leq 2(1+c)\tilde\sigma^2,
\end{align*}}%
where the second inequality follows from the conditional event $C_{(\ell)}^1$, which implies that $\norm{G(\tz)}>\epsilon$ and $\sigma_x^2(\tz)+\sigma_y^2(\tz)\leq 2(1+c)\tilde\sigma^2$; hence, replacing $\norm{G(\tz)}$ with $t_0$, $\sigma_x^2(\tz)$ and $\sigma_y^2(\tz)$ with $t_1$ and $t_2$, respectively, and taking the supremum over $t_0, t_1, t_2$, we end up with a deterministic bound with the supremum achieved at $t_0=\epsilon$, and $t_1=t_2=2(1+c)\tilde\sigma^2$, and finally using $H(u)\geq \frac{u^2/2}{1+u/3}$ for all $u>0$ (see~Lemma~\ref{lem:Bennett-helper} in \cref{sec:bounded-proofs}), we get
{\small
\begin{align}
\label{eq:w1-bound}
    w_1(\ell) 
    & \leq (n_x+1)\exp\Big(-\frac{M_x\epsilon^2/16}{2(1+c)\tilde\sigma^2+\frac{B_x\epsilon}{6\sqrt{2}}}\Big)+(n_y+1)\exp\Big(-\frac{M_y\epsilon^2/16}{2(1+c)\tilde\sigma^2+\frac{B_y\epsilon}{6\sqrt{2}}}\Big)\triangleq\bar w_1(\ell).
\end{align}}%
Next, we bound $w_2(\ell)$. Under the event $C_{(\ell)}^2$, we have $\sigma_x^2\big(\tz\big)+\sigma_y^2\big(\tz\big)> 2(1+c)\widetilde\sigma^2$; hence,
{\small
\begin{align*}
\p{A^v_{(\ell)}~\big|~C_{(\ell)}^2\,}
&\leq \p{(1-\frac{1}{M_x})v_{(t^*,\ell)}^x-\sigma_x^2(\tz)+(1-\frac{1}{M_y})v_{(t^*,\ell)}^y-\sigma_y^2(\tz)\leq -(1+c)\widetilde\sigma^2~\Big|~C_{(\ell)}^2\,}\\
&\leq \p{(1-\frac{1}{M_x})v_{(t^*,\ell)}^x-\sigma_x^2(\tz)\leq -(1+c)\tilde\sigma_x^2~\Big|~C_{(\ell)}^2\,}\\
&\quad +\p{(1-\frac{1}{M_y})v_{(t^*,\ell)}^y-\sigma_y^2(\tz)\leq -(1+c)~\tilde\sigma_y^2\Big|~C_{(\ell)}^2\,},
\end{align*}}%
\rv{which follows from the definition of $A^v_{(\ell)}$ in \eqref{eq:ql-unknown}.} Therefore, it follows from Theorem~\ref{thm:concentration-bounded} that
{\footnotesize
\begin{equation*}
\begin{aligned}
&\E{\p{A^s_{(\ell)}~\big|~\tz\,}\big|~C^2_{(\ell)}}\\
&\leq \E{\exp \left(\frac{-M_x\sigma_x^2(\tz)}{B_x^2}H\left(\frac{(1+c)\tilde\sigma_x^2}{2\sigma_x^2(\tz)}\right)\right)
				+(n_x+1)\exp\left(\frac{-M_x\sigma_x^2(\tz)}{B_x^2}\cdot H\left(\frac{B_x\sqrt{(1+c)\tilde\sigma_x^2/2}}{\sigma_x(\tz)}\right)\right)\big|~C^2_{(\ell)}}\\
& \ + \E{\exp \left(\frac{-M_y\sigma_y^2(\tz)}{B_y^2}H\left(\frac{(1+c)\tilde\sigma_y^2}{2\sigma_y^2(\tz)}\right)\right)
				+(n_y+1)\exp\left(\frac{-M_y\sigma_y^2(\tz)}{B_y^2}\cdot H\left(\frac{B_y\sqrt{(1+c)\tilde\sigma_y^2/2}}{\sigma_y(\tz)}\right)\right)\big|~C^2_{(\ell)}}.
\end{aligned}
\end{equation*}}%
Moreover, since $\sigma_x^2(\tz)\leq \sigma^2_x$ and $\sigma_y^2(\tz)\leq \sigma_y^2$, we first use the fact that for any fixed $C>0$, $u\mapsto uH(C/u)$ is decreasing in $u$ over $\{u\in\reals:~u>0\}$ in order to further bound the first and the third terms in the above inequality, and we also use $H(u)\geq\frac{u^2/2}{1+u/3}$ for $u>0$ together with $(1+c)\tilde\sigma^2/2<\sigma^2(\tz)/4\leq\sigma^2/4$ to bound the second and the fourth terms on the right hand side of the above inequality; hence, \eqref{eq:w2-bound-0} implies that
{\small
\begin{align*}
w_2(\ell)
&\leq \exp \left(-M_x\frac{\sigma_x^2}{B_x^2}H\left(\frac{(1+c)\tilde\sigma_x^2}{2\sigma_x^2}\right)\right)
				+(n_x+1)\exp\left(\frac{-M_x(1+c)\tilde\sigma_x^2/4}{\sigma_x^2+B_x\sigma/6}\right)\\
& \quad + \exp \left(-M_y\frac{\sigma_y^2}{B_y^2}H\left(\frac{(1+c)\tilde\sigma_y^2}{2\sigma_y^2}\right)\right)
				+(n_y+1)\exp\left(\frac{-M_y(1+c)\tilde\sigma_y^2/4}{\sigma_y^2+B_y\sigma/6}\right).
\end{align*}}%
Finally, under the event $C^2_{(\ell)}$, we have $\frac{(1+c)\tilde\sigma_x^2}{2\sigma_x^2}\leq 1/4$; moreover, since $H(u)\geq \frac{2}{5}u^2$ for $u\in(0,1/4)$, 
\begin{equation}
    \label{eq:w2-bound}
{\small
\begin{aligned}
w_2(\ell)
&\leq \exp \left(-M_x \frac{(1+c)^2}{10}\frac{\tilde\sigma_x^4}{B_x^2\sigma_x^2}\right)+(n_x+1) \exp\left(\frac{-M_x(1+c)\tilde\sigma_x^2/4}{\sigma_x^2+B_x\sigma/6}\right)\\
& \quad + \exp \left(-M_y \frac{(1+c)^2}{10}\frac{\tilde\sigma_y^4}{B_y^2\sigma_y^2}\right)+(n_y+1) \exp\left(\frac{-M_y(1+c)\tilde\sigma_y^2/4}{\sigma_y^2+B_y\sigma/6}\right)\triangleq \bar w_2(\ell).
\end{aligned}}%
\end{equation}
Clearly, for all $\ell\in\integers_+$, due to \eqref{eq:w0-C1C2-sets}, the following relation among $w_1(\ell)$, $w_2(\ell)$ and $w(\ell)$ holds:
\begin{equation}
\label{eq:w-unknown}
{\small
\begin{aligned}
    w(\ell)= 1- w_1(\ell)/q(\ell)\cdot \p{C_{(\ell)}^1}-w_2(\ell)/q(\ell)\cdot \p{C_{(\ell)}^2}\geq 1- \Big(w_1(\ell)+w_2(\ell)\Big)/q(\ell),
\end{aligned}}%
\end{equation}
where the inequality follows from the definition of $q(\ell)$ and $\p{C_{(\ell)}^1},\p{C_{(\ell)}^2}\in [0,1]$.
Thus, combining \eqref{eq:det-gmap-unknown}, \eqref{eq:w0-C1C2-sets}, \eqref{eq:w1-bound} and \eqref{eq:w2-bound}, we obtain
\begin{equation}
\label{eq:final-G-bound}
\begin{aligned}
    \p{\norm{G\left(\bx_\epsilon,y_\epsilon\right)}\leq \epsilon}=\sum_{\ell\geq 0} w(\ell) p(\ell)\geq 1-\sum_{\ell\geq 0}\Big(\bar w_1(\ell)+\bar w_2(\ell)\Big)\Pi_{\ell'=0}^{\ell-1}(1-q(\ell')),
\end{aligned}
\end{equation}
where we used $\sum_{\ell\geq 0}p(\ell)=1$ and $p(\ell)=q(\ell)\Pi_{\ell'=0}^{\ell-1}(1-q(\ell'))$. Note that Lemma~\ref{lem:stopping_probability-unknown} shows that $q(\ell)\geq 1-p$ for all $\ell\geq \bar\ell\triangleq\lceil\log_{\frac{1}{\gamma}}(\rv{\cR})\rceil$; therefore, $\Pi_{\ell'=0}^{\ell-1}(1-q(\ell'))\leq p^{\ell-\bar\ell}$ for $\ell\geq\bar\ell$. Thus, we get the desired result. 
\end{proof}
\rv{Before we prescribe how $M_x,M_y$ should be chosen for this setting to achieve the desired probability bound $\p{\norm{G\left(\bx_\epsilon,y_\epsilon\right)}\leq \epsilon}\geq 1-\bar p$ for any given $\bar p\in (0,1)$, we first state a technical result to bound infinite series arising in the proof of our high probability bound stated in Corollary~\ref{cor:sample-complexity-bounded-unkown}.}
\begin{lemma}
\label{lem:simple-bound}
    For $a,b>1$ and $C>0$, $\sum_{\ell=0}^\infty \frac{1}{a^{C b^\ell}}\leq \frac{b}{C\log a \log b}\cdot a^{-C/b}$. Moreover, for any $\bar p\in(0,1)$, $\frac{a b}{C\log a \log b}\cdot a^{-C/b}\leq \bar p$ holds for all $C>b+\frac{b}{\log a}\Big(\log\Big(\frac{1}{\bar p\log a\log b}\Big)+\log b\Big)$. Finally, for any $\bar p\in(0,1)$, $C>1$ and $b>1$, it holds that $\sum_{\ell=0}^\infty\exp\Big(-C\log\Big(1+\frac{4}{\bar p}\Big) b^\ell\Big)\leq \frac{\bar p}{\rv{8}}$ whenever $C\geq \frac{1}{\log(1+4/\bar p)}\max\{\log(16/\bar p),~\log(2)/\log(b)\}$.
\end{lemma}
\begin{corollary}
\label{cor:sample-complexity-bounded-unkown}
Under \cref{ass1,assumption:noise,as:bounded-oracle}, for any given $\bar p\in(0,1)$, $\p{\norm{G\left(\bx_\epsilon,y_\epsilon\right)}\leq \epsilon}\geq 1-\bar p$ holds when $M_x$ and $M_y$ are chosen such that $M_x=\frac{C_0}{\epsilon^2}(C_1^x\log(n_x+1)\tilde\sigma_x^2+C_2 (\sigma^0)^2/\bar\gamma^{\rv{2\ell}})$ and $M_y=\frac{C_0}{\epsilon^2}\Big(C_1^y\log(n_y+1)\tilde\sigma_y^2+ C_2 (\sigma^0)^2 \Big(1+\frac{6}{\rv{\widetilde\mu} \eta_y} \frac{2-\rv{\widetilde\mu} \eta_y}{1-\rv{\widetilde\mu} \eta_y}\Big)/\bar\gamma^{\rv{2\ell}}\Big)$ for \rv{$C_0\geq 48(1+c)$}, $C_1^x\geq (\sigma^0)^2/(\sigma^0_x)^2$, $C_1^y\geq (\sigma^0)^2/(\sigma^0_y)^2$, $C_2\geq \log\Big(1+\frac{4}{\bar p}\Big)$, $\bar\gamma\in(0,\gamma)$ and $\epsilon\in(0,\bar\epsilon)$, where $(\sigma^0)^2=(\sigma_x^0)^2+(\sigma_y^0)^2$ and
\begin{align*}
    \bar\epsilon^2=\cO\left((\underline{\sigma}^0)^2\min\left\{\varsigma_1\cdot \min\Big\{1,\frac{\log(\underline{n})}{\log(1/\bar p)+\log(1/\gamma)}\Big\},~\varsigma_2,~\varsigma_3\cdot\min\left\{1,~\frac{\log(1/\bar p)}{\log(1/p)}\right\}\right\}\right),
\end{align*}
for $\varsigma_1=\frac{\gamma^2}{\sigma^2+B\sigma}$, $\varsigma_2=\frac{(\underline{\sigma}^0)^4}{B^2\sigma^2}$, $\varsigma_3=c^2\min\{\sigma_x^2,~\sigma_y^2\}/B^2$,  $\underline{\sigma}^0=\min\{\sigma_x^0,\sigma_y^0\}$, $\sigma^2=\sigma_x^2+\sigma_y^2$, $B=\max\{B_x,B_y\}$ and $\underline{n}=\min\{n_x,n_y\}$.
\end{corollary}
\begin{proof}
Without loss of generality, suppose $\gamma/\bar\gamma= \sqrt{2}$. Given an arbitrary $p\in(0,1)$, for any $\ell\geq \bar\ell$, since $\tilde\sigma_x^2\geq\sigma_x^2$ and $\tilde\sigma_y^2\geq\sigma_y^2$, the event $A_{(\ell)}^v$ defined in \eqref{eq:ql-unknown} holds 
w.p. at least $1-p/2$ whenever
\begin{equation}
\label{eq:test-condition-unknown-bounded}
{\small
\begin{aligned}
    \p{\Big(1-\frac{1}{M_x}\Big)v^x_{(t^*,\ell)}\leq (1+c)\sigma_x^{2}}\geq 1-\frac{p}{4},\quad \p{\Big(1-\frac{1}{M_y}\Big)v^y_{(t^*,\ell)}\leq (1+c)\sigma_y^{2}}\geq 1-\frac{p}{4},
\end{aligned}}%
\end{equation}
and according to \cref{thm:concentration-bounded}, \eqref{eq:concentration-subGaussian-bounded} implies that \eqref{eq:test-condition-unknown-bounded} holds under \cref{assumption:noise,as:bounded-oracle} whenever $M_x,M_y\in\integers_+$ satisfy $M_x\geq \frac{1}{H(c)}\cdot\frac{B_x^2}{\sigma_x^2}\cdot\log\Big(\frac{4}{p}\Big)$ and $M_y\geq \frac{1}{H(c)}\cdot\frac{B_y^2}{\sigma_y^2}\cdot\log\Big(\frac{4}{p}\Big)$.

Our choice of $M_x,M_y$ immediately implies a trivial upper bound on $\bar w_1(\ell)$; indeed, it follows from \eqref{eq:w1-bound} and $\tilde\sigma\geq\sigma^0$ for all $\ell\geq 0$ that
{\small
\begin{align}
\label{eq:w1-bound2}
    \bar w_1(\ell) 
    & \leq (n_x+1)\exp\Big(-\frac{M_x\epsilon^2}{48(1+c)\tilde\sigma^2}\Big)+(n_y+1)\exp\Big(-\frac{M_y\epsilon^2}{48(1+c)\tilde\sigma^2}\Big)\leq \rv{2\exp\left(-C_2 \frac{\gamma^{2\ell}}{\bar\gamma^{2\ell}}\right)},
\end{align}}%
holds for all $\ell\geq 0$ and $\epsilon>0$ such that $\epsilon\leq 6\sqrt{2}\frac{(1+c)(\sigma^0)^2}{\max\{B_x,B_y\}}$. Therefore, given $\bar p\in (0,1)$, since $(\gamma/\bar\gamma)^2=2$, we get
    $\sum_{\ell\geq 0}\bar w_1(\ell)\leq\sum_{\ell\geq 0} 2\exp\left(-C_2 \frac{\bar\gamma^{2\ell}}{\gamma^{2\ell}}\right)\leq \bar p/2$ holds for all $C_2\geq \log\Big(1+\frac{4}{\bar p}\Big)$ since $\sum_{\ell\geq 1} \frac{2}{e^{C_2 \ell}}=\rv{2/(e^{C_2}-1)}\leq \bar p/2$ holds for such $C_2>0$.
    Next, we derive an upper bound on $\bar w_2(\ell)$ defined in \eqref{eq:w2-bound}. First, let $\zeta_x\triangleq \sigma_x^2+B_x\sigma/6$ and $\zeta_y\triangleq \sigma_y^2+B_y\sigma/6$, and consider
    \begin{align*}
        \MoveEqLeft \bar w_3(\ell)\triangleq (n_x+1) \exp\left(\frac{-M_x(1+c)\tilde\sigma_x^2/4}{\sigma_x^2+B_x\sigma/6}\right)+(n_y+1) \exp\left(\frac{-M_y(1+c)\tilde\sigma_y^2/4}{\sigma_y^2+B_x\sigma/6}\right)\\
        & \leq (n_x+1)\exp\Big(\frac{-12(1+c)^2}{\zeta_x\epsilon^2}\log(n_x+1)\tilde\sigma_x^2\Big)+(n_y+1)\exp\Big(\frac{-12(1+c)^2}{\zeta_y\epsilon^2}\log(n_y+1)\tilde\sigma_y^2\Big)\\
        & = \frac{n_x+1}{(n_x+1)^{\frac{12(1+c)^2(\sigma_x^0)^2}{\zeta_x\epsilon^2}\cdot \frac{1}{\gamma^{2\ell}}}}+\frac{n_y+1}{(n_y+1)^{\frac{12(1+c)^2(\sigma_y^0)^2}{\zeta_y\epsilon^2}\cdot \frac{1}{\gamma^{2\ell}}}};
    \end{align*}
    therefore, Lemma~\ref{lem:simple-bound} implies that 
    $\sum_{\ell=0}^\infty\bar w_3(\ell)\leq \bar p/4$ if
    \begin{align*}
        \frac{12(1+c)^2(\sigma_x^0)^2}{\zeta_x\epsilon^2}\gamma^2
        &>
        1+\frac{1}{\log(n_x+1)}\Big(\log\Big(\frac{
        \rv{2}}{\bar p\log(n_x+1)\log(1/\gamma)}\Big)+2\log(1/\gamma)\Big)
        \triangleq B_1(\bar p)\\
        \frac{12(1+c)^2(\sigma_y^0)^2}{\zeta_y\epsilon^2}\gamma^2
        &>
        1+\frac{1}{\log(n_y+1)}\Big(\log\Big(\frac{
        \rv{2}}{\bar p\log(n_y+1)\log(1/\gamma)}\Big)+2\log(1/\gamma)\Big)
        \triangleq B_2(\bar p).
    \end{align*}
    Thus, $\sum_{\ell=0}^\infty\bar w_3(\ell)\leq \bar p/4$ holds for all $\epsilon>0$ such that $\epsilon^2 < 12(1+c)^2\gamma^2 \min\Big\{\frac{(\sigma_x^0)^2}{\zeta_x B_1(\bar p)},~\frac{(\sigma_y^0)^2}{\zeta_y B_2(\bar p)}\Big\}$.

    Next, consider $\bar w_4(\ell)$ defined below such that $\bar w_2(\ell)=\bar w_3(\ell)+\bar w_4(\ell)$, i.e.,
    {\footnotesize
    \begin{align*}
        \MoveEqLeft \bar w_4(\ell)\triangleq 
        \exp \left(-M_x \frac{(1+c)^2}{10}\frac{\tilde\sigma_x^4}{B_x^2\sigma_x^2}\right)
        +\exp \left(-M_y \frac{(1+c)^2}{10}\frac{\tilde\sigma_y^4}{B_y^2\sigma_y^2}\right)\\
        & \leq \exp\Big(-4.8\frac{(1+c)^3}{\epsilon^2}\frac{(\sigma^0)^2(\sigma_x^0)^4}{B_x^2\sigma_x^2} \log\Big(1+\frac{4}{\bar p}\Big)\frac{1}{\bar\gamma^{\rv{2\ell}}\gamma^{4\ell}}\Big)+\exp\Big(-4.8\frac{(1+c)^3}{\epsilon^2}\frac{(\sigma^0)^2(\sigma_y^0)^4}{B_y^2\sigma_y^2} \log\Big(1+\frac{4}{\bar p}\Big)\frac{1}{\bar\gamma^{\rv{2\ell}}\gamma^{4\ell}}\Big),
    \end{align*}}%
    where the inequality follows from our choice of $M_x$, $M_y$, and from $C_2\geq \log\Big(1+\frac{4}{\bar p}\Big)$. Thus, Lemma~\ref{lem:simple-bound} implies that $\sum_{\ell=0}^\infty\bar w_4(\ell)\leq \bar p/4$ holds whenever
    {\small
    \begin{align*}
        &4.8\frac{(1+c)^3(\sigma^0)^2}{\epsilon^2}\cdot\min\left\{\frac{(\sigma_x^0)^4}{B_x^2\sigma_x^2},~\frac{(\sigma_y^0)^4}{B_y^2\sigma_y^2}\right\}\geq \frac{1}{\log(1+4/\bar p)}\max\Big\{\log(16/\bar p),~\frac{\log(2)}{2}\frac{1}{\log(1/(\bar\gamma\gamma^2))}\Big\}\triangleq B_3(\bar p),
    \end{align*}}%
    i.e., for all $\epsilon>0$ such that $\epsilon^2< 4\frac{(1+c)^3(\sigma^0)^2}{B_3(\bar p)}\cdot\min\left\{\frac{(\sigma_x^0)^4}{B_x^2\sigma_x^2},~\frac{(\sigma_y^0)^4}{B_y^2\sigma_y^2}\right\}$.

    Finally, we observed that \eqref{eq:test-condition-unknown-bounded} holds whenever $\min\{M_x,M_y\}\geq \max\left\{\frac{B_x^2}{\sigma_x^2},\frac{B_y^2}{\sigma_y^2}\right\}\cdot\frac{1}{H(c)}\cdot\log\Big(\frac{4}{p}\Big)$; hence, a sufficient condition on $M_x,M_y$ is
    \begin{align*}
        \frac{48(1+c)}{\epsilon^2} (\sigma^0)^2 \log\Big(1+\frac{4}{\bar p}\Big)\geq \max\left\{\frac{B_x^2}{\sigma_x^2},\frac{B_y^2}{\sigma_y^2}\right\}\cdot\frac{1}{H(c)}\cdot\log\Big(\frac{4}{p}\Big),
    \end{align*}
    which follows from $C_0\geq 48(1+c)$ and $C_2\geq \log\Big(1+\frac{4}{\bar p}\Big)$. Thus, for given $p,\bar p\in (0,1)$, \eqref{eq:test-condition-unknown-bounded} holds for all $\epsilon>0$ such that
    {\small
    \begin{align*}
        \epsilon^2 < 48(1+c)H(c)(\sigma^0)^2\cdot\min\left\{\frac{\sigma_x^2}{B_x^2},\frac{\sigma_y^2}{B_y^2}\right\}\cdot \log\Big(1+\frac{4}{\bar p}\Big)/\log\Big(\frac{4}{p}\Big);
    \end{align*}}%
    therefore, by simplifying all three bounds we derived above, we can conclude that $$\bar\epsilon^2=\cO\left((\underline{\sigma}^0)^2\min\left\{\Gamma_1,~\Gamma_2,~\Gamma_3\right\}\right),$$ where
    $\Gamma_1=\frac{\gamma^2}{\sigma^2+ B \sigma}\min\Big\{1,\frac{\log(\underline{n})}{\log(1/\bar p)+\log(1/\gamma)}\Big\}$, $\Gamma_2=\frac{(\underline{\sigma}^0)^4}{ B^2\sigma^2}\min\{1,\log(1/\bar p)\log(1/(\bar\gamma\gamma))\}$, and $\Gamma_3=H(c)\min\left\{\frac{\sigma_x^2}{B_x^2},\frac{\sigma_y^2}{B_y^2}\right\}\min\Big\{1,\frac{\log(1/\bar p)}{\log(1/p)}\Big\}$.
\end{proof}
\rv{The same sample complexity bound in Theorem~\ref{ncsc_thm} directly extends to the setting considered in this section, i.e., when $\sigma_x^2$ and $\sigma_y^2$ are unknown, using the same arguments; hence, we skip its proof in the interest of conciseness.}
\section{Weakly \mg{C}onvex-\saa{Merely} \mg{C}oncave~(WCMC) \mg{S}etting}
\label{sec:WCMC}
In this section, we 
\saa{consider the scenario with $\mu=0$, i.e., $f(\bx,\cdot)$} is merely concave for all \saa{$\bx\in \dom g$, under the bounded dual domain assumption.} 
\begin{assumption} \label{assp_ydomain}
{For the case $\mu=0$, we assume that} $\cD_y \triangleq \sup_{y_1,y_2\in \dom h} \norm{y_1 - y_2} \sa{<} \infty$.
\end{assumption}
Let \sa{$\rv{{\mu}_r}=\frac{\epsilon}{\cD_y}$}. \saa{To give a complexity result for the WCMC case, we} 
approximate \eqref{eq:main-problem} with
\begin{equation} \label{q:wcmc}
    \begin{aligned}
        \min_{\bx \in \cX} \max_{y\in \cY} \hat{\cL} (\bx,y) \triangleq 
        g(\bx) + \hat{f}(\bx,y) - h(y), \quad {\rm where} \quad \hat{f}(\bx,y) \triangleq f(\bx,y) - \frac{{\mu}_r}{2} \norm{y-\hat{y}}^2,
    \end{aligned}
\end{equation}
where $\hat{y}\in \dom g$ is an arbitrary given point. Note that \eqref{q:wcmc} is a WCSC minimax problem.
\begin{theorem}
\label{thm:mcwc}
Suppose $\mu=0$ and \cref{ass1,assumption:noise,as:bounded-oracle,assp_ydomain} hold \rv{with known variance bounds $\sigma_x^2$ and $\sigma_y^2$}.
Given any $(\bx^0,y^0)\in\dom g\times\dom h$, 
$\epsilon,L^{0}>0$ and $\gamma,p,\rv{\bar p}\in (0,1)$ such that $L^{0}>\mu^{0}=\mu_r$,
\sgdab, displayed in \cref{alg:GDA-B}, \rv{applied to \eqref{q:wcmc} with ${\mu}_r=\frac{\epsilon}{\cD_y}$ and sample sizes $M_x$ and $M_y$ chosen as in (i) of Corollary~\ref{cor:sample-complexity-bounded},} stops w.p.1 returning
$(\bx_\epsilon,y_\epsilon)$ satisfying \rv{$\p{\norm{G\left(\bx_\epsilon,y_\epsilon\right)}\leq 2\epsilon}\geq 1-\bar p$.} Moreover, 
with probability 
at least $1-p$, \sgdab{} stops within $\bar{\ell}\triangleq\lceil\log_{\frac{1}{\gamma}}(2L\cD_y/\epsilon)\rceil$ 
backtracking iterations which require $\cO\Big(L^3\cD_y^2\frac{1}{\epsilon^4}\Big)$ \rbsgda{} iterations 
with oracle complexity 
{\small
$$\cO\Big(\frac{L^3\cD_y^2}{\epsilon^6}\Big(F(\bx^0)-\bar{F}+\rv{\delta}\Big)\big(\sigma_x^2+\frac{L\cD_y}{\epsilon}\sigma_y^2\big)\log\Big(\frac{1}{\bar p}\Big)\log\Big(\frac{1}{p}\Big)\log\Big(\frac{L}{L^0}\Big)\Big)=\tilde{O}\Big(L^4 \cD_y^3{\epsilon^{-7}}\Big).$$}%
For the special case with $\sigma_x^2,\sigma_y^2=0$ and $N=1$, after setting $M_x,M_y=1$ and $T=1$, \sgdab{} can generate $(\bx_\epsilon,y_\epsilon)$ such that $\norm{G(\bx_\epsilon,y_\epsilon)} \leq \epsilon$ with 
\saa{gradient} complexity of $\cO\Big(L^3\cD_y^2\frac{1}{\epsilon^4}\Big)$.
\end{theorem}
\begin{proof}
\saa{Consider the 
gradient map 
of the regularized function $\hat f$, i.e.,  
${G}^r(\cdot)\triangleq [{G}_{\bx} (\cdot)^\top {G}_y^r(\cdot)^\top]^\top$ such that ${G}^r_y (\bx,y) \triangleq \Big[\prox{\eta_y h}\Big(y+\eta_y \grad_y \hat f(\bx,y)\Big)-y\Big]/ \eta_y$ for all $(\bx,y)\in \sa{\dom g\times \dom h}$. From \cref{ncsc_thm}, it follows that \sgdab, displayed in \cref{alg:GDA-B}, when applied to \eqref{q:wcmc}, stops w.p.~1 returning
$(\bx_\epsilon,y_\epsilon)$ such that
\rv{$\p{\norm{G^r\left(\bx_\epsilon,y_\epsilon\right)}\leq \epsilon}\geq 1-\bar p$.} Therefore, \saa{it follows from {${\mu}_r = \frac{\epsilon}{\cD_y}$} that} $(\bx_\epsilon,y_\epsilon)$ also satisfies the following bound \mg{with probability at least 
\rv{$1-\bar p$}}:} 
\begin{equation}
\label{eq:Gr-G}
{\small
    \begin{aligned}
        \norm{G(\bx_\epsilon,y_\epsilon)}&\leq \norm{G(\bx_\epsilon, y_\epsilon) - {{G}}^r(\bx_\epsilon,y_\epsilon)}+ \norm{{{G}}^r(\bx_\epsilon, y_\epsilon)}
         = \norm{G_y(\bx_\epsilon, y_\epsilon) - {{G}}^r_y(\bx_\epsilon,y_\epsilon)} + \norm{{{G}}^r(\bx_\epsilon, y_\epsilon)}\\
        &\leq \frac{1}{\eta_y}\norm{\prox{\eta_y h}(y_\epsilon + \eta_y \grad_y f(\bx_\epsilon,y_\epsilon))- \prox{\eta_y h}(y_\epsilon + \eta_y \grad_y\hat f(\bx_\epsilon,y_\epsilon)) } + \epsilon\\
        & \leq \|\grad_y f(\bx_\epsilon,y_\epsilon) - \grad_y \hat{f}(\bx_\epsilon,y_\epsilon)\| + \epsilon
        \leq {\mu}_r\cD_y + \epsilon
        \leq \epsilon + \epsilon = 2\epsilon,
    \end{aligned}}%
\end{equation}
where the third inequality is due to $\grad \hat{f} (\bx_\epsilon, y_\epsilon) = \grad_y f(\bx_\epsilon, y_\epsilon) - \hat{\mu} (y_\epsilon - \hat{y})$. Hence, we have $\mathbf{P}\Big(\norm{G(\bx_\epsilon,y_\epsilon)}\leq 
\rv{2\epsilon}\Big)\geq 
\rv{1-\bar p}$. 
Moreover, letting \rv{$L_r=L+\mu_r$ and $\kappa_r\triangleq\frac{L_r}{\mu_r}$,} 
\rv{the results of \cref{ncsc_thm} and Remark~\ref{rem:known-mu} yield} that with probability at least $1-p$, \sgdab{} stops within 
$\bar{\ell}\triangleq\lceil\log_{\frac{1}{\gamma}}(\rv{\cR})\rceil=\cO(\log(L/L^0))$
 backtracking iterations which require \rv{$\cO\Big(L_r{\kappa_r}^2\frac{1}{\epsilon^2}\Big)$} \rbsgda
~iterations and the oracle complexity is 
\rv{$\cO\Big({\frac{{L_r}{\kappa_r}^2}{\epsilon^4}\Big(
F_0-\bar{F}+
\delta\Big)}(\rv{\sigma_x^2}+{\kappa_r}\rv{\sigma_y^2})\log(\cR)\log\Big(\frac{1}{\bar p}\Big)\log\Big(\frac{1}{\rv{p}}\Big)\Big)$, where $\cR=\max\{1,~\frac{L_r}{L^0}\}=\cO(L/L^0)$ for $L^0>\mu_r$.}
Moreover, for the special case with $\sigma_x^2,\sigma_y^2=0$ and $N=1$, after setting $M_x,M_y=1$ and $T=1$, \sgdab{} can generate $(\bx_\epsilon,y_\epsilon)$ such that $\norm{ G^r(\bx_\epsilon,y_\epsilon)} \leq 2\epsilon$ with oracle complexity of $\cO\Big({L_r}\kappa_r^2\frac{1}{\epsilon^2}\Big)$. Substituting $\hat{\mu}=\frac{\epsilon}{\cD_y}$ and ${\kappa_r}\triangleq\frac{{L_r}}{{\mu_r}}=1+\cD_y\frac{L}{\epsilon}$ in the bounds above completes the proof. 
\end{proof}
\rv{The same sample complexity bound in Theorem~\ref{thm:mcwc} directly extends to the setting with \textit{unknown} $\sigma_x^2$ and $\sigma_y^2$ using the same arguments in the proof of Theorem~\ref{ncsc_thm}; hence, we skip its proof in the interest of conciseness.}
\section{Numerical Experiments}
\label{sec:numerics}

\saa{In our 
tests, we compare \sgdab{}, displayed in~\cref{alg:GDA-B}, with the adaptive method \tiada{}~{\citep{li2022tiada}} and also with other state-of-the-art algorithms for solving  WSCS minimax problems that are not agnostic to \rv{the problem parameters $L$, $\mu$ and $\sigma^2$:} \texttt{GDA}~{\citep{lin2020gradient}},  \texttt{AGDA}~{\citep{boct2020alternating}}, \texttt{sm-AGDA}~{\citep{yang2022faster}} and \vrlm{}~{\citep{mancino2023variance}} as benchmark. In our tests, 
\tiada{} has constantly performed better than \neada{} –similar to the comparison results in~\citep{li2022tiada}; this is why we only report results for \tiada{}, which is the main competitor of \sgdab{} as a parameter agnostic algorithm for stochastic WCSC
minimax problems.} 
The experiments with synthetic data are conducted \rv{on a MacBook Air equipped with an Apple M5 CPU and 24 GB
of unified memory while the DRO experiments with real data are run on Nvidia A100 80G}.\footnote{Our code is publicly available at \url{https://github.com/Qiushui-Xu/SGDA-B-WCSC}.}
\paragraph{\saa{Parameter settings.}}
To have a fair comparison, \saa{for each method we test, we adopt step sizes with theoretical guarantees for that method.} Indeed,
according to \cite{lin2020gradient,boct2020alternating}, 
we set $\tau=\Theta(1/(\kappa^2L))$, $\sigma=\Theta(1/L)$ for both \texttt{GDA} and \texttt{AGDA}. For \tiada{}, we set $\alpha=0.6$ and $\beta = 0.4$ as recommended in~\citep{li2022tiada}, 
and tune the initial step sizes $\tau_0$ and $\sigma_0$ from $\{100, 10, 1, 0.1, 0.01\}$ and set $v^x_0=v^y_0=1$. \saa{Indeed,} the step size rule for \tiada{} is 
\begin{equation}
\label{eq:tiada_step}
\sigma_t=\sigma_0/(v_{t+1}^y)^\beta,\quad \tau_t=\tau_0/(\max\{v_{t+1}^x, v_{t+1}^y\})^\alpha,
\end{equation}
where \saa{$v^x_{t+1}=v^x_t+\norm{\grad_x \tilde f(\bx_t,y_t;\omega)}^2$} and $v^y_{t+1}=v^y_t+\norm{\tilde G_y(\bx_t,y_t;\omega)}^2$ --see Definition~\ref{def:s-gradmapping} for 
$\tilde G_y$. 
For \texttt{sm-AGDA}, we set their parameters according to in~\cite[Theorem 4.1]{yang2022faster}. For \vrlm{}, we set their parameters according to  
~\cite[Lemma 3.10]{mancino2023variance} and \saa{use it with STROM type variance reduction, i.e., we set} \texttt{VR}-tag=STORM.

\rv{For \sgdab{}, we set $p=\bar p=0.1$ (hence, $T=\lceil\log_2(3/p)\rceil = 5$).
In our \mgrev{experiments} we assumed that $\mu$ is known and we initialize $\mu^0=\mu$, $L^0=\mu/\gamma$ --if both $L$ and $\mu$ are unknown, then one can initialize $L^0$ and $\mu^0$ based on Remark~\ref{rem:mu-L-initialization} using $M=100$ randomly generated points.} Given $\bx^0$, we used the approach in Remark~\ref{rem:x0-delta-choice} to set $F_0$, i.e., we ran the {Generalized AdaGrad~(\texttt{G-AdaGrad})} algorithm~\citep[Algorithm 4]{junchinest} for $T=10^4$ 
iterations to approximate $y_*(\bx^0)$ and set it to $y^0$; \rv{indeed, when implemented on $\max_{y\in Y}f(\bx^0,y)$, 
\texttt{G-AdaGrad} updates can be written as follows: given $y_0\in Y$ and $v_0>0$, for $v_{t+1}=v_t+\norm{\grad_y f(\bx^0,y_t)}^2$, $y_{t+1}=\Pi_{Y}\big(y_t+\frac{\eta}{v_{t+1}^\alpha}\grad_y f(\bx^0,y_t)\big)$ for $t=0,\ldots,T$, where $\Pi_{Y}(\cdot)$ denotes the Euclidean projection onto $Y$. 
We initialize \texttt{G-AdaGrad} using $\alpha=0.75$, $\eta=0.01 v_0$ and $v_0=\norm{\grad_y f(\bx^0,y_0)}^2$ for all the experiments in this section.} 
Since $F(\bx^0)=\max_{y\in Y}f(\bx^0,y)$ is a strongly concave optimization problem, after $T=10^4$ iterations of \texttt{G-AdaGrad}, $\norm{y_T-y_*(\bx^0)}/F(\bx^0)$ is relatively very small compared to $\norm{y_0-y_*(\bx^0)}/F(\bx^0)$ --indeed, for the synthetic data problems we tested below, $\norm{y_T-y_*(\bx^0)}/F(\bx^0)\approx 10^{-6}$; 
therefore, for all practical reasons, we set $F_0=\cL(\bx^0,y^0)$, $y^0=y_T\approx y_*(\bx^0)$ and $\delta=0$.
The target accuracy $\epsilon$ is 
chosen based on the initial primal gap
$\Delta_0 \triangleq F_0 - \bar{F}$, and we set $\epsilon = \sqrt{c_{\rm tol}~ \Delta_0}$ for some $c_{\rm tol}\in (0,1)$
-- note that for our test problems, $g(\cdot)=0$; hence, $F(\cdot)=\Phi(\cdot)$ is smooth, which implies that for any $\bx$, we have $\norm{\grad \Phi(\bx)}^2=\cO(F(\bx)-F^*)$, i.e., $\norm{\grad \Phi(\bx^0)}^2=\cO(\Delta_0)$. That is to say $c_{\rm tol}$ in our choice of $\epsilon$ can be seen as a relative accuracy value.
The inner iteration bound $K$ is then set depending on $\epsilon$ as in \sgdab{}, i.e.,
$K = \lceil \frac{64 N}{\epsilon^2 \eta_x}\Delta_0\rceil=\lceil 
\frac{64 \widetilde L^3}{ \rho \mu^2}\cdot \Delta_0/\epsilon^2 \rceil$.
%
We initialize 
the variance parameters $\sigma_x^{0}$ and $\sigma_y^{0}$ by estimating the variances at $(\bx^0,y^0)$
for both $\bx$- and $\by$-components of $\widetilde\nabla f$. Indeed, for each of the randomly generated $M=100$ points, we compute the sample variance using $20$  stochastic gradients, and we set $\sigma_x^{0}$ and $\sigma_y^{0}$ based on the largest sample variance observed among $M=100$ random points, and set $(\sigma^0)^2 = (\tilde\sigma_x^{0})^2 + (\tilde\sigma_y^{0})^2$.
{For $\ell\geq 0$, the 
mini-batch sizes $M_x$ and $M_y$ for the $\ell$-the backtracking iteration are set according to Corollary~\ref{cor:sample-complexity-bounded-unkown}, i.e.,
\begin{subequations}\label{eq:batchsize}
    \begin{align}
        M_x&\gets C_x\cdot\frac{C_0}{\epsilon^2}(C_1^x\log(n_x+1){\tilde\sigma_x}^2+C_2 \sigma_0^2/\bar\gamma^{\rv{2\ell}}),\\
        M_y&\gets C_y\cdot\frac{C_0}{\epsilon^2}\Big(C_1^y\log(n_y+1)\tilde\sigma_y^2+ C_2 \sigma_0^2 \Big(1+\frac{6}{{\mu} \eta_y} \frac{2-{\mu} \eta_y}{1-{\mu} \eta_y}\Big)/\bar\gamma^{\rv{2\ell}}\Big),
    \end{align}
\end{subequations}
where $C_0= 48(1+c)$, $C_1^x=(\sigma^0)^2/(\sigma^0_x)^2$, $C_1^y= (\sigma^0)^2/(\sigma^0_y)^2$, $C_2=\log\Big(1+\frac{4}{\bar p}\Big)$ and $C_x,C_y$ are some hyper-parameters {one can tune to the problem at hand --we fixed $C_x\approx 1e-4$ and $C_y\approx 1e-5$} throughout all the experiments conducted.}
%
In our \mgrev{experiments}, we employ a two-phase stopping criterion as described in \sgdab; indeed, for some $\ell\geq 0$, if the backtracking condition
$\tilde S_{(t^*,\ell)}\leq \frac{\epsilon^2}{4}$ for $t^*=\argmin\{\tilde S_{(t,\ell)}: t=1,\ldots,T\}$ is satisfied,
we perform a sample variance test:
after sampling $M_x$ and $M_y$ stochastic gradients at $(\tilde\bx_{(\ell)},\tilde y_{(\ell)})=(\bx^{\tilde k_t}_{(t,\ell)},y^{\tilde k_t}_{(t,\ell)})$ for $k_t\sim U[0,K-1]$ and $t=t^*$, we compute the sample variance
$v_x$ and $v_y$ for the $\bx$- and $\by$-components of $\gt f(\tilde\bx_{(\ell)},\tilde y_{(\ell)})$ separately.
If $(1-1/M_x)\,v_x \leq (1+c)\, (\widetilde\sigma_x)^2$
and $(1-1/M_y)\,v_y \leq (1+c)\, (\widetilde\sigma_y)^2$,
we stop the algorithm and output the current point, where we set $c=0.5$.
If either phase fails, we 
update the parameters as follows:
$\widetilde\mu \gets \mu$, $\widetilde L \gets \widetilde L/\gamma$,
$\widetilde\sigma_x \gets \widetilde\sigma_x/\gamma$,
$\widetilde\sigma_y \gets \widetilde\sigma_y/\gamma$,
and return to \cref{alg:outer_iter} of \sgdab. In the experiments, in case $\sigma_x^2$ and $\sigma_y^2$ are known, then we only check whether $\tilde S_{(t^*,\ell)}\leq \frac{\epsilon^2}{4}$ is satisfied. If it is satisfied then we stop; otherwise, we update the parameters as follows:
$\widetilde\mu \gets \mu$, $\widetilde L \gets \widetilde L/\gamma$,
$\widetilde\sigma_x \gets \sigma_x$,
$\widetilde\sigma_y \gets \sigma_y$.

{Next, we present two different experiment settings: one for a regularized bilinear problem with synthetic data and the other for a distributionally robust optimization problem with real data.}
\paragraph{Regularized Bilinear Problem with Synthetic Data.}
\label{sec:random-data}
\sa{We first test
on the regularized WCSC bilinear SP problem  of the form $\min_{\bx\in\reals^n}\max_{y\in\reals^m}f(\bx,y)$ such that
\begin{equation} \label{eq:bilinear-problem}
{\small
    f(\bx,y)=\bx^\top Q\bx
    +
    \bx^\top A y - \frac{\mu_y}{2}\|y\|^2,}%
    \vspace{-2mm}
\end{equation}
where 
$A\in\reals^{n\times m}$, $Q\in\mathbb{R}^{n\times n}$,  and $\mu_y>0$; hence, $\cL(\bx,y)=f(\bx,y)$ and $\mu=\mu_y$.
This class of problems include \emph{Polyak-Lojasiewicz game} {\citep{chen2022faster,zhang2023jointly}}, \emph{image processing}~{\citep{chambolle2011first}}, and \emph{robust regression}~{\citep{xu2008robust}}.}

\sa{In our experiments, we set $m=n=30$, $\mu_y=1$ and randomly generate $A$ and $Q$ such that $A = V\Lambda_A V^{-1}$ and $Q = V\Lambda_Q V^{-1}$, where $V\in\mathbb{R}^{30\times30}$ is an 
orthogonal matrix\footnote{\mg{We first generate a random matrix $\mathcal{M}$ with entries i.i.d. 
with $U[0,1]$, and then set $V$ as the orthogonal matrix in the QR decomposition of $\mathcal{M}$}.}, 
$\Lambda_A$ and $\Lambda_Q$ are diagonal matrices.
We set $\Lambda_Q = \frac{\Lambda^0_Q}{\|\Lambda^0_Q\|}\cdot L$ for $L\in\{5,10,50\}$, where $\Lambda^0_Q$ is a random diagonal matrix with diagonal elements being 
sampled uniformly at random from the interval $[-1, 1]$ --recall that for matrices $\norm{\cdot}$ denotes the spectral norm.} 
We choose $\Lambda_A$ such that $\Lambda_Q + \frac{1}{\mu_y}\Lambda_A^2 \succeq 0$ and $\norm{\Lambda_Q}\geq\norm{\Lambda_A}$. Those conditions imply that $\cL$ is 
WCSC and $L$-smooth. 
Moreover, given  $\bx$, we can compute $y_*(\bx) = \frac{1}{\mu_y}A^\top \bx$; consequently, the primal function at $\bx$ has the closed form: $F(\bx) = \bx^\top (Q+\frac{1}{\mu_y} A A^\top) \bx$ and $F(\cdot)$ is lower bounded by zero, \saa{i.e., $\bar F=0$}. 

We consider the additive noise setting, the unbiased stochastic oracle $\widetilde\grad \cL$ is given by $\grad \cL + \varsigma$ where $\varsigma\sim N(0,\sigma^2\mathbf{I})$ is Gaussian with mean $\mathbf{0}$ and $\sigma=1$.
We set \xz{$M=1$ for \vrlm{} \saa{with STORM-type variance reduction} and the special batchsize $M_0$ of \vrlm{} at iteration $0$ is set according to \cite[Corollary 2]{mancino2023variance}, and for all the other algorithms tested in this experiment, except for \sgdab{}, we fixed $M = 10$.} 
We initialize all the methods from the same randomly generated initial point $\bx^0$ such that all the entries of $\bx^0$ are i.i.d. with $U[90,100]$. Given $\bx^0$, to initialize $F_0$, $\delta$ and $y^0$, we use \texttt{G-AdaGrad} as described above at the beginning of this section. In Table~\ref{tab:adaptive_solver}, we list both the average and sample variance of some important metrics related to initialization of \sgdab: $\|y_{T+1} - y_{T}\| / \|y_{T}\|$, $\|y_T - y_*(\bx^0)\| / \|y_*(\bx^0)\|$, $\norm{y_T-y_*(\bx^0)}$,  $\cL(\bx^0,y_T)$ and $\cL(\bx^0,y_*(\bx^0))$; hence, for all practical reasons, we set $F_0=\cL(\bx^0,y^0)$ for $y^0=y_T$ and $\delta=0$. The accuracy parameter is set to $\epsilon = \sqrt{c_{\rm tol}~ \Delta_0}$ for $c_{\rm tol}=0.1$, where $\Delta_0 \triangleq F_0 - \bar{F}$. The contraction parameters are set to $\gamma={0.8}$ and $\bar\gamma=0.75$, and we fixed $C_x=2e-3$, and $C_y=7e-5$ for the batchsize update rule in~\eqref{eq:batchsize}.
\begin{table}[htbp]
\centering
\resizebox{\textwidth}{!}{%
\begin{tabular}{c|cc|cc|cc|cc|c}
\toprule
\multirow{2}{*}{$\kappa$} 
& \multicolumn{2}{c|}{$\|y_{T+1}-y_T\| / \|y_T\|$} 
& \multicolumn{2}{c|}{$\|y_T-y_*(\bx^0)\| / \|y_*(\bx^0)\|$} 
& \multicolumn{2}{c|}{$\|y_T-y_*(\bx^0)\|$} 
& \multicolumn{2}{c|}{$\cL(\bx^0,y_T)$} 
& $\cL(\bx^0,y_*(\bx^0))$ \\
& Mean & Var & Mean & Var & Mean & Var & Mean & Var &  \\
\midrule
5  
& $9.7 \times 10^{-5}$ & $5.3 \times 10^{-11}$ 
& $2.2 \times 10^{-4}$ & $9.3 \times 10^{-10}$ 
& $2.1 \times 10^{-1}$ & $8.7 \times 10^{-4}$ 
& $4.8 \times 10^{5}$ & $4.3 \times 10^{-5}$ 
& $4.8 \times 10^{5}$ \\
10 
& $9.4 \times 10^{-5}$ & $5.1 \times 10^{-11}$ 
& $1.9 \times 10^{-4}$ & $7.3 \times 10^{-10}$ 
& $2.4 \times 10^{-1}$ & $1.2 \times 10^{-3}$ 
& $6.7 \times 10^{5}$ & $7.6 \times 10^{-5}$ 
& $6.7 \times 10^{5}$ \\
50 
& $8.6 \times 10^{-5}$ & $3.8 \times 10^{-11}$ 
& $1.3 \times 10^{-4}$ & $3.3 \times 10^{-10}$ 
& $3.4 \times 10^{-1}$ & $2.1 \times 10^{-3}$ 
& $3.7 \times 10^{6}$ & $2.4 \times 10^{-4}$ 
& $3.7 \times 10^{6}$ \\
\bottomrule
\end{tabular}%
}
\caption{Performance metrics of the \texttt{G-AdaGrad}~\citep{junchinest} on 10 random simulations for different $\kappa$ values.}
\label{tab:adaptive_solver}
\end{table}

In both Figures~\ref{fig:Q} and ~\ref{fig:tildeQ}, we plot $\|\grad \cL\|^2$ on the $y$-axis against the number of 
\saa{stochastic oracle calls in the}
$x$-axis. \rv{For the results plotted, the solid lines show the average over 10 randomly generated instances, all starting from the same initial point, with shaded region indicating the range statistic.} We observe that \sgdab~
\saa{performs well on this test as it can effectively take larger steps with convergence guarantees.}

\begin{figure}[htb!]
     \centering
     \begin{subfigure}[b]{0.32\textwidth}
         \centering
        \includegraphics[width=\textwidth]{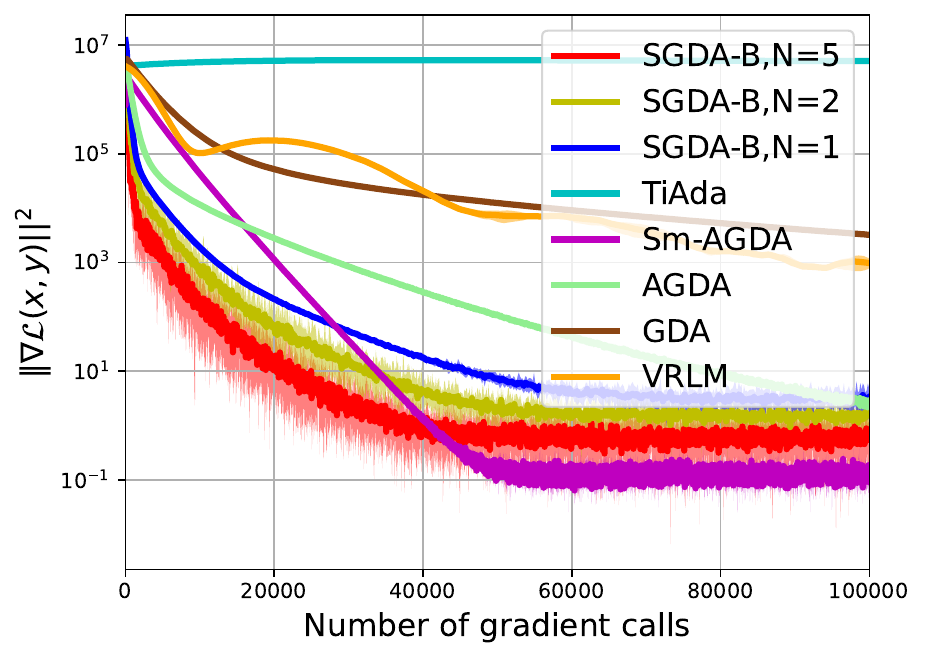}
         \caption{$\kappa=5$}
         \label{fig:Q-mu1}
     \end{subfigure}
     \begin{subfigure}[b]{0.32\textwidth}
         \centering
\includegraphics[width = \textwidth]{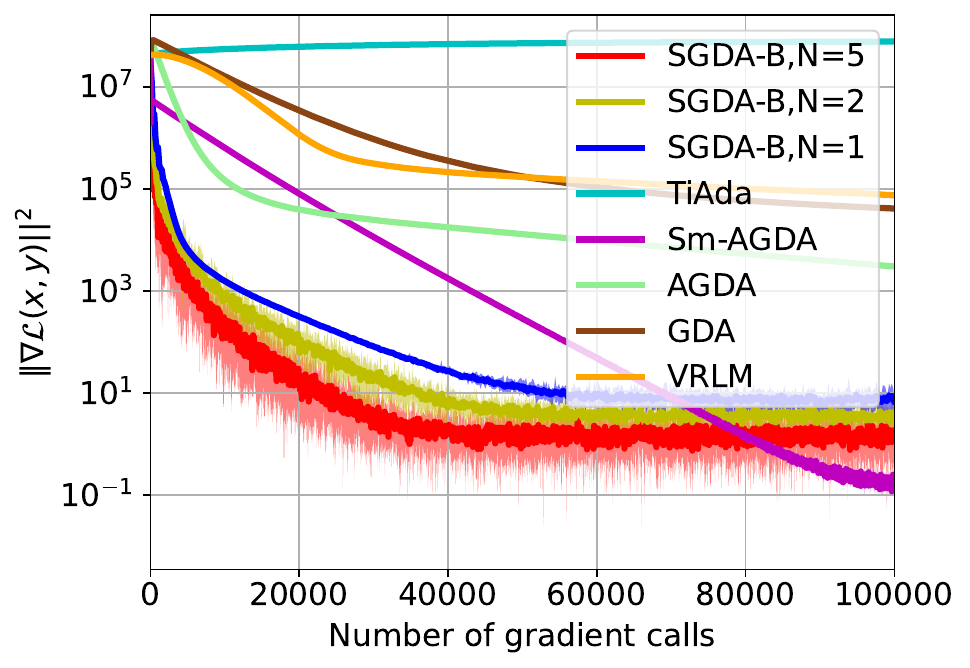}
         \caption{$\kappa=10$}
         \label{fig:Q-mu1e-1}
     \end{subfigure}
     \begin{subfigure}[b]{0.32\textwidth}
         \centering
\includegraphics[width = \textwidth]{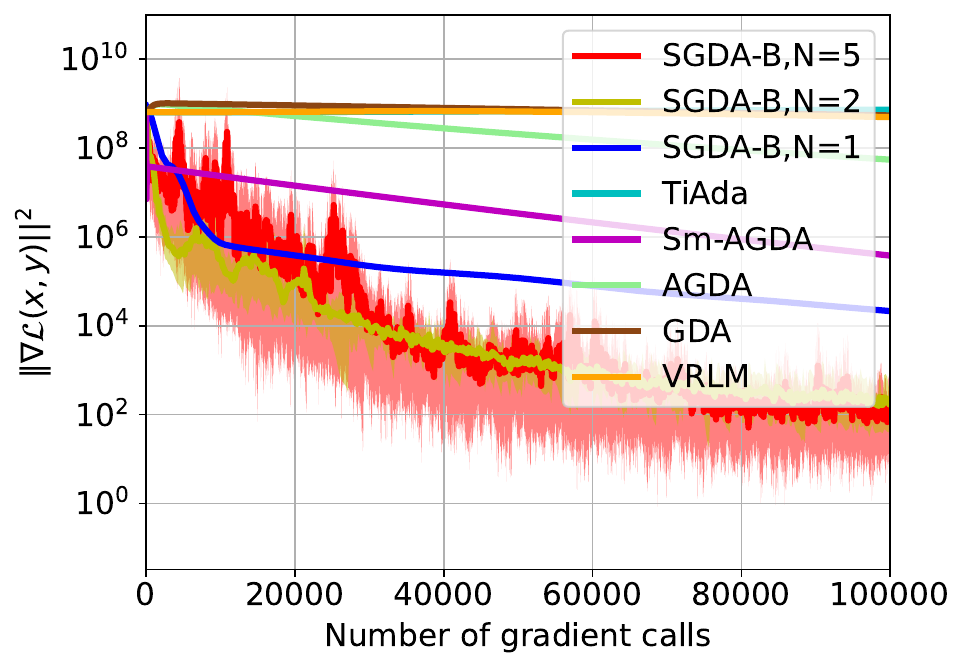}
         \caption{$\kappa=50$}
         \label{fig:Q-mu1e-2}
     \end{subfigure}
\caption{
Comparison of \sgdab~against 
\tiada{} {\citep{li2022tiada}}, \smagda{} {\citep{yang2022faster}}, \agda{} {\citep{boct2020alternating}}, \gda{} {\citep{lin2020gradient}}, and \vrlm{} {\citep{mancino2023variance}} for solving regularized bilinear minimax problem in \cref{eq:bilinear-problem} with synthetic data over $10$ randomly generated instances when $\sigma_x^2$ and $\sigma_y^2$ are \textit{known}.
}
\label{fig:Q}
\centering
\end{figure}

\begin{figure}[htb!]
     \centering
     \begin{subfigure}[b]{0.32\textwidth}
         \centering
        \includegraphics[width=\textwidth]{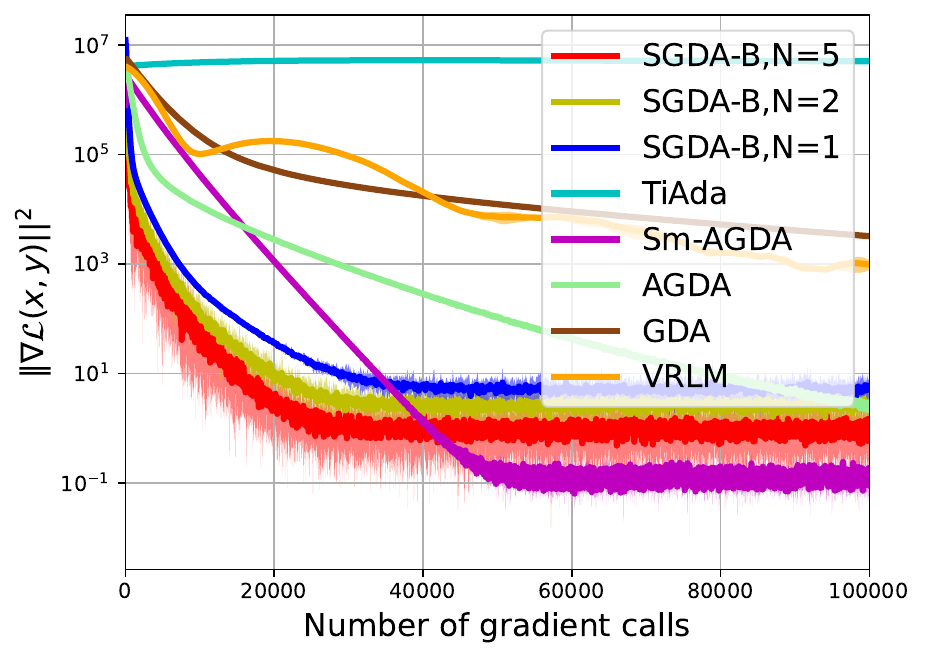}
         \caption{$\kappa=5$}
         \label{fig:tildeQ-mu1}
     \end{subfigure}
     \begin{subfigure}[b]{0.32\textwidth}
         \centering
\includegraphics[width = \textwidth]{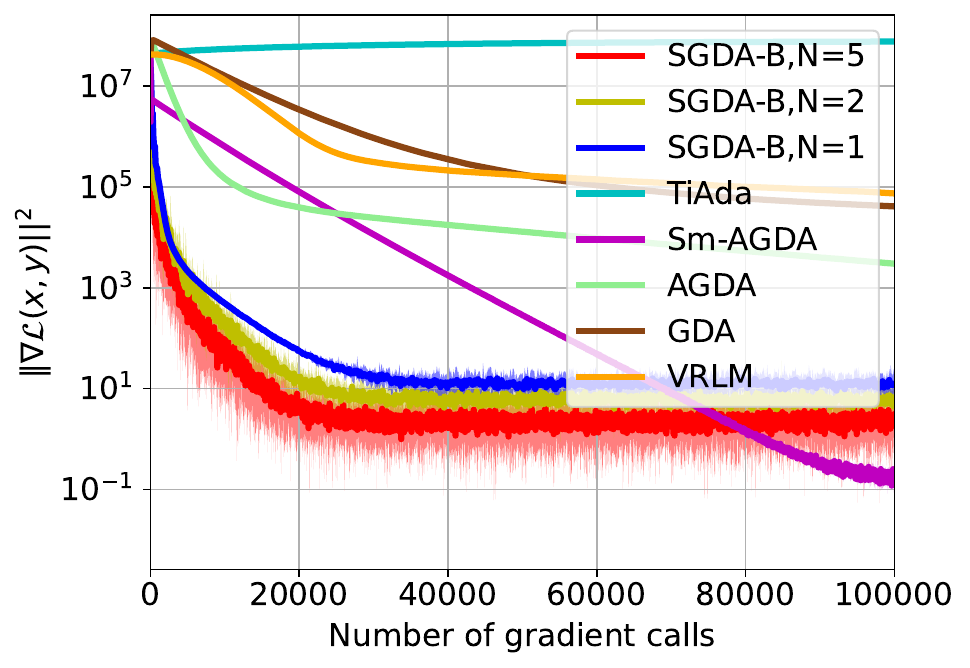}
         \caption{$\kappa=10$}
         \label{fig:tildeQ-mu1e-1}
     \end{subfigure}
     \begin{subfigure}[b]{0.32\textwidth}
         \centering
\includegraphics[width = \textwidth]{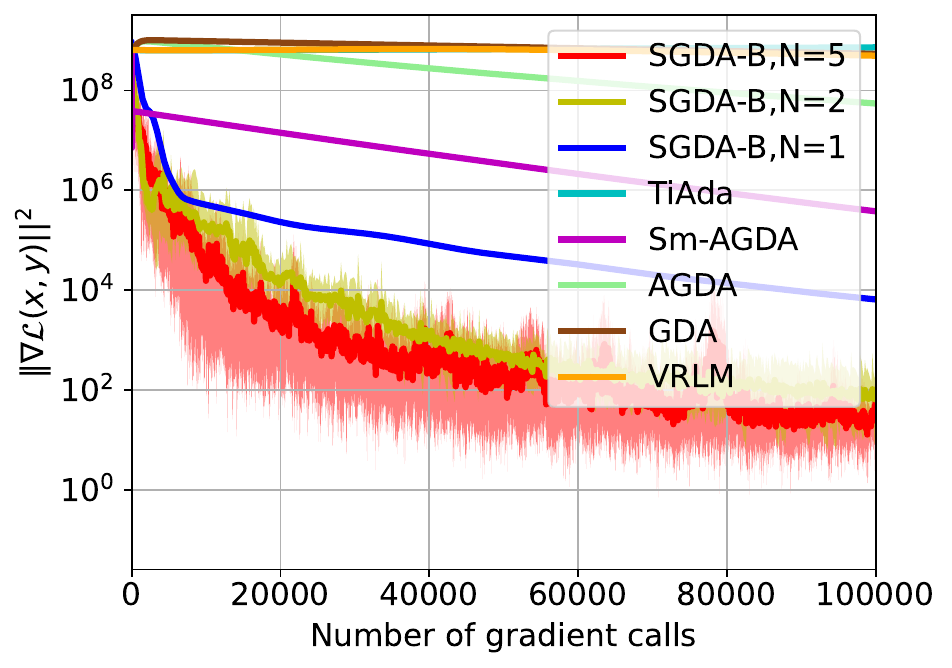}
         \caption{$\kappa=50$}
         \label{fig:tildeQ-mu1e-2}
     \end{subfigure}
\caption{
Comparison of \sgdab~against 
\tiada{} {\citep{li2022tiada}}, \smagda{} {\citep{yang2022faster}}, \agda{} {\citep{boct2020alternating}}, \gda{} {\citep{lin2020gradient}}, and \vrlm{} {\citep{mancino2023variance}} for solving regularized bilinear minimax problem in \cref{eq:bilinear-problem} with synthetic data over $10$ randomly generated instances when $\sigma_x^2$ and $\sigma_y^2$ are \textit{unknown}.
}
\label{fig:tildeQ}
\centering
\end{figure}

\paragraph{Distributed Robust Optimization with Neural Network.}
Next, we test \sgdab{} on the distributionally robust optimization problem from~\citep{namkoong2016stochastic}, i.e.,
\begin{equation}
    \label{eq:dro-problmm}
        \min_{\bx\in\mathbb{R}^n} \max_{y\in\mathbb{R}^m}\Big\{\sum_{i=1}^{m}y_i \ell_i(\saa{q}(\ba_i;\bx))- 
        \rv{r(y)}:\quad 
        \sum_{i=1}^{m}y_i=1,\;y_i\geq 0,\; \forall\;i=1,\ldots,m\Big\},
\end{equation}
where $\ba_i\in\reals^d$ and $b_i\in\{1,-1\}$ \saa{denote the feature vector and \mg{the} label corresponding to data point $i\in\{1,\ldots,m\}$ belonging to training data for binary classification;} the function \saa{$q:\mathbb{R}^d\rightarrow \reals$ represents three-layer perceptron} neural network \saa{with $\bx\in\reals^n$ denoting the parameters of the network with $n=610,369$}; and in this experiment we use the binary logistic loss function, i.e.,
$\saa{{\ell_i}(z)} 
=\ln(1+\exp(-b_i z))$ for all $i\in [m]$. The regularizer $\rv{r(y)} = \frac{\mu_y}{2}\|y - \mathbf{1}/m\|^2$ makes sure that the worst\mg{-}case distribution will not be too far from the uniform distribution –here, $\mathbf{1}$ denotes the vector with all entries equal to one, and we set \rv{$\mu_y = 0.01$}.
\saa{Since $\ell_i(\cdot)\geq 0$ for all $i\in[m]$, we have $\bar F=0$ for this experiment as well.} 
We conducted the experiment on the data set
\verb+gisette+ with $m = 6000$ and $d = 5000$, 
which can be downloaded from LIBSVM repository \footnote{https://www.csie.ntu.edu.tw/~cjlin/libsvmtools/datasets/binary.html}. 
We 
rescaled the dataset by \rv{$\ba_i \gets (\ba_i - \underline{a}_i\cdot\mathbf{1})/(\bar{a}_i-\underline{a}_i)$, where $\ba_i = [a_{ij}]_{j=1}^d$ is the feature vector of the $i$-th data point, $\underline{a}_i=\min\{a_{ij}:~j\in[d]\}$ and $\bar{a}_i=\max\{a_{ij}:~j\in[d]\}$.}

\begin{figure}[h!]
\centering
\includegraphics[width = 0.32\textwidth]{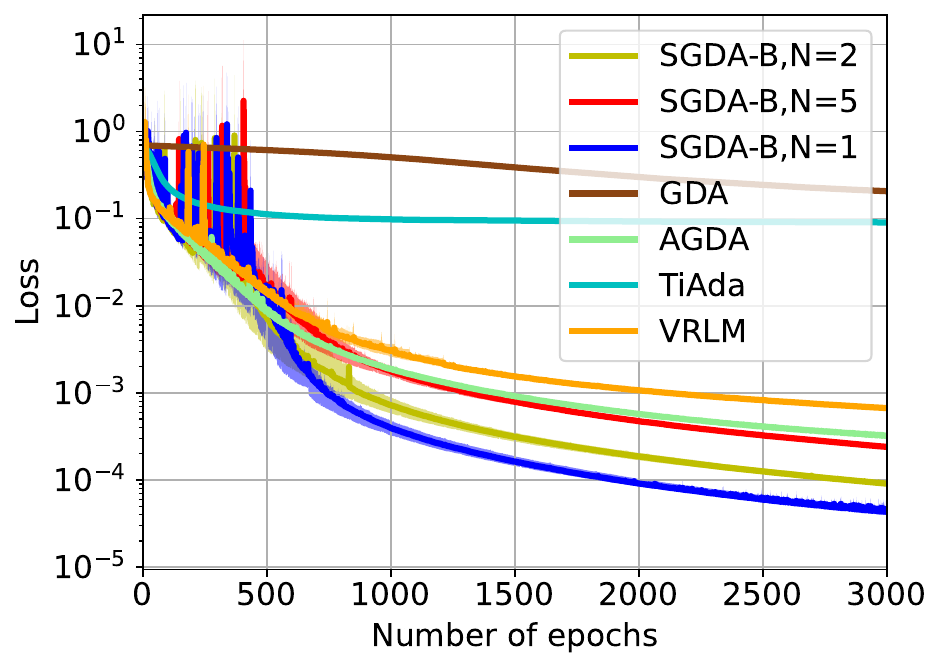}
\includegraphics[width = 0.32\textwidth]{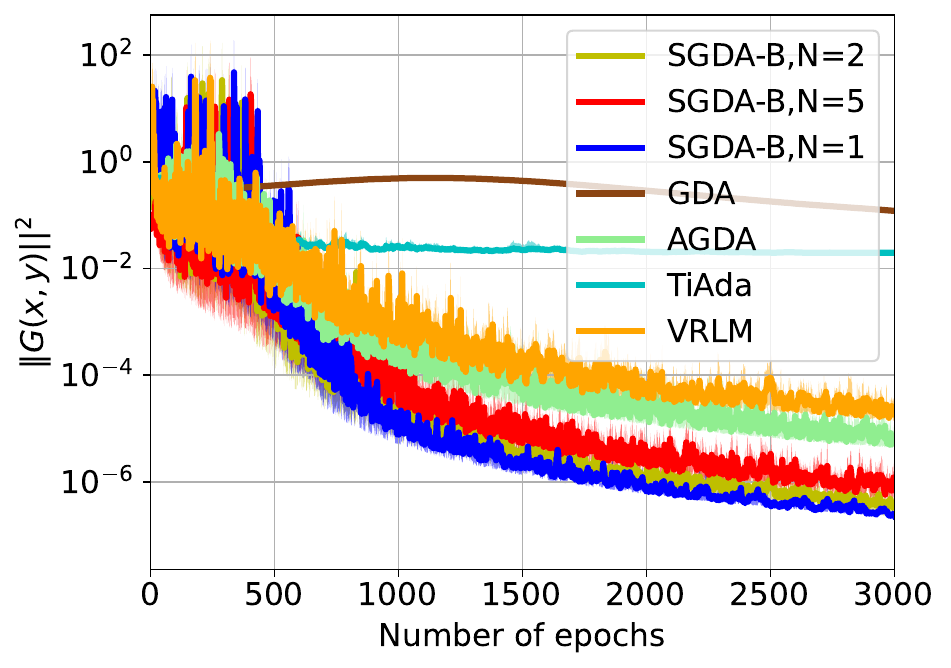}
\includegraphics[width = 0.32\textwidth]{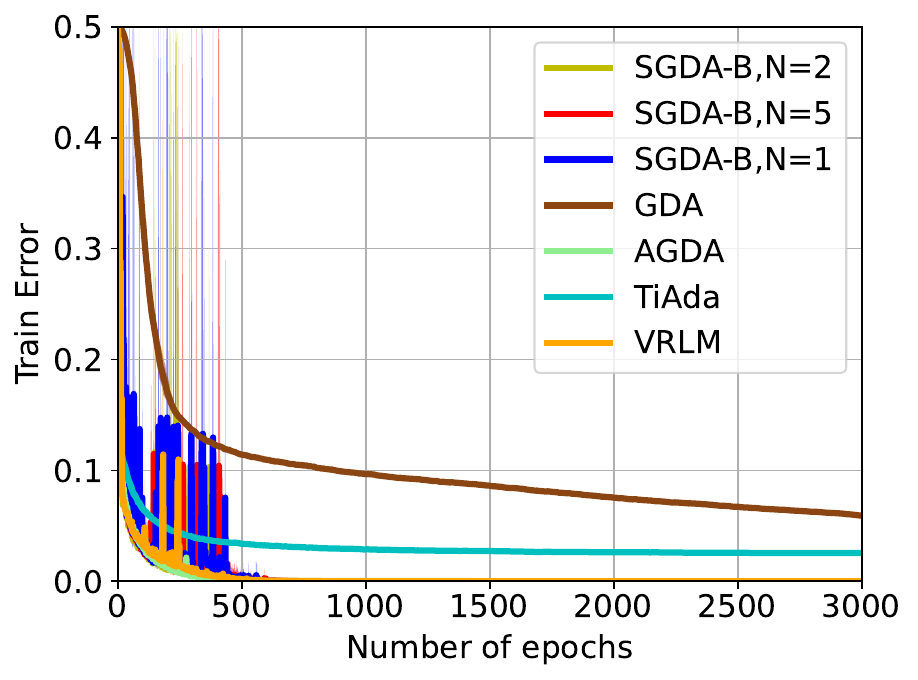} 
\caption{
Comparison of \sgdab{} against \tiada{} {\citep{li2022tiada}}, \agda{} {\citep{boct2020alternating}},
\gda{} {\citep{lin2020gradient}}, and \vrlm{} {\citep{mancino2023variance}} for solving the DRO problem in~\eqref{eq:dro-problmm} with real data {when $\sigma_x^2$ and $\sigma_y^2$ are unknown}. For each method, the solid line indicates the average behavior over 10 sample path with shaded region indicating the range statistic. ``Train error" denotes the fraction of wrong prediction, and ``loss" denotes $F(\bx)=\max_{y\in\cY}\cL(\bx,y)$.  One epoch means one complete pass over the whole data set.}
\label{figs:tilde LG Regression}
\end{figure}

We initialize all the methods from the same randomly generated initial point $\bx^0$  generated by the \mg{Xavier} method for initializing deep neural networks \citep{glorot2010understanding}.  Given $\bx^0$, to initialize $F_0$, $\delta$ and $y^0$, as described above at the beginning of this section, we ran \texttt{G-AdaGrad} for $T=10^4$ iterations starting from $y_0=\bold{1}/m$, and we set $F_0=\cL(\bx^0,y^0)$ for $y^0=y_T$ and $\delta=0$. The accuracy parameter is set to $\epsilon = \sqrt{c_{\rm tol}~ \Delta_0}$ for $c_{\rm tol}=0.005$, where $\Delta_0 \triangleq F_0 - \bar{F}$. The contraction parameters are set to $\gamma=0.9$ and $\bar\gamma = 0.85$, and we fixed $C_x=4e-4$, and $C_y=2e-5$ for the batchsize update rule in~\eqref{eq:batchsize}. 
Since $L$ is unknown, 
we estimated $L$ 
deriving a lower bound for it; indeed, since $\frac{\partial \cL}{\partial \by\partial \bx} =
 [-\frac{b_i\exp(-b_i q(\ba_i;\bx))}{1+\exp(-b_i q(\ba_i;\bx))}\frac{\partial q(\ba_i;\bx)}{\partial \bx}]_{i=1}^{m}$, 
 we have $L\geq \tilde{L}=\|\frac{\partial \cL(\bx,y)}{\partial \bx\partial y}\mid_{(\bx,y)=(\bx^0,y^0)}\|$. We used $\widetilde L$ to set the step sizes for \gda, \agda, \smagda, and \vrlm{}.
 In our tests, $\widetilde L\approx 0.0532$; thus, $\kappa\approx 5.32$. The competing methods \gda, \agda, and \vrlm{} worked well with $\widetilde L$; however, \smagda~sequence diverged so we removed it from the comparison. Moreover, using batch size of $100$, we estimated $\sigma_x,\sigma_y$ by random sampling $\gt f(\bx,y)$ at $(\bx,y)$ such that we fixed $y=\bold{1}/m$ and for $\bx$ we used $100$ random points generated by 
 \mg{Xavier}
 method, which resulted in the estimate values $\sigma_x\approx 0.44$, $\sigma_y\approx 53$. Except for \sgdab{}, for all the algorithms tested we fixed the batch size at $M = 200$. We observe that \sgdab{} is competitive against \tiada{} and others benchmark methods
on this test with real data as well.

\section{Conclusion}
In this paper, we considered 
nonconvex-(strongly) concave minimax problems in the form of \mg{$\min_{\bx} \max_y \sum_{i=1}^N \sa{g_i(x_i)}+f(\bx,y)-h(y)$, where $h$ and $g_i$ for $i=1,\cdots,N$ are closed convex, $f$ is smooth with $L$-Lipschitz gradients and $\mu$-strongly concave in $y$} \sa{for some $\mu\geq 0$}. 
We proposed a new method \texttt{SGDA-B} with backtracking that supports random block-coordinate updates in the primal variable. {To our knowledge, \texttt{SGDA-B} achieves the best computational complexity for WCSC problems} among the existing methods \mg{that are} agnostic to Lipschitz constant $L$, concavity modulus $\mu$ and variance bound $\sigma^2$ for the stochastic oracle. {More specifically,} we show that \texttt{SGDA-B} {attains computational complexities of} $\cO(\kappa^2 L \epsilon^{-2})$ and $\cO(\kappa^3 L \log^2(p^{-1}) \epsilon^{-4})$ in the deterministic and stochastic settings, respectively. 
We also showed that our method \mg{can support Gauss-Seidel-type updates and can handle WCMC problems.}


{\small
\bibliography{reference}}%

@article{mancino2023variance,
  title={Variance-reduced accelerated methods for decentralized stochastic double-regularized nonconvex strongly-concave minimax problems},
  author={Mancino-Ball, Gabriel and Xu, Yangyang},
  journal={arXiv preprint arXiv:2307.07113},
  year={2023}
}

@inproceedings{chen2022accelerated,
  title={Accelerated proximal alternating gradient-descent-ascent for nonconvex minimax machine learning},
  author={Chen, Ziyi and Ma, Shaocong and Zhou, Yi},
  booktitle={2022 IEEE International Symposium on Information Theory (ISIT)},
  pages={672--677},
  year={2022},
  organization={IEEE}
}

@inproceedings{yoon2021accelerated,
  title={Accelerated Algorithms for Smooth Convex-Concave Minimax Problems with $\mathcal{O}(1/k^2)$ Rate on Squared Gradient Norm},
  author={Yoon, TaeHo and Ryu, Ernest K},
  booktitle={International Conference on Machine Learning},
  pages={12098--12109},
  year={2021},
  organization={PMLR}
}

@article{hamedani2018_RBC,
  title={Randomized Primal-Dual Methods with Adaptive Step Sizes},
  author={Hamedani, E Yazdandoost and Jalilzadeh, A and Aybat, NS},
  journal={arXiv preprint arXiv:1806.04118, accepted to AISTATS 2023},
  year={2022}
}

@article{jiang2022generalized,
  title={Generalized optimistic methods for convex-concave saddle point problems},
  author={Jiang, Ruichen and Mokhtari, Aryan},
  journal={arXiv preprint arXiv:2202.09674},
  year={2022}
}

@article{carmon2020lower,
  title={Lower bounds for finding stationary points I},
  author={Carmon, Yair and Duchi, John C and Hinder, Oliver and Sidford, Aaron},
  journal={Mathematical Programming},
  volume={184},
  number={1-2},
  pages={71--120},
  year={2020},
  publisher={Springer}
}

@article{pethick2023escaping,
  title={Escaping limit cycles: Global convergence for constrained nonconvex-nonconcave minimax problems},
  author={Pethick, Thomas and Latafat, Puya and Patrinos, Panagiotis and Fercoq, Olivier and Cevher, Volkan},
  journal={arXiv preprint arXiv:2302.09831},
  year={2023}
}

@article{grimmer2022landscape,
  title={The landscape of the proximal point method for nonconvex--nonconcave minimax optimization},
  author={Grimmer, Benjamin and Lu, Haihao and Worah, Pratik and Mirrokni, Vahab},
  journal={Mathematical Programming},
  pages={1--35},
  year={2022},
  publisher={Springer}
}

@article{duchi2011adaptive,
  title={Adaptive subgradient methods for online learning and stochastic optimization.},
  author={Duchi, John and Hazan, Elad and Singer, Yoram},
  journal={Journal of machine learning research},
  volume={12},
  number={7},
  year={2011}
}

@inproceedings{junchinest,
  title={Nest Your Adaptive Algorithm for Parameter-Agnostic Nonconvex Minimax Optimization},
  author={Junchi, YANG and Li, Xiang and He, Niao},
  booktitle={Advances in Neural Information Processing Systems},
  year={2022}
}

@article{peng2016coordinate,
  title={Coordinate friendly structures, algorithms and applications},
  author={Peng, Zhimin and Wu, Tianyu and Xu, Yangyang and Yan, Ming and Yin, Wotao},
  journal={arXiv preprint arXiv:1601.00863},
  year={2016}
}

@article{jalilzadeh2019doubly,
  title={A Doubly-Randomized Block-Coordinate Primal-Dual Method for Large-scale Saddle Point Problems},
  author={Jalilzadeh, Afrooz and Hamedani, Erfan Yazdandoost and Aybat, Necdet S},
  journal={arXiv preprint arXiv:1907.03886},
  year={2019}
}

@article{chen2022faster,
  title={Faster stochastic algorithms for minimax optimization under polyak-$\{$$\backslash$L$\}$ ojasiewicz condition},
  author={Chen, Lesi and Yao, Boyuan and Luo, Luo},
  journal={Advances in Neural Information Processing Systems},
  volume={35},
  pages={13921--13932},
  year={2022}
}

@article{xu2008robust,
  title={Robust regression and lasso},
  author={Xu, Huan and Caramanis, Constantine and Mannor, Shie},
  journal={Advances in neural information processing systems},
  volume={21},
  year={2008}
}

@article{nakkiran2021deep,
  title={Deep double descent: Where bigger models and more data hurt},
  author={Nakkiran, Preetum and Kaplun, Gal and Bansal, Yamini and Yang, Tristan and Barak, Boaz and Sutskever, Ilya},
  journal={Journal of Statistical Mechanics: Theory and Experiment},
  volume={2021},
  number={12},
  pages={124003},
  year={2021},
  publisher={IOP Publishing}
}

@article{chang2008coordinate,
  title={Coordinate descent method for large-scale l2-loss linear support vector machines.},
  author={Chang, Kai-Wei and Hsieh, Cho-Jui and Lin, Chih-Jen},
  journal={Journal of Machine Learning Research},
  volume={9},
  number={7},
  year={2008}
}

@article{li2009coordinate,
  title={Coordinate descent optimization for l1 minimization with application to compressed sensing; a greedy algorithm},
  author={Li, Yingying and Osher, Stanley},
  journal={Inverse Problems and Imaging},
  volume={3},
  number={3},
  pages={487--503},
  year={2009},
  publisher={Inverse Problems and Imaging}
}

@article{wu2008coordinate,
  title={Coordinate descent algorithms for lasso penalized regression},
  author={Wu, Tong Tong and Lange, Kenneth},
  year={2008}
}

@inproceedings{richtarik2012efficient,
  title={Efficient serial and parallel coordinate descent methods for huge-scale truss topology design},
  author={Richt{\'a}rik, Peter and Tak{\'a}{\v{c}}, Martin},
  booktitle={Operations Research Proceedings 2011: Selected Papers of the International Conference on Operations Research (OR 2011), August 30-September 2, 2011, Zurich, Switzerland},
  pages={27--32},
  year={2012},
  organization={Springer}
}

@article{nakamura2021block,
  title={Block-cyclic stochastic coordinate descent for deep neural networks},
  author={Nakamura, Kensuke and Soatto, Stefano and Hong, Byung-Woo},
  journal={Neural Networks},
  volume={139},
  pages={348--357},
  year={2021},
  publisher={Elsevier}
}

@article{iusem2019variance,
  title={Variance-based extragradient methods with line search for stochastic variational inequalities},
  author={Iusem, Alfredo N and Jofr{\'e}, Alejandro and Oliveira, Roberto I and Thompson, Philip},
  journal={SIAM Journal on Optimization},
  volume={29},
  number={1},
  pages={175--206},
  year={2019},
  publisher={SIAM}
}

@article{lin2019gradient,
  title={On gradient descent ascent for nonconvex-concave minimax problems},
  author={Lin, Tianyi and Jin, Chi and Jordan, Michael I},
  journal={arXiv preprint arXiv:1906.00331},
  year={2019}
}

@article{malitsky2018proximal,
  title={Proximal extrapolated gradient methods for variational inequalities},
  author={Malitsky, Yu},
  journal={Optimization Methods and Software},
  volume={33},
  number={1},
  pages={140--164},
  year={2018},
  publisher={Taylor \& Francis}
}

@article{malitsky2018first,
  title={A first-order primal-dual algorithm with linesearch},
  author={Malitsky, Yura and Pock, Thomas},
  journal={SIAM Journal on Optimization},
  volume={28},
  number={1},
  pages={411--432},
  year={2018},
  publisher={SIAM}
}

@article{xiao2019dscovr,
  title={DSCOVR: Randomized Primal-Dual Block Coordinate Algorithms for Asynchronous Distributed Optimization.},
  author={Xiao, Lin and Yu, Adams Wei and Lin, Qihang and Chen, Weizhu},
  journal={J. Mach. Learn. Res.},
  volume={20},
  pages={43--1},
  year={2019}
}

@inproceedings{goodfellow2014generative,
  title={Generative adversarial nets},
  author={Goodfellow, Ian and Pouget-Abadie, Jean and Mirza, Mehdi and Xu, Bing and Warde-Farley, David and Ozair, Sherjil and Courville, Aaron and Bengio, Yoshua},
  booktitle={Advances in neural information processing systems},
  pages={2672--2680},
  year={2014}
}

@article{hamedani2021primal,
  title={A primal-dual algorithm with line search for general convex-concave saddle point problems},
  author={Hamedani, Erfan Yazdandoost and Aybat, Necdet Serhat},
  journal={SIAM Journal on Optimization},
  volume={31},
  number={2},
  pages={1299--1329},
  year={2021},
  publisher={SIAM}
}

@article{li2022tiada,
  title={TiAda: A Time-scale Adaptive Algorithm for Nonconvex Minimax Optimization},
  author={Li, Xiang and Yang, Junchi and He, Niao},
  journal={arXiv preprint arXiv:2210.17478},
  year={2022}
}

@article{dang2015convergence,
  title={On the convergence properties of non-euclidean extragradient methods for variational inequalities with generalized monotone operators},
  author={Dang, Cong D and Lan, Guanghui},
  journal={Computational Optimization and applications},
  volume={60},
  number={2},
  pages={277--310},
  year={2015},
  publisher={Springer}
}

@article{zhang2022sapd+,
  title={Sapd+: An accelerated stochastic method for nonconvex-concave minimax problems},
  author={Zhang, Xuan and Aybat, Necdet Serhat and Gurbuzbalaban, Mert},
  journal={arXiv preprint arXiv:2205.15084},
  year={2022}
}

@inproceedings{glorot2010understanding,
  title={Understanding the difficulty of training deep feedforward neural networks},
  author={Glorot, Xavier and Bengio, Yoshua},
  booktitle={Proceedings of the thirteenth international conference on artificial intelligence and statistics},
  pages={249--256},
  year={2010},
  organization={JMLR Workshop and Conference Proceedings}
}

@article{lee2021fast,
  title={Fast extra gradient methods for smooth structured nonconvex-nonconcave minimax problems},
  author={Lee, Sucheol and Kim, Donghwan},
  journal={Advances in Neural Information Processing Systems},
  volume={34},
  year={2021}
}

@article{song2020optimistic,
  title={Optimistic dual extrapolation for coherent non-monotone variational inequalities},
  author={Song, Chaobing and Zhou, Zhengyuan and Zhou, Yichao and Jiang, Yong and Ma, Yi},
  journal={Advances in Neural Information Processing Systems},
  volume={33},
  pages={14303--14314},
  year={2020}
}

@inproceedings{diakonikolas2021efficient,
  title={Efficient methods for structured nonconvex-nonconcave min-max optimization},
  author={Diakonikolas, Jelena and Daskalakis, Constantinos and Jordan, Michael I},
  booktitle={International Conference on Artificial Intelligence and Statistics},
  pages={2746--2754},
  year={2021},
  organization={PMLR}
}

@article{lu2020hybrid,
  title={Hybrid block successive approximation for one-sided non-convex min-max problems: algorithms and applications},
  author={Lu, Songtao and Tsaknakis, Ioannis and Hong, Mingyi and Chen, Yongxin},
  journal={IEEE Transactions on Signal Processing},
  volume={68},
  pages={3676--3691},
  year={2020},
  publisher={IEEE}
}

@inproceedings{zhang2021complexity,
  title={The complexity of nonconvex-strongly-concave minimax optimization},
  author={Zhang, Siqi and Yang, Junchi and Guzm{\'a}n, Crist{\'o}bal and Kiyavash, Negar and He, Niao},
  booktitle={Uncertainty in Artificial Intelligence},
  pages={482--492},
  year={2021},
  organization={PMLR}
}

@article{boct2020alternating,
  title={Alternating proximal-gradient steps for (stochastic) nonconvex-concave minimax problems},
  author={Bo{\c{t}}, Radu Ioan and B{\"o}hm, Axel},
  journal={arXiv preprint arXiv:2007.13605},
  year={2020}
}

@article{huang2022accelerated,
  title={Accelerated zeroth-order and first-order momentum methods from mini to minimax optimization},
  author={Huang, Feihu and Gao, Shangqian and Pei, Jian and Huang, Heng},
  journal={Journal of Machine Learning Research},
  volume={23},
  number={36},
  pages={1--70},
  year={2022}
}

@article{huang2021efficient,
  title={Efficient mirror descent ascent methods for nonsmooth minimax problems},
  author={Huang, Feihu and Wu, Xidong and Huang, Heng},
  journal={Advances in Neural Information Processing Systems},
  volume={34},
  year={2021}
}

@article{zhang2021robust,
  title={Robust accelerated primal-dual methods for computing saddle points},
  author={Zhang, Xuan and Aybat, Necdet Serhat and G{\"u}rb{\"u}zbalaban, Mert},
  journal={SIAM Journal on Optimization},
  volume={34},
  number={1},
  pages={1097--1130},
  year={2024},
  publisher={SIAM}
}

@article{li2021complexity,
  title={Complexity Lower Bounds for Nonconvex-Strongly-Concave Min-Max Optimization},
  author={Li, Haochuan and Tian, Yi and Zhang, Jingzhao and Jadbabaie, Ali},
  journal={arXiv preprint arXiv:2104.08708},
  year={2021}
}

@article{chambolle2011first,
  title={A first-order primal-dual algorithm for convex problems with applications to imaging},
  author={Chambolle, Antonin and Pock, Thomas},
  journal={Journal of Math. Imaging and Vision},
  volume={40},
  number={1},
  pages={120--145},
  year={2011},
  publisher={Springer}
}

@article{bottou2018optimization,
  title={Optimization methods for large-scale machine learning},
  author={Bottou, L{\'e}on and Curtis, Frank E and Nocedal, Jorge},
  journal={Siam Review},
  volume={60},
  number={2},
  pages={223--311},
  year={2018},
  publisher={SIAM}
}

@inproceedings{namkoong2016stochastic,
  title={Stochastic gradient methods for distributionally robust optimization with f-divergences},
  author={Namkoong, Hongseok and Duchi, John C},
  booktitle={Advances in {N}eural {I}nformation {P}rocessing {S}ystems},
  pages={2208--2216},
  year={2016}
}

@article{paquette2018stochastic,
  title={A stochastic line search method with convergence rate analysis},
  author={Paquette, Courtney and Scheinberg, Katya},
  journal={arXiv preprint arXiv:1807.07994},
  year={2018}
}

@article{rebegoldi2022scaled,
  title={Scaled, inexact, and adaptive generalized fista for strongly convex optimization},
  author={Rebegoldi, Simone and Calatroni, Luca},
  journal={SIAM Journal on Optimization},
  volume={32},
  number={3},
  pages={2428--2459},
  year={2022},
  publisher={SIAM}
}

@article{calatroni2019backtracking,
  title={Backtracking strategies for accelerated descent methods with smooth composite objectives},
  author={Calatroni, Luca and Chambolle, Antonin},
  journal={SIAM Journal on Optimization},
  volume={29},
  number={3},
  pages={1772--1798},
  year={2019},
  publisher={SIAM}
}

@article{ostrovskii2021efficient,
  title={Efficient search of first-order nash equilibria in nonconvex-concave smooth min-max problems},
  author={Ostrovskii, Dmitrii M and Lowy, Andrew and Razaviyayn, Meisam},
  journal={SIAM Journal on Optimization},
  volume={31},
  number={4},
  pages={2508--2538},
  year={2021},
  publisher={SIAM}
}

@inproceedings{lin2020gradient,
  title={On gradient descent ascent for nonconvex-concave minimax problems},
  author={Lin, Tianyi and Jin, Chi and Jordan, Michael},
  booktitle={International Conference on Machine Learning},
  pages={6083--6093},
  year={2020},
  organization={PMLR}
}

@article{rafique1810non,
  title={Non-convex min--max optimization: provable algorithms and applications in machine learning (2018)},
  author={Rafique, H and Liu, M and Lin, Q and Yang, T},
  journal={arXiv preprint arXiv:1810.02060},
  year={2018}
}

@article{gurbuzbalaban2020stochastic,
  title={A stochastic subgradient method for distributionally robust non-convex learning},
  author={G{\"u}rb{\"u}zbalaban, Mert and Ruszczy{\'n}ski, Andrzej and Zhu, Landi},
  journal={arXiv preprint arXiv:2006.04873},
  year={2020}
}

@article{li2021augmented,
  title={Augmented Lagrangian--Based First-Order Methods for Convex-Constrained Programs with Weakly Convex Objective},
  author={Li, Zichong and Xu, Yangyang},
  journal={INFORMS Journal on Optimization},
  volume={3},
  number={4},
  pages={373--397},
  year={2021},
  publisher={INFORMS}
}

@article{xu2023unified,
  title={A unified single-loop alternating gradient projection algorithm for nonconvex--concave and convex--nonconcave minimax problems},
  author={Xu, Zi and Zhang, Huiling and Xu, Yang and Lan, Guanghui},
  journal={Mathematical Programming},
  pages={1--72},
  year={2023},
  publisher={Springer}
}

@misc{dvurechensky2017gradient,
      title={Gradient Method With Inexact Oracle for Composite Non-Convex Optimization}, 
      author={Pavel Dvurechensky},
      year={2017},
      eprint={1703.09180},
      archivePrefix={arXiv},
      primaryClass={math.OC}
}

@inproceedings{li2022convergence,
  title={On convergence of gradient descent ascent: A tight local analysis},
  author={Li, Haochuan and Farnia, Farzan and Das, Subhro and Jadbabaie, Ali},
  booktitle={International Conference on Machine Learning},
  pages={12717--12740},
  year={2022},
  organization={PMLR}
}

@inproceedings{sebbouh2022randomized,
  title={Randomized stochastic gradient descent ascent},
  author={Sebbouh, Othmane and Cuturi, Marco and Peyr{\'e}, Gabriel},
  booktitle={International Conference on Artificial Intelligence and Statistics},
  pages={2941--2969},
  year={2022},
  organization={PMLR}
}

@inproceedings{yang2022faster,
  title={Faster single-loop algorithms for minimax optimization without strong concavity},
  author={Yang, Junchi and Orvieto, Antonio and Lucchi, Aurelien and He, Niao},
  booktitle={International Conference on Artificial Intelligence and Statistics},
  pages={5485--5517},
  year={2022},
  organization={PMLR}
}

@article{nouiehed2019solving,
  title={Solving a class of non-convex min-max games using iterative first order methods},
  author={Nouiehed, Maher and Sanjabi, Maziar and Huang, Tianjian and Lee, Jason D and Razaviyayn, Meisam},
  journal={Advances in Neural Information Processing Systems},
  volume={32},
  year={2019}
}

@article{zhang2020single,
  title={A single-loop smoothed gradient descent-ascent algorithm for nonconvex-concave min-max problems},
  author={Zhang, Jiawei and Xiao, Peijun and Sun, Ruoyu and Luo, Zhiquan},
  journal={Advances in neural information processing systems},
  volume={33},
  pages={7377--7389},
  year={2020}
}

@article{zhang2023jointly,
  title={Jointly Improving the Sample and Communication Complexities in Decentralized Stochastic Minimax Optimization},
  author={Zhang, Xuan and Mancino-Ball, Gabriel and Aybat, Necdet Serhat and Xu, Yangyang},
  journal={arXiv preprint arXiv:2307.09421},
  year={2023}
}

@article{kong2021accelerated,
  title={An accelerated inexact proximal point method for solving nonconvex-concave min-max problems},
  author={Kong, Weiwei and Monteiro, Renato DC},
  journal={SIAM Journal on Optimization},
  volume={31},
  number={4},
  pages={2558--2585},
  year={2021},
  publisher={SIAM}
}

@article{jin2021high,
  title={High probability complexity bounds for line search based on stochastic oracles},
  author={Jin, Billy and Scheinberg, Katya and Xie, Miaolan},
  journal={Advances in Neural Information Processing Systems},
  volume={34},
  pages={9193--9203},
  year={2021}
}

@article{jin2023sample,
  title={Sample Complexity Analysis for Adaptive Optimization Algorithms with Stochastic Oracles},
  author={Jin, Billy and Scheinberg, Katya and Xie, Miaolan},
  journal={arXiv preprint arXiv:2303.06838},
  year={2023}
}

@article{antonakopoulos2021sifting,
  title={Sifting through the noise: Universal first-order methods for stochastic variational inequalities},
  author={Antonakopoulos, Kimon and Pethick, Thomas and Kavis, Ali and Mertikopoulos, Panayotis and Cevher, Volkan},
  journal={Advances in Neural Information Processing Systems},
  volume={34},
  pages={13099--13111},
  year={2021}
}

@inproceedings{lin2020near,
  title={Near-optimal algorithms for minimax optimization},
  author={Lin, Tianyi and Jin, Chi and Jordan, Michael I},
  booktitle={Conference on Learning Theory},
  pages={2738--2779},
  year={2020},
  organization={PMLR}
}

@inproceedings{NIPS2017_6aed000a,
  title={Stochastic and Adversarial Online Learning without Hyperparameters},
  author={Cutkosky, Ashok and Boahen, Kwabena A},
  booktitle={Advances in Neural Information Processing Systems},
  pages={},
  year={2017},
  volume = {30},
  organization={}
}

@misc{lacostejulien2012simpler,
      title={A simpler approach to obtaining an O(1/t) convergence rate for the projected stochastic subgradient method}, 
      author={Simon Lacoste-Julien and Mark Schmidt and Francis Bach},
      year={2012},
      eprint={1212.2002},
      archivePrefix={arXiv},
      primaryClass={cs.LG}
}

@misc{rakhlin2012making,
      title={Making Gradient Descent Optimal for Strongly Convex Stochastic Optimization}, 
      author={Alexander Rakhlin and Ohad Shamir and Karthik Sridharan},
      year={2012},
      eprint={1109.5647},
      archivePrefix={arXiv},
      primaryClass={cs.LG}
}

@article{nesterov2012efficiency,
  title={Efficiency of coordinate descent methods on huge-scale optimization problems},
  author={Nesterov, Yu},
  journal={SIAM Journal on Optimization},
  volume={22},
  number={2},
  pages={341--362},
  year={2012},
  publisher={SIAM}
}

@inproceedings{layer_wise2021,
    author = {Sun, Wei and Zhou, Aojun and Stuijk, Sander and Wijnhoven, Rob and Nelson, Andrew Oakleigh and Li, hongsheng and Corporaal, Henk},
    title = {DominoSearch: Find layer-wise fine-grained N:M sparse schemes from dense neural networks},
    booktitle = {Advances in Neural Information Processing Systems},
    year = {2021},
    volume={34},
    pages={20721--20732}
}

@article{jin2019short,
  title={A short note on concentration inequalities for random vectors with subgaussian norm},
  author={Jin, Chi and Netrapalli, Praneeth and Ge, Rong and Kakade, Sham M and Jordan, Michael I},
  journal={arXiv preprint arXiv:1902.03736},
  year={2019}
}

@article{hazan2014beyond,
  title={Beyond the regret minimization barrier: optimal algorithms for stochastic strongly-convex optimization},
  author={Hazan, Elad and Kale, Satyen},
  journal={The Journal of Machine Learning Research},
  volume={15},
  number={1},
  pages={2489--2512},
  year={2014},
  publisher={JMLR. org}
}

@article{rosasco2020convergence,
  title={Convergence of stochastic proximal gradient algorithm},
  author={Rosasco, Lorenzo and Villa, Silvia and V{\~u}, Bang C{\^o}ng},
  journal={Applied Mathematics \& Optimization},
  volume={82},
  number={3},
  pages={891--917},
  year={2020},
  publisher={Springer}
}

@article{spokoiny2023concentration,
  title={Concentration of a high dimensional sub-gaussian vector},
  author={Spokoiny, Vladimir},
  journal={arXiv preprint arXiv:2305.07885},
  year={2023}
}

@article{tropp2012user,
  title={User-friendly tail bounds for sums of random matrices},
  author={Tropp, Joel A},
  journal={Foundations of computational mathematics},
  volume={12},
  number={4},
  pages={389--434},
  year={2012},
  publisher={Springer}
}

@book{boucheron2013concentration,
  title={Concentration Inequalities: A Nonasymptotic Theory of Independence},
  author={Boucheron, St{\'e}phane and Lugosi, G{\'a}bor and Massart, Pascal},
  year={2013},
  publisher={Oxford University Press},
  address={Oxford},
  isbn={978-0199535255}
}

@article{fercoq-bianchi,
author = {Fercoq, Olivier and Bianchi, Pascal},
title = {A Coordinate-Descent Primal-Dual Algorithm with Large Step Size and Possibly Nonseparable Functions},
journal = {SIAM Journal on Optimization},
volume = {29},
number = {1},
pages = {100-134},
year = {2019},
doi = {10.1137/18M1168480}
}

@article{madry2017towards,
  title={Towards deep learning models resistant to adversarial attacks},
  author={Madry, Aleksander and Makelov, Aleksandar and Schmidt, Ludwig and Tsipras, Dimitris and Vladu, Adrian},
  journal={arXiv preprint arXiv:1706.06083},
  year={2017}
}

@inproceedings{singh2023minimax,
  title={When do minimax-fair learning and empirical risk minimization coincide?},
  author={Singh, Harvineet and Kleindessner, Matth{\"a}us and Cevher, Volkan and Chunara, Rumi and Russell, Chris},
  booktitle={International Conference on Machine Learning},
  pages={31969--31989},
  year={2023},
  organization={PMLR}
}

\appendix
\section{Bounded variables}
\label{sec:bounded-proofs}
We first recall an important concentration result for independent bounded random variables.
\begin{lemma}[Bennett's Inequality]
\label{lem:Bennett}
    Let $\{Z_j\}_{j\in[M]}$ be independent random variables such that for all $j\in[M]$, $\mathbf{E}Z_j=0$, $\mathbf{E}Z_j^2<\infty$ and $Z_j \leq B$ a.s. for some $B>0$. Let $H:\reals_+\to\reals$ such that $H(u)=(1+u)\log(1+u)-u$. 
    Then, for all $t>0$, it holds that 
    \begin{align*}
        \p{\frac{1}{M}\sum_{j\in [M]}Z_j \geq t}\leq \exp\left(-\frac{\sum_{j\in[M]}\mathbf{E} Z_j^2}{B^2}\cdot H\left(\frac{MBt}{\sum_{j\in[M]}\mathbf{E}Z_j^2}\right)\right).
    \end{align*}
\end{lemma}
For proof one can refer to \cite[Theorem~2.9]{boucheron2013concentration}. Moreover, the same tail bound continues to hold for $\p{\frac{1}{M}\sum_{j\in [M]}Z_j \leq -t}$ in case there exists $B>0$ such that $Z_j \geq -B$ a.s. holds for all $j\in[M]$. Before we give the proof of Theorem~\ref{thm:concentration-bounded}, we next present an important technical lemma related to certain monotonicity properties of $H(\cdot)$.
\begin{lemma}
\label{lem:Bennett-helper}
Let $H:\reals_+\to\reals$ such that $H(u)=(1+u)\log(1+u)-u$ for $u\geq 0$. $H(\cdot)$ is an increasing function; moreover, $H(u)\geq \frac{u^2/2}{1+u/3}$ for all $u\geq 0$. Next, given some arbitrary $C>0$, define $W:\reals_+\to\reals$ such that $W(u)=u H(C/u)$. Then, $W(\cdot)$ is monotonically decreasing function.
\end{lemma}
\subsection{Proof of Theorem~\ref{thm:concentration-bounded}}
\begin{proof}
Let $\E{XX^\top}=\Sigma$, and
$(\sigma_{ii'})_{i,i'\in[d]}$ denote the entries of $\Sigma$, i.e., $\Sigma_{ii'}=\sigma_{ii'}$ for all $i,i'\in[d]$. Thus, $\mathbf{E}[\norm{X}^2]=\sigma^2$ implies that 
$\sigma^2=\sum_{i\in[d]}\sigma_{ii}$; moreover, since $\Sigma\succeq 0$, we have $\norm{\Sigma}=\lambda_{\rm max}(\Sigma)\leq\Tr(\Sigma)=\sigma^2$. 
Let $\{X_j\}_{j\in[M]}$ be i.i.d. copies of $X$. 
Therefore,
\begin{align*}
    \max\left\{\left\|\sum_{j\in[M]}\E{X_j^\top X_j}\right\|,
    ~\left\|\sum_{j\in[M]}\E{X_jX_j^\top}\right\|\right\}
    =M\max\{\sigma^2,\norm{\Sigma}\}
    =M\sigma^2,
\end{align*}
and \cite[Remark 6.3]
{tropp2012user} 
implies the desired inequality in \eqref{eq:mean-concentration-bounded}. Next, we show the tail bounds in~\eqref{eq:svar-tail-bounds}. Let $Y\triangleq \norm{X}^2-\sigma^2$; trivially, $\E{Y}=0$. Note that $\sigma^2=\E{\norm{X}^2}\leq B^2$ a.s. Thus, $Y\geq -\sigma^2\geq -B^2$ and $Y\leq\norm{X}^2\leq B^2$, which imply that $|Y|\leq B^2$ a.s. Moreover, 
\begin{align*}
    \E{Y^2}=\E{(\norm{X}^2-\sigma^2)^2}=\E{\norm{X}^4-2\sigma^2\norm{X}^2+\sigma^4}=\E{\norm{X}^4}-\sigma^4\leq \sigma^2(B^2-\sigma^2),
\end{align*} 
where the last inequality follows from $\norm{X}^2\leq B^2$ a.s. Hence, we can conclude that $\E{Y^2}\leq B^2\sigma^2$. Now, let $Y_j=\norm{X_j}^2-\sigma^2$ for $j\in[M]$, and set $\bar Y=\frac{1}{M}\sum_{j\in[M]}Y_j$. Thus, using $|Y_j|\leq B^2$ and $\bE Y_j^2\leq B^2\sigma^2$, Lemma~\ref{lem:Bennett} implies that for all $t>0$,
\begin{align}
\label{eq:barY-bound}
    \p{\bar Y\geq t}\vee\p{\bar Y\leq -t}\leq \exp\left(-M\frac{\bE Y^2}{B^4}\cdot H\left(\frac{B^2 t}{\bE Y^2}\right)\right)\leq \exp\left(-M\frac{\sigma^2}{B^2}\cdot H\left(\frac{t}{\sigma^2}\right)\right),
\end{align}
where in the second inequality we used Lemma~\ref{lem:Bennett-helper}. Recall that $\bar{X}\triangleq\frac{1}{M} \sum_{j=1}^{M} X_{j}$ and $\rv{v}\triangleq\frac{1}{M-1} \sum_{j=1}^{M} \| X_{j}-\bar{X} \|^{2}$, which imply that
\begin{align*}
 \rv{v} 
=\frac{1}{M-1} \sum_{j=1}^{M} \| X_{j}\|^{2}-\frac{M}{M-1}\| \bar{X} \|^{2}.
\end{align*}
Since $\{X_j\}_{j=1}^M$ are independent, we have $\E{\norm{\bar X}^2}=\frac{1}{M}\sigma^2$; hence, $\E{v}=\sigma^2$ and $\Big(1-\frac{1}{M}\Big)v-\sigma^2=\frac{1}{M}\sum_{j\in[M]}\norm{X_j}^2-\sigma^2-\norm{\bar X}^2$. Therefore, we have
    $\Big(1-\frac{1}{M}\Big)v-\sigma^2=\bar Y-\norm{\bar X}^2$, which implies that
\begin{align*}
    \p{\Big(1-\frac{1}{M}\Big)v-\sigma^2\leq-t}\leq\p{\{\bar Y\leq -t/2\}\cup\{\norm{\bar X}^2\geq t/2\}},\quad\forall~t>0.
\end{align*}
We get the desired inequality in \eqref{eq:svar-lower-tail-bounded} by combining \eqref{eq:mean-concentration-bounded} and \eqref{eq:barY-bound}. On the other hand, \eqref{eq:svar-upper-tail-bounded} follows from $\p{\Big(1-\frac{1}{M}\Big)v-\sigma^2\geq t}\leq \p{\bar Y\geq t}$ and the tail bound in \eqref{eq:barY-bound}. Finally, given any $\alpha\in(0,1)$, $\bar\sigma\geq\sigma$ and $c>0$, note that \eqref{eq:barY-bound} and $\bar\sigma\geq\sigma$ imply
{\small
\begin{align*}
    \p{\Big(1-\frac{1}{M}\Big)v \leq (1+c)\bar\sigma^2} 
    \geq 1-\p{\Big(1-\frac{1}{M}\Big)v-\sigma^2>c\bar\sigma^2}
    \geq 1-\exp\Big(-M\frac{\sigma^2}{B^2}\cdot H\Big(c \frac{\bar\sigma^2}{\sigma^2}\Big)\Big);
\end{align*}}%
hence, it follows from Lemma~\ref{lem:Bennett-helper} that $\p{\Big(1-\frac{1}{M}\Big)v \leq (1+c)\bar\sigma^2} 
    \geq 1-\exp\Big(-M H(c)\frac{\bar\sigma^2}{B^2}\Big)$, and this immediately leads to the desired result in \eqref{eq:concentration-subGaussian-bounded}.
\end{proof}
\section{Convergence Guarantees Under subGaussian Noise Setting}
\label{sec:subGaussian}
In this section, we relax the bounded noise assumption, i.e., \cref{as:bounded-oracle}, and we consider {the} light-tail setting where we assume that the stochastic gradient oracle is norm-subGaussian.
\begin{defn}
    The random vector $X\in\reals^n$ is norm-subGaussian with variance proxy $\sigma_p^2>0$ 
    if 
    $\mathbf{P}(\norm{X-\mathbf{E}[X]}>t)\leq 2\exp(-\frac{t^2}{2\sigma_p^2})$ for all $t>0$.
\end{defn}
In the rest, we argue that the guarantees for \sgdab{} continue to hold when \cref{as:bounded-oracle} is replaced by the following assumptions.
\begin{assumption}
\label{ass3}
    \rv{For any given $(\bx,y)\in\dom g\times\dom h$, the stochastic gradient $\gt_{\bx}f(\bx,y;\bom)$ and $\gt_{y}f(\bx,y;\bzt)$ are norm-subGaussian with proxy $p_x^2(\bx,y)$ and $p_y^2(\bx,y)$, respectively. 
    {We also assume that} there exist proxy bounds $p_x^2$ and $p_y^2$ such that $p_x^2(\bx,y)\leq p_x^2$ and $p_y^2(\bx,y)\leq p_y^2$ for all $(\bx,y)\in\dom g\times\dom h$.}
\end{assumption}
\begin{assumption}
\label{as:bounded-proxy}
    There exists a known constant $r_\sigma>0$ such that $r_\sigma\geq \max\{p_x^2(\bz)/\sigma_x^2(\bz),~p_y^2(\bz)/\sigma_y^2(\bz)\}$ for all $\bz\in\dom g\times\dom h$, where $\sigma_x^2(\bz)$ and $\sigma_y^2(\bz)$ denote variances of the stochastic oracles at $\bz$, i.e., $\sigma_x^2(\bz)=\E{\norm{\gt_{\bx}f(\bz;\bom)-\grad_\bx f(\bz)}^2}$ and $\sigma_y^2(\bz)=\E{\norm{\gt_{y}f(\bz;\bzt)-\grad_y f(\bz)}^2}$.
\end{assumption}
\rv{There are particular scenarios arising in practice for which $r_\sigma$ can be known in advance. Let $X\in\reals^d$ be a random vector. For instance, if $X~\sim\cN(\mu,\Sigma)$, then $X$ is norm-subGaussian with proxy $p^2=2\sigma^2$ where $\sigma^2=\E{\norm{X-\mu}^2}=\Tr(\Sigma)$; hence, whenever $\gt f(\bx,y;\bxi)$ is Gaussian, then $r_\sigma=2$. Now consider a more general scenario where $X\in\reals^d$ such that $\E{X}=0$ is sub-Gaussian in every direction, with proxy proportional to its directional variance, i.e., let $\Sigma\triangleq \E{X X^\top}$, and suppose for some $K\ge 1$,
\begin{equation}
\label{eq:directional-subGaussian}
\E{\exp\!\big(\theta \langle u,X\rangle\big)}
\;\le\;
\exp\!\left(\frac{K^2\theta^2}{2}\,u^\top \Sigma u\right)
\quad
\forall u\in\mathbb R^d,\ \forall \theta\in\mathbb R.
\end{equation}
Under this assumption, it follows from \cite[Eq.(2.2)]{spokoiny2023concentration} that $X$ is norm-subGaussian with proxy $p^2\leq 8K^2\sigma^2$ where $\sigma^2=\Tr(\Sigma)$; thus, $r_\sigma=8 K^2$. Next, we discuss some scenarios for which we know a bound on $K$ for $X\in\reals^d$ satisfying \eqref{eq:directional-subGaussian} such that $\E{X}=0$ and $\Sigma\succeq \sigma_{\rm \min}I$ for some $\sigma_{\rm \min}>0$. Note that when $X$ is Gaussian, we have $K=1$; moreover, in case $\norm{X}\leq B$ a.s. for some $B>0$, since $-G\norm{u}\leq \langle u,X\rangle\leq G\norm{u}$ using Hoeffding’s lemma, we get $\E{\exp\!\big(\theta \langle u,X\rangle\big)}
\;\le\;
\exp\!\left(\frac{B^2\theta^2}{2}\,\norm{u}^2\right)$
for all $u\in\mathbb R^d$ and $\theta\in\mathbb R$;
hence, we observe that \eqref{eq:directional-subGaussian} holds for $X$ with $K=B/\sqrt{\sigma_{\rm \min}}$. Finally, suppose $\{X_j\}_{j=1}^M$ are i.i.d and we set $X=\frac{1}{M}\sum_{j=1}^MX_j$; in this mini-batch setting, by the central limit theorem, the distribution of $X$ becomes close to Gaussian when $M$ is moderately large. Hence, we can conclude that \eqref{eq:directional-subGaussian} holds for $X$ with $K\approx 1$ and to be conservative for $K=2$, we get $r_\sigma=32$.}

\rv{
Next, under Assumption~\ref{ass3} we state a concentration result for $s(\bx,y;\bxi)$ defined in Definition~\ref{def:s-gradmapping}. Note that this result is similar to \cref{thm:concentration-bounded}.
\begin{theorem}
\label{thm:concentration}
    Let $\{X_j\}_{j=1}^M\subset\reals^d$ be i.i.d. norm-subGaussian random vectors with variance proxy $\sigma_p^2$, mean $\mu$ and variance $\sigma^2$, i.e., $X_j$ satisfies Assumption~\ref{ass3}, $\mathbf{E}[X_j]=\mu$ and $\mathbf{E}[\norm{X_j-\mu}^2]=\sigma^2$ for all $j=1,\ldots, M$. Then, the sample mean $\bar X=\frac{1}{M}\sum_{j=1}^M X_j$ satisfies
    \begin{align}
    \label{eq:mean-concentration}
        \p{\norm{\bar X-\mu}>t}\leq (d+1) \exp\Big(-\frac{Mt^2}{16\sigma_p^2}\Big),\quad\forall~t>0.
    \end{align}
    Moreover, let $v=\frac{1}{M-1}\sum_{j=1}^M\norm{X_j-\bar{X}}^2$ denote the sample variance {for $M\geq 2$}. Then,
    {\small
    \begin{subequations}
    \begin{align}
        \p{\Big(1-\frac{1}{M}\Big)v-\sigma^{2}\leq -t}&\leq \exp \left(-M \min \left\{\left(\frac{t}{32 \sigma_p^2}\right)^2,~\frac{t}{32 \sigma_p^2}\right\}\right)+(d+1) \exp\Big(-M\cdot\frac{t}{32 \sigma_p^2}\Big),\label{eq:svar-lower-tail}\\
        \p{\Big(1-\frac{1}{M}\Big)v-\sigma^{2}\geq t}&\leq \exp \left(-M \min \left\{\left(\frac{t}{16 \sigma_p^2}\right)^2,~\frac{t}{16 \sigma_p^2}\right\}\right),\quad\forall~t>0. \label{eq:svar-upper-tail}
    \end{align}
    \end{subequations}}%
    Finally, for any $\alpha\in (0,1)$, $\bar\sigma\geq\sigma$ and 
    $c>0$, it follows that \mgrev{for $M\geq 2$ we have}
    \begin{equation}
    \label{eq:concentration-subGaussian}
    {\small
    \begin{aligned}
    \p{\Big(1-\frac{1}{M}\Big)v\leq (1+c)\bar\sigma^{2}}\geq 1-\alpha,\quad 
    \mgrev{\mbox{whenever}}\quad
    M\geq \max\Big\{\Big(\frac{16}{c}\cdot\frac{\sigma_p^2}{\bar\sigma^2}\Big)^2,\Big(\frac{16}{c}\cdot\frac{\sigma_p^2}{\bar\sigma^2}\Big)\Big\}\cdot\log\Big(\frac{1}{\alpha}\Big). 
    \end{aligned}}%
    \end{equation}
\end{theorem}
The proof of \cref{thm:concentration} is provided in  \cref{sec:concentration-proof}.}
\sa{\sgdab, using the batch size of $M_x=\cO(\mg{1+}\rv{\widetilde\sigma_x^2}/\epsilon^2)$ and $M_y=\cO(\mg{1+}\rv{\widetilde\sigma_y^2}/\epsilon^2)$, 
terminates w.p.1 
(see Lemma~\ref{lem:stopping_probability}), and at the termination, it generates $(\bx_\epsilon,y_\epsilon)$ such that $\norm{\tilde G(\bx_\epsilon,y_\epsilon;\bxi)}\leq \epsilon$ w.p.1. 
\mg{Next, we show that under Assumption~\ref{ass3}, for batchsizes $M_x$ and $M_y$ as given above, the \saa{deterministic gradient map $G(\bx_\epsilon,y_\epsilon)$ is close to its stochastic estimate $\tilde G(\bx_\epsilon,y_\epsilon;\bxi)$} in a probabilistic sense 
\saa{implying that the output $(\bx_\epsilon,y_\epsilon)$ has a small deterministic gradient map}.} 
}
\begin{lemma}
\label{lem:Chebyshev-subG}
    Suppose \cref{ass1,assumption:noise,ass3} hold. For any given sample sizes $M_x,M_y\in\integers_+$, there exists $b:\reals_+\times\integers_+\times\integers_+\to\reals_+$ such that $b(r,M_x,M_y)$ decreasing in all the arguments, and
\begin{equation}
\label{eq:diff-bound-subG}
    \p{\norm{\tilde G(\bx,y,\bxi)-G(\bx,y)}>r}\leq b(r,M_x,M_y),\quad \forall~r>0,\quad \forall~(\bx,y)\in\dom g\times \dom h;
\end{equation} 
{in particular,} $b(r,M_x,M_y)=(n_x+1)e^{-\frac{M_x r^2}{32p_x^2}}+(n_y+1)e^{-\frac{M_y r^2}{32p_y^2}}$ satisfies \eqref{eq:diff-bound-subG}.
\end{lemma}
\begin{proof}
The proof follows from the same arguments used in the proof of Lemma~\ref{lem:Chebyshev}. Indeed, since \eqref{eq:G-xy-split} holds for any $r>0$, using \eqref{eq:nonexpansive_prox} and \eqref{eq:mean-concentration} in \cref{thm:concentration}, we get the desired result.
\end{proof}
\begin{remark}
\label{rem:small-eps-subG}
\rv{Suppose the variance bounds $\sigma_x$ and $\sigma_y$ are unknown. We assume that \cref{assumption:noise,ass3} hold, and in \sgdab{}, the sample sizes for the $\ell$-th iteration satisfy $M_x\geq \frac{128}{\epsilon^2}\widetilde\sigma_x^2$ and $M_y\geq \frac{128}{\epsilon^2}\widetilde\sigma_y^2$. This choice implies that
     $M_x \geqslant \frac{128}{\epsilon^{2}}\left(\frac{\sigma_x^0}{\gamma^\ell}\right)^{2} \geqslant \frac{128}{\epsilon^{2}}\left(\sigma_x^0\right)^{2}$ for all iterations $\ell\geq 0$ since the estimate $\tilde\sigma_x$ is updated using $\widetilde\sigma_x\gets\min\{\widetilde\sigma_x/\gamma,\bar{\sigma}_x\}$ starting from any given $\sigma_x^{0}>0$ as long as the stopping condition in \emph{\texttt{line}}~\ref{algeq:stop} in \sgdab{} does not hold. Similarly, $M_y\geq \frac{128}{\epsilon^2}\big(\sigma_y^0\big)^2$ for all $\ell\geq 0$.} 
    
    \rv{For any  $(\bx,y)$, let \rv{$(v^x,v^y)\gets V_{\bxi}(\bx,y)$ as defined in \eqref{eq:sample_variance}.} Thus, it follows from \cref{thm:concentration} that for any $\alpha\in(0,1)$, $\p{(1+c)\sigma_x^2\geq (1-\frac{1}{M_x})v^x}\geq 1-\alpha$ holds whenever
    \begin{align*} 
    \frac{128}{\epsilon^{2}}\left(\sigma_x^{0}\right)^{2} \geqslant\left(\frac{16}{c} \frac{p_x^2}{\sigma_x^2}\right)^2 \log \frac{1}{\alpha}
 \mgrev{\iff}
 \frac{\sigma_x^{0}}{\sqrt{2}} \frac{\sigma_x^2}{p_x^2} \frac{c}{\sqrt{\log(1/{\alpha})}} \geqslant \epsilon>0.
\end{align*}
A similar argument is also valid for $\p{(1+c)\sigma_y^2\geq (1-\frac{1}{M_y})v^y}\geq 1-\alpha$.} 

\rv{Therefore, for all sufficiently small $\epsilon>0$, i.e., $0<\epsilon=\cO\big(c\min\{\sigma_x^0,\sigma_y^0\}/\sqrt{\log(1/\alpha)}\big)$, one has $\p{(1+c)\sigma_x^2\geq (1-\frac{1}{M_x})v_{(t,\ell)}^x}\geq 1-\alpha$ and $\p{(1+c)\sigma_y^2\geq (1-\frac{1}{M_y})v^y_{(t,\ell)}}\geq 1-\alpha$ for all $\ell\geq 0$ and $t=1,\ldots,T$ where $\Big(v_{(t,\ell)}^x,v_{(t,\ell)}^y\Big)=V_{\tilde\bxi_{(t,\ell)}}\Big(\tilde\bx_{(t,\ell)},\tilde y_{(t,\ell)}\Big)$ as defined in \emph{\texttt{\cref{algeq:RBSGDA_runs}}} of \sgdab. Thus, to handle a larger range of tolerance values $\epsilon>0$, one can pick larger initial values for the variance bound estimates $\sigma_x^0,\sigma_y^0>0$, which are free design parameters.}
\end{remark}
Next, since \cref{thm:p-bound} continues to hold in this setting with $b(r,M_x,M_y)$ as defined in Lemma~\ref{lem:Chebyshev-subG}, we state a result similar to Corollary~\ref{cor:sample-complexity-bounded}, describing how sample sizes $M_x$ and $M_y$ should be selected when proxy bounds $p_x^2$ and $p_y^2$ are known.
\begin{corollary}
\label{cor:sample-complexity}
Under \cref{ass1}, for any given $\bar p\in(0,1)$, $\p{\norm{G\left(\bx_\epsilon,y_\epsilon\right)}\leq \epsilon}\geq 1-\bar p$ holds when $M_x$ and $M_y$ are chosen as follows:
\begin{enumerate}
    \item[(i)] under \cref{assumption:noise,ass3}, for any $C\geq \log\Big(\frac{2}{\bar p}\Big(\bar \ell+\frac{1}{1-p}\Big)\Big)$, let $M_x=\frac{128}{\epsilon^2}p_x^2(\log(n_x+1)+C)$ and $M_y=\frac{128}{\epsilon^2}p_y^2\Big(\log(n_y+1)+ \Big(1+\frac{6}{\rv{\widetilde\mu} \eta_y} \frac{2-\rv{\widetilde\mu} \eta_y}{1-\rv{\widetilde\mu} \eta_y}\Big)C\Big)$;
    \item[(ii)] under \cref{assumption:noise,ass3}, for any $C\geq \frac{\pi^2}{3\bar p}$, let $M_x=\frac{128}{\epsilon^2}p_x^2\Big(\log(n_x+1)+\log\Big(C(\ell+1)^2\Big)\Big)$ and $M_y=\frac{128}{\epsilon^2}p_y^2\Big(\log(n_y+1)+ \Big(1+\frac{6}{\rv{\widetilde\mu} \eta_y} \frac{2-\rv{\widetilde\mu} \eta_y}{1-\rv{\widetilde\mu} \eta_y}\Big)\log\Big(C(\ell+1)^2\Big)\Big)$.
\end{enumerate}
\end{corollary}
\begin{proof}
Recall that under \cref{assumption:noise,ass3}, according to Lemma~\ref{lem:Chebyshev-subG}, we get
    \begin{equation*}
        b_\ell\Big(\frac{\epsilon}{2}\Big) 
        =(n_x+1)e^{-\frac{M_x \epsilon^2}{128 p_x^2}}+(n_y+1)e^{-\frac{M_y \epsilon^2}{128 p_y^2}},
    \end{equation*}
    {where $b_\ell(\epsilon)\triangleq b(\epsilon,M_x,M_y)$ for $M_x,M_y$ values corresponding to the backtracking iteration $\ell\geq 0$ and $b(\epsilon,M_x,M_y)$ is defined in Lemma~\ref{lem:Chebyshev-subG}}. For the scenario \textit{(i)}, our choice of $M_x$ and $M_y$ immediately implies that
    \begin{equation*}
        b_\ell\Big(\frac{\epsilon}{2}\Big) \leq 2e^{-C};
    \end{equation*}
    therefore, from \cref{thm:p-bound}, for $C\geq \log\Big(\frac{2}{\bar p}\Big(\bar \ell+\frac{1}{1-p}\Big)\Big)$, we get
    \begin{equation*}
\begin{aligned}
\p{\norm{G\left(\bx_\epsilon,y_\epsilon\right)}\leq\epsilon}\geq 1-2\bar\ell e^{-C}-\sum_{k\geq 0} 2e^{-C} p^{k}=1-2e^{-C}\Big(\bar\ell+\frac{1}{1-p}\Big)\geq 1-\bar p.
\end{aligned}
\end{equation*}
On the other hand, for the scenario \textit{(ii)}, our choice of $M_x$ and $M_y$ immediately implies that
        $b_\ell\Big(\frac{\epsilon}{2}\Big) \leq \frac{2}{C}\cdot\frac{1}{(\ell+1)^2}$;
    therefore, from \cref{thm:p-bound}, for $C\geq \frac{\pi^2}{3\bar p}$, we get
    \begin{equation*}
\begin{aligned}
\p{\norm{G\left(\bx_\epsilon,y_\epsilon\right)}\leq\epsilon}\geq 1-\frac{2}{C}\sum_{\ell\geq 0} \frac{1}{(\ell+1)^2}=1-\frac{\pi^2}{3C}\geq 1-\bar p.
\end{aligned}
\end{equation*}
\end{proof}
Furthermore, due to \cref{thm:concentration} and Remark~\ref{rem:small-eps-subG}, the result of Lemma~\ref{lem:stopping_probability-unknown} continues to hold for this setting for $\epsilon=\cO\Big(\min\{\sigma_x^{(0)}\frac{\sigma_x^2}{p_x^2},~\sigma_y^{(0)}\frac{\sigma_y^2}{p_y^2}\}\frac{c}{\log(1/p)}\Big)$. Therefore, we are now ready to show a result similar to \cref{thm:p-bound-unknown}. 
\begin{theorem}
\label{thm:p-bound-unknown-bproxy}
    Let $\bar\ell\triangleq\lceil\log_{\frac{1}{\gamma}}(\rv{\cR})\rceil$ for $\cR$ as in \eqref{eq:R-unknown}. Under the premise of Lemma~\ref{lem:stopping_probability-unknown} with \cref{as:bounded-oracle} replaced by \cref{ass3}, for any $(\bx^0,y^0)\in\dom g\times\dom h$, $\epsilon>0$, $p\in (0,1)$ and all other parameters chosen as in the input line of \cref{alg:GDA-B} with \rv{$T\triangleq\lceil\log_2(2/p)\rceil$}, \sgdab{} employing $\mathrm{StopCond}(\ell)$ in \eqref{eq:stop-cond-simple} stops w.p.1 returning $(\bx_\epsilon,y_\epsilon)$ such that
\begin{equation*}
\begin{aligned}
\p{\norm{G\left(\bx_\epsilon,y_\epsilon\right)}\leq\epsilon}\geq 1-\sum_{\ell=0}^{\bar\ell-1} \tilde b_\ell(\epsilon)-\sum_{\ell\geq \bar\ell} \tilde b_\ell(\epsilon) p^{\ell-\bar\ell},
\end{aligned}
\end{equation*}
where
$\tilde b_\ell(\epsilon)\triangleq \tilde b^x_\ell(\epsilon)+\tilde b^y_\ell(\epsilon)$ and under Assumption~\ref{as:bounded-proxy}, $\tilde b^x_\ell(\epsilon),\tilde b^y_\ell(\epsilon)$ take the form:
\begin{align*}
    \tilde b^x_\ell(\epsilon)=(n_x+1)\Big[\exp\Big(\frac{-M_x^{(\ell)} \epsilon^2}{256 r_\sigma (1+c)\tilde\sigma^2}\Big)+\exp\Big(-M_x^{(\ell)}\cdot\frac{1}{64 r_\sigma}\Big)\Big]+\exp \left(-M_x^{(\ell)} \left(\frac{1}{64 r_\sigma}\right)^2\right)\\
    \tilde b^y_\ell(\epsilon)=(n_y+1) \Big[\exp\Big(\frac{-M_y^{(\ell)} \epsilon^2}{256 r_\sigma (1+c)\tilde\sigma^2}\Big)+\exp\Big(-M_y^{(\ell)}\cdot\frac{1}{64 r_\sigma}\Big)\Big]+ \exp \left(-M_y^{(\ell)} \left(\frac{1}{64 r_\sigma}\right)^2\right),
\end{align*}
for $M_x^{(\ell)},M_y^{(\ell)}$ denoting $M_x,M_y$ values corresponding to the backtracking iteration $\ell\geq 0$.
\end{theorem}
\begin{proof}
    The proof of this result follows from the same arguments used in the proof of Theorem~\ref{thm:p-bound-unknown} except for the bounds we derive for $w_1(\ell)$ and $w_2(\ell)$ defined in \eqref{eq:w-bounds}. More precisely, under Assumption~\ref{as:bounded-proxy}, which is a weaker requirement than Assumption~\ref{as:bounded-oracle} we can bound $w_1(\ell)$ in \eqref{eq:w1-bound-0} as follows. Indeed, \eqref{eq:w1-bound-0} and Lemma~\ref{lem:Chebyshev-subG} together imply that
\begin{align*}
    w_1(\ell) 
    &\leq \E{(n_x+1) \exp\Big(\frac{-M_x \big(\norm{G(\tz)}-\frac{\epsilon}{2}\big)^2}{32 p_x^2(\tz)}\Big)+(n_y+1) \exp\Big(\frac{-M_y \big(\norm{G(\tz)}-\frac{\epsilon}{2}\big)^2}{32p_y^2(\tz)}\Big)~\Big|~C_{(\ell)}^1}\\
    & \leq \sup_{t_1,t_2\geq 0}\quad (n_x+1) \exp\Big(\frac{-M_x \big(t_0-\frac{\epsilon}{2}\big)^2}{32 r_\sigma t_1}\Big)+(n_y+1) \exp\Big(\frac{-M_y \big(t_0-\frac{\epsilon}{2}\big)^2}{32 r_\sigma t_2}\Big)\\
    & \qquad \mbox{s.t.}\qquad t_0>\epsilon,~ t_1,t_2\leq 2(1+c)\tilde\sigma^2,
\end{align*}
where the second inequality follows from the conditional event $C_{(\ell)}^1$, for which it holds that $\norm{G(\tz)}>\epsilon$ and $\sigma_x^2(\tz)+\sigma_y^2(\tz)\leq 2(1+c)\tilde\sigma^2$; hence, after upper bounding $p_x^2(\tz)$ and $p_y^2(\tz)$ with $r_\sigma \sigma_x^2(\tz)$ and $r_\sigma \sigma_y^2(\tz)$, then replacing $\norm{G(\tz)}$ with $t_0$, $\sigma_x^2(\tz)$ and $\sigma_y^2(\tz)$ with $t_1$ and $t_2$, respectively, and taking the supremum over $t_0, t_1, t_2$, we end up with a deterministic bound with the supremum achieved at $t_0=\epsilon$, and $t_1=t_2=2(1+c)\tilde\sigma^2$, i.e.,
\begin{align*}
\label{eq:w1-bound-bproxy}
    w_1(\ell) 
    & \leq (n_x+1) \exp\Big(\frac{-M_x \epsilon^2}{256 r_\sigma (1+c)\tilde\sigma^2}\Big)+(n_y+1) \exp\Big(\frac{-M_y \epsilon^2}{256 r_\sigma (1+c)\tilde\sigma^2}\Big)\triangleq\bar w_1(\ell).
\end{align*}
Next, we bound $w_2(\ell)$. Under the event $C_{(\ell)}^2$, we have $\sigma_x^2\big(\tz\big)+\sigma_y^2\big(\tz\big)> 2(1+c)\widetilde\sigma^2$; hence, 
\begin{align*}
\p{A^v_{(\ell)}~\big|~\tz\,}
&\leq \p{(1-\frac{1}{M_x})v_{(t^*,\ell)}^x+(1-\frac{1}{M_y})v_{(t^*,\ell)}^y\leq \frac{1}{2}\sigma_x^2(\tz)+\frac{1}{2}\sigma_y^2(\tz)~\Big|~\tz\,}\\
&\leq \p{(1-\frac{1}{M_x})v_{(t^*,\ell)}^x\leq \frac{1}{2}\sigma_x^2(\tz)~\Big|~\tz\,}+\p{(1-\frac{1}{M_y})v_{(t^*,\ell)}^y\leq \frac{1}{2}\sigma_y^2(\tz)~\Big|~\tz\,},
\end{align*}
\rv{which follows from the definition of $A^v_{(\ell)}$ in \eqref{eq:ql-unknown}.} Therefore, it follows from Theorem~\ref{thm:concentration} that
{\small
\begin{align*}
\p{A^v_{(\ell)}~\big|~\tz\,}
&\leq \exp \left(-M_x \min \left\{\left(\frac{\sigma_x^2(\tz)}{64 p_x^2(\tz)}\right)^2,~\frac{\sigma_x^2(\tz)}{64 p_x^2(\tz)}\right\}\right)+(n_x+1) \exp\left(-M_x\cdot\frac{\sigma_x^2(\tz)}{64 p_x^2(\tz)}\right)\\
& \quad + \exp \left(-M_y \min \left\{\left(\frac{\sigma_y^2(\tz)}{64 p_y^2(\tz)}\right)^2,~\frac{\sigma_y^2(\tz)}{64 p_y^2(\tz)}\right\}\right)+(n_y+1) \exp\left(-M_y\cdot\frac{\sigma_y^2(\tz)}{64 p_y^2(\tz)}\right).
\end{align*}}%
Moreover, since $\sigma_x^2(\tz)\leq 4 p_x^2(\tz)$ and $\sigma_y^2(\tz)\leq 4 p_y^2(\tz)$, using \eqref{eq:w2-bound-0} and the bounds $p_x^2(\tz)\leq r_\sigma \sigma_x^2(\tz)$ and $p_y^2(\tz)\leq r_\sigma \sigma_y^2(\tz)$, we get 
{\small
\begin{equation*}
    \label{eq:w2-bound-bproxy}
\begin{aligned}
w_2(\ell)\leq \E{\p{A^s_{(\ell)}~\big|~\tz\,}\big|~C^2_{(\ell)}}
&\leq \exp \left(-M_x \left(\frac{1}{64 r_\sigma}\right)^2\right)+(n_x+1) \exp\Big(-M_x\cdot\frac{1}{64 r_\sigma}\Big)\\
& \quad + \exp \left(-M_y \left(\frac{1}{64 r_\sigma}\right)^2\right)+(n_y+1) \exp\Big(-M_y\cdot\frac{1}{64 r_\sigma}\Big)\triangleq \bar w_2(\ell).
\end{aligned}
\end{equation*}}%
Finally, the desired inequality follows from \eqref{eq:final-G-bound}.
\end{proof}
Finally, we are ready to state a result similar to Corollary~\ref{cor:sample-complexity-bounded-unkown}, describing how sample sizes $M_x$ and $M_y$ should be selected when proxy bounds $p_x^2$ and $p_y^2$ are unknown. 
\begin{corollary}
Under \cref{ass1,assumption:noise,ass3,as:bounded-proxy}, for any given $\bar p\in(0,1)$, $\p{\norm{G\left(\bx_\epsilon,y_\epsilon\right)}\leq \epsilon}\geq 1-\bar p$ holds when $M_x$ and $M_y$ are chosen such that $M_x=\frac{C_0}{\epsilon^2}(C_1^x\log(n_x+1)\tilde\sigma_x^2+C_2 (\sigma^0)^2/\bar\gamma^{\rv{2\ell}})$ and $M_y=\frac{C_0}{\epsilon^2}\Big(C_1^y\log(n_y+1)\tilde\sigma_y^2+ C_2 (\sigma^0)^2 \Big(1+\frac{6}{\rv{\widetilde\mu} \eta_y} \frac{2-\rv{\widetilde\mu} \eta_y}{1-\rv{\widetilde\mu} \eta_y}\Big)/\bar\gamma^{\rv{2\ell}}\Big)$ for $C_0\geq 256 r_\sigma (1+c)$, $C_1^x\geq (\sigma^0)^2/(\sigma^0_x)^2$, $C_1^y\geq (\sigma^0)^2/(\sigma^0_y)^2$, $C_2\geq \log\Big(1+\frac{4}{\bar p}\Big)$, $\bar\gamma\in(0,\gamma)$ and $\epsilon\in(0,\bar\epsilon)$, where $(\sigma^0)^2=(\sigma_x^0)^2+(\sigma_y^0)^2$ and
\begin{align*}
    \bar\epsilon^2=\cO\left((\underline{\sigma}^0)^2\min\left\{\gamma^2\min\Big\{1,\frac{\log(\underline{n})}{\log(1/\bar p)+\log(1/\gamma)}\Big\},~\frac{1}{r_\sigma},~\frac{c^2}{r_\sigma}\frac{\log(1/\bar p)}{\log(1/p)}\right\}\right).
\end{align*}
for $\underline{\sigma}^0=\min\{\sigma_x^0,\sigma_y^0\}$, and $\underline{n}=\min\{n_x,n_y\}$. Thus, fixing $p,\bar p\in(0,1)$ such that $p=\bar p$, for $n_x,n_y> 1/\bar p$, a simpler bound can be written $\bar\epsilon=\cO\Big(\underline{\sigma}^0\min\Big\{\gamma, \frac{c}{\sqrt{r_\sigma}}\Big\}\Big)$.
\end{corollary}
\begin{proof}
Without loss of generality, suppose $\gamma/\bar\gamma= \sqrt{2}$. Given an arbitrary $p\in(0,1)$, for any $\ell\geq \bar\ell$, since $\tilde\sigma_x^2\geq\sigma_x^2$ and $\tilde\sigma_y^2\geq\sigma_y^2$, the event $A_{(\ell)}^v$ defined in \eqref{eq:ql-unknown} holds with probability at least $1-p/2$ whenever
\begin{equation}
\label{eq:test-condition-unknown}
    \p{\Big(1-\frac{1}{M_x}\Big)v^x_{(t^*,\ell)}\leq (1+c)\sigma_x^{2}}\geq 1-\frac{p}{4},\quad \p{\Big(1-\frac{1}{M_y}\Big)v^y_{(t^*,\ell)}\leq (1+c)\sigma_y^{2}}\geq 1-\frac{p}{4},
    \end{equation}
and according to \cref{thm:concentration}, \eqref{eq:concentration-subGaussian} implies that \eqref{eq:test-condition-unknown} holds under \cref{assumption:noise,ass3} whenever $M_x,M_y\in\integers_+$ satisfy $\min\{M_x,M_y\}\geq \Big(\frac{16}{c}\cdot r_\sigma\Big)^2\cdot\log\Big(\frac{4}{p}\Big)$.

Our choice of $M_x,M_y$ immediately implies a trivial upper bound on $\bar w_1(\ell)$ defined in \eqref{eq:w1-bound}, i.e.,
    \begin{equation*}
        \bar w_1(\ell)=(n_x+1)\exp\Big(\frac{-M_x \epsilon^2}{256 r_\sigma (1+c)\tilde\sigma^2}\Big)+(n_y+1)\exp\Big(\frac{-M_y \epsilon^2}{256 r_\sigma (1+c)\tilde\sigma^2}\Big)\leq \rv{2\exp\left(-C_2 \frac{\gamma^{2\ell}}{\bar\gamma^{2\ell}}\right)};
    \end{equation*}
    therefore, given $\bar p\in (0,1)$, since $(\gamma/\bar\gamma)^2=2$, we get
    $\sum_{\ell\geq 0}\bar w_1(\ell)\leq\sum_{\ell\geq 0} 2\exp\left(-C_2 \frac{\bar\gamma^{2\ell}}{\gamma^{2\ell}}\right)\leq \bar p/2$ holds for all $C_2\geq \log\Big(1+\frac{4}{\bar p}\Big)$ since $\sum_{\ell\geq 1} \frac{2}{e^{C_2 \ell}}=\rv{2/(e^{C_2}-1)}\leq \bar p/2$ holds for such $C_2>0$.
    Next, we derive an upper bound on $\bar w_2(\ell)$ defined in \eqref{eq:w2-bound}. First, we consider
    \begin{align*}
        \MoveEqLeft \bar w_3(\ell)\triangleq (n_x+1) \exp\Big(-M_x\cdot\frac{1}{64 r_\sigma}\Big)+(n_y+1) \exp\Big(-M_y\cdot\frac{1}{64 r_\sigma}\Big)\\
        & \leq (n_x+1)\exp\Big(\frac{-4(1+c)}{\epsilon^2}\log(n_x+1)\tilde\sigma_x^2\Big)+(n_y+1)\exp\Big(\frac{-4(1+c)}{\epsilon^2}\log(n_y+1)\tilde\sigma_y^2\Big)\\
        & = \frac{n_x+1}{(n_x+1)^{\frac{4(1+c)(\sigma_x^0)^2}{\epsilon^2}\cdot \frac{1}{\gamma^{2\ell}}}}+\frac{n_y+1}{(n_y+1)^{\frac{4(1+c)(\sigma_y^0)^2}{\epsilon^2}\cdot \frac{1}{\gamma^{2\ell}}}};
    \end{align*}
    therefore, Lemma~\ref{lem:simple-bound} implies that 
    $\sum_{\ell=0}^\infty\bar w_3(\ell)\leq \bar p/4$ if
    \begin{align*}
        \frac{4(1+c)(\sigma_x^0)^2}{\epsilon^2}\gamma^2
        &>
        1+\frac{1}{\log(n_x+1)}\Big(\log\Big(\frac{
        \rv{2}}{\bar p\log(n_x+1)\log(1/\gamma)}\Big)+2\log(1/\gamma)\Big)
        \triangleq B_1(\bar p)\\
        \frac{4(1+c)(\sigma_y^0)^2}{\epsilon^2}\gamma^2
        &>
        1+\frac{1}{\log(n_y+1)}\Big(\log\Big(\frac{
        \rv{2}}{\bar p\log(n_y+1)\log(1/\gamma)}\Big)+2\log(1/\gamma)\Big)
        \triangleq B_2(\bar p).
    \end{align*}
    Thus, $\sum_{\ell=0}^\infty\bar w_3(\ell)\leq \bar p/4$ holds for all $\epsilon>0$ such that $\epsilon^2 < 4(1+c)\gamma^2 \min\Big\{\frac{(\sigma_x^0)^2}{B_1(\bar p)},~\frac{(\sigma_y^0)^2}{B_2(\bar p)}\Big\}$.

    Next, consider $\bar w_4(\ell)$ defined below such that $\bar w_2(\ell)=\bar w_3(\ell)+\bar w_4(\ell)$, i.e.,
    \begin{align*}
        \MoveEqLeft \bar w_4(\ell)\triangleq 
        \exp \left(-M_x \left(\frac{1}{64 r_\sigma}\right)^2\right)
        +\exp \left(-M_y \left(\frac{1}{64 r_\sigma}\right)^2\right)\\
        & \leq 2\exp\Big(\frac{-(1+c)(\sigma^0)^2}{16 r_\sigma \epsilon^2} \log\Big(1+\frac{4}{\bar p}\Big)\frac{1}{\bar\gamma^{\rv{2\ell}}}\Big),
    \end{align*}
    where the inequality follows from our choice of $M_x$, $M_y$, and from $C_2\geq \log\Big(1+\frac{4}{\bar p}\Big)$. Thus, Lemma~\ref{lem:simple-bound} implies that $\sum_{\ell=0}^\infty\bar w_4(\ell)\leq \bar p/4$ holds whenever
    \begin{align*}
        \frac{(1+c)(\sigma^0)^2}{16 r_\sigma \epsilon^2}\geq \frac{1}{\log(1+4/\bar p)}\max\Big\{\log(16/\bar p),~\frac{\log(2)}{2}\frac{1}{\log(1/\bar\gamma)}\Big\}\triangleq B_3(\bar p), 
    \end{align*}
    i.e., for all $\epsilon>0$ such that $\epsilon^2< \frac{(1+c)(\sigma^0)^2}{16 r_\sigma}\cdot\frac{1}{B_3(\bar p)}$.

    Finally, we observed that \eqref{eq:test-condition-unknown} holds whenever $\min\{M_x,M_y\}\geq \Big(\frac{16}{c}\cdot r_\sigma\Big)^2\cdot\log\Big(\frac{4}{p}\Big)$; hence, a sufficient condition on $M_x,M_y$ is
    \begin{align*}
        \frac{256 r_\sigma (1+c)}{\epsilon^2} (\sigma^0)^2 \log\Big(1+\frac{4}{\bar p}\Big)\geq \Big(\frac{16}{c}\cdot r_\sigma\Big)^2\cdot\log\Big(\frac{4}{p}\Big),
    \end{align*}
    which follows from $C_0\geq 256 r_\sigma (1+c)$ and $C_2\geq \log\Big(1+\frac{4}{\bar p}\Big)$. Thus, for given $p,\bar p\in (0,1)$, \eqref{eq:test-condition-unknown} holds for all $\epsilon>0$ such that
    \begin{align*}
        \epsilon^2 < (1+c)c^2 \frac{(\sigma^0)^2}{r_\sigma} \log\Big(1+\frac{4}{\bar p}\Big)/\log\Big(\frac{4}{p}\Big);
    \end{align*}
    therefore, by simplifying all three bounds we derived above, we can conclude that
    {\small
    \begin{align*}
    \bar\epsilon=\cO\left((\underline{\sigma}^0)^2\min\left\{\gamma^2\min\Big\{1,\frac{\log(\underline{n})}{\log(1/\bar p)+\log(1/\gamma)}\Big\},~\frac{1}{r_\sigma}\min\{1,\log(1/\bar p)\log(1/\bar\gamma)\},~\frac{c^2}{r_\sigma}\frac{\log(1/\bar p)}{\log(1/p)}\right\}\right).
\end{align*}}%
\end{proof}
\vspace*{-7.5mm}
\section{Proof of \cref{thm:concentration}}
\label{sec:concentration-proof}
In this section, we show the following concentration result. \rv{Given the sample size $M\geq 2$,} let $\{X_j\}_{j=1}^M$ be i.i.d. norm-subGaussian random vectors with \rv{mean $\mu$}, variance proxy \rv{$\sigma_p^2$} and variance $\sigma^2$, i.e., $X_j$ satisfies Assumption~\ref{ass3}, \rv{$\E{X_j}=\mu$ and $\sigma^2=\mathbf{E}[\norm{X_j-\mu}^2]$} for all $j=1,\ldots, M$. Let $\bar X=\frac{1}{M}\sum_{j=1}^M X_j$ denotes the sample mean, and $v^2$ denote the sample variance, i.e., $v^2=\frac{1}{M-1}\sum_{j=1}^M\norm{X_j-\bar{X}}^2$. 
We first state some standard results on subGaussian and subexponential random variables.
\begin{lemma}
\label[lemma]{lem:moment-bound}
    Let $Z$ be a 
    random variable such that for some \rv{$\sigma_p>0$}, 
    $\p{|Z| \geq t} \leq 2 e^{-t^2 / (2 \sigma_p^2)}$ holds for all $t>0$. Then, $\bE|Z|^k\leq (2 \sigma_p^2)^{k/2} k\Gamma\left(\frac{k}{2}\right)$ for any $k \geq 1$, and it holds that $(\bE|Z|^k)^{1/k} \leq e^{\frac{1}{e}}\sigma_p\sqrt{k}$ for any $k\geq 2$. 
\end{lemma}
\begin{proof}
Given any $k\geq 1$, it holds that
\begin{equation}
    \begin{aligned}
      \bE|Z|^k &= \int_0^\infty \p{|Z|^k > t} dt = \int_0^\infty \p{|Z| > t^{1/k}} dt\\
      & \leq k (2 \sigma_p^2)^{k/2}\int_0^\infty  e^{-u} u^{\frac{k}{2}-1} du= k (2 \sigma_p^2)^{k/2} \Gamma\left(\frac{k}{2}\right),
    \end{aligned}
\end{equation}
where $u = \frac{t^{2/k}}{2 \sigma_p^2}$ and $dt = \frac{k}{2u} (2 \sigma_p^2 u)^{k/2} du$. \rv{Therefore, for $k\geq2$, it holds that $\bE|Z|^k\leq  k (2 \sigma_p^2)^{k/2} \left(\frac{k}{2}\right)^{\frac{k}{2}}$ since $\Gamma\left(\frac{k}{2}\right)\leq (\frac{k}{2})^{\frac{k}{2}}$.}
Hence, we have
\begin{equation}
    (\bE|Z|^k)^{1/k} \leq\sqrt{2\sigma_p^2} k^{\frac{1}{k}} \sqrt{\frac{k}{2}} \leq e^{\frac{1}{e}}\sigma_p\sqrt{k}, \quad \forall k \geq 2,
\end{equation}
where the last inequality is due to $k^{\frac{1}{k}} \leq e^{\frac{1}{e}}$ for all $k\geq 2$.
\end{proof}
\begin{remark}
\label{rem:sigma-p-relation}
    \rv{Under the premise of \cref{lem:moment-bound}, for $Z$ such that $\bE[Z]=0$, the variance of $Z$ can be bounded as follows: $\sigma^2=\bE|Z|^2\leq 4\sigma_p^2$.}
\end{remark}
\begin{defn}
    A random variable $Y$ is subexponential with parameters $(\lambda,b)$ if $\E{Y}=\mathbf{0}$ and $\E{e^{\theta Y}}\leq \exp\Big(\frac{\lambda^2\theta^2}{2}\Big)$ holds for all $\theta\in\reals$ such that $|\theta|\leq\frac{1}{b}$.
\end{defn}
\begin{lemma}
\label{lem:subexponential-constants}
Consider a random vector  $X \in \mathbb{R}^n$ such that 
$\bE X=\mathbf{0}$ and $\p{\|X\| > t} \leq 2\exp\Big(\frac{-t^2}{2\sigma_p^2}\Big)$ for all  $t>0$, for some $\sigma_p > 0$. Then $Y \triangleq \|X\|^2 - \E{\|X\|^2}$ satisfies $\E{e^{\theta Y}} \leq e^{\sa{64} \sigma_p^4 \theta^2}$ for all $\theta:\ |\theta| \leq \frac{1}{8\sigma_p^2}$, i.e., $Y$ is subexponential with $(\lambda, b) = \left(\rv{8\sqrt{2}} \sigma_p^2, 8\sigma_p^2\right)$ and $\E{Y}=0$.
\end{lemma}
\begin{proof}
Note that $\bE Y=0$; hence, from the dominated convergence theorem, we have
\begin{equation}
    \begin{aligned}
        \bE e^{\theta Y} 
        &= 1 + \sum_{k=2}^{\infty} \frac{\theta^k \E{\left( \|X\|^2 - \bE \|X\|^2 \right)^k}}{k!} 
        \leq 1 + \sum_{k=2}^{\infty} \frac{\theta^k}{k!} 2^{k-1}\E{\|X\|^{2k} + \left(\bE \|X\|^2\right)^k }\\
        &\leq 1 + \sum_{k=2}^{\infty} \frac{\theta^k}{k!} 2^{k} \bE\|X\|^{2k}\leq 1 + \sum_{k=2}^{\infty} \frac{\theta^k}{k!} 2^k (2\sigma_p^2)^k 2k \Gamma(k)\\
        &= 1 + 2(4 \sigma_p^2\theta)^2 \sum_{k=0}^{\infty} (4 \sigma_p^2 \theta)^k
= 1 + 2 \frac{(4 \sigma_p^2 \theta)^2}{1 - 4 \sigma_p^2 \theta}\\
&\leq 1 + 4 (4 \sigma_p^2 \theta)^2 \leq e^{\rv{64} \sigma_p^4 \theta^2} ,
    \end{aligned}
\end{equation}
where the first inequality is due to $(a - b)^k \leq (|a| + |b|)^k \leq 2^{k-1} (|a|^k + |b|^k)$, the second inequality is due to $(\bE|Z|^r)^{1/r} {\leq}(\bE|Z|^s)^{1/s},\forall 1 < r < s < +\infty$, the third inequality is due to Lemma~\ref{lem:moment-bound} and the last two inequalities {are} 
due to $|\theta| \leq \frac{1}{8 \sigma_p^2}$, and \sa{we also use $1+a\leq e^a$ for all $a\in\reals$.} 
\end{proof}
The following is a standard Bernstein inequality result for subexponential random variables.
\rv{
\begin{theorem}
\label{thm:subexponential-bernstein}
 Let $\{Y_{j}\}_{j=1}^M$ be i.i.d. subexponential random variables with parameters $(\lambda, b)$. Then $\bar Y\triangleq \frac{1}{M}\sum_{j=1} Y_{j}$ is also subexponential random variable with parameters $(\lambda,b)$, and
\begin{align}
\label{eq:bernstein-probability-bound}
    \p{
    \bar Y\leqslant -t}\vee \p{
    \bar Y \geqslant t} \leq \exp \left(-\frac{M}{2} \min \left\{t^{2}/\lambda^{2}, t/b\right\}\right),\qquad\forall~t>0.
\end{align} 
\end{theorem}
\begin{proof}
    From Chernoff bound, we get $\p{\bar Y\geq t}\leq \Pi_{j=1}^M\E{e^{\theta Y_j}}e^{-\theta M t}$ for any $\theta>0$. Moreover, for $|\theta|\leq \frac{1}{b}$, we also have $\E{e^{\theta Y_j}}\leq e^{\frac{\lambda^2\theta^2}{2}}$ for all $j\in[M]$; hence, $\p{\bar Y\geq t}\leq e^{\frac{M\lambda^2\theta^2}{2}-\theta M t}$ and choosing $\theta=\min\{\frac{t}{\lambda^2},~\frac{1}{b}\}\in[0,\frac{1}{b}]$, we get the desired inequality. The bound on $\p{\bar Y\leqslant -t}$ follows from the same arguments.
\end{proof}}%

The following concentration inequality is shown in \cite{jin2019short} without specifying the universal constant $C>0$; moreover, although the result is correct, there is a mistake in their proof\footnote{\cite{jin2019short} argues that $\bE Y^{2k+1}=\mathbf{0}$ for all $k\geq 1$; however, this is not necessarily true when it is assumed that $X$ is norm-subGaussian with $\bE X=\mathbf{0}$. When $X\in\reals$ is a continuous random variable, one sufficient condition for their claim to be true is that $X$ has a symmetric density with respect to $\mathbf{0}$.} --which affects the value of the universal constant $C>0$. Therefore, for the sake of completeness, we provide the proof of this result here and a prescribe a particular value for the constant $C>0$.

\begin{lemma}[Corollary~7 in~\cite{jin2019short}]
\label{lem:subGaussian-concentration-bound}
Consider a random vector  $X \in \mathbb{R}^n$ such that $\bE X=\mathbf{0}$ and $\p{\|X\| > t} \leq 2\exp\Big(\frac{-t^2}{2\sigma_p^2}\Big)$ for all  $t>0$, for some $\sigma_p > 0$. Let $\{X_j\}_{j=1}^M$ be i.i.d. random vectors following the same distribution with $X$, and \rv{$\bar X\triangleq \frac{1}{M}\sum_{j=1}^M X_j$}. Then, 
\begin{equation}
    \label{eq:barx-concentration}
    \rv{\p{\norm{\bar X}^2 \geq t} \leq (d+1) \exp\Big(\frac{-Mt}{C^2\sigma_p^2}\Big),\quad \mbox{where}\quad C\triangleq 4.}
\end{equation}
\end{lemma}
\begin{proof}
Let $Y(X) \triangleq 
\begin{bmatrix}
    0 & X^\top \\ X & \mathbf{0}
\end{bmatrix}\in \mathbb{S}^{d+1}$. Note that $\E{Y(X)}=\mathbf{0}$ and $\Rank(Y(X))=2$. It can be checked that $u_+=[\norm{X},~ X^\top]^\top$ and $u_-=[-\norm{X},~ X^\top]^\top$ are eigenvectors of $Y(X)$ corresponding to the two non-zero eigenvalues $\lambda_+=\norm{X}$ and $\lambda_-=-\norm{X}$, respectively. In the rest, for the simplicity of the notation, let $Y=Y(X)$. The non-zero eigenvalues corresponding to the powers of $Y$ have the following form: $\lambda(Y^{2k})=\norm{X}^{2k}$ and $\lambda(Y^{2k+1})=\pm\norm{X}^{2k+1}$ for all $k\geq 0$. Thus, $\lambda_{\max}(Y^k)=\norm{X}^k$ for $k\geq 1$, which implies that $\norm{X}^k \bI\succeq Y^k$ for $k\geq 1$. Then, since $\bE Y=\mathbf{0}$ and $X$ is norm-subGaussian, for any $\theta\in\reals$ we have
\begin{equation}
    \bE e^{\theta Y} = \bI + \sum_{k=2}^\infty \frac{\theta^k \bE Y^k}{k!} \preceq \bI \left( 1 + \sum_{k=2}^\infty \frac{\theta^{k} \bE \norm{X}^{k}}{k!}\right)\preceq \bI \left( 1 + \sum_{k=2}^\infty \frac{\theta^{k} (2\sigma_p^2)^{k/2}k\Gamma(k/2)}{k!}\right),
\end{equation}
where the last inequality follows from invoking Lemma~\ref{lem:moment-bound} for $Z=\norm{X}$, which is subGaussian; thus,
\begin{equation}
    \begin{aligned}
        \bE e^{\theta Y} 
        &\preceq 
        \bI \left(1 + \sum_{k=1}^\infty \frac{(2\theta^2\sigma_p^2)^k2k\Gamma(k)}{(2k)!}
        + \sum_{k=1}^\infty \frac{(2\theta^2\sigma_p^2)^{k+1/2}(2k+1)\Gamma(k+1/2)}{(2k+1)!}\right) \\
        & \preceq 
        \bI \left(1 + \Big(2+\sqrt{2\theta^2\sigma_p^2}\Big) \sum_{k=1}^\infty \frac{(2\theta^2\sigma_p^2)^k k!}{(2k)!}
        \right)\preceq 
        \bI \left(1 + \Big(1+\sqrt{\theta^2\sigma_p^2/2}\Big) \sum_{k=1}^\infty \frac{(2\theta^2\sigma_p^2)^k}{k!}
        \right)\\
        &=\bI\left(\exp(2\theta^2\sigma_p^2)+\sqrt{\theta^2\sigma_p^2/2}(\exp(2\theta^2\sigma_p^2)-1)\right)\preceq \exp(4\theta^2\sigma_p^2)\bI,
    \end{aligned}
\end{equation}
where the second inequality is due to $\Gamma(k+1/2)<\Gamma(k+1)$ for $k\geq 1$, the third inequality follows from $2(k!)^2\leq (2k)!$ for $k\geq 1$, and the final inequality holds because of $e^{2x^2}+\frac{|x|}{\sqrt{2}}(e^{2x^2}-1)\leq e^{4x^2}$ for all $x\in\reals$. 

Define $\cF_k \triangleq \sigma\Big(\{X_j\}_{j=1}^k\Big)$ and $Y_j\triangleq Y(X_j)$ for $j=1,\ldots, M$. Thus, using the tower law of expectation, we get
\begin{equation}
\label{eq:mgf-Y}
    \begin{aligned}
         &\E{\Tr\left( \exp\Big(-4M\theta^2 \sigma_p^2 \bI + \theta\sum_{j=1}^M Y_j\Big) \right)} \\
= &\E{\bE\left[\Tr \left( \exp\left( -4M\theta^2  \sigma_p^2 \bI + \theta \sum_{j=1}^M Y_j \right) \right)\middle| \mathcal{F}_{M-1} \right]}\\
\leq & \E{ \Tr \left( \exp\left( -4M\theta^2 \sigma_p^2 \bI + \theta \sum_{j=1}^{M-1} Y_j + \log \E{e^{\theta Y_M}| \mathcal{F}_{M-1}}\right) \right) } \\
\leq & \E{\Tr \left( \exp\left( -4(M-1)\theta^2 \sigma_p^2 \bI + \theta \sum_{j=1}^{M-1} Y_j \right) \right)}\\
\leq &\ldots \leq \Tr \exp(0\bI) = d+1,
    \end{aligned}
\end{equation}
where the first inequality follows from $\bE \Tr \exp(A+B)\leq \Tr \exp(A+\log(\bE\exp(B)))$ for any fixed $A$ and random $B$, and the second inequality 
follows from
the fact that $\{Y_j\}_{j=1}^M$ are independent and from the inequality $e^{A+B}\leq e^{A+C}$ when $B\preceq C$. 

As discussed earlier, we have $\lambda(\sum_{j=1}^MY_j)=\pm\norm{\sum_{j=1}^M X_j}$, which implies that for any $t,\theta>0$,
\begin{equation}
\label{eq:unoptimized-Xbound}
\begin{aligned}
\bP\left(\left\| \sum_{j=1}^M  X_j \right\| \geq 4M \sigma_p^2 \theta + \frac{t}{\theta} \right)
= & \bP \left( \lambda_{\rm max}\left( \sum_{j=1}^M  Y_j \right) \geq 4M\theta \sigma_p^2 + \frac{t}{\theta}  \right)\\
=& \bP \left( \lambda_{\rm max}\left( \exp\left(\theta\sum_{j=1}^M  Y_j\right) \right) \geq \exp(4M\theta^2 \sigma_p^2 + t)  \right)\\
\leq & \bP \left( \Tr\left( e^{-4M\theta^2 \sigma_p^2} \exp\left(\theta\sum_{j=1}^M  Y_j\right) \right) \geq e^t  \right)\\
=& \bP \left( \Tr\left( e^{- 4M\theta^2 \sigma_p^2 \bI+\theta\sum_{j=1}^M  Y_j} \right) \geq e^t  \right)\\
\leq & \E{\Tr\left( e^{- 4M\theta^2 p^2 \bI+\theta\sum_{j=1}^M  Y_j} \right)} e^{-t} \leq (d+1)e^{-t}, 
\end{aligned}
\end{equation}
where in the first inequality we used $e^A\succ 0$ for any symmetric $A$; hence, $\Tr(e^A)>\lambda_{\max}(e^A)$, and the last inequality follows from Markov's inequality and \eqref{eq:mgf-Y}.

\rv{This result in \eqref{eq:unoptimized-Xbound} holds for any $\theta>0$; therefore, for any given $t>0$, choosing $\theta=\sqrt{\frac{t}{4M\sigma_p^2}}$ leads to $\p{\norm{\sum_{j=1}^M X_j}\geq 4\sigma_p\sqrt{Mt}}\leq (d+1)e^{-t}$, which implies the desired result.}
\end{proof}

Now we are ready to prove \cref{thm:concentration}. Let $\{X_j\}_{j=1}^M \in \reals^d$ be i.i.d. random vectors such that for some $\mu\in\reals^d$, $\sigma, \sigma_p>0$, we have
\begin{equation}
\E{X_j} = \mu, \quad \E{\|X_{j}-\mu\|^{2}}=\sigma^{2}, \quad \p{\left\|X_{j}-\mu\right\|>t} \leq 2 e^{-t^{2} / (2 \sigma_p^{2})},\quad \forall~t>0,
\end{equation}
for all $j=1,\ldots,M$. Let $\bar{X}\triangleq\frac{1}{M} \sum_{j=1}^{M} X_{j}$ and $\rv{v}\triangleq\frac{1}{M-1} \sum_{j=1}^{M} \| X_{j}-\bar{X} \|^{2}$. Note that
\begin{align}
\label{eq:sample-variance}
 \rv{v} 
=\frac{1}{M-1} \sum_{j=1}^{M} \| X_{j}-\mu\|^{2}-\frac{M}{M-1}\| \bar{X}-\mu \|^{2}.
\end{align}
Since $\{X_j\}_{j=1}^M$ are independent, we have
\begin{align}
\label{eq:mu-barx}
\bE \| \bar{X}-\mu \|^{2} =\frac{1}{M^{2}} \E{  \| \sum_{j=1}^{M}\left(X_{j}-\mu\right) \|^{2}}
=\frac{1}{M} \sigma^{2};
\end{align}
hence, using \eqref{eq:sample-variance} and \eqref{eq:mu-barx}, we immediately get $\bE v = \sigma^2$.
\rv{Let $Y_j\triangleq\norm{X_j-\mu}^2-\sigma^2$ for $j=1,\ldots,M$, and $\bar Y\triangleq \frac{1}{M}\sum_{j=1}^M Y_j$.} Note that from \eqref{eq:sample-variance}, we get 
\begin{equation}
\label{eq:sample-variance-decomposition}
\frac{M-1}{M}v-\sigma^{2}=\frac{1}{M} \sum_{j=1}^{M} \| X_{j}-\mu \|^{2}-\sigma^{2}-\| \bar{X}-\mu \|^{2}=\rv{\bar Y - \| \bar{X}-\mu \|^{2}};
\end{equation} 
thus, \rv{we provide tail bounds for $\frac{M-1}{M}v-\sigma^{2}$ by using the tail bounds for $\bar Y$ and $\| \bar{X}-\mu \|^{2}$. Indeed,} from \eqref{eq:barx-concentration}, we have
\begin{equation}
\label{eq:subGaussian-bound}
\rv{\p{\|\bar{X}-\mu\|^{2} \geqslant \frac{t}{2}} \leqslant (d+1) \exp\Big(\frac{-Mt}{32 \sigma_p^2}\Big),\quad\forall~t>0,\quad \forall~M\geq 1.}
\end{equation}
Moreover, according to Lemma~\ref{lem:subexponential-constants}, $Y_j=\norm{X_j-\mu}^2-\sigma^2$ is subexponential with $(\lambda,b)=(\rv{8\sqrt{2}}\sigma_p^2,8\sigma_p^2)$ 
for all $j=1,\ldots,M$. Thus, Theorem~\ref{thm:subexponential-bernstein} implies that \rv{$\bar Y=\frac{1}{M}\sum_{j=1}^MY_j$ is subexponential with parameters $(\lambda, b)=(\rv{8\sqrt{2}}p^2, 8p^2)$ 
with the tail bounds as in \eqref{eq:bernstein-probability-bound}, i.e.,
\begin{align}
\label{eq:subExponential-bound}
    \p{
    \bar Y\leqslant -t}\vee \p{
    \bar Y \geqslant t} \leq \exp \left(-\frac{M}{2} \min \left\{\frac{1}{2}\left(\frac{t}{8p^2}\right)^2,~\frac{t}{8p^2}\right\}\right),\quad\forall~t>0,\quad \forall~M\geq 1.
\end{align}
Thus, for any $t>0$ and $M\in\integers_+$, using \eqref{eq:subGaussian-bound} and \eqref{eq:subExponential-bound}, we obtain \eqref{eq:svar-lower-tail} as follows:
\begin{equation}
\label{eq:Bernstein-lower-tail}
\begin{aligned}
    \MoveEqLeft \p{\frac{M-1}{M}v-\sigma^{2}\leq -t}\\
    &\leq \p{\{\bar Y\leq -t/2\}\cup\{\norm{\bar X-\mu}^2\geq t/2\}}\\
    &\leq \p{\bar Y\leq -t/2}+\p{\norm{\bar X-\mu}^2\geq t/2}\\
    &\leq \exp \left(-M \min \left\{\left(\frac{t}{32 \sigma_p^2}\right)^2,~\frac{t}{32 \sigma_p^2}\right\}\right)+(d+1) \exp\Big(-M\cdot\frac{t}{32 \sigma_p^2}\Big).
\end{aligned}
\end{equation}
On the other hand, for any $t>0$ and $M\in\integers_+$, using \eqref{eq:subExponential-bound} and \eqref{eq:sample-variance-decomposition}, we also get \eqref{eq:svar-upper-tail}, i.e.,
\begin{equation}
\label{eq:Bernstein-upper-tail}
\begin{aligned}
    \p{\frac{M-1}{M}v-\sigma^{2}\geq t} \leq \p{\bar Y\geq t}\leq \exp \left(-M \min \left\{\left(\frac{t}{16 \sigma_p^2}\right)^2,~\frac{t}{16 \sigma_p^2}\right\}\right).
\end{aligned}
\end{equation}
Now, consider the event $\{\frac{M-1}{M}v\leq (1+c)\sigma^{2}\}$ for some $c\in[0,4]$. Note that
\begin{equation}
\label{eq:stop1-tail}
\begin{aligned}
    \MoveEqLeft[4]\p{\frac{M-1}{M}v\leq (1+c)\bar\sigma^{2}}\\ 
    &\geq 1 - \p{\frac{M-1}{M}v-\sigma^2\geq c\bar\sigma^{2}}\geq 1-\exp \left(-M \min \left\{\left(\frac{c\bar\sigma^2}{16 \sigma_p^2}\right)^2,~\frac{c\bar\sigma^2}{16 \sigma_p^2}\right\}\right).
\end{aligned}
\end{equation}
Thus, for any given 
$c>0$ and $\alpha\in(0,1)$, we obtain the desired bound in \eqref{eq:concentration-subGaussian}, i.e.,
{\small
\begin{equation*}
    \p{\left(1-\frac{1}{M}\right)v\leq (1+c)\bar\sigma^{2}}\geq 1-\alpha,\quad \forall~M\in\integers_+:\ M\geq \max\Big\{\Big(\frac{16}{c}\cdot\frac{\sigma_p^2}{\bar\sigma^2}\Big)^2,~\Big(\frac{16}{c}\cdot\frac{\sigma_p^2}{\bar\sigma^2}\Big)\Big\}\cdot\log\Big(\frac{1}{\alpha}\Big),
\end{equation*}}}%
{which completes the proof.}
\section{\sgdab{} with Gauss-Seidel Updates}
\label{sec:GS}
In this section, we discuss how our results extend to \sgdab{} with Gauss-Seidel updates, \saa{i.e., in~\texttt{\cref{algeq:RBSGDA_runs}} of \sgdab, displayed in \cref{alg:GDA-B}, instead of calling \rbsgda{} in \cref{alg:GDA}, we call \mgb{randomized block stochastic alternating GDA} (\texttt{RB-SAGDA}), which is stated below in \cref{alg:AGDA}.
In this Gauss-Seidel variant,} the analysis for the primal 
\saa{updates} follows the same 
\saa{arguments used for analyzing the Jacobi updates, and for the new variant of \sgdab{} that employs \texttt{RB-SAGDA}, we only need to analyze the dual updates, i.e., to update $y^{k+1}$ instead of using $\bx^k$ as in~\rbsgda, we now use $\bx^{k+1}$ --see \texttt{\cref{eq:AGDA_y}} of \texttt{RB-SAGDA}.}  
\begin{algorithm}[bht!]
\caption{$(\tilde\bx,\tilde y,\sa{\tilde G},v^x,v^y)=$\sa{\texttt{RB-SAGDA}}($\eta_x,\eta_y,M_x,M_y,N,K,\bx^0,y^0$)}
\label{alg:AGDA}
\small
\begin{algorithmic}[1]
\State \sa{$\tilde k\gets\cU[0,K-1]$}
\For{$k=0,\dots,\tilde k$}
\State \sa{$i_k\gets\cU[1,N]$ \Comment{$i_k$ is distributed uniformly on $\{1,\ldots,N\}$}}
\State \sa{$\bx^{k+1}\gets\bx^{k}$}
\State \sa{$x_{i_{k}}^{k+1}\gets \prox{\eta_x g_{i_k}}\Big(x_{i_{k}}^k-\eta_x s_{x_{i_k}}(\bx^k,y^k;\bom^k)\Big)$} \label{eq:AGDA_x}
\State\sa{$y^{k+1}\gets \prox{\eta_y h}\Big(y^{k}+\eta_y s_y(\bx^{k+1},y^k;\bzt^k)\Big)$} \label{eq:AGDA_y}
\State \saa{$\widetilde G_\bx^k \gets \tilde G_\bx(\bx^{k},y^{k};\bom^{k})$} \label{algeq:Gx}
\State \saa{$\widetilde G_y^k\gets \Big[\prox{\eta_y h}\Big(y^{k}+\eta_y s_y(\bx^{k},y^{k};\boldsymbol{\psi}^{k})\Big)-y^{k}\Big]/\eta_y$}  \Comment{$\boldsymbol{\psi}^{k}\sim\bzt$ --for distribution of $\bzt$, see Def.~\ref{def:s-gradmapping}} \label{algeq:Gy}
\EndFor
\State $(\tilde\bx,\tilde y)\gets(\bx^{\tilde k},y^{\tilde k}),\quad \saa{\tilde G\gets [\widetilde G_\bx^{\tilde k}, \widetilde G_y^{\tilde k}]}$,\quad \rv{$(v^x,v^y)\gets V(\bx^{\tilde k},y^{\tilde k};\bxi^{\tilde k})$} \label{algeq:RB-SAGDA_N}
\Comment{$V$ is defined in \eqref{eq:sample_variance}, and $\bxi^{\tilde k} = (\bom^{\tilde k}, \boldsymbol{\psi}^{\tilde k})$}
\State \Return \sa{$(\tilde\bx,\tilde y, \mg{\tilde G}, v^x, v^y)$}
\State \saa{\texttt{\# In practice \cref{algeq:Gx} and \cref{algeq:Gy} are not computed, this is done once in \cref{algeq:RB-SAGDA_N} \#}}
\end{algorithmic}
\end{algorithm}

\saa{To analyze \rbagda{}, due to its alternating updates, we need to define the filtration slightly different from \rbsgda.} 
\begin{defn}
\rv{For all $k\geq 0$, let $\cF^k=\sigma(\bx^k,y^k)$ denote the $\sigma$-algebra generated by the \rbagda{} iterate {$(\bx^{k},y^k)$} and let $\cF^{k+\frac{1}{2}}$ be a sub-$\sigma$-algebra of $\cF^{k+1}$ such that $\cF^{k+\frac{1}{2}}=\sigma(\{\bx^{k+1},y^k\})$, where $\bx^{k+1}$ follows the update rule in \texttt{\cref{eq:AGDA_x}} of the \texttt{RB-SAGDA} algorithm.} 
\end{defn}
\saa{In the next result, we bound the average squared norm of stochastic gradient maps over past iterations in expectation. This result is a slight modification of \cref{lemma:primal-bound-gda}.}
\begin{lemma}\label{lemma:primal-bound-agda}
For any $K\geq 1$ \mg{and $\eta_x>0$}, the \rbagda{}
iterate sequence $\{\bx^{k},y^{k}\}_{k\geq 0}$ satisfies
\sa{
\begin{equation}\label{eq:partial_sum_grad_delta_tmpp_jac-agda}
    \begin{aligned}
        &\Big(\frac{1}{2} - \eta_xL\kappa\Big)\sum_{k=0}^{K-1}\mathbf{E}\Big[\norm{\tilde G_{\bx}(\bx^{k},y^{k};\bom^k)}^2\Big]+\frac{1}{4} \sum_{k=0}^{K-1}\mathbf{E}\Big[\norm{\tilde G_y(\bx^{k},y^{k};\boldsymbol{\psi}^k)}^2\Big]\\
        &\leq \mathbf{E}\Big[\frac{N}{\eta_x}\Big(F(\bx^0)-F(\bx^K)\Big)
        +\frac{3}{2} \Big(L^2+\frac{1}{\eta_y^2}\Big)
        \sum_{k=0}^{K-1} \delta^{k}\Big]+K\Big(\frac{\sigma_x^2}{M_x}+ \frac{\sigma_y^2}{M_y}\Big),
    \end{aligned}
\end{equation}}%
where $\sa{\delta^k\triangleq}\|{y}^k_{*}-{y}^k\|^2$ and ${y}^k_{*}\triangleq{y}_{*}(\bx^k)$ \sa{for $k\geq 0$}.
\end{lemma}
\begin{proof}
    \saa{The proof is the same with the proof of \rv{Lemma}~\ref{lemma:primal-bound-gda} with the exception that we use $\tilde G_y(\bx^{k},y^{k};\boldsymbol{\psi}^k)$ for \rbagda{} rather than $\tilde G_y(\bx^{k},y^{k};\bzt^k)$ we used for \rbsgda.}
\end{proof}
\saa{For \rbsgda~dual iterate sequence $\{y^k\}$, the bound on $\delta^k\triangleq\norm{y_*^k-y^k}^2$ given in \rv{Lemma}~\ref{lemma:deltat_tmpp} does not hold for \rbagda, where ${y}^k_{*}\triangleq{y}_{*}(\bx^k)$ and $y_*(\cdot)$ is given in Definition~\ref{def-ystar}. Next, we provide a bound on $\delta^k$ for $\{y^k\}$ generated by \rbagda.}  
\begin{lemma}\label{lemma_delta_bd_sagda}
 Suppose \rv{$\eta_y \in(0, 1/\widetilde\mu)$ for some $0<\widetilde\mu\leq\mu$} and \saa{\mg{c}onsider $\{(\bx^k,y^k)\}_k$ generated by \rbagda, displayed in \cref{alg:AGDA}.} Then, \mg{$\delta^k=\|{y}^k_{*}-{y}^k\|^2$ admits \saa{the following bound,}} 
  {\small\begin{equation*}
     \begin{aligned}
         \mathbf{E}\Big[\delta^k\Big]&\leq \mathbf{E}\Big[ \delta^0 R_1^k+ \sum_{i=1}^k R_1^{k-i}\left(R_2 \|\tilde G_{\bx}(\bx^{i-1},y^{i-1};\bom^{i-1})\|^2/N  -R_3  \|\tilde G_y(\bx^{i},y^{i-1};\bzt^{i-1})\|^2 \right)\Big] + \sa{R_4}\frac{\sigma_y^2}{M_y}\sum_{i=0}^{k-1}R_1^{i},
    \end{aligned}
 \end{equation*}}%
 for all $k\geq 1$, where \rv{$R_1\triangleq (1+a)(1-\widetilde\mu\eta_y)<1$,  $R_2\triangleq \frac{(1-\eta_y \widetilde\mu)(2-\eta_y \widetilde\mu )}{\eta_y \widetilde\mu} \kappa^2 \eta_x^2$, $R_3\triangleq\eta_y^2 (1-\eta_y L)$ and $R_4\triangleq 2\eta_y^2$, and $a = \frac{1}{2}\frac{\widetilde\mu \eta_y}{1- \widetilde\mu \eta_y}$.}
 \end{lemma}
\begin{proof}
Fix $k\geq 1$. Since $y^{k}=\prox{\eta_y h}\Big(y^{k-1}+\eta_y s_y(\bx^{k},y^{k-1};\bzt^{k-1})\Big)$ and $\tilde G_y (\bx^{k},y^{k-1};\bzt^{k-1})=(y^k-y^{k-1})/\eta_y$, we have
\begin{equation}\label{delta_primal_AGDA}
    \begin{aligned}
        \delta^{k}&=\|y_*^{k}-y^{k}\|^2\\
        &=\|y_*^{k}-y^{k-1}-\eta_y\tilde G_y (\bx^{k},\by^{k-1};\bzt^{k-1})\|^2\\
        &=\|y_*^{k}-y^{k-1}\|^2 + 2\eta_y \langle \tilde G_y (\bx^{k},y^{k-1};\bzt^{k-1}),y^{k-1} - y_*^{k}\rangle +\eta_y^2 \|\tilde G_y (\bx^{k},y^{k-1};\bzt^{k-1})\|^2.
    \end{aligned}
\end{equation}
Since $y_*^{k}=\saa{\argmax_{y\in\cY}} f(\mg{\bx^{k}}
,y)\saa{-h(y)}$, \rv{Corollary}~\ref{cor:str-concavity} 
with \saa{$\bar\bx=\bx^k$, $\bar y=y^{k-1}$ and $\bar{y}^+=y^{k}$ implies that} 
\begin{equation}\label{lemma_use_AGDA}
    \begin{aligned}
        2\eta_y &\langle \tilde G_y (\bx^{k},y^{k-1};\bzt^{k-1}), y^{k-1}-y_*^{k}\rangle + \eta_y^2\|\tilde G_y (\bx^{k},y^{k-1};\bzt^{k-1})\|^2\\
        \leq &-\eta_y \mu \|y^{k-1}-y_*^{k}\|^2 - \eta_y^2(1-L\eta_y ) \|\tilde G_y (\bx^{k},y^{k-1};\bzt^{k-1})\|^2\\
        &+2\eta_y \fprod{s_y(\bx^k, y^{k-1};\bzt^{k-1})-\grad_y f(\bx^k, y^{k-1}),~y^k-\saa{y_*^k}}. 
    \end{aligned}
\end{equation}
By plugging \cref{lemma_use_AGDA} into \cref{delta_primal_AGDA}, \rv{using $\widetilde\mu\leq\mu$,} we obtain that
\begin{equation}\label{delta_ineq1_AGDA}
    \begin{aligned}
        \delta^{k}&\leq (1-\eta_y \rv{\widetilde\mu} )\|y_*^{k}-y^{k-1}\|^2 - \eta_y^2(1-L\eta_y )\|\tilde G_y (\bx^{k},y^{k-1};\bzt^{k-1})\|^2\\
        &\quad +2\eta_y \fprod{s_y(\bx^k, y^{k-1};\bzt^{k-1})-\grad_y f(\bx^k, y^{k-1}),y^k-\saa{y_*^k}}.
    \end{aligned}
\end{equation}
Next, $\|y_*^{k}-y^{k-1}\|^2$ can be bounded as follows:
\begin{equation}
\label{eq:y-bound-GS}
    \begin{aligned}
        \|y_*^{k}-y^{k-1}\|^2=\|y_*^{k}-y_*^{k-1}+y_*^{k-1}-y^{k-1}\|^2\leq (1+\frac{1}{\saa{a}})\kappa^2 \|\bx^{k}-\bx^{k-1}\|^2 + (1+\saa{a}) \delta^{k-1},
    \end{aligned}
\end{equation}
\saa{for any \saa{$a>0$},} where we use Young's inequality \saa{and the fact that $y_*(\cdot)$ is $\kappa$-Lipschi\mg{t}z from \cite[Lemma A.3]{nouiehed2019solving}.} Then, using 
\eqref{eq:y-bound-GS} and also the fact that $\norm{\tilde G_{x_{i_{k-1}}} (\bx^{k-1},\by^{k-1};\bom^{k-1})}=\norm{\bx^k-\bx^{k-1}}/\eta_x$ within \eqref{delta_ineq1_AGDA}, we get
\begin{equation*}
{\small
    \begin{aligned}
        \delta^{k}
        &\leq (1-\eta_y \rv{\widetilde\mu} )(1+a)\delta^{k-1}+ (1-\eta_y \rv{\widetilde\mu})(1+\frac{1}{a})\kappa^2 \eta_x^2 \|\tilde G_{x_{i_{k-1}}} (\bx^{k-1},\by^{k-1};\bom^{k-1})\|^2\\
        &\quad - \eta_y^2 (1-L\eta_y )\|\tilde G_y (\bx^{k},\by^{k-1};\bzt^{k-1})\|^2 +2\eta_y \fprod{s_y(\bx^k, y^{k-1};\bzt^{k-1})-\grad_y f(\bx^k, y^{k-1}),y^k-\saa{y_*^k}}. 
    \end{aligned}}%
\end{equation*}
Let $a = \frac{1}{2}\frac{\rv{\widetilde\mu}\eta_y}{1-\rv{\widetilde\mu} \eta_y}$ \sa{--we have $a> 0$ since $\eta_y< 1/\rv{\widetilde\mu}$ 
for all $k\geq 0$. This choice} implies that 
\begin{equation}
\label{delta_bd_inner}
{\small
    \begin{aligned}
        \delta^{k}&\leq (1-\frac{1}{2}\eta_y \rv{\widetilde\mu})\delta^{k-1} + \frac{(1-\eta_y \rv{\widetilde\mu})(2-\eta_y \rv{\widetilde\mu})}{\eta_y \rv{\widetilde\mu}} \kappa^2 \eta_x^2 \|\tilde G_{x_{i_{k-1}}} (\bx^{k-1},y^{k-1};\bom^{k-1})\|^2\\
        &\quad -\eta_y^2 (1-L\eta_y) \|\tilde G_y (\bx^{k},y^{k-1};\bzt^{k-1})\|^2+2\eta_y \fprod{s_y(\bx^k, y^{k-1};\bzt^{k-1})-\grad_y f(\bx^k, y^{k-1}),y^k-\saa{y_*^k}}. 
    \end{aligned}}%
\end{equation}
{Let $\hat{y}^k\triangleq\prox{\sa{\eta_y} h}\Big(y^{k-1}+\eta_y \grad_y f(\bx^{k},y^{k-1})\Big)$. We can bound the inner product in \eqref{delta_bd_inner} as follows:
\begin{equation}\label{eq:cond-bound-y-bias_gs}
    \begin{aligned}
    \MoveEqLeft\mathbf{E}\Big[
            \fprod{
                s_y(\bx^{k}, y^{k-1};\bzt^{k-1})-\grad_y f(\bx^{k}, y^{k-1}),~y^k-y_*(\bx^{k})}~\Big|~\xqs{\cF^{k-\frac{1}{2}}}
        \Big]
        \\
        =& \mathbf{E}\Big[
            \fprod{
                s_y(\bx^{k}, y^{k-1};\bzt^{k-1})-\grad_y f(\bx^{k}, y^{k-1}),~y^k-\hat{y}^k)}~\Big|~\xqs{\cF^{k-\frac{1}{2}}}
        \Big]
        \\
        \leq
        & \eta_y \mathbf{E}\Big[
            \|s_y(\bx^{k}, y^{k-1};\bzt^{k-1})-\grad_y f(\bx^{k}, y^{k-1})\|^2~\Big|~\xqs{\cF^{k-\frac{1}{2}}}
        \Big]
        \leq \eta_y\frac{\sigma_y^2}{M_y},
    \end{aligned}
\end{equation}}%
\sa{where in the first equality we used the fact that $\mathbf{E}\Big[s_y(\bx^{k}, y^{k-1};\bzt^{k-1})~\Big|~\cF^{k-\saa{\frac{1}{2}}}\Big]=\grad_y f(\bx^{k},y^{k-1})$ and $\hat y^k-y_*(\bx^{k})$ is \xqs{$\cF^{k-\frac{1}{2}}$}-measurable; in the first inequality, we used Cauchy-Schwarz inequality and the nonexpansivity of the proximal map; and the final inequality follows from \cref{assumption:noise}.
Therefore, \cref{delta_bd_inner} and \cref{eq:cond-bound-y-bias_gs} together with \cref{lemeq:cond-exp-Gik} imply that}
\begin{equation}\label{eq_delta_one_tmpp_exp2}
    \begin{aligned}
        \mathbf{E}\Big[ \delta^k \Big]\leq  \mathbf{E} \Big[R_1\delta^{k-1}+ R_2 \norm{\tilde G_{\sa{\bx}} (\bx^{k-1},y^{k-1};\bom^{k-1})}^2/N-R_3\norm{\tilde G_y (\bx^{k},y^{k-1};\bzt^{k-1})}^2\Big] + R_4{\frac{\sigma_y^2}{M_y}.}
    \end{aligned}
\end{equation}
\sa{Thus, \eqref{eq_delta_one_tmpp_exp2} implies the desired result.} \xqs{Indeed, $R_1 = C_1 < 1$, where $C_1$ is defined in Lemma~\ref{lemma:deltat_tmpp}.}
\end{proof}
 \begin{lemma}
 Let $\eta_x>0$ and $\eta_y\in (0, 1/\rv{\widetilde\mu})$ for some $0<\widetilde\mu\leq\mu$. Then, for any \saa{$K>1$,}
\begin{equation}\label{eq:line-search-sketch_sAGDA}
    \begin{aligned}
        \MoveEqLeft\mathbf{E}\bigg[\Big(2 - 4\eta_x\kappa L\Big)\|\tilde G_{\bx}(\bx^{K-1},y^{K-1};\bom^K)\|^2+\sa{\sum_{k=0}^{K-1}\|\tilde G_y(\bx^k,y^{k};\xz{\boldsymbol{\psi}^k)}\|^2}\\
        &\quad+\sum_{k=0}^{K-2} 6 R_3(L^2+\frac{1}{\eta_y^2})\sum_{i=k}^{K-2}R_1^{i-k}\|\tilde G_y(\bx^{k+1},y^{k};\bzt^k)\|^2\\
        &\quad+4\sum_{k=0}^{K-\sa{2}}\Big(\frac{1}{2} - \eta_x\kappa L-{ \frac{3}{2N}R_2\Big(L^2+\frac{1}{\eta_y^2}\Big) \sum_{i=k}^{K-2} R_1^{i-k}}\Big)\|\tilde G_{\bx}(\bx^{k},y^{k};\bom^k)\|^2\bigg]\\
        &\leq 4\mathbf{E}\Big[\frac{N}{\eta_x}\Big(F(\bx^0)-F(\bx^K)\Big)+\delta^0{\sum_{k=0}^{K-1} \frac{3}{2}\Big(L^2+\frac{1} {\eta_y^2}\Big)  R_1^k}\Big] \\
        &\quad+4K \Big(\frac{\sigma_x^2}{M_x} +\Big(1+\frac{6\eta_y}{\rv{\widetilde\mu}}\Big(L^2+\frac{1}{\eta_y^2}\Big)\Big)\frac{\sigma_y^2}{M_y}  \Big).
        \end{aligned}
        \end{equation}
        \end{lemma}
\begin{proof}
\saa{The proof of this result follows from Lemma~\ref{lemma:primal-bound-agda}, Lemma~\ref{lemma_delta_bd_sagda} and from the same arguments used in the proof of Lemma~\ref{linesearch_scratch_stoc}.}
\end{proof}
\begin{lemma}\label{lem:linesearch_sAGDA}
   \sa{Given some $\gamma\in(0,1)$, \rv{suppose $\eta_y=1/{\widetilde L}$ and $\eta_x = N\rho\cdot\eta_y^3\widetilde\mu^2$ for some $\widetilde L,\widetilde\mu>0$ such that $0<\widetilde\mu\leq\mu\leq L\leq \widetilde L$,} and $\rho = \frac{\sqrt{1+\frac{12}{N}}-1}{24}\in (0,1)$. Let $\{\bx^k,y^k\}_{k\geq 0}$ be generated by \rbagda. Then,
    \begin{equation}\label{eq:line-search-condition_sagda}
    \begin{aligned}
        \frac{1}{K}\sum_{k=0}^{K-1}\mathbf{E}\Big[\|\tilde G(\bx^{k},{y}^{k};\bxi^k)\|^2\Big]
        \leq \frac{4N}{\eta_x K}\mathbf{E}\Big[F(\bx^0)-F(\bx^K)+\rv{6\rho\delta^0\widetilde\mu}
          \Big]+\frac{4\sigma_x^2}{M_x}+ \Big(1+ \frac{ 12}{\eta_y \rv{\widetilde\mu} }\Big)\frac{4\sigma_y^2}{M_y}
    \end{aligned}
\end{equation}
     holds for all $K\geq 1$, where $\tilde G(\bx^k,y^k;\bxi^k)\triangleq [\tilde G_{\bx}(\bx^k,y^k;\bom^k)^\top,~\tilde G_y(\bx^k,y^k;\xqs{\boldsymbol{\psi}^k})^\top]^\top$ for all 
     $k\geq 0$.}
\end{lemma}
\begin{proof}
\saa{The proof is similar to that of Lemma~\ref{lem:linesearch}}. 
\xz{Since $R_1 = C_1$ and $R_2 = \frac{ (1-\eta_y\widetilde\mu)(2-\eta_y\widetilde\mu)}{2}C_2\leq C_2$, where $C_1$ and $C_2$ are defined in 
Lemma~\ref{lemma:deltat_tmpp},
we obtain that}
\begin{equation}\label{eq:long-LS-leftx2-AGDA}
    \begin{aligned}
        &\frac{1}{2} - \eta_x \kappa L-R_2\frac{3}{2N}\Big(L^2+\frac{1}{\eta_y^2}\Big) \sum_{i=k}^{K-2}\sa{R_1^{i-k}}\\
        & =
        \frac{1}{2} - \eta_x \kappa L-\frac{ (1-\eta_y\widetilde\mu)(2-\eta_y\widetilde\mu)}{2}C_2\frac{3}{2N}\Big(L^2+\frac{1}{\eta_y^2}\Big) \sum_{i=k}^{K-2}\sa{C_1^{i-k}}\\
        & \geq 
        \frac{1}{2} - \eta_x \kappa L-C_2\frac{3}{2N}\Big(L^2+\frac{1}{\eta_y^2}\Big) \sum_{i=k}^{K-2}\sa{C_1^{i-k}}
        \geq \frac{1}{4},
    \end{aligned}
\end{equation}
\xz{where the last inequality follows from 
\eqref{eq:long-LS-leftx2}. 
\saa{The definition of $\rho$ implies that $N\rho=\frac{1}{4}-12N\rho^2\leq\frac{1}{4}$; hence, using $\eta_y\leq\frac{1}{L}$ and \rv{$\widetilde\mu\leq\mu$}, we get $\eta_x=N\rho\cdot\eta_y^3\widetilde\mu^2<\frac{1}{4L\kappa^2}\leq\frac{1}{4L\kappa}$ since $\kappa\geq 1$.}
Moreover, since $\eta_x\leq\frac{1}{4L\kappa}$, all the $\tilde G_\bx$ terms on the left-hand side of \eqref{eq:line-search-sketch_sAGDA} can be lower bounded by $\sum_{k=0}^{K-1}\|\tilde G(\bx^{k},{y}^{k};\bxi^k)\|^2$. Finally, $\eta_y\leq \frac{1}{L}$ implies that $R_3\geq 0$; thus, $\mathbf{E}[\sum_{k=0}^{K-1}\|\tilde G(\bx^{k},{y}^{k};\bxi^k)\|^2]$ is a lower bound for the left-hand side of \eqref{eq:line-search-sketch_sAGDA}.}
Next, we provide an upper bound for the right-hand side of \eqref{eq:line-search-sketch_sAGDA}. Note that since $L\leq \frac{1}{\eta_y}$, 
\begin{equation}
    1+\frac{6\eta_y}{\mu}(L^2 + \frac{1}{\eta_y^2})\leq 1+ \frac{12}{\eta_y \mu}.
\end{equation}
The desired inequality follows from combining all these bounds with \eqref{eq:long-LS-right}, 
and finally dividing both sides by \saa{$K$}. 
\end{proof}
\saa{The results given in Lemma~\ref{lem:small-eta}, Lemma~\ref{lem:stopping_probability}, and Theorem~\ref{thm:p-bound} continue to hold for the Gauss-Seidel variant of \sgdab{} with $M_x\geq \lceil \frac{128}{\epsilon^2}\widetilde\sigma_x^2+1\rceil$, $M_y\geq \lceil (1+\frac{12}{\eta_y\widetilde\mu})\frac{128}{\epsilon^2}\widetilde\sigma_y^2+1\rceil$ and $K=\lceil\frac{64N}{\epsilon^2\eta_x}\left(\rv{F_0-\bar{F}+6\rho\delta\widetilde\mu}\right)\rceil$. Moreover, the results of Corollary~\ref{cor:sample-complexity-bounded} and Theorem~\ref{ncsc_thm} continue to hold for all three $(M_x,M_y)$ choices given in Corollary~\ref{cor:sample-complexity-bounded} when the coefficient $1+\frac{6}{\rv{\widetilde\mu} \eta_y} \frac{2-\rv{\widetilde\mu} \eta_y}{1-\rv{\widetilde\mu} \eta_y}$ appearing in all $M_y$ choices is replaced by $1+\frac{12}{\eta_y\widetilde\mu}$. The results of Section~\ref{sec:unknown-bounds} continue to hold for the Gauss-Seidel update setting as well.}

\section{Related Results for the Deterministic Setting}\label{sec-appendix-deterministic}
\renewcommand{\arraystretch}{1.4}
    \begin{table}[h!]
        \centering
        \resizebox{\textwidth}{!}{
        \begin{tabular}{|l"l|c|c|c|c|l|l|}
            \hline
            \textbf{Work} & \textbf{Ref.} & \textbf{Agnostic} & $g$ & $h$ & \textbf{BCU} & \textbf{Complexity} $(\mu>0)$ & \textbf{Complexity} $(\mu=0)$ \\
            \thickhline
            \gda{} & ~\cite{lin2020gradient} & \XSolidBrush & \XSolidBrush & $\mathds{1}_{Y}$ & \XSolidBrush & M1: $\cO\left(\rv{(\kappa^2 L B_0+\kappa L^2 \cD_y^2)}  \epsilon^{-2}\right)$ & M3: $\cO(\rv{L^3\ell^2\cD_y^2\hat B_0}\epsilon^{-6})$\\
            \hline
            \agda{} & ~\cite{boct2020alternating} & \XSolidBrush & \CheckmarkBold & \CheckmarkBold & \XSolidBrush & M1: $\cO\left(\rv{(\kappa^2 L B_0+\kappa L^2 \cD_y^2)}  \epsilon^{-2}\right)$ & M3: $\cO(\rv{L^3\ell^2\cD_y^2 B_0}\epsilon^{-6})$\\
            \hline
            \texttt{MDA} & ~\cite{huang2021efficient} & \XSolidBrush & \CheckmarkBold & \CheckmarkBold & \XSolidBrush & M1: $\cO(\kappa^{3}\rv{B_0}
            {\epsilon^{-2}})$ & \qquad\qquad\XSolidBrush \\
            \hline
            \smagda{} & ~\cite{yang2022faster} & \XSolidBrush & \XSolidBrush & \XSolidBrush & \XSolidBrush & M2: $\cO(\kappa L\rv{B_0}\epsilon^{-2})$ & \qquad\qquad\XSolidBrush \\
            \hline
            \texttt{AltGDAm} & ~\cite{chen2022accelerated} &\XSolidBrush & \CheckmarkBold & \CheckmarkBold & \XSolidBrush & M1: $\cO\left(\rv{(\kappa^{\frac{11}{6}} L B_0+L^2\cD_y^2)}\epsilon^{-2}\right)$ & \qquad\qquad\XSolidBrush\\
            \hline
            \texttt{AGP} & ~\cite{xu2023unified} & \XSolidBrush & \sa{$\mathds{1}_{X}$} & \sa{$\mathds{1}_{Y}$} & \XSolidBrush & M2: $\cO(\rv{\kappa^5L(B'_0+\kappa L\cD_y^2)}\epsilon^{-2})$ & M2: $\cO(L^4 \cD_y^4\epsilon^{-4})$
            \\ \hline
            \texttt{Sm-GDA} &~\cite{zhang2020single} & \XSolidBrush & $\mathds{1}_{X}$ & $\mathds{1}_{Y}$ & \CheckmarkBold & \qquad\qquad\XSolidBrush & M2: 
            {$\cO(L^4\cD_y^4\epsilon^{-4})$}\\ \hline
            \texttt{HiBSA} &~\cite{lu2020hybrid} & \XSolidBrush & \CheckmarkBold & \CheckmarkBold & \CheckmarkBold & M2: 
            {$\cO(\rv{\kappa^5L(B'_0+ \kappa^3L\cD_y^2)}\epsilon^{-2})$} & \qquad\qquad\XSolidBrush\\
            \thickhline
            \neada{} & ~\cite{junchinest}  & \rv{$L,\mu$} & \XSolidBrush & $\mathds{1}_{Y}$ & \XSolidBrush & M2: $
            {{\cO}(\rv{\sqrt{\kappa}(\kappa^2L^2+B_0)^2}\epsilon^{-2}\log(\epsilon^{-1}))}$ & \qquad\qquad\XSolidBrush\\
            \hline
            \tiada{} & ~\cite{li2022tiada}  & {$L,\mu$} & \XSolidBrush & $\mathds{1}_{Y}$ & \XSolidBrush &
            M2: $
            {\cO(\kappa^{10}\epsilon^{-2})}$ & \qquad\qquad\XSolidBrush\\
            \hline
            \texttt{GDA-B} &  \textbf{Ours} & {$L$} & \CheckmarkBold & \CheckmarkBold & \CheckmarkBold & M2: $\cO({\kappa^2 LB_0\epsilon^{-2}})$ & M2: $\cO(L^3\sa{\cD_y^2}B_0\epsilon^{-4})$\\
            \hline
            \texttt{GDA-B} &  \textbf{Ours} & {$L,\mu$} & \CheckmarkBold & \CheckmarkBold & \CheckmarkBold & M2: {$\cO({\hat{\kappa}^2 \hat{L}B_0\epsilon^{-2}})$} & {M2: $\cO(\hat{L}^3\sa{\cD_y^2}B_0\epsilon^{-4})$}\\
            \hline
        \end{tabular}%
        }
        \caption{\sa{Comparison of \texttt{GDA-B}~(this paper), \neada{}~{\citep{junchinest}} and \tiada{}~{\citep{li2022tiada}} that are agnostic to $L$ with \textit{single-loop} methods that require $L$ for solving \textit{deterministic} WCSC and WCMC minimax problems. The derivation of complexity results for the methods listed here is provided in \cref{sec:complexity_comparison} of the online supplementary document. In the table $\cD_y$ denotes the diameter of the dual domain, \rv{$B_0=F(\bx^0)-F^*$ denotes the initial suboptimality with respect to the primal function $F(\cdot)$ where $F^*=\min_{\bx\in\cX}F(\bx)$, and $B'_0=\cL(\bx^0,y^0)-\inf_{\bx,y}\cL(\bx,y)$. Finally, $\hat B_0=F_\lambda(\bx^0)-F^*$ for $\lambda=\frac{1}{2L}$, where $F_\lambda(\bx^0)$ denotes the Moreau envelope of $F(\cdot)$.}
        }}
        \label{table_deter}
    \end{table}

 \textbf{Existing work for the {deterministic setting: WCSC and WCMC minimax problems.}} \sa{For the case \saa{the global Lipschitz constant $L$} {and the concavity  modulus $\mu$ are} \textit{known}, there are many theoretically efficient methods in the literature, e.g., see~\citep{lin2020near,ostrovskii2021efficient,kong2021accelerated} and references therein. When $L$ {and $\mu$ are} known, both 
    ~\cite{lin2020near} and 
    ~\cite{ostrovskii2021efficient} show an iteration complexity of $\tilde\cO(L\kappa^{1/2}\epsilon^{-2})$ for finding an $\epsilon$-stationary point in terms of metric \textbf{(M2)} 
    for deterministic WCSC problems --both algorithms 
    have \textit{three nested loops} and impose restrictions on the closed convex functions, i.e., $g$ and $h$ can only be the indicator functions of some convex sets. Furthermore, both \cite{lin2020near} and \cite{ostrovskii2021efficient} establish an iteration complexity of $\tilde\cO(\epsilon^{-2.5})$ for deterministic WCMC problems, which is later improved to $\cO(\epsilon^{-2.5})$ by~\cite{kong2021accelerated}. Due to huge amount of recent work in this area, to limit the discussion in \cref{table_deter}, we only list some closely related \textit{single-loop} methods that require $L$ for solving \textit{deterministic} WCSC and WCMC minimax problems.}
    \sa{Although there are many methods for the deterministic setting, to the best of our knowledge, there are \textit{very few} FOPD algorithms that are \textit{agnostic} to $L$ and $\mu$, 
     and that use adaptive step sizes for solving minimax problems in the nonconvex setting~{\citep{dvurechensky2017gradient,junchinest,li2022tiada,lee2021fast,pethick2023escaping}.}}
    \sa{For a special case of \eqref{eq:main-problem} with $h(\cdot)=0$ and \textit{bilinear} coupling $f(x,y)=s(x)-r(y)+\fprod{Ax,y}$ such that $s,r$ are smooth and $r$ is strongly convex, 
    \cite{dvurechensky2017gradient} proposes a 
    \saa{triple-loop method where in the outer loop, primal iterate $x^k$ is updated using the step sizes satisfying a backtracking condition based on the primal function, one searches for an admissible primal step size using the middle loop which repeats until the backtracking condition holds, and for each backtracking iteration, the inner maximization problem $\max_{y\in\cY} f(x^k,y)$ is solved inexactly using the inner loop;} 
    {the gradient complexity of the proposed method is \xz{$\tilde\cO(L\kappa^{1.5}\epsilon^{-2})$} for computing an $\epsilon$-stationary point in terms of metric \textbf{(M1)}.}
    For the more general case of WCSC minimax problems when the coupling function $f$ is not bilinear, Nested Adaptive~(\neada)~{\citep{junchinest}} and Time-Scale Adaptive~(\tiada)~{\citep{li2022tiada}} algorithms are recently proposed, both of which are based on \saa{variants of} AdaGrad
    step sizes~{\citep{duchi2011adaptive}}; \saa{hence,} \sa{these methods do not use backtracking.} 
    \neada, similar to \citep{dvurechensky2017gradient},} is a double-loop method for solving \eqref{eq:main-problem} with $g(\cdot)=0$ and $h(\cdot)=\mathbf{1}_{Y}(\cdot)$, i.e., the indicator function of a closed convex set $Y$; in the inner loop, a strongly concave problem is \saa{inexactly solved using the} AdaGrad method, of which solution is then used to compute an inexact gradient for the primal objective function; finally, 
    \sa{
    a primal step is taken along the inexact gradient using AdaGrad stepsize}. 
    \saa{\neada{}} 
    can compute $(x_\epsilon,y_\epsilon)$ such that \sa{$\norm{\grad_x f(x_\epsilon,y_\epsilon)}\leq \epsilon$ and $\norm{y_\epsilon-y_*(x_\epsilon)}\leq \epsilon$ within $\cO(
    {\kappa^{4.5}}L^4\epsilon^{-2}\log(1/\epsilon))$ 
    gradient complexity}, where $y_*(\cdot)=\argmax_{y\in Y} f(\cdot,y)$. On the other hand, \tiada{}~by \cite{li2022tiada} is a single-loop scheme that can be considered as an extension of AdaGrad to solve minimax problems in a similar spirit to two-time scale GDA algorithms. \tiada{} can compute $(x_\epsilon,y_\epsilon)$ such that \sa{$\norm{\grad f(x_\epsilon,y_\epsilon)}\leq \epsilon$ within $\cO(
    {\kappa^{10}}\epsilon^{-2})$ 
    gradient calls. 
    The complexity for \neada{} and \tiada{} in the deterministic case are provided in Table~\ref{table_deter}; see~\cref{sec:complexity_comparison} for the derivation of $\cO(1)$ constants.} The analyses of these algorithms are given for \textit{smooth} \sa{WCSC} problems with $g(\cdot)=0$ and $h(\cdot)=\mathbf{1}_{Y}(\cdot)$.
    
    \mg{
    \saa{Finally,} for smooth WCWC problems, extragradient~(EG) methods \sa{with backtracking are proposed} in~\citep{lee2021fast,pethick2023escaping}; 
    but, they make additional assumptions} \sa{(\textit{stronger} than smoothness of $f$)}
    that in general \sa{do} not hold for WCSC and WCMC problems \sa{we consider as in} 
    \eqref{eq:main-problem}. To make the discussion \sa{complete}, 
    \sa{next we provide a brief discussion on the state-of-the-art results for FOPD methods to solve WCWC minimax problems.}

    {\textbf{Existing work for the deterministic setting: {smooth WCWC minimax problems}.}} Consider the {smooth} WCWC case corresponding to the formulation in~\eqref{eq:main-problem} with $g=h=0$ and $f$ is $L$-smooth, {i.e., $\nabla f$ is Lipschitz with constant $L$}, and let $\cM(x,y)\triangleq[\grad_x f(x,y)^\top ~-\grad_y f(x,y)^\top]^\top$. For this setting, 
    {\cite{lee2021fast}} incorporated backtracking 
    within a variant of EG method; under the additional assumption that $\cM$ is a $\rho$-comonotone operator for some $\rho<0$, the authors established $\norm{\grad f(x_k,y_k)}^2=\cO(1/k^2)$ --this result extends the $\cO(1/k^2)$ rate shown in \citep{yoon2021accelerated} for the 
    \sa{merely convex-merely concave minimax problems}, i.e., {the authors extend the} monotone operator case {(which corresponds to $\rho=0$)} to negative-comonotone operators\footnote{\sa{Given a coupling function $f:\cX\times\cY\to\reals$} that is WCWC and twice continuously differentiable, \sa{the corresponding map} $\cM$ is negative-comonotone if $\grad^2_{xy} f$ dominates any negative curvature in block Hessians $\grad^2_{xx} f$ and $-\grad^2_{yy} f$ --see~\citep{grimmer2022landscape}.} (which corresponds to $\rho<0$). In a separate body of work, the WCWC case (with $g=h=0$ and $f$ is $L$-smooth) is analyzed under the additional assumption that Minty VI, abbreviated as MVI, corresponding to the operator $\cM$ has a solution\footnote{\sa{MVIs corresponding to quasiconvex-concave and starconvex-concave problems have a solution; however, since there are 
    {smooth} functions that are neither quasiconvex nor star-convex, there are WCWC minimax problems corresponding to a smooth coupling 
    $f$ for which the corresponding MVI does not necessarily have a solution.}}; $\cO(1/k)$ rate has been shown~{\citep{dang2015convergence,song2020optimistic,diakonikolas2021efficient}} for $\norm{\grad f}^2$.
    {\sa{The conditions} ``$\cM$ being negative-comonotone" and ``MVI corresponding to $\cM$ having a solution" are not equivalent, and neither one implies the other.}
    \sa{Later,~\cite{diakonikolas2021efficient} 
    introduced a 
    weaker condition compared to the other two: ``the \textit{weak} MVI corresponding to $\cM$ has a nonempty solution set," i.e., 
    this condition holds whenever $\cM$ is negative comonotone~{\citep{diakonikolas2021efficient,lee2021fast}} or 
    whenever a solution to the MVI exists~{\citep{dang2015convergence,song2020optimistic}}.}

    \sa{
    Recently,
    ~\cite{pethick2023escaping} have 
    proposed an extragradient-type algorithm with adaptive stepsizes (using backtracking 
    on the unknown Lipschitz constant) for \eqref{eq:main-problem} assuming that $g,h$ are closed convex and $f$ is $L$-smooth such that the corresponding weak MVI has a solution; and they show that the limit points of the iterate sequence belongs to the zero-set of the operator defining the weak MVI \textit{without} providing any complexity guarantee for the adaptive algorithm that uses backtracking.\footnote{In contrast, we 
    propose backtracking methods for WCMC and WCSC problems with complexity guarantees.} 
    It is crucial to emphasize that
    since $\cM$ is non-monotone, assuming $\grad f$ is $L$-Lipschitz does not imply that the corresponding weak MVI has a solution. Therefore, the existing methods developed for smooth WCWC minimax problems do not apply to the WCSC and WCMC settings we consider since the corresponding weak MVIs may not have a solution.}  
    In other words, all of these conditions, i.e., $\cM$ being negative comonotone, \sa{existence of a solution to MVI or even to its weaker version (weak MVI) corresponding to $\cM$ are both} stronger than the standard assumption of $\grad f$ being Lipschitz, \sa{and do not necessarily hold even if we further assume $f$ is (strongly) concave in $y$.} 
    Indeed, negative comonotonicity of $\cM$, leading to $\cO(\epsilon^{-1})$ complexity~{\citep{lee2021fast}}, is a strong {assumption}; 
    because the standard lower complexity bounds established for
FO methods to compute an $\epsilon$-stationary point for general non-convex smooth optimization problems imply that at least
$\Omega(\epsilon^{-2})$ gradient evaluations are required~{\citep{carmon2020lower}}.\footnote{This complexity bound also applies to computing $(x_\epsilon,y_\epsilon)$ such that $\norm{\grad f(x_\epsilon,y_\epsilon)}\leq \epsilon$ when $f$ in~\eqref{eq:main-problem} is WCWC and $g=h=0$. \sa{Furthermore, let $\Phi(\cdot)\triangleq \max_{y\in\cY}f(\cdot,y)$ denote the primal function corresponding to the smooth WCSC problem $(P_s):~\min_{x\in\cX}\max_{y\in\cY} f(x,y)$; for deterministic FOPD methods to solve $(P_s)$, 
~\cite{zhang2021complexity} (see also~\citep{li2021complexity}) have shown that a lower bound for finding an $\epsilon$-stationary point of $\Phi$, i.e., $x_\epsilon\in\cX$ such that $\norm{\grad \Phi(x_\epsilon)}\leq \epsilon$, is $\Omega\Big(L\sqrt{\kappa}/\epsilon^2\Big)$.}}

\section{Derivation of Computational Complexities in \mg{T}ables \ref{table_stoc} and  \ref{table_deter}}
\label{sec:complexity_comparison}
\subsection{Derivation of complexity in \citep{junchinest}}
The nested adaptive (\neada) algorithm by \cite{junchinest} \sa{is a double-loop method, where the inner loop is to inexactly maximize the coupling function for a given primal iterate 
and in the outer loop, the primal variable is updated using a (stochastic) gradient computed at the current primal iterate and \mg{an} inexact dual maximizer.} For 
\saa{WCSC} minimax problems, \neada{} with adaptive stepsizes can achieve the near-optimal $\Tilde{\cO}(\epsilon^{-2})$ and $\Tilde{\cO}(\epsilon^{-4})$ gradient complexities respectively in the deterministic and stochastic settings. \saa{However, 
the dependence of $\cO(1)$ constant on $\kappa$ and $L$ is not explicitly stated in~\citep{junchinest}; and in the rest of this section, we compute the important terms 
within the $\cO(1)$ constant.}

\saa{According to the proof of \cite[Theorem 3.2]{junchinest}, 
\neada{} in the deterministic setting 
has a theoretical guarantee that
\begin{equation*}
    \norm{\grad_x f(x^K,y^K)}^2+\norm{G_y(x^K,y^K)}^2\leq \epsilon^2
\end{equation*}
for $K\leq \Tilde{\cO}\Big( \Big((A+\cE)^2+\sqrt{v_0}(A+\cE) 
\Big)\epsilon^{-2}\Big)$} with $A+\cE = \tilde{\cO}\Big(\frac{\Phi(x_0)- \min_x \Phi(x)}{\eta} + 2\sigma + \kappa L \eta + \frac{\kappa^2 (L+1)^2}{\sqrt{v_0}} \Big)$ for some $v_0,\eta>0$ \saa{--see~\cite[Theorem B.1]{junchinest} in the appendix of the paper for the details; moreover, for each outer iteration $k=1,\ldots, K$, subroutine $\cA$ to inexact solve $\max_{y\in Y} f(x^k,y)$ requires $\cO(\frac{1}{a_2}\log(k))$ gradient calls for some $a_2\in (0,1)$ constant specific to the subroutine $\cA$ that depend on structural properties of $f(x,\cdot)$ uniformly in $x\in\cX$ such as $L$ and $\mu$. Since $\sigma = 0$, if one uses accelerated \xz{backtracking method for solving strongly convex problems}~\citep{calatroni2019backtracking,rebegoldi2022scaled} 
as the subroutine $\cA$, then $a_2 = \frac{1}{\sqrt{\kappa}}$. 
Therefore, the total gradient complexity for \neada{} in the deterministic case} is at least $\saa{\Tilde{\cO}\Big( \Big((A+\cE)^2+\sqrt{v_0}(A+\cE) 
\Big)\epsilon^{-2}/a_2\Big)}=\tilde \cO (\kappa^{4.5} L^4 \epsilon^{-2})$.

\saa{The stochastic case $\sigma>0$ is considered in \cite[Theorem 3.3]{junchinest}, and according to this result, stochastic \neada{} can compute
\begin{equation*}
    \E{\norm{\grad_x f(x^K,y^K)}^2+\norm{y^K-y_*(x^K)}^2}\leq \epsilon^2
\end{equation*}
for some} $K = \tilde \cO \Big(\Big(A+\cE)^2 + \sqrt{v_0}(A+\cE)(1+\sigma)+b_3 + (b_3+\saa{\delta^0})e^{2b_1}\Big)\epsilon^{-2} \Big)$ with $A= \tilde \cO\Big( \frac{\Phi(x_0)-\min_x \Phi(x)}{\eta} + (\frac{2\sigma}{\sqrt{M}}+\kappa L \eta)(1+b_1)\Big)$, $\cE = \frac{L^2}{2\sqrt{v_0}}[b_3(1+\log K) + b_3 e^{2b_1} + \saa{\delta^0} e^{2b_1}]$, \saa{$\delta^0=\norm{y^0-y_*(x^0)}^2$} and $b_3 = 2\kappa^2 \eta^2 b_1 + b_2$, where $v_0,\eta ,b_1, b_2> 0$ are some constants, \saa{and $b_1,b_2$ can depend on problem parameters such as $\mu,L$ and $\kappa$.} 

\saa{Indeed, it is assumed that there exists a subroutine $\cA$ such that for any given smooth strongly concave function $h$, after $T=t\log^p(t)+1$ iterations, it guarantees that $\E{\norm{y^T-y^*}^2}\leq \frac{b_1\norm{y^0-y^*}^2+b_2}{T}$ where $y^*=\argmax h(y)$ and $p\in\integers_{+}$. In \cite[Remarks 5 and 8]{junchinest}, they cited some parameter-agnostic algorithms, i.e.,~\citep{NIPS2017_6aed000a, lacostejulien2012simpler, rakhlin2012making}, that can be used as $\cA$ for solving the inner max problems with a sub-linear rate. All these three methods need some extra assumptions, e.g., the stochastic gradients are bounded, i.e., $\max\{\norm{\tilde \grad f(x,y)}:~x\in\cX,~y\in Y\}\leq G$ for some $G>0$,} and in \citep{lacostejulien2012simpler, rakhlin2012making}, strongly concavity modulus $\mu$ is \saa{assumed to be} known; \saa{under these assumptions, these methods satisfy $b_1 = 0$ and $b_2 = \frac{G^2}{\mu^2}$.} FreeRexMomentum~{\citep{NIPS2017_6aed000a}} is the only algorithm which is agnostic to both $\mu$ and $G$, and \saa{it satisfies the conditions of subroutine $\cA$ with $b_1 = 0$, $b_2 = \frac{G^2}{\mu^2}$ and $p=2$ --see~\cite[Footnote 4]{junchinest}.} \rv{Next, setting $b_1=0$ and $b_2=\frac{G^2}{\mu^2}$, we get $b_3=b_2=\frac{G^2}{\mu^2}$; therefore, $\cE=\cO(L^2b_3)=\cO(G^2\kappa^2)$ and $\cA=\cO(B_0+\kappa L)$. Moreover, since $K=\tilde\cO((\cE+\cA)^2\epsilon^{-2})$, we get $K=\tilde\cO((G^2\kappa^2+L\kappa+B_0)^2)$}. \saa{Since the batch size $M = \cO(\epsilon^{-2})$, the oracle complexity for $\tilde\grad_x f$ is $KM = \tilde \cO(G^4\kappa^4  \epsilon^{-4})$} and the oracle complexity for $\tilde\grad_y f$ is $K^2 \log^{\saa{2}} K + K = \tilde \cO(G^8\kappa^8  \epsilon^{-4})$. Hence, the total oracle complexity for stochastic NeAda is at least $\tilde \cO(G^8\kappa^8  \epsilon^{-4})$.

\subsection{Derivation of complexity in \cite{li2022tiada}}
To get the near-optimal complexities, their theoretical analysis 
\saa{requires} $\alpha>\frac{1}{2}>\beta$. 
\saa{To come up with explicit $\kappa$ and $L$ dependence of their $\cO(1)$ constant, we will use this as well.} 

\saa{The complexity for the deterministic case is discussed in \cite[Theorem C.1]{li2022tiada}, and they show that 
\begin{equation}
\label{eq:x-complexity}
\sum_{k=0}^{K-1}\norm{\grad_x f(x^k,y^k)}^2\leq \max\{5C_1,2C_2\}\triangleq C,
\end{equation}
for some constants $C_1$ and $C_2$, which we will analyze next.}  
\xz{Indeed, both $C_1$ and $C_2$ are explicitly defined in the statement of Theorem C.1 depending on some other positive constants, i.e., $c_1,c_2,c_3,c_4>0$ and $c_5>0$, which are defined as follows\footnote{\saa{In~\citep{li2022tiada}, there are typos in the definition of $c_1$ and $c_3$; $v_{t_0}^y$ appearing in $c_1$ and $c_3$ should be $v_0^y$.}}:
\begin{equation}
\begin{aligned}
        & c_1= \frac{\eta_x \kappa^2}{\eta_y(v_{0}^y)^{\alpha-\beta}},\quad  c_2 =\max\Big\{\frac{4\eta_y \mu L}{\mu+ L},\ \eta_y (\mu+L)\Big\},\quad c_3 = 4(\mu+L) \Big(\frac{1}{\mu^2} + \frac{\eta_y}{(v_{0}^y)^\beta} \Big)c_2^{1/\beta},\\
        & c_4= (\mu+L)\Big( \frac{2\kappa^2}{(v_0^y)^\alpha}+ \frac{(\mu+L)\kappa^2}{\eta_y \mu L}\Big),\quad c_5 = c_3 + \frac{\eta_y v_0^y}{(v_0^y)^\beta}+ \frac{\eta_y c_2^{\frac{1-\beta}{\beta}}}{1-\beta}.
\end{aligned}
\end{equation}
Because $C_1$ and $C_2$ are monotonically increasing in $\{c_i\}_{i=1}^5$ and \saa{they have complicated forms to compute exact quantities, we compute a lower bound for each $c_i$, $i=1,\ldots,5$ and use these bounds to further bound $C_1$ and $C_2$ from below. This will allow us to provide a lower bound on \texttt{TiAda} complexity results in terms of its dependence on $L,\mu$ and $\kappa$. Initialization of \tiada{} requires setting six parameters: $\eta_x,\eta_y,v_0^x,v_0^y>0$ and $\alpha,\beta\in(0,1)$ such that $\alpha>\beta$. Since the problem parameters $L,\mu$ and $\kappa$ are unknown, we treat all these parameters as $O(1)$ constants.} Indeed,
\begin{enumerate}
    \item[(i)] consider $c_1$, treating $\frac{\eta_x}{\eta_y(v_{0}^y)^{\alpha-\beta}}$ as $O(1)$ constant, we have $c_1= \Theta(\kappa^2)$;
    \item[(ii)] for $c_2$, we similarly treat $\eta_y$ as $O(1)$ constant, and get $c_2= \Theta(L)$;
    \item[(iii)] for $c_3$, since $c_2= \Theta(L)$ and we treat $ \frac{\eta_y}{(v_{0}^y)^\beta}$ as $O(1)$ constant, we get $c_3= \Theta(\frac{\kappa L^{1/\beta}}{\mu})$;
    \item[(iv)]  similarly, we have the bounds: $c_4= \Theta(\kappa^3)$ and $c_5= \Theta(\frac{\kappa L^{1/\beta}}{\mu})$. 
\end{enumerate}
Next, we 
\saa{derive} lower bound for $C_1$ and $C_2$ \saa{in terms of $L,\mu$ and $\kappa$}. We first 
\saa{focus on} $C_1$. Indeed, \saa{it follows from the definition of $C_1$ in \cite[Theorem C.1]{li2022tiada} that} 
\begin{equation}
    C_1 = \Omega\Big( 
    (c_1c_4)^{\frac{1}{\alpha-\beta}}\mathbf{1}_{2\alpha-\beta<1}+ (c_1c_4)^{\frac{2}{1-\alpha}}\mathbf{1}_{2\alpha-\beta\geq 1} \Big)
\end{equation}
where \saa{$\mathbf{1}_{2\alpha-\beta<1}= 1$ if $2\alpha-\beta<1$; otherwise, it is $0$, and we define $\mathbf{1}_{2\alpha-\beta\geq 1}=1-\mathbf{1}_{2\alpha-\beta<1}$.} 
\saa{Although $\alpha,\beta\in (0,1)$, it is necessary to consider their effects when they appear as exponents.}
Therefore, we consider the following two cases:
\begin{enumerate}
    \item when $2\alpha-\beta \geq 1$, since $\beta>0$, we can conclude that $\alpha > \frac{1}{2}$. \saa{As a result, $(c_1c_4)^{\frac{2}{1-\alpha}}\geq(c_1c_4)^4=\Theta(\kappa^{20})$; therefore,} 
    \begin{equation*}
        C_1 
        \saa{=\Omega\Big(\kappa^{20}\Big);}
    \end{equation*}
    \item when $2\alpha-\beta  < 1$, it implies $\alpha-\beta<\frac{1-\beta}{2}<\frac{1}{2}$. As a result, we can 
    \saa{bound $C_1$ from below as follows: $(c_1c_4)^{\frac{1}{\alpha-\beta}}\geq (c_1c_2)^2=\Theta(\kappa^{10})$; therefore,}
    \begin{equation*}
        C_1 =\Omega\Big(\kappa^{10}\Big).
    \end{equation*}
\end{enumerate}
\saa{Therefore, using \eqref{eq:x-complexity} and $C\geq C_1$, we can conclude that for any $\epsilon>0$, the gradient complexity of \tiada{} for $\sum_{k=0}^{K-1}\norm{\grad_x f(x^k,y^k)}^2\leq \epsilon^2$ to hold is at least $\cO(\kappa^{10}\epsilon^{-2})$. Therefore, for the deterministic WCSC minimax problems, the gradient complexity of \tiada{} to compute an $\epsilon$-stationary point in terms of metric \textbf{(M2)} is at least $\cO(\kappa^{10}\epsilon^{-2})$.}}

{Next, we will analyze the oracle complexity of \tiada{} for the stochastic case. Indeed, the analysis in  \cite[Theorem C.2]{li2022tiada} implies that
$
\mathbf{E}\left[\frac{1}{K} \sum_{k=0}^{K-1}\left\|\nabla_x f\left(x^k, y^k\right)\right\|^2\right]
\leq 
\saa{\tilde C_x(K)}$ \saa{holds for all $K\geq 1$,}
where \saa{$\tilde C_x(K)$ is a monotonically decreasing function of $K$ such that} 
\begin{equation}\label{eq:tilde-C-bound}
    \saa{\tilde C_x(K)} \geq\frac{\widehat{L}^2 G^2\left(\eta_x\right)^2}{\mu \eta_y\left(v_0^y\right)^{2 \alpha-\beta}} \frac{2 L \kappa\left(\eta_x\right)^2}{(1-\alpha) \eta_y\left(v_0^y\right)^{\alpha-\beta}} \frac{G^{2(1-\alpha)}}{K^\alpha} + \frac{2 L \kappa \eta_y G^{2(1-\beta)}}{(1-\beta) K^\beta},
\end{equation}
where $\hat{L} = \frac{\bar L+\bar L\kappa}{\mu}+ \frac{L(\bar L+\bar L\kappa)}{\mu^2}= \Theta(\frac{\kappa^2 \bar L}{\mu})$, and it is assumed that there exists a constant $\bar L>0$ such that for any $x_1,x_2\in \cX$ and $y_1,y_2\in Y$, 
\begin{equation*}
    \begin{aligned}
    \norm{\grad^2_{xy} f(x_1,y_1)- \grad^2_{xy} f(x_2,y_2)}\leq \bar L (\norm{x_1-x_2}+ \norm{y_1-y_2}),\\
    \norm{\grad^2_{yy} f(x_1,y_1)- \grad^2_{yy} f(x_2,y_2)}\leq \bar L (\norm{x_1-x_2}+ \norm{y_1-y_2}).
    \end{aligned}
\end{equation*}
The bound in \saa{\cref{eq:tilde-C-bound} implies that for $\frac{1}{K}\sum_{k=0}^{K-1}\norm{\grad_x f(x^k,y^k)}^2\leq \epsilon^2$ to hold, it is necessary that}
\begin{equation}\label{eq:T-bound-sto-tiada}
\begin{aligned}
      \saa{K} &\geq   \Big(\frac{\widehat{L}^2 G^2\left(\eta_x\right)^2}{\mu \eta_y\left(v_0^y\right)^{2 \alpha-\beta}} \frac{\saa{4} L \kappa\left(\eta_x\right)^2}{(1-\alpha) \eta_y\left(v_0^y\right)^{\alpha-\beta}} \frac{G^{2(1-\alpha)}}{\epsilon^{2}}\Big)^{\frac{1}{\alpha}}+
   \Big(\frac{\saa{4L} \kappa \eta_y G^{2(1-\beta)}}{(1-\beta)}\Big)^{\frac{1}{\beta}} \\
       & = \Omega\Big((\kappa^{6}\bar L^2 \saa{G^{2(2-\alpha)}}\mu^{-2}\epsilon^{-2})^{\frac{1}{\alpha}}+ (\kappa L\saa{G^{2(1-\beta)}}\epsilon^{-2})^{\frac{1}{\beta}}\Big).
\end{aligned}
\end{equation}
As explained in \citep{li2022tiada}, \saa{the best oracle complexity of \tiada{} in terms of $\epsilon$ dependence is $\cO\Big(\epsilon^{-(4+\delta)}\Big)$, which holds for any small $\delta>0$. To achieve this complexity under the condition $\alpha>\beta$, one has to choose $\alpha$ and $\beta$ as follows:} $\alpha = 0.5 + \delta/(8+2\delta)$ and $\beta = 0.5 - \delta/(8+2\delta)$. As a result, \cref{eq:T-bound-sto-tiada} implies that 
\begin{equation}
    \begin{aligned}
        K = \saa{\Omega}\Big( G^{6-\frac{4\delta}{2+\delta}}\kappa^{12-\frac{6\delta}{2+\delta}}\bar L^{4-\frac{2\delta}{2+\delta}} \mu^{-4+\frac{2\delta}{2+\delta}} \epsilon^{-4+\frac{2\delta}{2+\delta}} + G^{2+\delta}\kappa^{(2+\frac{\delta}{2})} L^{(2+\frac{\delta}{2})} \epsilon^{-(4+\delta)}\Big).
    \end{aligned}
\end{equation}
Furthermore, \cite[Theorem C.2]{li2022tiada} also implies that 
$
 \mathbf{E}\left[\frac{1}{K} \sum_{t=0}^{K-1}\left\|\nabla_y f\left(x_t, y_t\right)\right\|^2\right] 
\leq  \saa{\tilde{C}_y(K)}
$
\saa{holds for all $K\geq 1$, $\tilde C_y(K)$ is a monotonically decreasing function of $K$, which can be bounded from below as follows:}
$$\saa{\tilde{C}_y(K)} \geq \frac{\hat{L}^2 G^2 (\eta_x)^2}{\mu \eta_y (v_0^y)^{2\alpha - \beta}} \frac{L (\eta_x)^2 G^{2-2\alpha}}{(1-\alpha)\eta_y (v_0^y)^{\alpha-\beta}T^{\alpha}}\geq \Theta(\frac{\kappa^5 \bar L^2}{\mu^2 T^{\alpha}});$$
hence, for any given $\epsilon>0$, in order for $\frac{1}{K}\sum_{k=0}^{K-1}\norm{\grad_y f(x^k,y^k)}^2\leq \epsilon^2$ to hold, it is necessary to have
$$K=\Omega\Big(G^{6-\frac{4\delta}{2+\delta}}\kappa^{10-\frac{5\delta}{2+\delta}}\bar L^{4-\frac{2\delta}{2+\delta}} \mu^{-4+\frac{2\delta}{2+\delta}} \epsilon^{-4+\frac{2\delta}{2+\delta}}\Big).$$
\saa{Therefore, we can conclude that to guarantee an $\epsilon$-stationary point in metric \textbf{M2}, the oracle complexity of \tiada{} for computing is $\Omega(G^6\bar{L}^4\kappa^{12}\mu^{-4}\epsilon^{-4})$ as $\delta\to 0$.}}
\subsection{Derivation of complexity in \citep{lu2020hybrid}}
The hybrid successive approximation~(\texttt{HiBSA}) algorithm is a single-loop proximal alternating algorithm for solving deterministic nonconvex minimax problems $\min_{\cX}\max_{y\in\cY}\cL(\bx,y)$, where $\cL$ has the same form in~\eqref{eq:main-problem}. \saa{\texttt{HiBSA} uses block coordinate updates for the primal variable that has $N\geq 1$ blocks. Moreover, it employs} constant step sizes for primal and dual updates, and to set these step sizes one needs to know the \mg{L}ipschitz constant $L$ and strongly concavity modulus $\mu$. \saa{Convergence results provided in the paper are in terms of \textbf{(M2)} metric.}

\saa{For the WCSC setting, according to \cite[Lemma 3]{lu2020hybrid}, whenever dual step size $\rho$ and primal step size $\frac{1}{\beta}$ are chosen satisfying the following two conditions: $0< \rho < \frac{\mu}{4L^2}$ and $\beta > L^2(\frac{2}{\mu^2 \rho}+\frac{\rho}{2})+\frac{L}{2}-{L}$, it is guaranteed that the \texttt{HiBSA} iterate sequence satisfies the following bound for all $k\geq 0$:
\begin{equation}
    \begin{aligned}
         c_1\norm{y^{k+1}-y^k}^2 + c_2 \sum_{i=1}^N \norm{x_i^{k+1}-x_i^k}^2 \leq \cP^k - \cP^{k+1},
    \end{aligned}
\end{equation}
where $c_1\triangleq 4(\frac{1}{\rho}- \frac{L^2}{2\mu})-\frac{7}{2\rho}>0$ and $c_2\triangleq \beta + {L}-\frac{L}{2}-L^2(\frac{2}{\mu^2 \rho}+\frac{\rho}{2})>0$, and 
$\cP^{k}\triangleq \cL(x^{k},y^{k})+\Big(\frac{2}{\rho^2 \mu}+\frac{1}{2\rho} - 4(\frac{1}{\rho} - \frac{L^2}{2\mu}) \Big)\norm{y^{k}-y^{k-1}}^2$, which is the value of the potential function at iteration $k\geq 0$ --here, $c_1,c_2>0$ follows from the conditions on $\rho$ and $\beta$. Furthermore, from the proof of \cite[Theorem 1]{lu2020hybrid}, for all $k\geq 0$ we get}
\begin{equation}
    \begin{aligned}
       &\norm{G_{x_i}(x^k,y^k)}\leq (\beta+ 2L) \norm{x_i^{k+1}-x_i^k},\quad i=1,\ldots,N,\\
       &\norm{G_y(x^k,y^k)}\leq L\norm{x^{k+1}-x^k}  + \frac{1}{\rho}\norm{y^{k+1}-y^k},
    \end{aligned}
\end{equation}
\saa{which immediately implies that}
\begin{equation}
    \begin{aligned}
        \norm{G(x^k,y^k)}^2 &\leq \Big((\beta +2L)^2 + 2L^2 \Big) \norm{x^{k+1}-x^k}^2 + \frac{2}{\rho^2}\norm{y^{k+1}-y^k}^2 \leq \frac{\cP^{k}-\cP^{k+1}}{d_1},\quad \forall~k\geq 0,
    \end{aligned}
\end{equation}
where \saa{$d_1 \triangleq \min \{ 4(\frac{1}{\rho}-\frac{L^2}{2\mu})-\frac{7}{2\rho},~\beta +{L} - \frac{L}{2}-L^2(\frac{2}{\mu^2 \rho} + \frac{\rho}{2})\}/ \max \{\frac{2}{\rho^2},~(\beta +2L)^2 + 2L^2\}$ --note that $d_1>0$ due to conditions on $\rho$ and $\beta$. Therefore, for any $K\geq 1$, one has 
\begin{equation}
\label{eq:hibsa-bound}
    \frac{1}{K}\sum_{k=1}^K\norm{G(x^k,y^k)}^2\leq \frac{1}{K}\cdot\frac{P^1-\underline{\cL}}{d_1}\leq\frac{1}{K}\cdot\frac{\cL(x^1,y^1)-\underline{\cL}+c_3\cD_y^2}{d_1},
\end{equation}
where $c_3\triangleq \frac{2}{\rho^2 \mu}+\frac{1}{2\rho} - 4(\frac{1}{\rho} - \frac{L^2}{2\mu})$, $\cD_y$ is the diameter of the dual domain, and it is assumed that there exists $\underline{\cL}>-\infty$ such that $\cL(x,y)\geq \underline{\cL}$ for all $x\in\cX$ and $y\in\cY$. Note that $c_3>\frac{2}{\rho^2\mu}>32 L\kappa^3$, where we used $0<\rho<\frac{\mu}{4L^2}$; moreover, we also have $\beta>\frac{2L^2}{\mu^2\rho}-
{\frac{L}{2}}>\frac{2\kappa^2}{\rho}-
{L}$; therefore, 
{\begin{eqnarray*}
   \frac{1}{d_1}\geq \frac{(\beta +2L)^2 + 2L^2}{4(\frac{1}{\rho}-\frac{L^2}{2\mu})}&\geq& \frac{\kappa^2}{2}\cdot\frac{\beta^2}{\frac{2\kappa^2}{\rho}-L\kappa^3}\\
   &\geq& \frac{\kappa^2\beta}{2}\cdot\frac{\frac{2\kappa^2}{\rho}-
   {L}}{\frac{2\kappa^2}{\rho}-L\kappa^3}\geq \frac{\kappa^2\beta}{2}\geq \frac{\kappa^2}{2}\Big(\frac{2\kappa^2}{\rho}-
 {L}\Big)\geq{\kappa^2}\Big(4\kappa^3 L-\frac{
   {L}}{2}\Big)=\Omega(L\kappa^5).
\end{eqnarray*}%
}
Thus, combining this bound with \eqref{eq:hibsa-bound} and using $c_3>32 L\kappa^3$, we get $\frac{\cL(x^1,y^1)-\underline{\cL}+c_3\cD_y^2}{d_1}=\Omega(L^2\kappa^8\cD_y^2)$.
This result implies that the total complexity of HiBSA to compute an $\epsilon$-stationary point for WCSC minimax problems is at least $\cO(\kappa^8 L^2\epsilon^{-2})$.}

\subsection{Derivation of complexity in \citep{zhang2020single}}
\xz{In this paper, the objective function is defined as 
$
\min_{x\in \cX} \max_{y\in \cY}  g(x)+f(x,y)-h(y) 
$,
where $g(\cdot)\triangleq\mathds{1}_{X}(\cdot)$ and $h(\cdot)\triangleq\mathds{1}_{Y}(\cdot)$ denote the indicator functions of the closed convex sets $X\subset \cX$ and $Y\subset \cY$. \cite{zhang2020single} adopt the metric \textbf{(M4)} 
\mg{for characterizing} 
$\epsilon$-stationarity; \mg{the metric \textbf{(M4)} and its connections to some other convergence metrics is discussed in Section \ref{sec:metrics}}. 
\saa{\mg{Basically, }
for given $\epsilon>0$, $(x_\epsilon,y_\epsilon)$ is $\epsilon$-stationary \mg{(in the sense of the \textbf{(M4)} metric)} if
\begin{equation}
    \label{eq:eps-stationary-m4}
    \exists u \in \grad_x f(x_\epsilon,y_\epsilon)+\partial g(x_\epsilon),\quad \exists v\in -\grad_y f(x_\epsilon,y_\epsilon) + \partial h(y_\epsilon)\; :\quad
    \max\{\|u\|,~\|v\|\}\leq \epsilon.
\end{equation}}}%
For convenience of 
\saa{the reader}, in this section we use the notation from \citep{zhang2020single}. Following the \saa{parameter} conditions in \cite[Theorem 3.4]{zhang2020single}, \saa{we consider selecting the algorithm parameters $p$ and $c$ such that $p = 4L$ and $c = \frac{1}{6L}$, where $L$ denotes the Lipschitz constant of $\grad f$. This choice implies the following bounds on other parameters showing up in the proof of \cite[Theorem 3.4]{zhang2020single}:} 
\begin{enumerate}
    \item \cite[Theorem 3.4]{zhang2020single}, $\alpha = \cO(1/L),\;\beta = \cO\Big(\min\Big\{\frac{1}{\sqrt{\saa{K}}}, \frac{1}{L}\Big\}\Big)$ --see also \cite[page 29]{zhang2020single};
    \item  \cite[Lemma B.2, page 19]{zhang2020single}, $\sigma_1=\cO(1)$, $\sigma_2=\cO(1)$, $\sigma_3 = \cO(1)$;
    \item  \cite[Lemma B.9, page 23]{zhang2020single}, $\kappa=\cO(1)$;\footnote{\saa{Unlike the notation in our paper, $\kappa$ does not denote the condition number $L/\mu$ in \citep{zhang2020single}, it is just some $\cO(1)$ constant.}}
    \item \cite[Lemma B.11, page 26]{zhang2020single}, $\bar \lambda = \cO(L)$;
    \item  \cite[eq.~(B.63), page 29]{zhang2020single}, $\lambda_1 = \cO(\cD_y),\;\lambda_2 = \cO(L\cD_y^2),\;\lambda_3 = \cO(\cD_y^2)$.
\end{enumerate}
\saa{For any given $K\in\integers_{++}$, the proof of \cite[Theorem 3.4, \saa{page 29}]{zhang2020single} considers two exhaustive cases where $\beta$ is set to $\frac{1}{\sqrt{K}}$: if \textbf{case 1} holds,} then \mg{the bound},
\begin{equation*}
\begin{aligned}
\max \left\{\left\|x^k-x^{k+1}\right\|^2,\left\|y^k-y_{+}^k\left(z^k\right)\right\|^2,\left\|z^k-x^{k+1}\right\|^2\right\} \leq \max \left\{\lambda_2 \beta^2, \lambda_1^2 \beta^2, \lambda_3 \beta\right\}\leq B_1
\end{aligned}
\end{equation*}
holds for some \saa{$k_* \leq K$}, where $B_1=\cO\Big(\cD_y^2\cdot
\saa{\frac{\max\{L,~1\}}{K}}
\Big)$. \mg{Furthermore,} according to \cite[Lemma B.11, page 26]{zhang2020single}, this point \saa{$(x^{k_*},y^{k_*})$ is an $\bar\lambda\sqrt{B_1}$-stationary; hence, $(x^{k_*},y^{k_*})$ is $\cO\Big(L\cD_y\cdot\sqrt{\frac{\max\{L,1\}}{K}}\Big)$-stationary since $\bar\lambda=\cO(L)$.} 
\saa{Therefore, $K \geq \cO\Big(\frac{L^4\cD_y^4}{\epsilon^4}\Big)$ \mg{iterations} is required to guarantee $\epsilon$-stationary point for this case.} \saa{On the other hand, if \textbf{case 2} holds, then there exists $k_*\leq K$ such that $(x^{k_*},y^{k_*})$ is $\sqrt{(\phi_0-\underline{f})/(K\beta)}$-stationary; hence, $(x^{k_*},y^{k_*})$ is $\sqrt{(\phi_0-\underline{f})/\sqrt{K}}$-stationary since $\beta=1/\sqrt{K}$, where $\underline{f}\in\reals$ such that $\underline{f}\leq \max_{y\in Y}f(x,y)$ for all $x$, and $\phi^0=F(x^0,z^0;y^0)-2d(y^0,z^0)+2P(z^0)$ for $F(x,z;y)\triangleq f(x,y)+\frac{p}{2}\norm{x-z}^2$, $d(y,z)\triangleq \min_{x\in X}F(x,z;y)$ and $P(z)=\min_{x\in X}\max_{y\in Y}F(x,z;y)$. Therefore, $K \geq \cO\Big(\frac{(\phi_0-\underline{f})^2}{\epsilon^4}\Big)$ is required to guarantee $\epsilon$-stationary point for this case. Thus, combining the two exhaustive cases, we conclude that $K\geq \cO\Big(\frac{L^4\cD_y^4+(\phi_0-\underline{f})^2}{\epsilon^4}\Big)$ to ensure that the algorithm can compute an $\epsilon$-stationary point within $K$ iterations.}

\subsection{Derivation of complexity in \citep{dvurechensky2017gradient}}
In this paper, a special case of \eqref{eq:main-problem} \saa{is considered when} $h(\cdot)=0$ and \saa{the primal-dual coupling is \textit{bilinear} of the form} $f(x,y)=s(x)-r(y)+\fprod{Ax,y}$ such that $s(\cdot)$ is $L_s$-smooth (possibly nonconvex) and $r(\cdot)$ is $L_r$-smooth and $\mu$-strongly convex. \saa{An adaptive triple-loop method is proposed; in the outer loop the method takes inexact proximal gradient steps for the primal function using the inexact oracle values satisfying the backtracking condition, in the middle loop backtracking steps are implemented for the primal function, and each backtracking step and checking the backtracking condition require solving $\max_{y\in Y} f(x^k,y)$ inexactly. Let $F(\cdot)=g(\cdot)+\Phi(\cdot)$ denote the primal function, where $\Phi(\cdot)=\max_{y\in Y}f(\cdot,y)$, and $F^*=\min_{x\in\cX}F(x)$.} 
From the discussion in
\cite[Corollary 1]{dvurechensky2017gradient}, \saa{the 
number of times the \mg{backtracking} 
condition checked is bounded by
$\cO\big(L(F(x^0)-F^*)\epsilon^{-2}\big)$,} where $L = L_v+L_s$ and $L_v=\frac{\|A\|^2}{\mu}$. Furthermore, according to \cite[Algorithm 1]{dvurechensky2017gradient} and the discussion in \cite[eq.(21)-(32)]{dvurechensky2017gradient}, given $x^k$ corresponding to the $k$-th outer iteration, computing the inexact oracle \saa{for the pair $(\Phi(x^k), \grad \Phi(x^k))$, in the sense of \cite[\mg{Definition} 1 on page 4]{dvurechensky2017gradient},
which is $(f(x^k,\tilde y^k),~\grad s(x^k)+A^\top \tilde y^k)$. This step requires to compute $\tilde y^k$ such that}
\begin{equation}
    \Phi(x^k) - f(x^k,\tilde y^k)=\cO(\saa{\epsilon^2}/M_k), 
\end{equation} 
which requires $\cO(\sqrt{L_r/\mu}\log(\saa{L/\epsilon^2}))$ \saa{projected gradient iterations (that involve computing projections onto $Y$ and computing $\grad_y f(x_k,\cdot)$)} by employing an accelerated backtracking method for strongly convex problems~{\citep{calatroni2019backtracking,rebegoldi2022scaled}}, \saa{where $M_k$ is an estimate of $L$, and according to \cite[Corollary 2]{dvurechensky2017gradient}, $M_k\leq 2 L$ for all $k\geq 0$. After each computation of $\tilde y^k$, a projected inexact gradient step using the inexact oracle is computed from the point $x^k$ to compute a candidate primal point $w^k$, and the backtracking condition is checked. If it holds for $w^k$, then $x^{k+1}\gets w^k$; otherwise, $M_k\gets 2 M_k$ and $\tilde y^k$ is computed again.} 
Thus, 
the total \saa{gradient complexity, i.e., number of times $\grad f$ is computed, to compute an $\epsilon$-stationary point in the sense of metric \textbf{(M1)} is} $\cO(\sqrt{\frac{L_r}{\mu}}\Big(L_s+\frac{\|A\|^2}{\mu}\Big)\saa{(F(x^0)-F^*)}\epsilon^{-2}\log(\epsilon^{-1}) )$. Indeed, \saa{since the Lipschitz constant of $\grad f$ satisfies $L\geq \max\{L_r,~L_s,~\|A\|\}$}, 
the total gradient complexity can be simplified as $\cO(L\kappa^{1.5}\epsilon^{-2}\log(\epsilon^{-1}))$.

\subsection{Derivation of complexity in~\cite{xu2023unified}}
\rv{The authors of~\cite{xu2023unified} consider $\min_{x\in \cX}\max_{y\in \cY}\cL(x,y)=\mathds{1}_{X}(x)+f(x,y)-\mathds{1}_{Y}(y)$, where $X$ and $Y$ are closed convex sets, and $f$ is a smooth function such that $f$ is strongly concave in $y$. It is shown that an alternating GDA method is guaranteed to generate an $\epsilon$-stationary point in terms of metric \textbf{(M2)} within $\frac{F_0-\underline{F}}{d_1\epsilon^2}$ iterations whenever the primal ($\tau$) and dual ($\sigma$) step sizes satisfy $\tau<\min\{\frac{1}{L},~\frac{1}{L^2\sigma+4\kappa^2/\sigma}\}$ and $\sigma\leq \frac{1}{4L\kappa}$, where $d_1=\min\{\frac{1}{2\tau}-\frac{\sigma L^2}{2}-\frac{2\kappa^2}{\sigma},~\frac{3\mu-\sigma L^2}{2}+\frac{\mu-4\sigma L^2}{2\sigma\mu}\}/\max\{\frac{1}{\tau^2}+2L^2,~\frac{2}{\sigma^2}\}$, $F_0=\cO(\cL(x_0,y_0))$, $\underline{F}=\cL^*-(\mu+\frac{7}{2\sigma}-\frac{\sigma L^2}{2}-2L\kappa)\cD_y^2$ and $\cL^*=\inf_{x\in X, y\in Y}\cL(x,y)$. Thus, choosing $\tau=\cO(\frac{1}{L\kappa^3})$ and $\sigma=\Theta(\frac{1}{L\kappa})$ implies that $\frac{1}{d_1}=\Omega(L\kappa^5)$ and $F_0-\underline{F}=\Omega(\cL(x_0,y_0)-\cL^*+L\kappa\cD_y^2)$; therefore, the upper complexity provided by the authors can be bounded from below as follows:
\begin{align*}
    \frac{F_0-\underline{F}}{d_1\epsilon^2}=\Omega\Big(L\kappa^5(\cL(x_0,y_0)-\cL^*+L\kappa\cD_y^2)\epsilon^{-2}\Big).
\end{align*}
For $\cO(L^4\cD_y^4/\epsilon^4)$ complexity bound for WCMC setting, see \cite[Theorem 3.2]{xu2023unified}.} 

\section{Discussion on the Convergence Metric\mg{s}}
\label{sec:metrics}
In this paper we 
\saa{adopt \textbf{(M2)}, also defined} in \cref{Def:stationary} below, as our primary measure of convergence for both the WCSC and WCMC cases.
This choice is motivated by the fact that \textbf{(M2)} \saa{is easy to check in practice.} Also note that in the absence of 
\saa{closed convex functions $g$ and $h$, $\norm{G(x,y)}\leq \epsilon$ reduces to $\norm{\grad f(x,y)}\leq \epsilon$.} 
Prior to 
\saa{discussing} 
various forms of stationarity and their 
\saa{connection with each other, we give some preliminary definitions and notation.} 
\mg{We} \saa{first} recall that
    $\Phi(x) \triangleq
    \max_{y\in\cY} f(x,y) - h(y)$, and \saa{the primal function is given by} $F(\cdot) \triangleq g(\cdot) + \Phi(\cdot)$.
\mg{Next}, we 
\mg{provide} the definition 
of \mg{the} Moreau envelope \saa{of a real-valued function.}
\begin{definition}\label{Def: Moreau envelope}
Let \saa{$r:\mathbb{R}^n\rightarrow\mathbb{R}\cup\{+\infty\}$} be $L$-weakly convex. Then, for any  $\lambda\in(0,L^{-1})$, \mg{the} Moreau envelope of $r$ is defined as \mg{the function} $r_{\lambda}:\mathbb{R}^n\rightarrow\mathbb{R}$ such that $r_{\lambda}(x)\triangleq\min_{w\in\cX} r(w) + \frac{1}{2\lambda}\|w-x \|^2$. 
\end{definition}%

\saa{In the rest of this section, we make the following assumption.}
\begin{assumption}\label{aspt:weakly-convex}
    \saa{There exists $m\geq 0$ such that for all} $y\in\saa{\dom h}$, the function $f(\cdot,y)$ is $m$-weakly convex, \saa{i.e.,} $f(\cdot,y) + \frac{m}{2}\|\cdot\|^2$ is convex.
\end{assumption}
\begin{remark}
Under Assumption \ref{ass1}, $f(\cdot,y)$ is naturally $L$-weakly convex; 
\saa{that said,} it is possible that $f$ satisfies \mg{Assumption} 
\ref{aspt:weakly-convex} with \saa{$m\ll L$}. 
\end{remark}%
\saa{Under Assumption~\ref{aspt:weakly-convex}, the primal function $F$ is $m$-weakly convex; hence,} 
\saa{for any $\lambda\in(0,1/m)$, $F_\lambda$} is well-defined for both WCSC and WCMC cases and it \mg{is differentiable such that} $\grad F_\lambda(x) = \frac{1}{\lambda}(x-\prox{\lambda F}(x))$ \saa{--this is a standard result, for details see, e.g., \cite[Lemma 2]{zhang2022sapd+}. Finally, we recall the gradient map.}
\begin{definition}
    For given $\eta_x,\eta_y>0$, the gradient map $G(x,y)=[G_x(x,y)^\top, G_y(x,y)^\top]^\top$, and $G_x(x,y)\triangleq(x-\prox{\eta_x g}(x-\eta_x\grad_x f(x,y)))/\eta_x$ and $G_y(x,y)\triangleq(\prox{\eta_y h}(y+\eta_y\grad_y f(x,y))-y)/\eta_y$.
\end{definition}
Now, we are ready to state 
\saa{various important} types of stationarity \saa{adopted} in the literature \saa{on nonconvex minimax problems}.
\begin{definition}\label{Def:stationary}
Given $\epsilon>0$, consider the minimax problem in~\eqref{eq:main-problem}. 
\begin{enumerate}[align=left]
    \item[\textbf{(M1):}] $x\in\cX$ is $\epsilon$-stationary in \textbf{(M1)} if $\|\cG(x)\|\leq \epsilon$ for some $\alpha>0$, where the gradient map of $\Phi$ is defined as $\cG(x) = \big(x-\prox{\alpha g}\big(x-\saa{\alpha}\grad \Phi(x)\big)\big)/\alpha$.
    \item[\textbf{(M2):}]  $(x,y)\in\cX\times\cY$ is $\epsilon$-stationary in \textbf{(M2)} if 
    $\|G(x,y)\|\leq \epsilon$ for some $\eta_x,\eta_y>0$.
    \item[\textbf{(M3):}] $x\in\cX$ is $\epsilon$-stationary in \textbf{(M3)} if $\|\grad F_{\lambda}(x)\|\leq \epsilon$ \saa{for some $\lambda\in (0,1/m)$}.
    \item[\textbf{(M4):}] $(x,y)\in\cX\times\cY$ is $(\epsilon_x,\epsilon_y)$-stationary in \textbf{(M4)} if $\exists u \in \grad_x f(x,y)+ \saa{\partial} g(x)$ and $\exists v\in \grad_y f(x,y) - \saa{\partial} h(y)$ 
    such that $\|u\|\leq\epsilon_x$ and  $\|v\|\leq \epsilon_y$.
    \item[\textbf{(M5):}] $x\in\cX$ is $\epsilon$-stationary in \textbf{(M5)} if $\exists \hat{x} \in \cX$ such that $\inf_{\|d\| \leq 1} F^{\prime}(\hat{x} ; d) \geq-\epsilon$ and $\|\hat{x}-x\| \leq \epsilon$, where $F'(\hat{x};d)\triangleq\lim_{t\rightarrow 0} \frac{F(\hat{x}+t\cdot d) - F(\hat{x})}{t}$.
    \item[\textbf{(M6):}] for the case $g(\cdot)=0$, $(x,y)\in\cX\times\cY$ is $\epsilon$-stationary in \textbf{(M6)} if for some $\eta_y>0$, $\|\grad_x f(x,\hat y)\|\leq \epsilon$ and $\|G_y(x,y)\|\leq \epsilon$ where $\hat{y} = \prox{\eta_y h}(y+\eta_y \grad_y f(x,y))$.
\end{enumerate}
\end{definition}  
\begin{remark}
    \saa{The metric \textbf{(M1)} requires differentiability of $\Phi$, which holds when $f(x,\cdot)$ is strongly concave for any given $x\in\cX$. Hence, this metric applies to WCSC setting, not to the WCMC setting. 
    Moreover, for any given $\epsilon>0$, \textbf{(M2)} and \textbf{(M4)} are equivalent when $g(\cdot)=h(\cdot)=0$ and $\epsilon_x=\epsilon_y=\epsilon$.}
\end{remark}
{The next result shows that the guarantees in 
\textbf{(M2)} can be converted into guarantees in \textbf{(M4)}.}
\begin{theorem}\label{thm:m2-to-m4}
       Suppose Assumption~\ref{ass1} 
       holds and $(x,y)$ is $\epsilon$-stationary in \textbf{(M2)} for some $\eta_x,\eta_y> 0$, then $(\hat{x},\hat{y})$ is $(\epsilon_x,\epsilon_y)$-stationary in \textbf{(M4)}, where  \saa{$\hat{x}=\prox{\eta_x g}(x-\eta_x \grad_x f(x,y))$} 
       and \saa{$\hat{y} =  \prox{\eta_y h}(y + \eta_y \grad_y f(x,y))$} 
       for \saa{$\epsilon_x=\epsilon_y=\big(1+(\eta_x +\eta_y)L\big)\epsilon$.}
\end{theorem}
\begin{proof}
\mg{Note that} by the optimality conditions \mg{corresponding to the} 
proximal operator, we have
    \begin{equation}
        G_x(x,y) = (x-\hat{x})/\eta_x \in \saa{\grad_x} f(x,y) + \partial g(\hat{x}),\quad G_y(x,y) = (\hat{y}-y)/\eta_y\in \grad_y f(x,y) - \partial h(\hat{y}). \label{eq-opt-cond}
        \end{equation} 
\saa{Since $(x,y)$ is $\epsilon$-stationary in \textbf{(M2)}, \eqref{eq-opt-cond} implies that $\exists u' \in\grad_x f(x,y)+\saa{\partial} g(\hat x)$ and $\exists v'\in \grad_y f(x,y) - \saa{\partial} h(\hat y)$ 
    such that $\max\{\|u'\|,~\|v'\|\}\leq \epsilon$. Let $\hat u\triangleq u'+\grad_x f(\hat{x},\hat{y}) - \grad_x f(x,y)$ and $\hat v\triangleq v'+\grad_y f(\hat{x},\hat{y}) - \grad_y f(x,y)$. Then, it follows from the smoothness of $f$ that}
\begin{equation*}
\begin{aligned}
        \|\hat u\| 
        \leq  \norm{u} + \norm{\grad_x f(\hat{x},\hat{y}) - \grad_x f(x,y)} \leq \epsilon + L\|\hat{x} - x\| +   L\|\hat{y} - y\|\leq (1 + \eta_x L + \eta_y L)\epsilon.
\end{aligned}
\end{equation*}
In a similar manner, we can show that $\norm{\hat v}\leq \big(1 + (\eta_x+\eta_y) L\big)\epsilon$, establishing the desired result.
\end{proof}

{
Next we provide some results from the literature describing how guarantees in different 
$\epsilon$-stationarity metrics can be transferred from one to another. In particular, besides its relation to the \textbf{(M4)} metric we studied in Theorem \ref{thm:m2-to-m4}, our results in \textbf{(M2)}, can be transferred to any of the other metrics mentioned in Definition~\ref{Def:stationary}. 
}
\begin{theorem}
Suppose Assumptions \ref{ass1} and \ref{aspt:weakly-convex} hold. 
    \begin{enumerate}
        \item[\textbf{i.}] \cite[Theorem 2]{zhang2022sapd+} Given $\lambda \in (0,1/m)$ and $\epsilon>0$, under Assumption~\ref{assumption:noise}, if $x\in\cX$ is $\epsilon$-stationary in expectation with respect to \textbf{(M3)}, i.e., $\mathbf{E}[\|\grad F_\lambda(x)\|]\leq \epsilon$, then an $\cO(\epsilon)$-stationary point $\hat{x}\in\cX$ in \textbf{(M1)} \saa{with $\alpha=\lambda$} can be computed within $\tilde\cO(\epsilon^{-2})$ stochastic oracle calls. 
        \item[\textbf{ii.}] \cite[Proposition 1]{kong2021accelerated} If $(x,y)\in\cX\times\cY$ is $(\epsilon_x,\epsilon_y)$-stationary in \textbf{(M4)}, then $\hat x\in\cX $ is $\hat{\epsilon}$-stationary in \textbf{(M5)} for $\hat{\epsilon} = \max\Big\{\epsilon_x + 2\sqrt{2m\cD_y\epsilon_y}, \sqrt{\frac{2\cD_y\epsilon_y}{m}}\Big\}$.
        \item[\textbf{iii.}] \cite[Proposition 2]{kong2021accelerated} Given $\lambda \in (0,1/m)$, $\epsilon,\hat\epsilon>0$ such that $\hat\epsilon\leq\frac{\lambda^3 \epsilon}{\lambda^2+2(1-\lambda m)(1+\lambda)}$, if $x$ is $\hat\epsilon$-stationary in \textbf{(M5)}, then $x$ is $\epsilon$-stationary in \textbf{(M3)}. 
        \item[\textbf{iv.}] \cite[Proposition 2]{kong2021accelerated} Given $\lambda \in (0,1/m)$, $\epsilon,\hat\epsilon>0$ such that $\epsilon\leq \hat\epsilon\cdot\min\{1,1/\lambda\}$, if $x$ is  $\epsilon$-stationary in \textbf{(M3)}, then 
        \saa{x} is $\hat\epsilon$-stationary in \textbf{(M5)} \saa{with $\hat x=x_\lambda$}, where $x_\lambda = \argmin_{u\in\cX} F(u) + \frac{1}{2\lambda}\|u-x\|$.
        \item[\textbf{v.}]  \cite[Proposition 4.11]{lin2020gradient} 
        Consider \eqref{eq:main-problem} with $g(\cdot)=0$, $h(\cdot)=\mathds{1}_{Y}(\cdot)$ for some convex set $Y\subset \cY$ \saa{such that $\cD_y<\infty$}, 
        and $\mu>0$. If 
        a point $x\in\cX$ is $\epsilon$-stationary in \textbf{(M1)}, \saa{i.e., $\norm{\grad\Phi(x)}\leq \epsilon$,} then in the sense of \textbf{(M6)} \saa{with $\eta_y = \frac{1}{L}$}, an $\cO(\epsilon)$-stationary point 
        can be computed within
        $\cO(\kappa\log(\epsilon^{-1}))$ oracle calls for the deterministic setting, and an $\cO(\epsilon)$-stationary point in expectation can be computed within $\cO(\epsilon^{-2})$ oracle calls for the stochastic setting under Assumption~\ref{assumption:noise}. Conversely, if a point $(x',y')\in\cX\times\cY$ is $\frac{\saa{\epsilon}}{\kappa}$-stationary in \textbf{(M6)} \saa{with $\eta_y=\frac{1}{L}$}, then $x'$ is $\cO(\epsilon)$-stationary in \textbf{(M1)}.
        \item[\textbf{vi.}] \cite[Proposition 4.12]{lin2020gradient}  
        Consider \eqref{eq:main-problem} with $g(\cdot)=0$, $h(\cdot)=\mathds{1}_{Y}(\cdot)$ for some convex set $Y\subset \cY$ \saa{such that $\cD_y<\infty$}, and $\mu=0$. If
        a point $x\in\cX$ is \saa{$\epsilon$-stationary 
        in \textbf{(M3)} with $\lambda=\frac{1}{2L}$}, then in the sense of \textbf{(M6)} \saa{with $\eta_y = \frac{1}{L}$}, an $\cO(\epsilon)$-stationary point can be computed within
        $\cO(\epsilon^{-2})$ oracle calls for the deterministic setting, and $\cO(\epsilon)$-stationary point in expectation can be computed within $\cO(\epsilon^{-4})$ oracle calls for the stochastic setting under Assumption~\ref{assumption:noise}. Conversely, if a point $(x',y')\in\cX\times\cY$ is  $\frac{\epsilon^2}{L\cD_y}$-stationary in \textbf{(M6)} with $\eta_y=\frac{1}{L}$, then $x'$ is $\cO(\epsilon)$-stationary in \textbf{(M3)} \saa{with $\lambda=\frac{1}{2L}$.}
        \item[\textbf{vii.}] \cite[Proposition 2.1]{yang2022faster} 
        Consider \eqref{eq:main-problem} with $g(\cdot)=h(\cdot)=0$ and $\mu>0$. If a point $x \in \cX$ is $\epsilon$-stationary in \textbf{(M1)}, \saa{i.e., $\norm{\grad\Phi(x)}\leq \epsilon$,} and $\norm{\grad_y f(x,\tilde y)}\leq \tilde\epsilon$ \saa{for some $\tilde y\in\cY$}, then an $\cO(\epsilon)$-stationary point $(x,y)$ in \textbf{(M2)}, i.e., $\norm{\grad f(x,y)}\leq \epsilon$, can be computed within $\cO(\kappa \log(\frac{\kappa \tilde\epsilon}{\epsilon}))$ oracle calls for the deterministic setting, and an $\cO(\epsilon)$-stationary point $(x,y)$ in \textbf{(M2)} (in expectation) can be computed within $\tilde\cO(\kappa + \kappa^3 \sigma^2 \epsilon^{-2})$ stochastic oracle calls under Assumption~\ref{assumption:noise}. Conversely, if a point $(\tilde x, \tilde y)$ satisfies $\|\grad_x f(\tilde x,\tilde y)\|\leq \epsilon$ and $\|\grad_y f(\tilde x, \tilde y)\|\leq \epsilon/\sqrt{\kappa}$, then a point $x\in\cX$ which is $\cO(\epsilon)$-stationary in \textbf{(M1)}, i.e., $\norm{\grad \Phi(x)}\leq \epsilon$, can be computed within $\cO(\kappa \log(\kappa))$ oracle calls for the deterministic setting, and within $\tilde\cO(\kappa + \kappa^5 \sigma^2 \epsilon^{-2})$ oracle calls for the stochastic setting under Assumption~\ref{assumption:noise}.
    \end{enumerate}
\end{theorem}

\end{document}